\documentclass[aip,cha,reprint,onecolumn]{revtex4-1}

\usepackage[pdfstartview=FitH,colorlinks=true,citecolor=blue,linkcolor=blue,urlcolor=blue,filecolor=blue,bookmarks=false]{hyperref}
\usepackage[T1]{fontenc}
\usepackage{ae}
\usepackage{aecompl}

\usepackage{lipsum}
\usepackage{amsfonts}
\usepackage{graphicx}
\usepackage{epstopdf}
\usepackage{algorithmic}
\ifpdf
  \DeclareGraphicsExtensions{.eps,.pdf,.png,.jpg}
\else
  \DeclareGraphicsExtensions{.eps}
\fi

\usepackage[normalem]{ulem}
\usepackage{soul}

\usepackage{xcolor}
\definecolor{green1}{rgb}{0,0.314,0} 
\definecolor{brown1}{rgb}{0.545,0.270,0.0745} 
\definecolor{Saddle}{HTML}{8b4513}
\definecolor{NewGrey}{HTML}{85888D}

\usepackage[normalem]{ulem}

\usepackage{bbm}
\usepackage{mathtools}
\usepackage{amsmath}
\usepackage{amssymb}
\usepackage{amsthm}
\usepackage{graphics}%

\usepackage{overpic}%

\usepackage{stackengine}
\usepackage{csquotes}
\usepackage[normalem]{ulem}

\usepackage{wasysym}
\usepackage{relsize}

\newtheorem{theorem}{Theorem}[section]

\graphicspath{ {./Figures/} }

\newcommand{\cP}{\mathcal{P}}

\renewcommand{\epsilon}{\varepsilon}

\renewcommand{\leq}{\leqslant}
\renewcommand{\geq}{\geqslant}

\usepackage{enumitem}
\setlist[enumerate]{leftmargin=.5in}
\setlist[itemize]{leftmargin=.5in}

\usepackage{amsopn}

\ifpdf
\hypersetup{
  pdftitle={HSCsBifurcation},
  pdfauthor={D.C. De Souza and A.R.  Humphries}
}
\fi

\begin{document}
\title{Dynamics of a Mathematical Hematopoietic Stem-Cell Population Model}
\author{Daniel C. De Souza}
\email[]{daniel.desouza@ed.ac.uk}
\affiliation{Department of Mathematics and Statistics, McGill University,
Montreal, Quebec, H3A 0B9, Canada.}
\affiliation{Currently at Institute of Immunology and Infection Research, University of Edinburgh, Ashworth Labs, Edinburgh, EH9 3JT, Scotland.}

\author{Antony R. Humphries}
\email[]{tony.humphries@mcgill.ca}
\homepage[]{http://www.math.mcgill.ca/humphries/}
\affiliation{Department of Mathematics and Statistics,
McGill University, Montreal, Quebec, H3A 0B9, Canada}
\affiliation{Department of Physiology, McGill University,
Montreal, Quebec, H3G 1Y6, Canada}

\date{\today}

\begin{abstract}
We explore the bifurcations and dynamics of a scalar differential equation with a single constant delay which models the population of human hematopoietic stem cells in the bone marrow. One parameter continuation reveals that with a delay of just a few days,
stable periodic dynamics can be generated of all periods from about one week up to one decade!
The long period orbits seem to be generated by several mechanisms, one of which is a canard explosion, for which we approximate the dynamics near the slow manifold.
Two-parameter continuation reveals parameter regions with even more exotic dynamics including quasi-periodic and phase-locked tori, and chaotic solutions. The
panoply of dynamics that we find in the model
demonstrates that instability in the stem cell dynamics
could be sufficient to generate the
rich behaviour seen in dynamic hematological diseases.
\end{abstract}


\maketitle


\section{Introduction}\label{sec.introduction}

We study the dynamics and bifurcations of the delay-differential equation (DDE)
\begin{equation}\label{Qprime}
Q^{\prime}(t) = -(\kappa+\beta(Q(t)))Q(t)+A\beta(Q(t-\tau))Q(t-\tau),
\end{equation}
where $Q(t)\geq0$ represents the concentration of hematopoietic stem cells (HSCs) in the
bone marrow, $A\in(1,2)$ is the amplification factor for cells undergoing division, $\tau>0$
is the division time, the rate at which cells enter division, $\beta(Q)$, is a monotonically
decreasing function of $Q$ with $\lim_{Q\to\infty}\beta(Q)=0$, and $\kappa>0$ is the rate that
the stem cells differentiate to the progenitors of circulating blood cells.

The DDE~\eqref{Qprime} represents the $G_0$ cell proliferation model
of Burns and Tannock~\cite{Burns_1970}.
A full derivation of the DDE can be found in Mackey and Rudnicki~\cite{Mackey_1994}, though it
was first stated in the form~\eqref{Qprime} in Mackey~\cite{Mackey_1978}.
In Bernard \textit{et al}~\cite{Bernard_2003} equation~\eqref{Qprime}
was used to describe the stem cell dynamics as one component of a larger model describing the regulation of circulating
neutrophil concentrations.
Since then many mathematical models have appeared which contain~\eqref{Qprime}, or a variant, within
larger hematopoiesis models for the production and regulation of neutrophils, erythrocytes and platelets~\cite{Adimy_2006a,Colijn_2005a,Colijn_2005b,Craig_2016,Langlois2017}.
Multiple versions of~\eqref{Qprime} have
also been coupled together to model discrete levels of stem cell maturity~\cite{Adimy_2006c,Qu_2010}.

The human hematopoietic system produces about $10^{11}$ blood cells of various
types per day~\cite{Kaushansky_2016}, of which erythrocytes (red blood cells), neutrophils
(a type of white blood cell) and platelets are the most common, in a production process
which is tightly regulated by a myriad of feedback loops. In dynamical diseases including cyclic neutropenia (CN), cyclic thrombocytopenia (CT) and
periodic chronic myelogenous leukemia (PCML), oscillations are observed
in the circulating concentrations of one or more of the cell lines~\cite{Foley_2009a}.

In mathematical models of hematopoiesis these oscillations arise through Hopf bifurcations as one or more parameters are varied in the model. Two principle mechanisms have been proposed to drive the oscillations in different dynamical diseases~\cite{Bernard_2003}. There can be an instability in the production of the HSCs themselves, with the oscillating HSC numbers then leading to oscillations in the concentrations of peripheral cells. This occurs in PCML where leukemic HSCs typically present a chromosome abnormality~\cite{Pujo_Menjouet_2005}. An
alternative mechanism is that oscillations in one cell line can be created through an abnormality in the production of precursor cells in that cell lineage, with the feedback loops from that lineage causing oscillations in the numbers of HSCs differentiating into other cell lines creating concomitant oscillations in the other cell lineages. This occurs in CN, for which a mutation in the ELANE gene that encodes neutrophil
elastase leads to increased apoptosis in the neutrophil
progenitor cells during mitosis~\cite{Dale_2002}.

In other dynamical diseases it remains an open question whether the oscillations
are driven by an inherent instability in the HSCs, or whether an instability
in production of one of the blood cell
lineages is creating the oscillations seen in circulating concentrations.
The second possibility is difficult to investigate directly, due to the complexity of the
hematopoietic models, with for example the granulopoiesis model of Craig \textit{et al}~\cite{Craig_2016} having five equations, over twenty parameters, and state-dependent delays.
In the current
work, motivated by the first possibility, we investigate the dynamics of the simple HSC model~\eqref{Qprime} as parameters are varied and explore the dynamics that arise. Such an approach alone will not definitively answer the question
of whether or not the oscillations in specific dynamical diseases
are driven by inherent instability in the HSC dynamics.
However, equation~\eqref{Qprime} is often incorporated in more
complicated hematopoietic models, and if the HSCs can oscillate in the decoupled equation~\eqref{Qprime},
these oscillations could drive oscillations in the production rates of the mature circulating blood cells which are all produced from the HSCs. Thus, studying the dynamics of equation~\eqref{Qprime} allows us to determine when instabilities in the HSC dynamics may arise, and hence if it is feasible for these to drive oscillations in the circulating blood cells concentrations.

Although~\eqref{Qprime} has been used and studied in numerous models, the codimension-one bifurcation analysis is incomplete, and little is known about codimension-two bifurcations.
In Section~\ref{sec.decoup.hsc} we review the model~\eqref{Qprime} and its basic dynamical properties including existence and positivity of solutions, non-dimensionalised formulation, homeostasis (the stable state of an organism maintained by physiological processes) parameter values and existence and stability of steady states.
In Section~\ref{sec.stab.bound} we discuss the stability boundary of the steady state with respect to the delay $\tau$.

In Section~\ref{sec.bifurc.hsc} we carry out a
numerical bifurcation analysis of Eq.~\eqref{Qprime} by
performing parameter continuation on solutions as three of the parameters that control the dynamics are varied individually or pairwise. The one parameter continuations reveal sub- and supercritical Hopf bifurcations, fold bifurcations of periodic orbits and period-doubling bifurcations. This results in bistability between a stable periodic orbit and a steady state, and bistability of two periodic orbits. We also find limit cycles of periods ranging from a week to over 9 years, and an apparent canard explosion~\cite{Benoit1981}. The two-parameter continuations allow us to map out the curves of Hopf, period-doubling and fold bifurcations to determine regions of parameter space for which interesting dynamics occur, and also reveal torus (or Neimark-Sacker) bifurcations.

In the following sections we explore some of the more interesting dynamics in more detail.
In Section~\ref{sec.longp.hsc} we study a canard explosion for which the period of solutions increases from about 50 days to over 700 days over an exponentially small parameter interval. We show how to approximate the slow manifold associated with these solutions and show that this manifold has both stable and unstable components.
In Section~\ref{sec.chaos.hsc} we consider non-periodic and chaotic solutions.
First in Section~\ref{sec.torus} we investigate the torus bifurcations found in Section~\ref{sec.bifurc.hsc}, and find a stable invariant torus in the dynamics. We compute Lyapunov exponents and a Poincar\'e section to show that the dynamics do indeed correspond to a quasi-periodic orbit which envelopes the unstable-periodic orbit from which the torus bifurcated. We also find parameter values for which there is phase-locking on the torus and present the resulting stable periodic orbits.
In Section~\ref{sec.chaos} we study the dynamics between the period-doublings and find period doubling cascades leading to chaos. The chaotic nature of the dynamics is
verified numerically by showing that the leading Lyapunov exponent is positive, and visualisations of the attractor which reveal some of its fractal structure.
Parameter continuation in opposite directions reveals hysteresis with parameter intervals for which stable chaotic dynamics can co-exist with a stable periodic orbit, or even co-exist with a second chaotic attractor.
We also find parameter values for which there appears to be transient chaos.
In Section~\ref{sec.snaking} we present an example of a branch of periodic orbits
which snakes in parameter space resulting in a small parameter region in which over 50 limit cycles co-exist. Period-doubling cascades either side of this region lead to additional parameter regions of chaotic dynamics.

In Section~\ref{sec.physiol} we discuss the physiological plausibility and implications of the results with regards to periodic hematological disorders,
and Section~\ref{sec.conc} includes further discussion and conclusions.

%
\section{Hematopoietic Stem Cell Equation} \label{sec.decoup.hsc}
HSC dynamics can be described by the classic $G_{0}$ cell-cycle model of Burns and Tannock~\cite{Burns_1970}.
The HSCs are distinguished between two phases, the proliferating phase and the resting or $G_0$ phase.
We denote the concentration of HSCs in the resting phase by $Q$.
From the resting phase HSCs may enter the proliferating phase at a
rate $\beta(Q)$, or differentiate at a constant rate $\kappa$, or remain in the resting phase.
Once HSCs enter the proliferating phase they are lost by apoptosis with a constant rate $\gamma$ or undergo mitosis. The time to complete the cell cycle is $\tau$.
After mitosis cells return to the $G_{0}$ resting phase, from whence the cycle may begin again.
In the resting phase HSCs are quiescent, while in the proliferating phase they are active and
distinguished between four subphases: $G_{1}$, $S$, $G_{2}$ and $M$. Cells at gap $G_{1}$ increase in size and
are committed to go through the cell cycle and undergo mitosis. At $S$ phase DNA synthesis occurs, at gap $G_{2}$
cells continue to grow, while in the mitotic phase $M$ cells stop growing and undergo cell division.
\begin{figure}[t]
\centering
\begin{overpic}[scale=0.3]{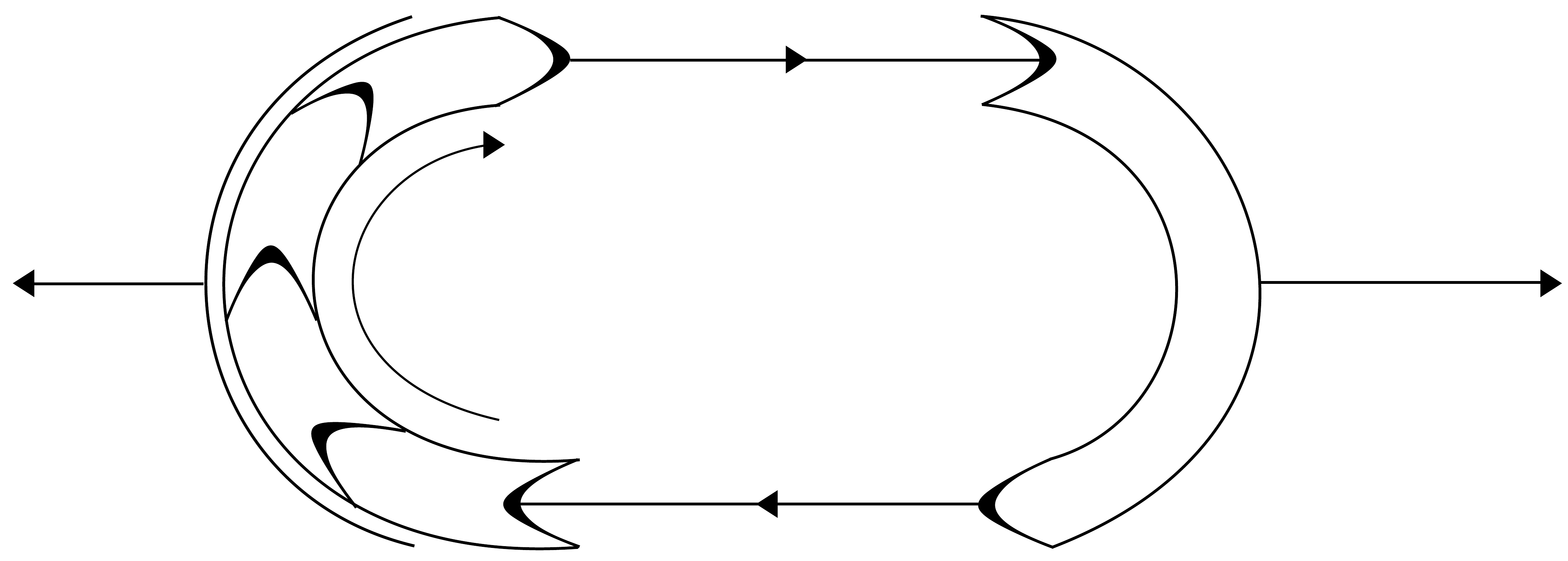}
\put(6.5,18.9){$\gamma$}
\put(2.1,15.4){\footnotesize{Apoptosis}}
\put(5.2,12.6){\footnotesize{Rate}}
\put(17.2,23.8){$G_{2}$}
\put(26.9,30.5){$M$}
\put(24.5,19){\footnotesize{Cell}}
\put(24,16.5){\footnotesize{Cycle}}
\put(25.9,14.0){$\tau$}
\put(17.5,11.2){$S$}
\put(26.0,3.8){$G_1$}
\put(43.4,33.2){\footnotesize{Self Renewal}}
\put(43.6,29.3){\footnotesize{Quiescence/}}
\put(44.0,26.8){\footnotesize{Senescence}}
\put(46.8,24.5){\footnotesize{Entry}}
\put(46.5,5.5){$\beta(Q)$}
\put(41.5,1.1){\footnotesize{Cell Cycle Entry}}
\put(65.8,17.9){\footnotesize{Resting}}
\put(66.6,15.4){\footnotesize{Phase}}
\put(76.1,16.5){$G_{0}$}
\put(88.5,18.4){$\kappa$}
\put(81.0,15.4){\footnotesize{Differentiation/}}
\put(83.0,13){\footnotesize{Death Rate}}
\end{overpic}
\vspace*{-3mm}
\caption{Schematic representation of the classical $G_{0}$ model for HSCs. The proliferating phase of the cell cycle is divided between 4 subphases:  gap one $G_{1}$, synthesis $S$, gap two $G_{2}$, and mitosis $M$. Cells in the resting phase (gap zero $G_{0}$), may differentiate with rate $\kappa$ or entry the cell cycle with rate $\beta(Q)$. Cells in the proliferation phase may be lost by apoptosis with rate $\gamma$, otherwise they re-enter the resting phase after mitosis, $\tau$ time units after they left the resting phase.}
\label{fig_SchematicHSC}
\end{figure}
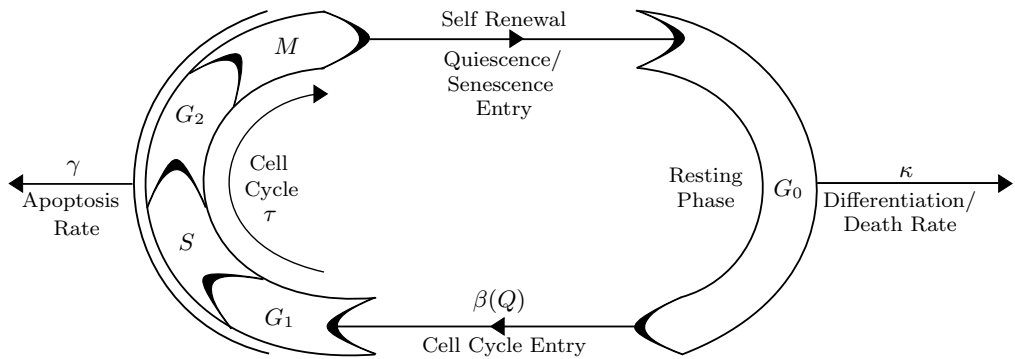

A schematic of the model is presented in Figure~\ref{fig_SchematicHSC}.
Following Mackey~\cite{Mackey_1978,Mackey_1994} this model can be stated as the DDE~\eqref{Qprime} where $\beta(Q)$ is a Hill function defined by
\vspace{-0.0mm}
\begin{equation}\label{beta2}
\beta(Q) = f\frac{\theta^{s}}{\theta^{s}+Q^{s}},
\end{equation}
and the HSC amplification factor $A$ is given by
\begin{equation}\label{AQ2}
A = 2\mathnormal{e}^{-\gamma\tau}.
\end{equation}
The parameters $\kappa$, $\gamma$, $\tau$, $\theta$, $f$ and $s$ are all strictly positive,
and we are interested in non-negative solutions $Q(t)\geq0$,
since $Q(t)$ represents a blood cell population.


Aspects of the dynamics of the HSC model~\eqref{Qprime} have been studied
by a number of authors, mainly concentrating on the existence and stability of the steady states, and one parameter continuation of some of the periodic orbits that
arise~\cite{Andersen_2001,Bernard_2004,Mackey_1994}.
A special case of~\eqref{Qprime}
with $\beta(Q)$ replaced by a step function, corresponding to the limit
as $s\to\infty$ in~\eqref{beta2}, allows explicit stable periodic solutions to be constructed~\cite{Pujo_Menjouet_2005,Pujo_Menjouet_2004}.
The existence of stable periodic solutions for $s$ large was subsequently established~\cite{Pujo_Menjouet_2006}, as a perturbation of the $s=\infty$ solutions.
In contrast, we will study the dynamics of~\eqref{Qprime} with $\beta(Q)$ a Hill function, with $s$ small, as is usually considered to be the case
in hematopoiesis models.


In order to solve the DDE~\eqref{Qprime} for $t\geq0$
it is necessary to define an initial function $Q(t)=\varphi(t)$ for $t\in[-\tau,0]$. For
$\varphi\in C\coloneqq C([-\tau,0],[0,\infty))$, that is continuous, bounded and non-negative, from the following
theorem the solution of~\eqref{Qprime} is also bounded and non-negative. It follows that the DDE~\eqref{Qprime} can be considered as an infinite dimensional dynamical system with phase space $C$~\cite{Smith_2010}.
%
%
%
%
\begin{theorem}\label{theorem.bound}
If $Q(t)=\varphi(t)$ for $t\in[-\tau,0]$ where $\varphi\in C([-\tau,0],[0,\infty))$, then
equation~\eqref{Qprime} has a unique solution $Q(t)$ defined for all $t\geq0$ and which satisfies $Q(t)\in[0,M]$ for all $t\geq0$ for some $M<\infty$.
\end{theorem}
\begin{proof}
As already noted by other authors, uniqueness and local existence of solutions follows from the method of steps. It suffices to show that solutions are bounded to obtain global existence and complete the proof~\cite{Pujo_Menjouet_2006,Pujo_Menjouet_2005,Pujo_Menjouet_2004}.

To show positivity of the solution, let $t^*\geq0$ such that $Q(t)\geq0$ for $t\in[-\tau,t^*]$. It follows easily from
\eqref{Qprime} that for all $t\in(t^*,t^*+\tau)$ for which $Q(t)\in(-\theta,0)$ that $Q'(t)>0$. This leads to a contradiction unless $Q(t)\geq0$ for all $t\in(t^*,t^*+\tau)$. Hence $Q(t)\geq0$ for all $t\geq0$. The existence of an upper bound $Q(t)\leq M$ follows from $\limsup_{t\to\infty}Q(t)$ being bounded, which was shown
by Mackey and Rudnicki~\cite{Mackey_1994}
for non-negative solutions with a general class of monotonic functions $\beta(Q)$ that includes~\eqref{beta2}.
\end{proof}

In~\eqref{Qprime}-\eqref{AQ2}
there are six parameters $\{\kappa,\gamma,\tau,\theta,f,s\}$, but we can reduce these to four by non-dimensionalising the equations. Let
$\hat{t}\coloneqq t/\tau$,  $\hat{f}\coloneqq\tau f$$, \hat{\kappa}\coloneqq \kappa/f$, $\hat{\gamma}\coloneqq \gamma\tau$, and $\hat{Q}(\hat t)\coloneqq Q(t)/\theta$, then
$\hat{Q}(\hat t-1)= Q(t-\tau)/\theta$.
It is also convenient to define
$\hat{A}\coloneqq 2\mathnormal{e}^{-\hat{\gamma}}=2\mathnormal{e}^{-\gamma\tau}=A$,
and notice that
$\beta(\hat{Q})=f/(1+\hat{Q}^{s}).$
Then
\eqref{Qprime} becomes
%
\begin{equation}\label{dQdt_hat}
\frac{1}{\hat{f}}\frac{d\hat{Q}}{d\hat{t}}(\hat t)= -\hat{\kappa}\hat{Q}(\hat t)
-\frac{\hat{Q}(\hat t)}{1+\hat{Q}(\hat t)^{s}}+\hat{A}\frac{\hat{Q}(\hat t-1)}{1+\hat{Q}(\hat t-1)^{s}},
\end{equation}
which depends on the four parameters $\hat{f}$, $\hat{\kappa}$, $s$ and $\hat{A}$ (or $\hat\gamma$). Although many mathematicians would prefer to study the non-dimensionalised DDE~\eqref{dQdt_hat} instead of~\eqref{Qprime}, we chose to work with~\eqref{Qprime} so that the solutions and bifurcations that we find have direct physiological interpretations.
But, as suggested by the non-dimensionalisation, we need only vary four parameters in~\eqref{Qprime}.
It is easily seen that varying the four parameters $\gamma$, $\kappa$, $\tau$ and $s$ in~\eqref{Qprime}, with the other parameters held constant, we can reproduce all possible values of $\hat{f}$, $\hat{\kappa}$, $s$ and $\hat{A}$ in~\eqref{dQdt_hat}, and hence we will only need to consider the variation of parameters from amongst these four.


In Table~\ref{tab.model.par} we state homeostatic values of the parameters for the model~\eqref{Qprime}-\eqref{AQ2}. The values in Table~\ref{tab.model.par} are all taken from
Craig \textit{et al}~\cite{Craig_2016}. The first two parameters in the table,
$Q^{h}$ the homeostatic concentration of HSCs, and
$\beta(Q^h)$ the homeostatic rate that cells enter the cell cycle, do not appear explicitly in
the model~\eqref{Qprime}-\eqref{AQ2}, but are used to calculate the last two parameters.
To ensure that $Q^{h}$ is exactly the homeostatic concentration of HSCs in the model~\eqref{Qprime}-\eqref{AQ2}, the last three parameters are computed in double precision. While $A$ is given directly
by~\eqref{AQ2}, rearranging ~\eqref{Qprime} and~\eqref{beta2} at homeostasis implies that
\begin{equation} \label{homeoparas}
\theta=Q^h\left(\frac{f}{\beta(Q^h)} - 1\right)^{-1/s},
\qquad
\kappa=(A-1)\beta(Q^h).
\end{equation}
We will use the parameters of Table~\ref{tab.model.par}
as a starting point for our bifurcation studies.
\begin{table}[t]
\begin{center}
\begin{tabular}{|c|c|c|c|c|}
\hline
Name & Interpretation &Value & Units\\
\hline
$Q^{h}$ & HSC homeostasis concentration & $1.1$ & ${10}^{6}\mathrm{cells/kg}$\\
$\beta(Q^h)$ & homeostasis cell cycle entry rate  & $0.043$ & $\mathrm{days}^{-1}$ \\
\hline
$\gamma$ & HSC apoptosis rate & $0.1$ & $\mathrm{days}^{-1}$\\
$\tau$ & Time for HSC re-entry & $2.8$ & $\mathrm{days}$\\
$f$ & Maximal HSC re-entry rate & $8$ &  $\mathrm{days}^{-1}$\\
$s$ & HSC re-entry Hill coefficient & $2$ & $-$\\
\hline
$A$ & HSC Amplification Factor & $1.512^*$ & $-$\\
$\theta$ & Half-effect HSC concentration & $0.08086^*$ & ${10}^{6}\mathrm{cells/kg}$\\
$\kappa$ & HSC differentiation rate to all lines & $0.022^*$ & $\mathrm{days}^{-1}$\\
\hline
\end{tabular}
\end{center}
\vspace*{2mm}
\caption{Homeostasis values of the parameters for the mathematical granulopoiesis model~\eqref{Qprime}-\eqref{AQ2},
which are all taken from Craig \textit{et al}~\cite{Craig_2016}.
The last three parameters with values denoted by $^*$, are only stated to 4 significant figures here, but actual values are
computed to full double precision in MATLAB~\cite{Matlab} using Eqs.~\eqref{AQ2} and~\eqref{homeoparas}.}
\label{tab.model.par}
\end{table}


For general parameter values, from~\eqref{Qprime},
steady states satisfy
\begin{equation}\label{dQdt.star}
0 = [(A-1)\beta(Q)-\kappa]Q.
\end{equation}
Hence the DDE~\eqref{Qprime} has the trivial steady state $Q(t)=0$ for all values of the parameters. The trivial steady state has been shown to be globally asymptotically stable if $\kappa>f(A-1)$~\cite{Mackey_1994}.
%
%
Eq.~\eqref{dQdt.star} has another solution $Q^*$ which satisfies
\begin{equation}\label{betaQ.star}
\beta(Q^{*}) = \frac{\kappa}{(A-1)}.
\end{equation}
Since $\beta(Q)$ is monotonic, this defines a unique nontrivial steady state $Q^*$.
Using~\eqref{beta2} we have
\begin{equation}\label{Q.star}
Q^{*} = \theta\left[f\frac{(A-1)}{\kappa} - 1\right]^{1/s}.
\end{equation}
From~\eqref{Q.star}, along with the relation~\eqref{AQ2} we obtain that $Q^*>0$ if and only if
the upper bounds
\begin{equation}\label{kgt}
\kappa < f(A-1),\qquad \gamma\tau < \ln\left(\frac{2f}{\kappa+f}\right),
\end{equation}
on the parameters $\kappa$, $\gamma$ and $\tau$ are satisfied. In the rest of this work,  we will consider the case where~\eqref{kgt} holds and
there are two steady states, $Q=0$ and $Q^*>0$. From~\eqref{AQ2} and~\eqref{kgt} we require $A\in(1+\kappa/f,2)$ for $Q^*>0$ to exist. Regarding the steady state $Q^*>0$ as a function of the
differentiation rate $\kappa$, the death rate $\gamma$, or the cell cycle duration $\tau$,
it is easy to see from~\eqref{Q.star} that $Q^*$ is a monotonically decreasing function with respect to each of these parameters, as would be expected from a physiological point of view. Furthermore
$Q^*\to0$ as equality is approached in~\eqref{kgt}.
We denote the non-zero steady state by $Q^h$ only when the parameters take their homeostasis values from Table~\ref{tab.model.par}, and by $Q^*$ otherwise. Unless otherwise stated, all the values in this paper are given in the same units as in Table~\ref{tab.model.par}.


%
%

\subsection{Stability Boundary}
\label{sec.stab.bound}

To determine the stability of the steady state $Q^*$,
linearise the DDE~\eqref{Qprime} around $Q^{*}$, with $z(t)\coloneqq Q(t)-Q^{*}$ to get
\begin{equation}\label{dQdtLin}
z^{\prime}(t)=az(t)+bz(t-\tau).
\end{equation}
It is convenient to define
\begin{equation}\label{hx}
h(x)\coloneqq x\beta(x),
\end{equation}
then the parameters $a$ and $b$ can be written as
\begin{equation}\label{ab}
a =-\kappa-\beta(Q^{*})-Q^{*}\beta^{\prime}(Q^{*})=-\kappa-h'(Q^*),\qquad
b =  A\left(\beta(Q^{*})+Q^{*}\beta^{\prime}(Q^{*})\right)=Ah'(Q^*).
\end{equation}
Seeking a nontrivial solution $z(t)=z_{0}\mathnormal{e}^{\lambda t}$ for equation~\eqref{dQdtLin},
we get the characteristic equation
\begin{equation}\label{hayes}
p(\lambda)\coloneqq  \lambda - a  - b\mathnormal{e}^{-\lambda\tau} = 0,
\end{equation}
first studied by Hayes~\cite{Hayes_1950}.
According to the \textit{Principle of Linearised Stability}~\cite{Smith_2010}
the stability analysis of the steady states for the nonlinear DDE~\eqref{Qprime} is reduced to the stability analysis of the steady state of the linearised equation~\eqref{dQdtLin}. Stability analysis of equation~\eqref{dQdtLin} is a standard example in DDEs, and can be found in~\cite{Breda_2015,Hale_1993,Insperger_2011,Smith_2010}.
The steady state is unstable if $b>-a$ with a characteristic value $\lambda>0$, and it is asymptotically stable if
$a<-|b|<0$, which is sometimes called the delay-independent stability region~\cite{Kolmanovskii_1999}. The interesting parameter region is for $b\leq-|a|<0$, where the steady state is asymptotically stable for
\begin{equation} \label{tau1}
\tau<\tau_1(a,b)\coloneqq\cos^{-1}(-a/b)/\sqrt{b^2-a^2}
\end{equation}
and unstable if $\tau>\tau_1(a,b)$. The curve $C_0$ described by $\tau_1(a,b)$ is contained in the region
$a\leq1/\tau$ and $b\leq-1/\tau$ and can be parameterised~\cite{Smith_2010} as
\begin{equation} \label{C0}
C_0=\bigl\{(a,b)=(\omega\cot(\omega)/\tau,-\omega\csc(\omega)/\tau), \quad \omega\in[0,\pi)\bigr\}.
\end{equation}
On this curve the characteristic equation~\eqref{hayes} has an imaginary solution $\lambda=\pm i\omega$.
The parameter region $(a\tau,b\tau)\in\mathbb{R}^2$ for which $Q^*$ is stable is illustrated in
Figure~\ref{fig:Hayesabstab}.

In the context of the DDE~\eqref{Qprime}
the trivial steady state $Q=0$ is unstable when $Q^*>0$~\cite{Qu_2010}. For the stability of $Q^*$, from~\eqref{betaQ.star} and~\eqref{ab} we have
$a+b=(A-1)Q^*\beta'(Q^*),$
and hence using~\eqref{beta2} and~\eqref{Q.star} we find that $-s\kappa<-s\kappa(1-\kappa/[f(A-1)])=a+b<0.$
Thus when $Q^*>0$ we have $a+b<0$, and $Q^*$ can only lose stability if $(a,b)$ crosses the curve $C_0$ as parameters are varied.

The steady state $Q^*>0$ is stable for all $\tau$ sufficiently small, since $\lambda=-(a+b)<0$ when $\tau=0$.
It is also stable for all $\tau$ sufficiently large. This follows from noting that
$b>0$ when $h'(Q^*)>0$, which holds
whenever $\tau>\tau_2$  where $\tau_2$ is defined by \eqref{tau12} below. Then, for
$\tau>\tau_2$ the parameters $(a,b)$ are in the upper half of the delay-independent stability region
$a<-|b|<0$.

The steady state $Q^*$ may lose stability in a Hopf bifurcation if the parameters cross the curve $C_0$ in the $(a,b)$ parameter space, but since the steady state is stable for $\tau$ small and for $\tau$ large, such Hopf bifurcations will occur in pairs, corresponding to crossings of $C_0$ in opposite directions.
For many DDEs 
the steady state gains ever more characteristic values
with positive real part as the delay $\tau$ is increased~\cite{Longtin1998}. The Burns-Tannock DDE~\eqref{Qprime} does not behave like that because of the exponential term in $A=2e^{-\gamma\tau}$ representing mortality during the cell cycle.

For the homeostasis parameter values in Table~\ref{tab.model.par} we have
$f(A-1)=4.093>\kappa=0.022$ so \eqref{kgt} is satisfied and there
exists a unique positive steady state, $Q^h=1.1>0$. Furthermore, at homeostasis
$a=0.020540$ and  $b=-0.064298$ and $\tau=2.8$ which is inside the stability region, as
illustrated in Figure~\ref{fig:Hayesabstab},
so the homeostasis steady state $Q^h$ is asymptotically stable.

If $\tau$ is varied and all the other parameters
are at their homeostasis values from Table~\ref{tab.model.par}, we find that
$Q^*>0$ is stable for $\tau\in[0,\tau_1^-)$ and for $\tau\in(\tau_1^+,\tau_{\max})$ where
\begin{equation} \label{hstabtau}
0 < \tau^{h} = 2.8 <  \tau_1^- = 5.74851 < \tau_1^+ = 6.87437 <  \tau_2 = 6.87662 < \tau_{\max} = 6.90401.
\end{equation}
Here, $\tau_{\max}$ is given by~\eqref{kgt}, while
\begin{equation} \label{tau12}
\tau_1^\pm=\frac{\cos^{-1}((\kappa+h'(Q^*))/A h'(Q^*))}{\sqrt{(A h'(Q^*))^2-(\kappa+h'(Q^*))^2}}, \qquad
\tau_2=\frac{1}{\gamma}\ln\left(\frac12\left[1+\frac{\kappa s}{f(s-1)}\right]\right)^{-1}.
\end{equation}
The formula for $\tau_2$ is found by using the expressions for $h(Q)$ and $Q^*$ and solving for $b=0$.
This formula was already stated in
Pujo Menjouet \textit{et al}~\cite{Pujo_Menjouet_2005,Pujo_Menjouet_2004}, where it was
erroneously claimed that $\tau=\tau_2$ was a stability boundary. As noted above already, and as shown in
Figure~\ref{fig:Hayesabstab}, the parameters corresponding to $\tau=\tau_2$ are in the interior of the stability region.

\sloppy{The expression for $\tau_1^\pm$ in \eqref{tau12}
is obtained from~\eqref{tau1} on substituting the values from~\eqref{ab}.
The two values for $\tau_1$ correspond to the parameters $(a,b)$ defined by~\eqref{ab} crossing the curve $C_0$ twice as $\tau$ is increased. The equation $\tau=\tau_1$ from \eqref{tau1}
can have
two solutions, because in the expression for $\tau_1$ the value of $A$ itself depends on
$\tau$, making the equation $\tau=\tau_1$ implicit in $\tau$.}
The locus of the parameters $(a,b)$ and the corresponding $\tau$ values from~\eqref{hstabtau} are illustrated in Figure~\ref{fig:Hayesabstab}.

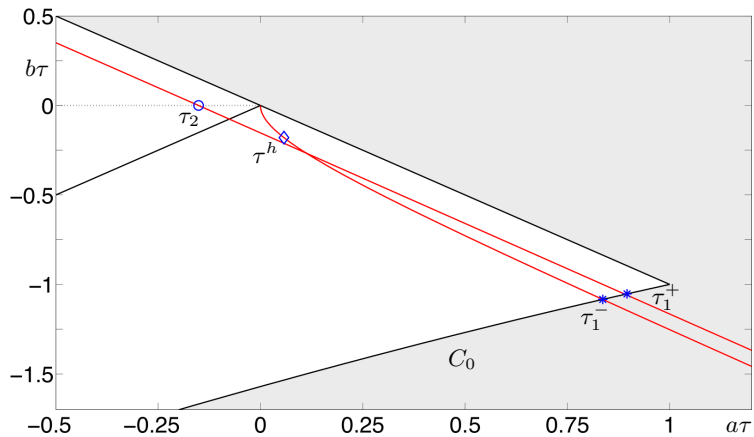
\begin{figure}[t]
\centering
\includegraphics[width=0.55\textwidth]{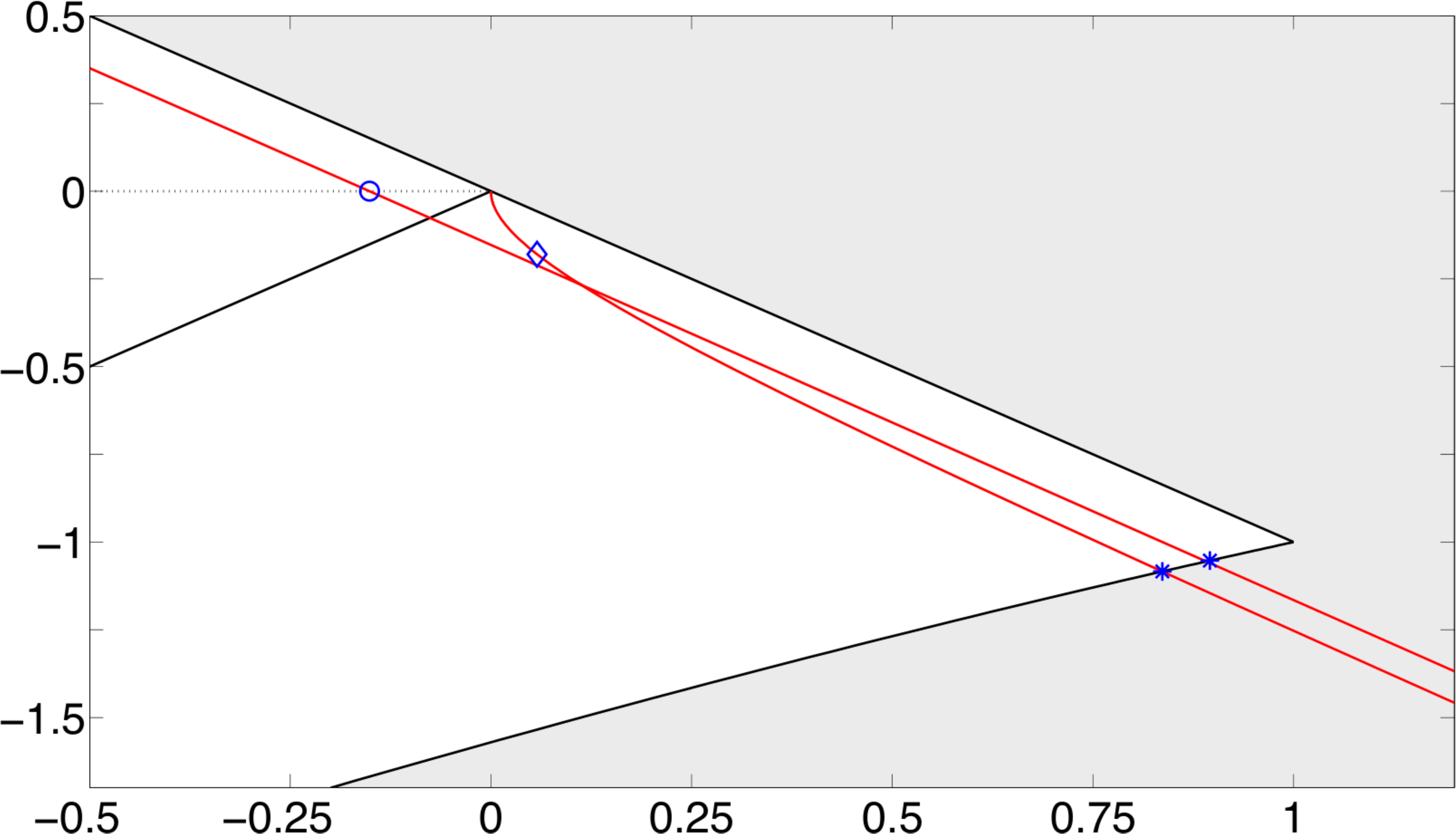}
\put(-10,0){$a\tau$}
\put(-275,135){$b\tau$}
\put(-65,41){$\tau_1^-$}
\put(-38,48.5){$\tau_1^+$}
\put(-217,117){$\tau_2$}
\put(-189,102){$\tau^h$}
\put(-115,25){$C_{0}$}
\vspace*{-1mm}
\caption{The parameter regions for which the steady state of equation~\eqref{dQdtLin} is stable (white) and unstable (shaded). Also shown (red) is the locus in $(a,b)$ of the parameters defined by
\eqref{ab} as $\tau$ is varied with the other parameters all at their values from Table~\ref{tab.model.par}. The values of $\tau$ from~\eqref{hstabtau} are indicated on this curve.}
\label{fig:Hayesabstab}
\end{figure}

If parameters are varied so that $\tau$ crosses the boundary $\tau_1$, either by varying $\tau$ itself, or by varying parameters to change the values of $\tau_1$, then the stability of $Q^*$ changes
and a periodic orbit with period $2\pi/\omega$
is created in a Hopf bifurcation.
The periodic orbits thus created have been explored to some extent~\cite{Fowler_Mackey_2002,Pujo_Menjouet_2006,Pujo_Menjouet_2005,Pujo_Menjouet_2004},
but the details of the Hopf bifurcation and its normal form have only been studied more recently~\cite{Qu_2010}.

Because~\eqref{hayes} has infinitely many roots, it is possible for additional pairs of complex conjugate
characteristic values to cross the imaginary axis resulting in additional Hopf bifurcations. These will occur
on curves $C_{2n}$, which are defined by \eqref{C0} but for $\omega\in[2n\pi,(2n+1)\pi)$~\cite{Smith_2010}. None of these curves intersect, so there are no double-Hopf bifurcations in which two pairs of characteristic values cross the imaginary axis at the same time.

\section{Bifurcations of the HSC Equation}
\label{sec.bifurc.hsc}

In this section we vary parameters from their homeostasis values so that the perturbed steady state $Q^*$ becomes unstable, and survey the bifurcations and dynamics that arise.
We begin our numerical bifurcation analysis of the DDE~\eqref{Qprime}, by studying codimension-one bifurcations
performing parameter continuation on $\tau$, $\gamma$ and $\kappa$, one at a time. We will also study
codimension-two bifurcations carrying out a two-parameter continuation for each pair of these three parameters.
Our numerical bifurcation diagrams are constructed using the well-established DDEBiftool
package~\cite{Engelborghs_2002,DDEBiftool15}, which runs under MATLAB~\cite{Matlab}.
This software finds periodic orbits using a boundary value approach, and is able to find stable and
unstable solutions, and continue the solutions as parameter(s) are varied and detect stability and
bifurcations. 

As noted after~\eqref{dQdt_hat}, the dynamics of the HSC DDE~\eqref{Qprime} only depend on the four
parameters $\tau$, $\gamma$, $\kappa$ and $s$.
The dependence
of the dynamics on the parameter $s$ has previously been studied for small integer values of $s$ (through numerical simulation)
and analytically in the limit as $s\to\infty$ (in which case the Hill function~\eqref{beta2} simplifies
to a Heaviside function)~\cite{Pujo_Menjouet_2006,Pujo_Menjouet_2005,Pujo_Menjouet_2004}. In the current work we will keep $s=2$ fixed and equal to its value
in Craig \textit{et al}~\cite{Craig_2016},
and we will only vary the parameters $\tau$, $\gamma$, and $\kappa$. In the following, unless mentioned otherwise,
all parameters take the values stated in Table~\ref{tab.model.par}.
Recall that the homeostasis steady state $Q^h$ is stable.

\subsection{One Parameter Continuation}
\label{sec.1p}

\begin{figure}[t]
\includegraphics[width=0.49\textwidth,height=44mm]{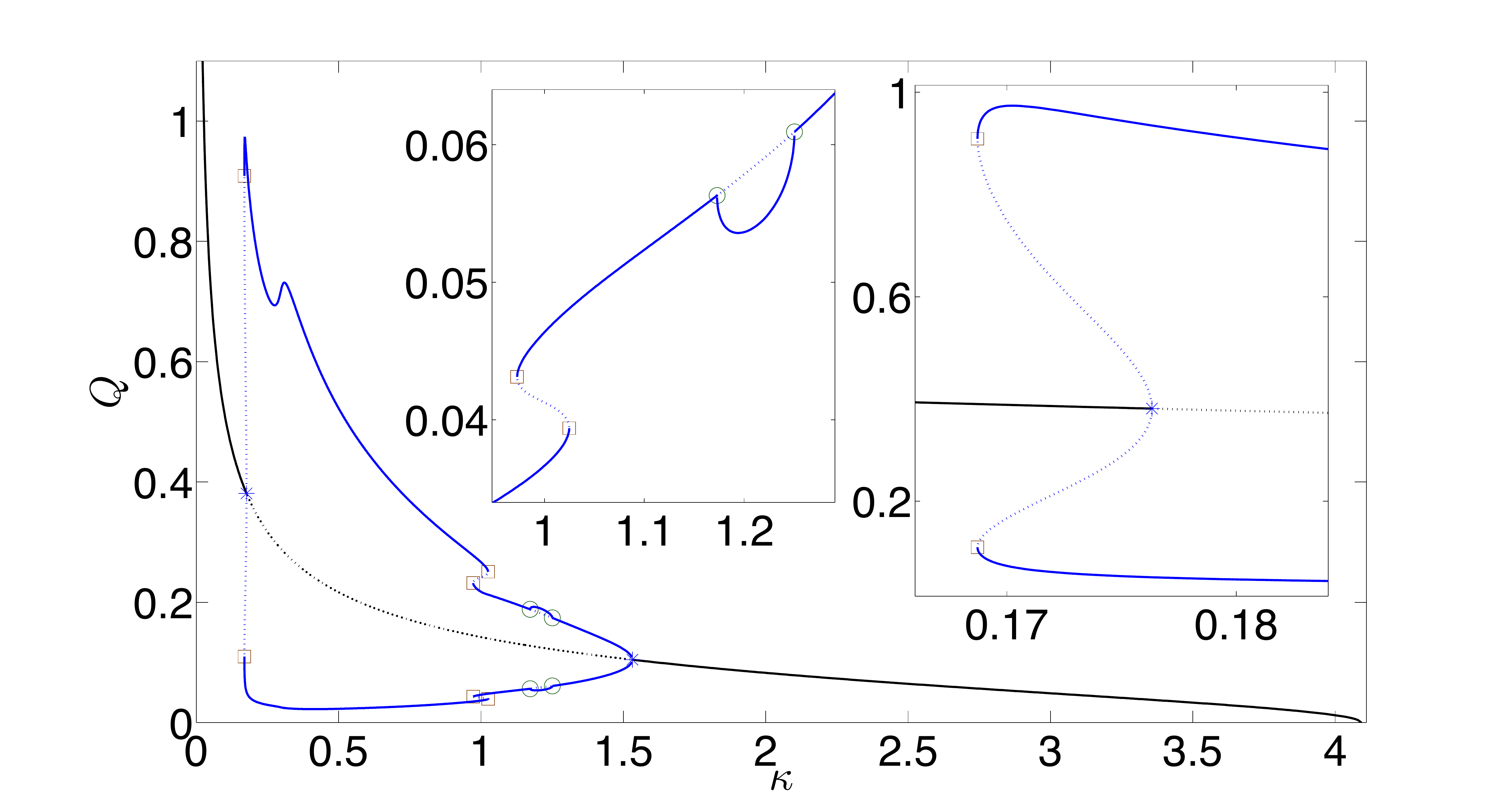}
\put(-245,2.5){\footnotesize{\textit{(i)}}}
\hspace*{1mm}
\includegraphics[width=0.49\textwidth,height=44mm]{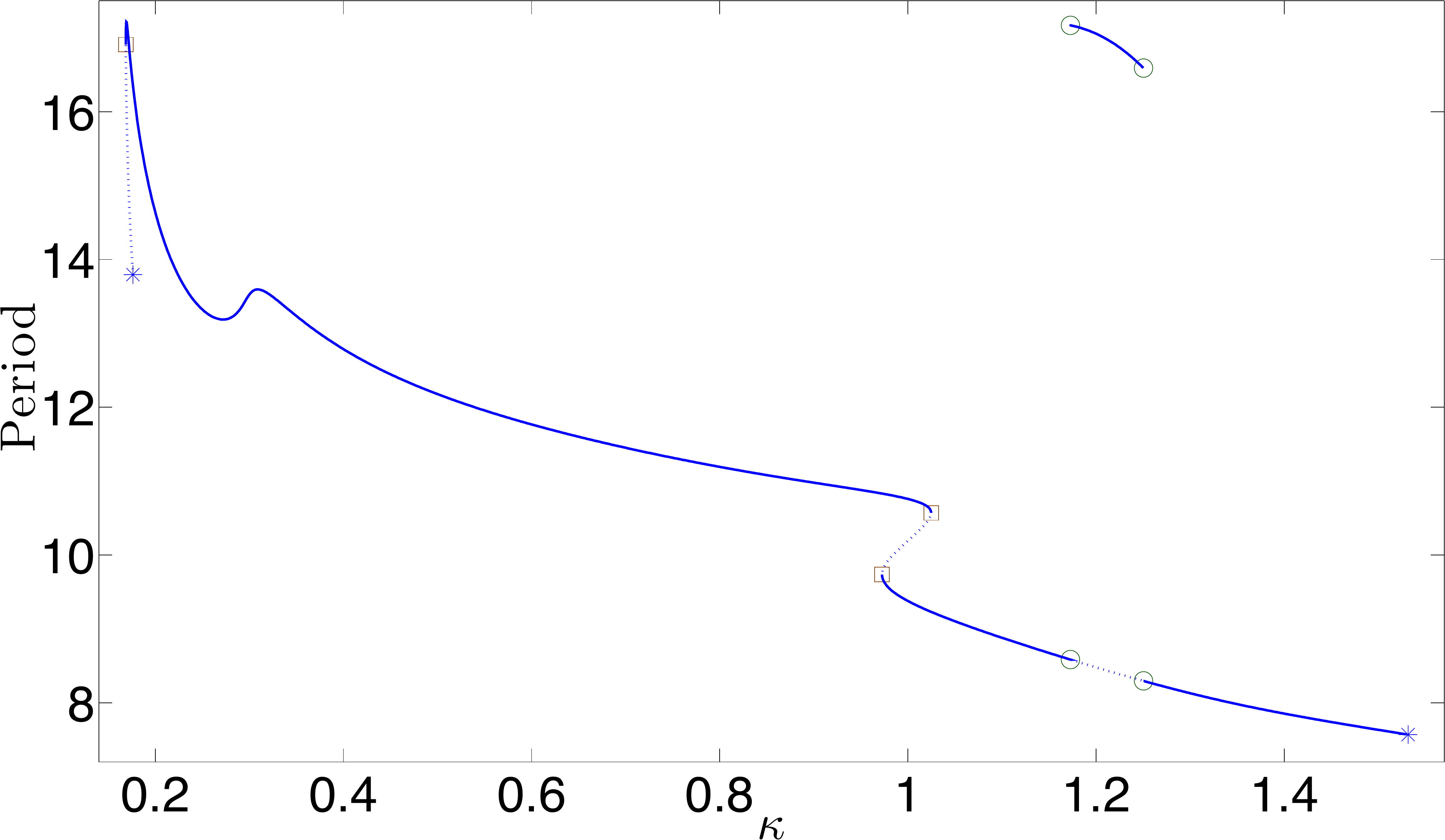}
\put(-245,2.5){\footnotesize{\textit{(ii)}}}
\vspace*{-03mm}
\caption{(i) Bifurcation diagram showing stability of the steady state $Q^*$, along with the branch of periodic orbits which bifurcates from $Q^*$ at the Hopf bifurcation points.
All parameters except $\kappa$ take values from Table~\ref{tab.model.par}.
In all bifurcation diagrams stable (unstable) steady states are represented by solid (dotted) black lines while the stable (unstable) limit cycles are represented by solid (dotted) blue lines.
For periodic orbits we plot both the maximum and minimum value of $Q$ over the period, so the upper and lower curve both represent the same periodic orbit. Hopf bifurcation points \textlarger[2]{$\textcolor{blue}{*}$}, saddle-node bifurcation of limit cycles points \textlarger[0]{$\textcolor{brown1}{\Box}$},  and period-doubling bifurcation points \textlarger[-1]{$\textcolor{green1}{\Circle}$},
are indicated, and also highlighted in insets.
The saddle-node bifurcation in the right inset creates an interval for $\kappa\in(0.16872,0.17632)$ of bistability between the periodic orbit and the steady state $Q^*$, while the pair of steady fold bifurcations seen in the left inset create an interval for $\kappa\in(0.97254,1.0247)$ of bistability between two different periodic orbits.
(ii) Shows the period of the periodic orbits seen in (i).
Two period doubling bifurcations and an interval for $\kappa\in(1.173,1.2506)$ of period-doubled solutions are also shown. These can also be seen  in the left inset in (i).}
\label{fig:KappaCont}
\end{figure}


We begin by varying the differentiation rate $\kappa$ (with all the other parameters held constant), with the bifurcation diagram presented in Figure~\ref{fig:KappaCont}.
The non-trivial steady state $Q^*$ (given by~\eqref{Q.star})
is seen to be unstable for an interval of $\kappa$ values between two Hopf bifurcation points, and stable for $\kappa$
outside this interval.
The Hopf bifurcation point at $\kappa\approx0.17632$ is subcritical leading to a branch of unstable periodic orbits
which becomes stable at $\kappa\approx0.16872$ in a fold bifurcation of periodic orbits. This creates a small interval of
bistability between the stable steady state $Q^*$ and the stable periodic orbit of amplitude close to $1$.
Similar bistability has been observed before in hematopoiesis models.
Bernard \textit{et al}~\cite{Bernard_2003,Bernard_2004} studied a model for white blood cell (WBC) production which incorporated~\eqref{Qprime}
to model the stem cell dynamics, and found bistability for WBCs between a stable steady state and a stable periodic
orbit. But in that model the bistability was seen as an amplification parameter in the WBC proliferation was varied,
and was not associated with the variation of any parameter in the HSC equation. We are not aware of bistability
having been observed previously in a stand-alone model for HSC dynamics.

There is a pair of fold bifurcations of periodic orbits which creates
an interval for $\kappa\in(0.97254,1.0247)$ of bistability of periodic orbits. Bistability is interesting, as it allows the possibility for a short term perturbation of the system (such as during treatment) to cause the solution to switch between one solution and another, and for the new stable dynamics to persist indefinitely. The two instances of bistability observed here occur for relatively small parameter intervals far from the homeostasis parameters and so are unlikely to be of great direct physiological relevance for healthy subjects. However the existence of bistability is interesting in the context of dynamical diseases which are related to bifurcations that occur when parameters in the system are varied.
There is also a pair of period-doubling bifurcations on the branch of periodic orbits illustrated in
Figure~\ref{fig:KappaCont}.
This leads to an interval for $\kappa\in(1.173,1.2506)$ where a period-doubled orbit is stable.

The amplitude of the periodic orbits on the main branch tends to decrease along the branch, and the periodic orbits disappear in
a super-critical Hopf bifurcation at $\kappa\approx1.5317$. For $\kappa$ between this value and its bound given by~\eqref{kgt},
the steady-state $Q^*$ is stable.
The periods of the orbits shown in Figure~\ref{fig:KappaCont}(ii)
strongly correlate with the amplitude of the orbits. The periods are all larger than a week, much larger than the delay $\tau=2.8$ days.

\begin{figure}[t]
\includegraphics[width=0.49\textwidth,height=44mm]{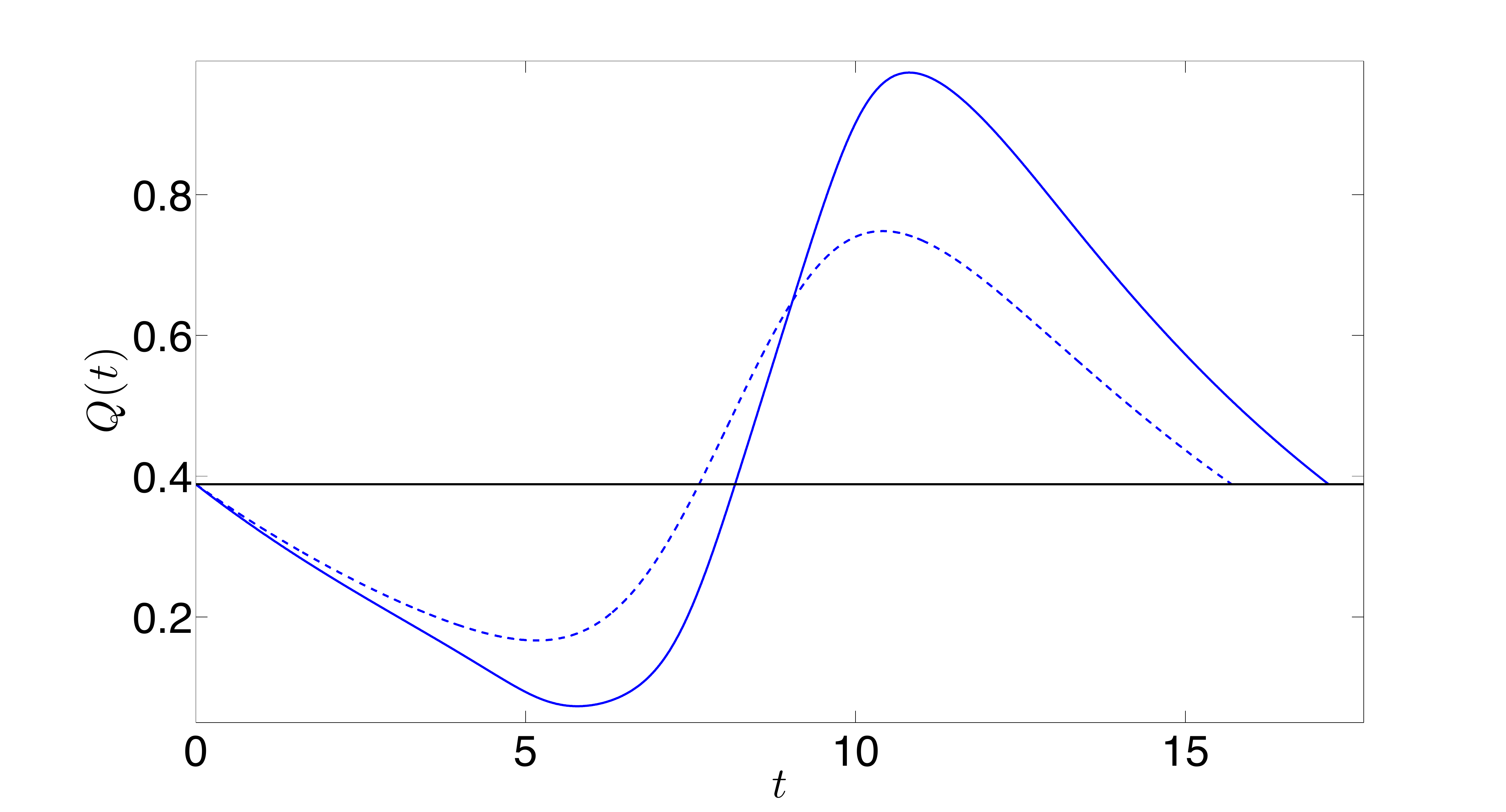}
\put(-244,2.5){\footnotesize{\textit{(i)}}}
\includegraphics[width=0.49\textwidth,height=44mm]{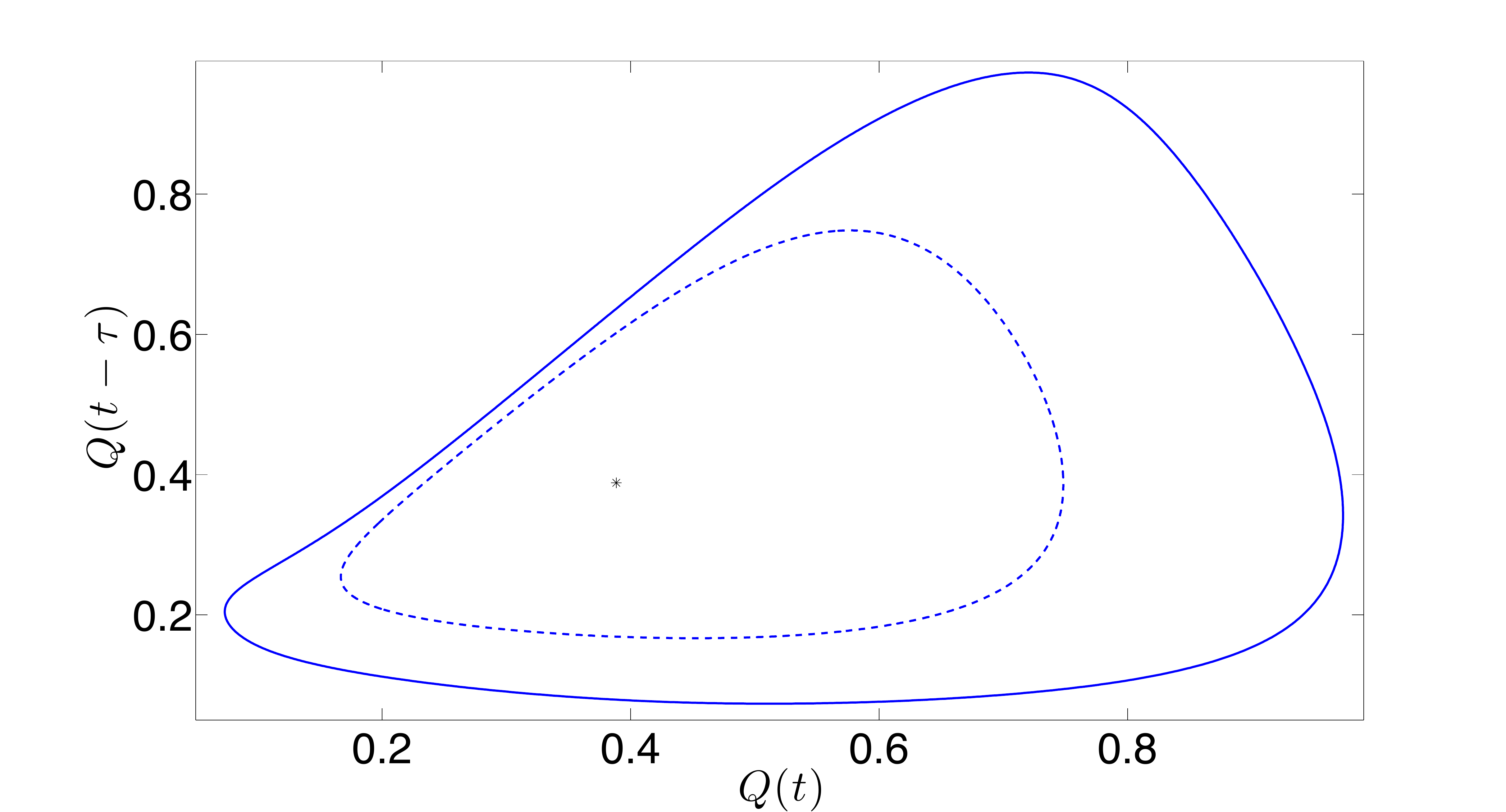}
\put(-244,2.5){\footnotesize{\textit{(ii)}}}

\vspace*{2mm}

\includegraphics[width=0.49\textwidth,height=44mm]{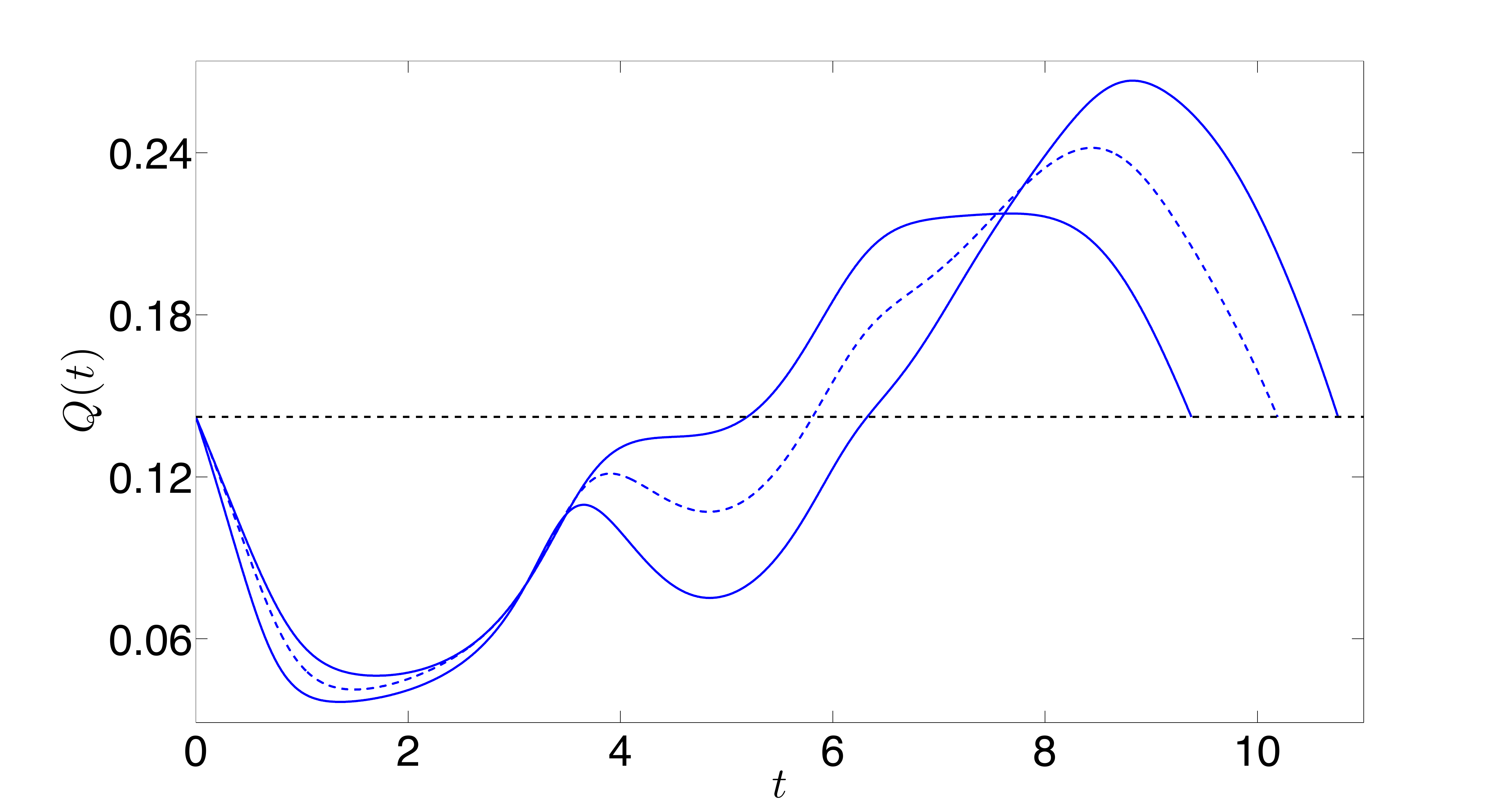}
\put(-244,2.5){\footnotesize{\textit{(iii)}}}
\includegraphics[width=0.49\textwidth,height=44mm]{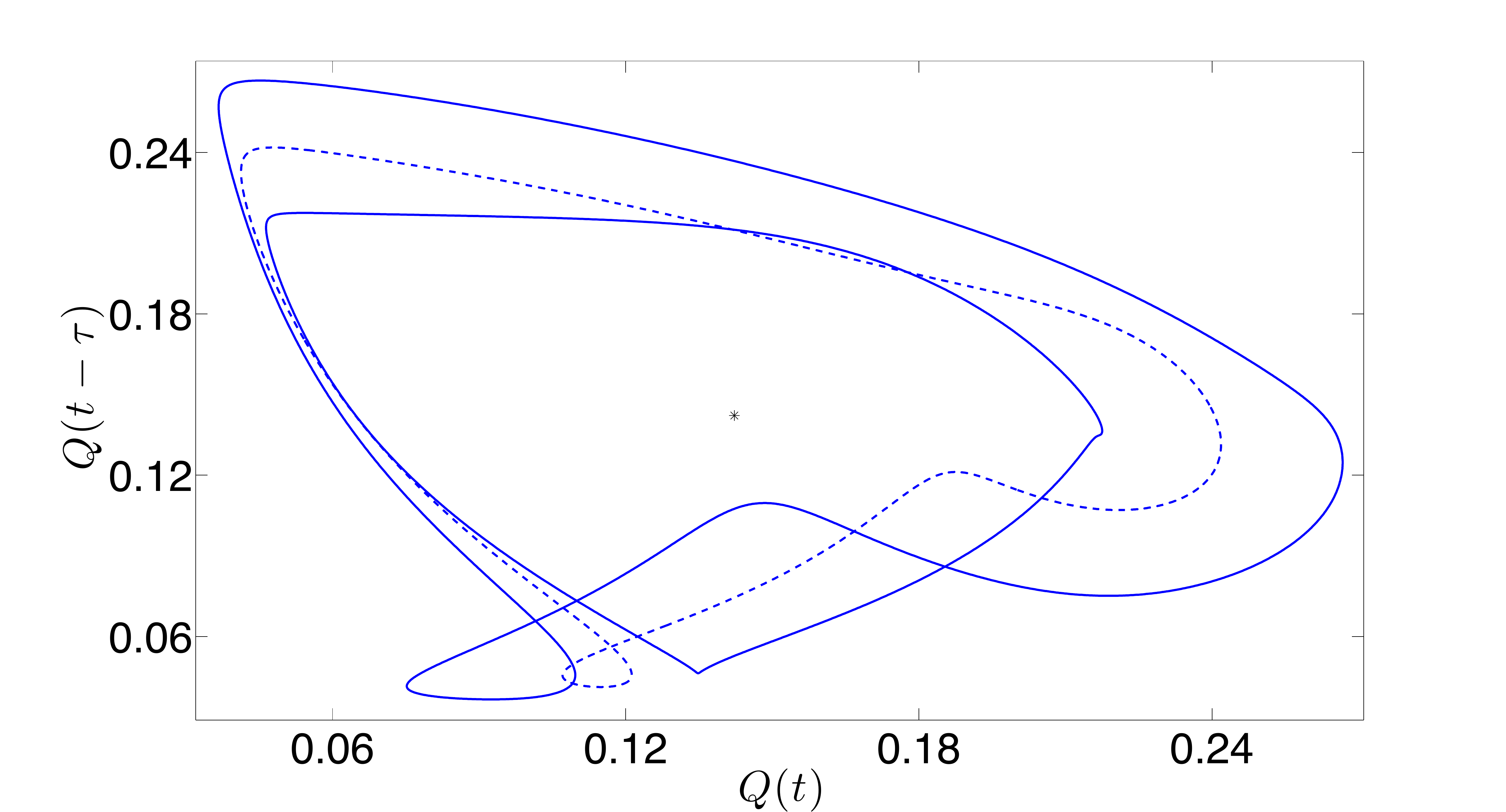}
\put(-244,2.5){\footnotesize{\textit{(iv)}}}\\

\vspace*{2mm}

\includegraphics[width=0.49\textwidth,height=44mm]{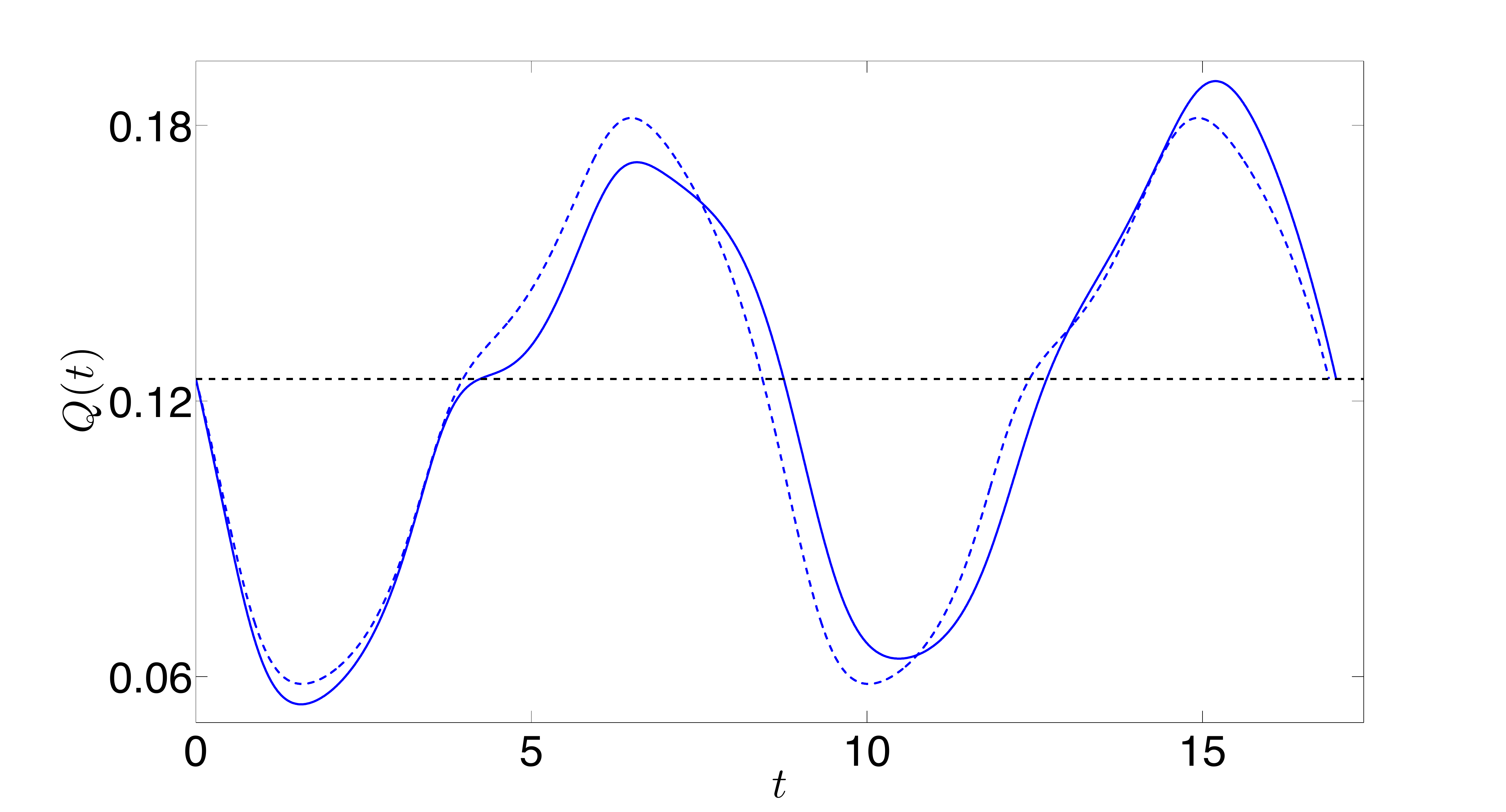}
\put(-244,2.5){\footnotesize{\textit{(v)}}}
\includegraphics[width=0.49\textwidth,height=44mm]{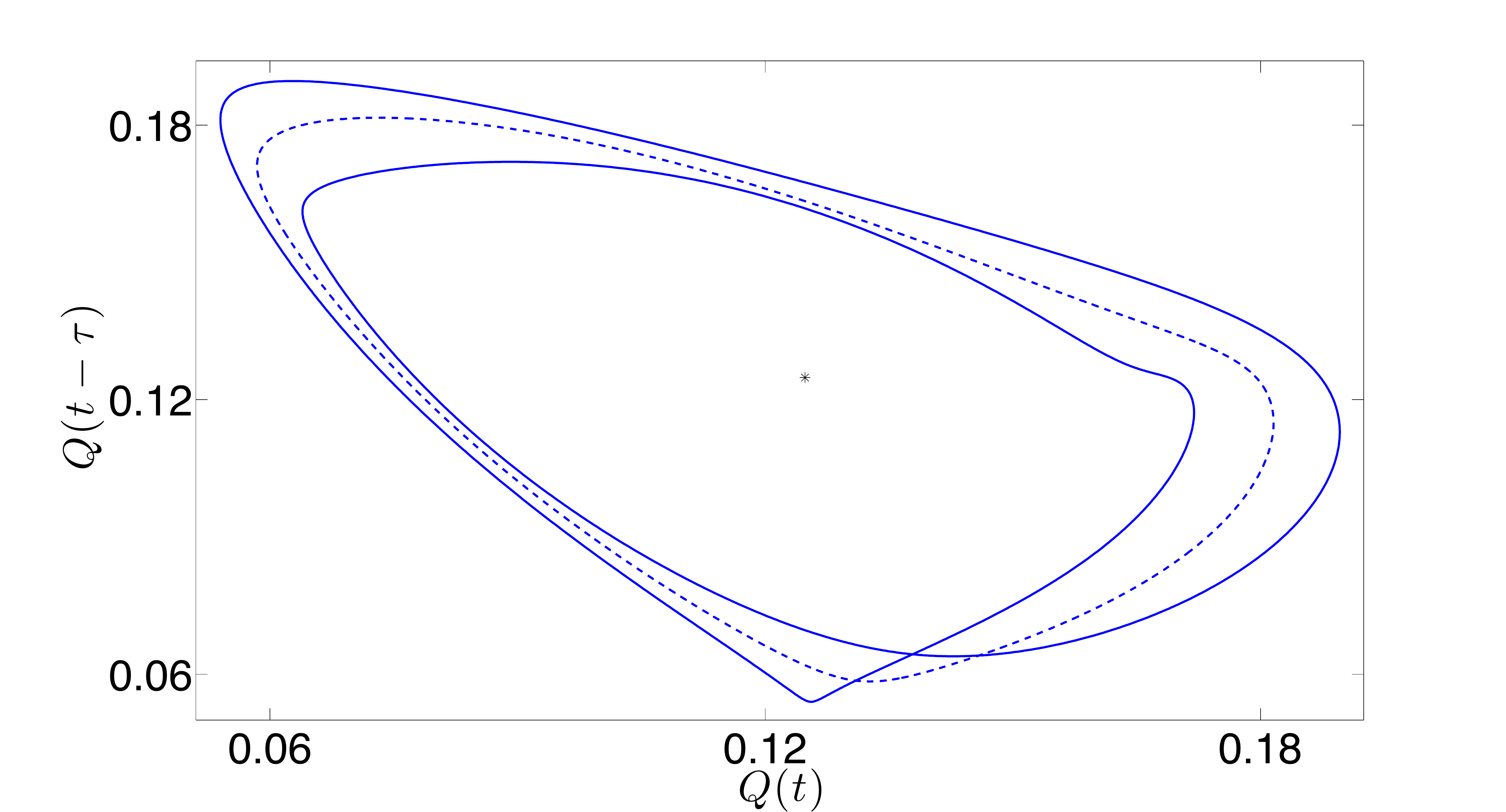}
\put(-244,2.5){\footnotesize{\textit{(vi)}}}
\vspace*{-2mm}
\caption{Periodic Orbits from the branch seen in Figure~\protect\ref{fig:KappaCont}.
(i) For $\kappa=0.17$ the stable (solid line) and unstable (dotted line) periodic orbits which coexist with the stable steady state $Q^*$ just before the subcritical Hopf bifurcation. (ii) The same periodic orbits shown in the
two-dimensional $(Q(t),Q(t-\tau))$ projection of phase space.
(iii-iv) For $\kappa=1$ the two coexisting stable periodic orbits along with the unstable periodic orbit and unstable steady state.
(v-vi) For $\kappa=1.2$ the stable period-doubled orbit (solid line),
along with the unstable periodic orbit
from which it bifurcates (in panel (v) we show two periods of this orbit).
}
\label{KappaCont_Orbits_Bistab_kappa_plots}
\end{figure}

Figure~\ref{KappaCont_Orbits_Bistab_kappa_plots} illustrates some of the more interesting periodic orbits
found during the $\kappa$ continuation, including examples of bistability and period doubled orbits.
The left panels show periodic solution profiles over one period,
while the right panels display the time-delay embedding of the same solutions.
Recalling that $Q(t)$ is a scalar, but that
the DDE~\eqref{Qprime} defines an infinite dimensional dynamical system, $(Q(t),Q(t-\tau))$ gives a useful
two-dimensional projection of the infinite dimensional solutions, which has been used widely since it was
introduced by Glass and Mackey~\cite{Glass_Mackey_1979}.
Since it is only a projection of phase space, orbits may appear to
cross each other, but because it incorporates both the terms $Q(t)$ and $Q(t-\tau)$ that appear in~\eqref{Qprime} this projection is often very revealing.

\begin{figure}[ht]
\includegraphics[width=0.49\textwidth,height=44mm]{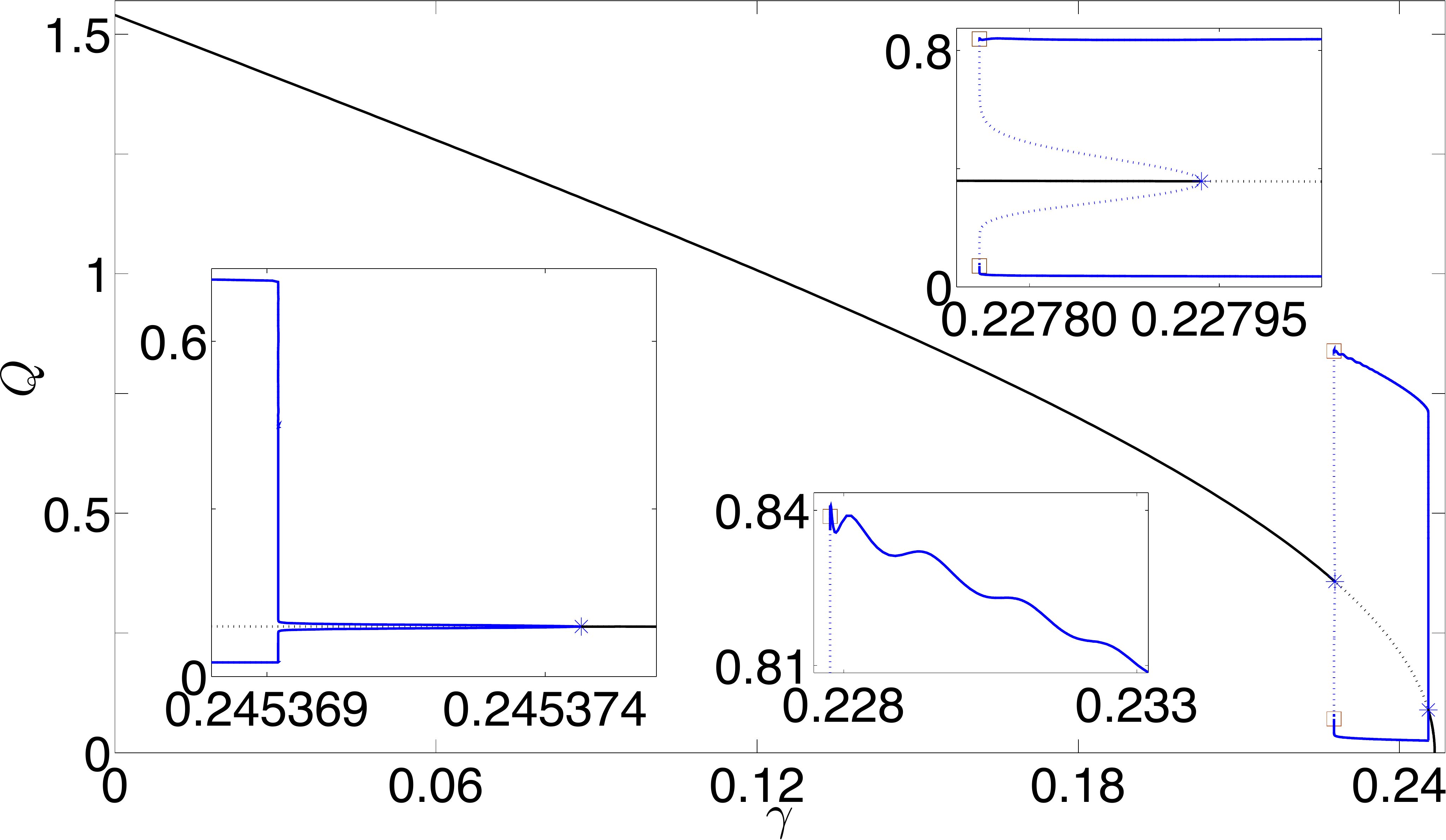}
\put(-245,2.5){\footnotesize{\textit{(i)}}}
\hspace*{1mm}
\includegraphics[width=0.49\textwidth,height=44mm]{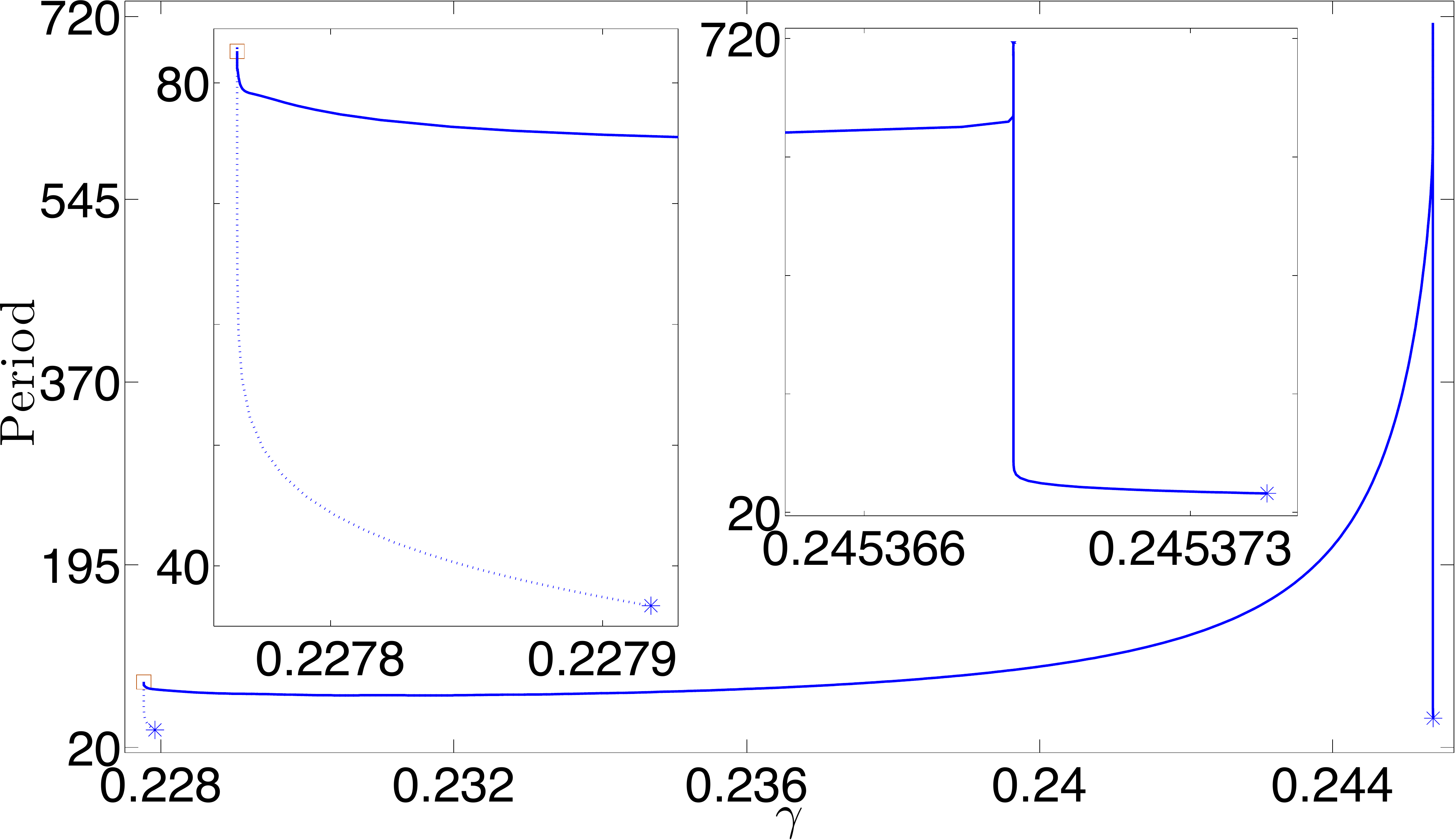}
\put(-245,2.5){\footnotesize{\textit{(ii)}}}
\vspace*{-2mm}
\caption{Continuation in $\gamma$ with other parameters taking values from Table~\ref{tab.model.par}.
(i) Bifurcation diagram showing stability of the steady state $Q^*$, along with the branch of periodic orbits
which bifurcates from the steady state at the Hopf bifurcation points at $\gamma=0.227918$ and $\gamma=0.245375$.
A saddle-node bifurcation seen in the top inset creates an interval for
$\kappa\in(0.227766,0.227918)$ of bistability between the periodic orbit and the steady state $Q^*$.
(ii) Shows the period of the periodic orbits seen in (i).
}
\label{fig:GammaCont}
\end{figure}
\begin{figure}[ht]
\includegraphics[width=0.49\textwidth,height=44mm]{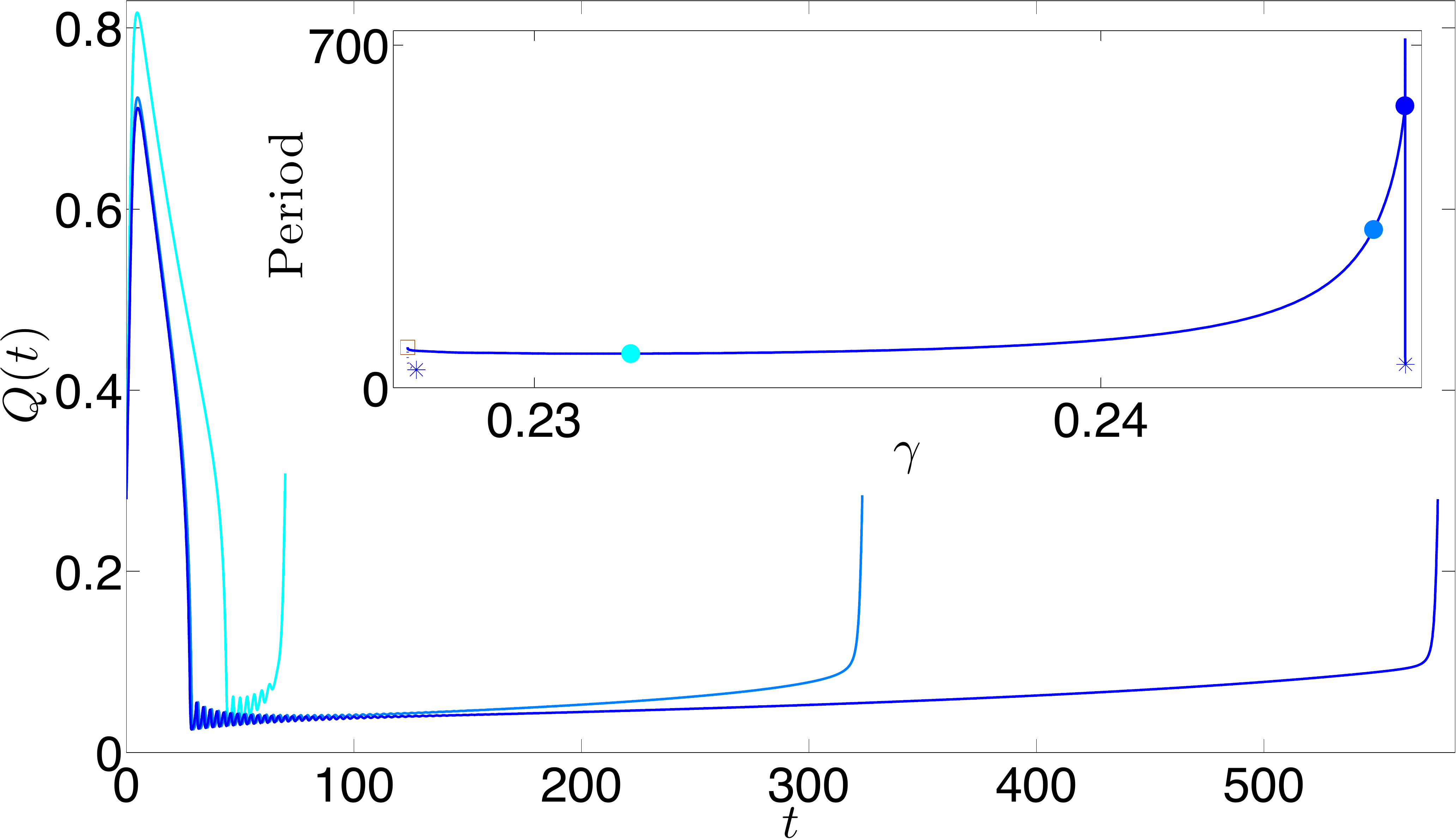}
\put(-245,2.5){\footnotesize{\textit{(i)}}}
\hspace*{1mm}
\includegraphics[width=0.49\textwidth,height=44mm]{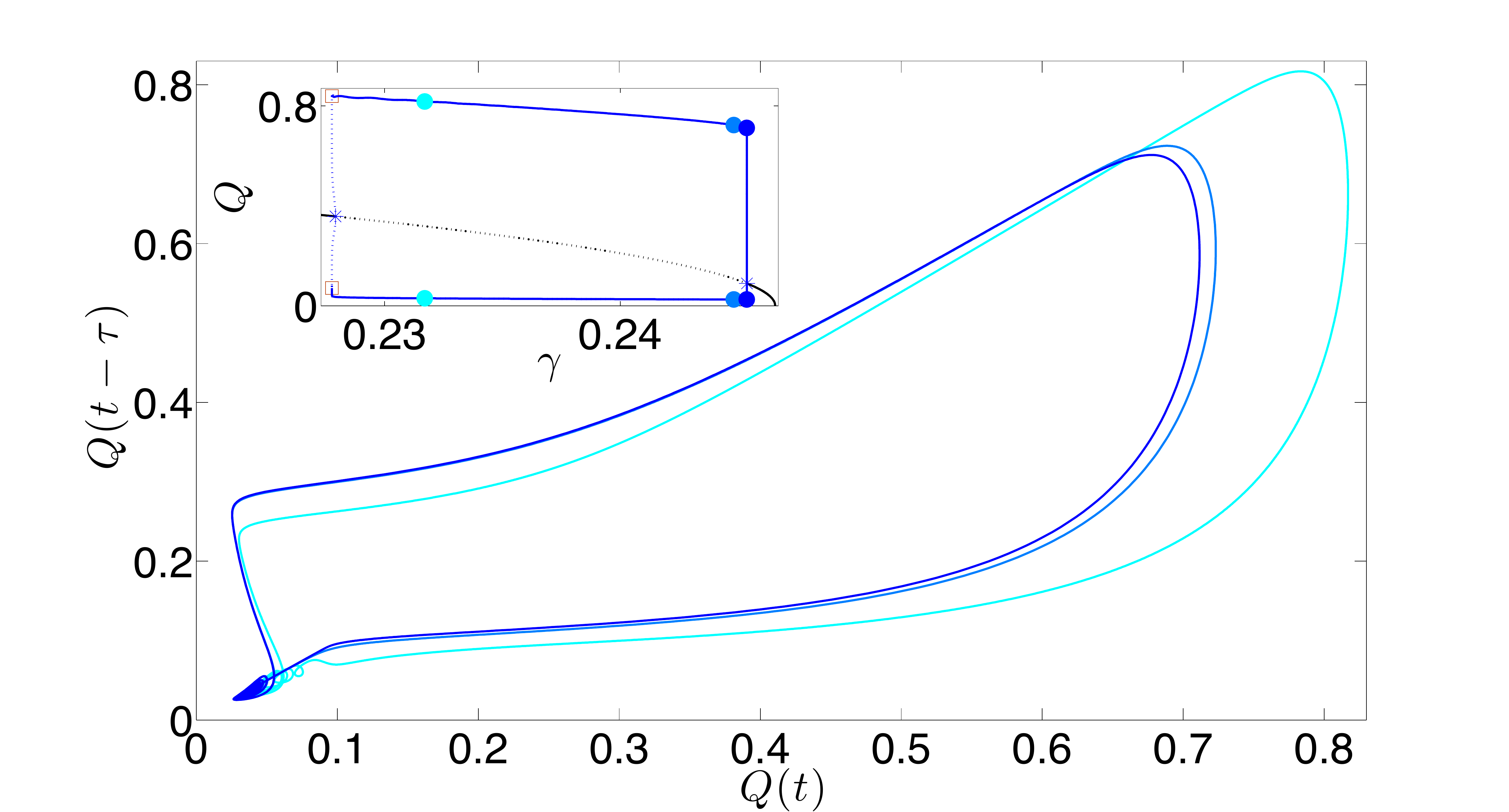}
\put(-245,2.5){\footnotesize{\textit{(ii)}}}
\vspace*{-2mm}
\caption{Three example periodic orbits from the branch shown in Figure~\ref{fig:GammaCont}.
(i) Solution profiles and (ii) the corresponding $(Q(t),Q(t-\tau))$ delay embeddings, with the location of the orbits
on the bifurcation branch indicated on the insets.}
\label{fig:GammaOrb}
\end{figure}

Figure~\ref{fig:GammaCont} shows the results of applying
one parameter continuation in the apoptosis rate $\gamma$, with the other parameters all held at their values in
Table~\ref{tab.model.par}. The steady state $Q^*$ is stable unless $\gamma$ is close to its
upper bound defined by~\eqref{kgt}. There is again a pair of Hopf bifurcations with the left bifurcation
at $\gamma\approx0.227918$ subcritical leading to an unstable periodic orbit with period $\approx36.7$ days at the bifurcation point.
The period grows to about $82$ days at a fold bifurcation of periodic orbits with $\gamma\approx0.227766$.
The periodic orbit becomes stable at the fold bifurcation leading to a very short
interval of bistability between the stable periodic orbit and the stable steady state.
As $\gamma$ is increased from the fold bifurcation the stable periodic orbits gradually decrease in amplitude but increase
in period reaching a maximum period of about $714$ days when
$\gamma\approx0.2453692$. Some of these stable periodic orbits are illustrated in Figure~\ref{fig:GammaOrb}.
These orbits all have a single peak above $Q^*$ which is only achieved once per period. After this peak the value of $Q$ quickly drops to below $0.1$, and there is then a very low amplitude oscillation in $Q$ with a period of about
$3$ days (slightly larger than the delay $\tau=2.8$) which decays in amplitude before the next spike in the number of HSCs. At $\gamma=0.2453692$ equation~\eqref{Q.star} gives that $Q^*=0.089673$ and the longest period orbit shown in
Figure~\ref{fig:GammaOrb} is close to homoclinic to the steady state $Q^*$ ($Q'(t)\approx0.00025$ when $Q(t)=Q^*$).

The long-period orbits appear to be relaxation oscillations; these have been observed and studied previously for~\eqref{Qprime}
~\cite{Colijn2006,Fowler_Mackey_2002}.
Visually, these
solution profiles are more reminiscent of a spiking neuron~\cite{Izhikevich_2007} than what one would naively expect to see in blood cell concentrations.
After the maximum period is achieved at $\gamma\approx0.2453692$ the period declines precipitously to approximately $48$ days at the Hopf bifurcation when $\gamma\approx0.245375$ in an apparent canard explosion.  We
are not aware of a canard being observed in a scalar DDE before, and we will investigate
this phenomenon in Section~\ref{sec.longp.hsc}.
\begin{figure}[t]
\includegraphics[width=0.49\textwidth,height=44mm]{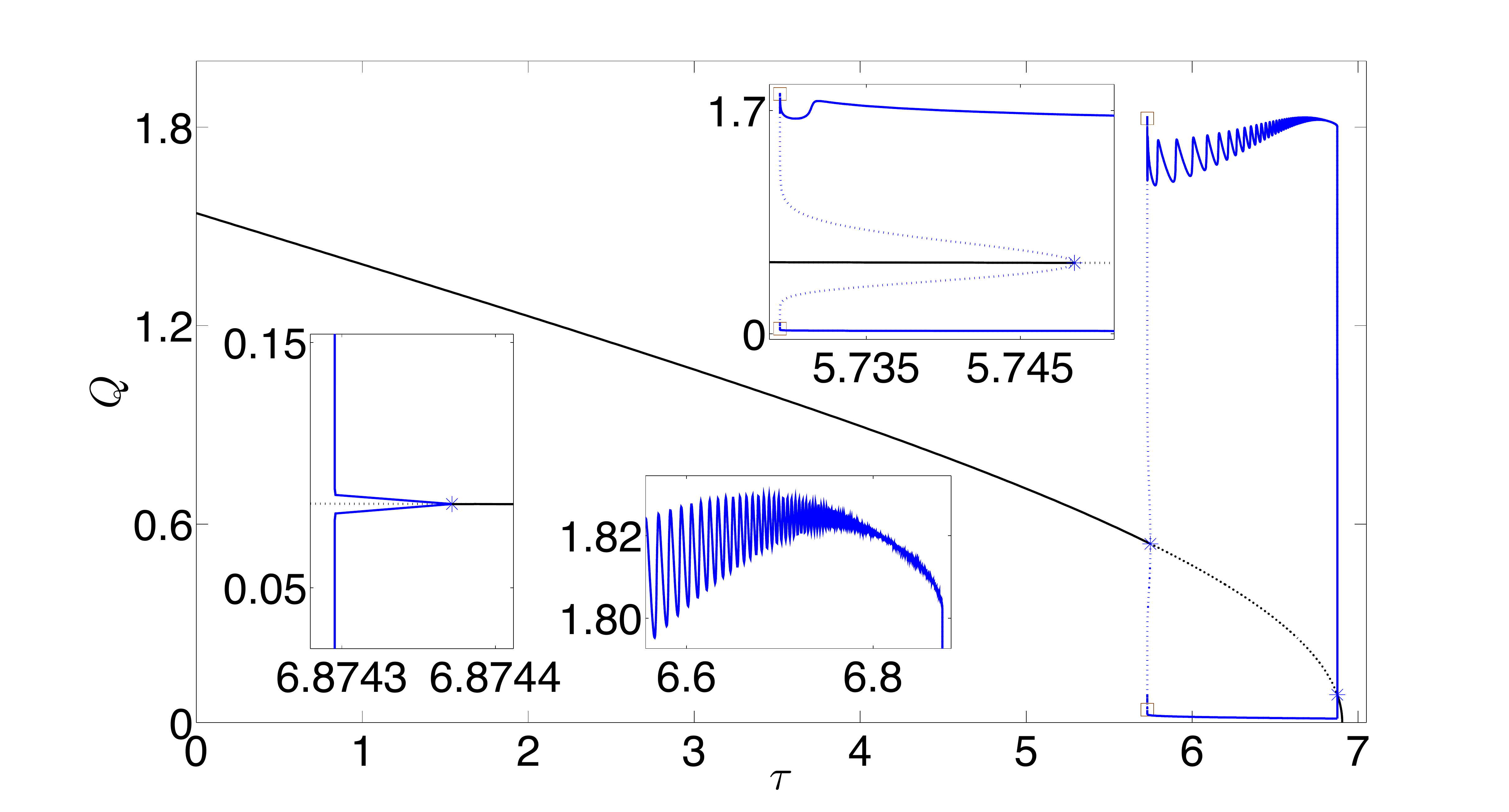}
\put(-245,2.5){\footnotesize{\textit{(i)}}}
\hspace*{1mm}
\includegraphics[width=0.49\textwidth,height=44mm]{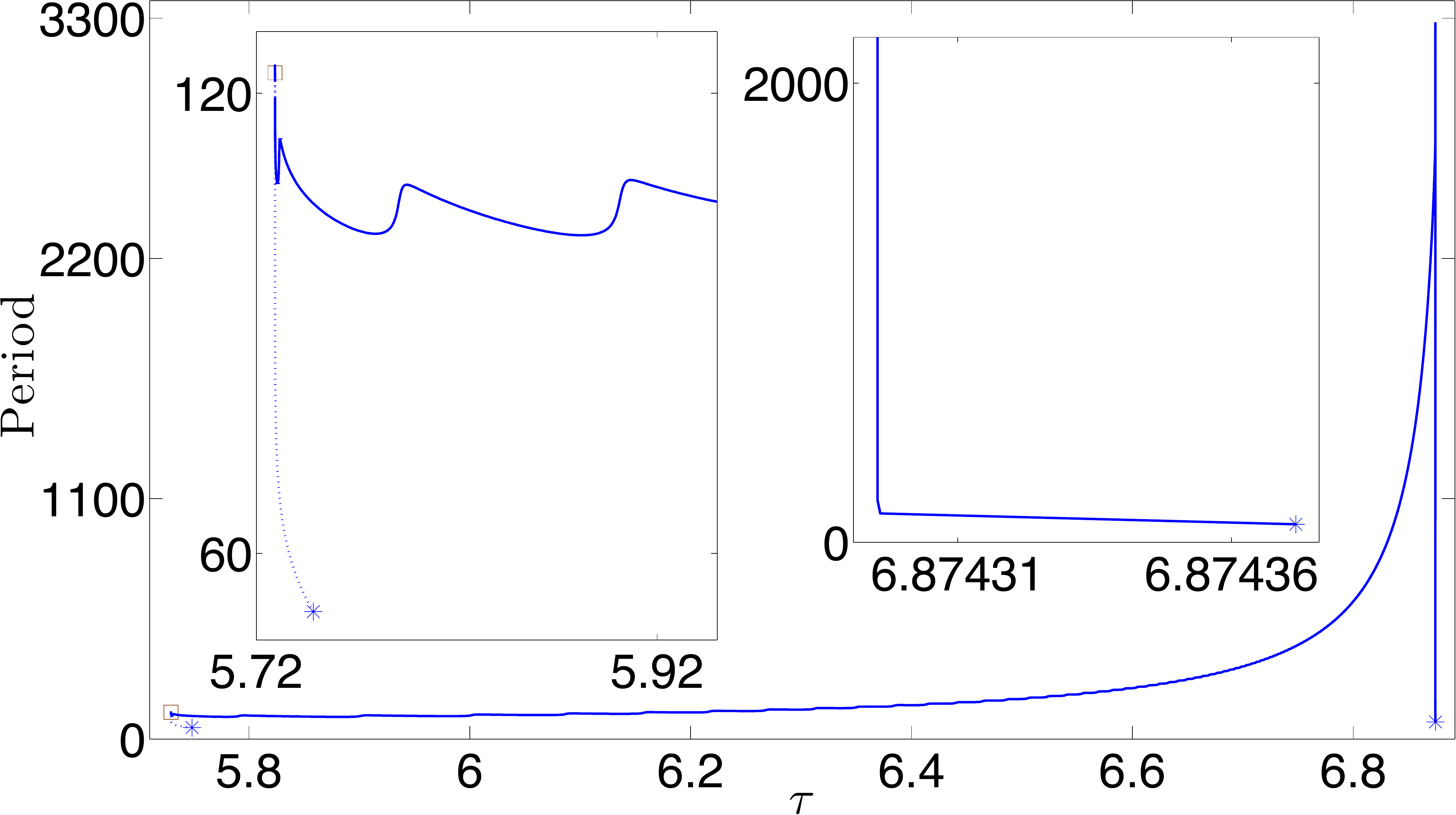}
\put(-245,2.5){\footnotesize{\textit{(ii)}}}
\vspace*{-2mm}
\caption{Continuation in $\tau$ with other parameters taking values from Table~\ref{tab.model.par}.
(i) Bifurcation diagram showing stability of the steady state $Q^*$, along with the branch of periodic orbits
which bifurcates from the steady state at the Hopf bifurcation points at $\tau=5.74851$ and $\tau=6.87437$.
A saddle-node bifurcation seen in the top inset creates an interval for
$\tau\in(5.72939,5.74851)$ of bistability between the periodic orbit and the steady state $Q^*$.
(ii) Shows the period of the periodic orbits seen in (i).}
\label{fig:TauCont}
\end{figure}
\begin{figure}[t]
\includegraphics[width=0.49\textwidth,height=44mm]{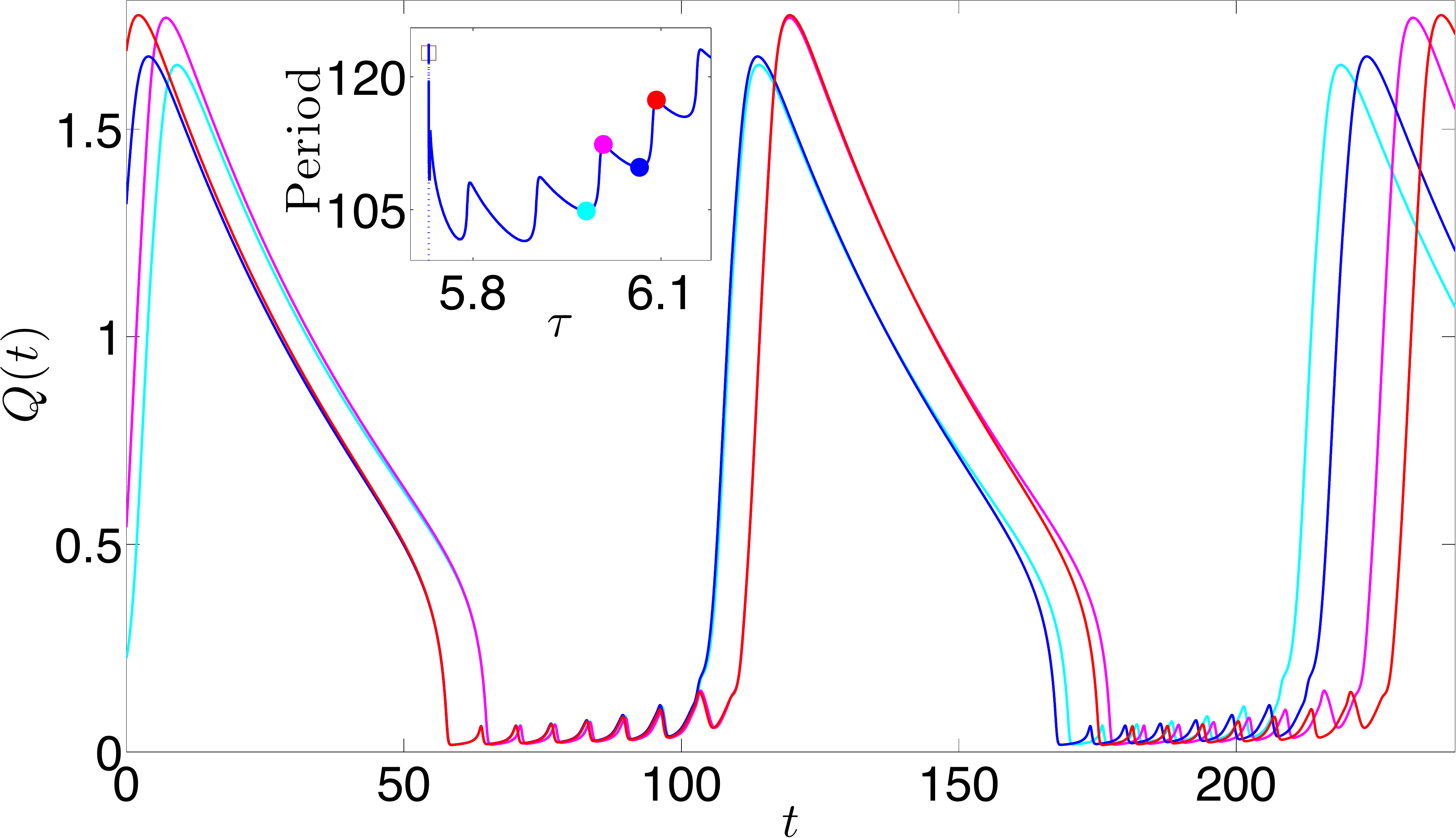}
\put(-245,2.5){\footnotesize{\textit{(i)}}}
\hspace*{1mm}
\includegraphics[width=0.49\textwidth,height=44mm]{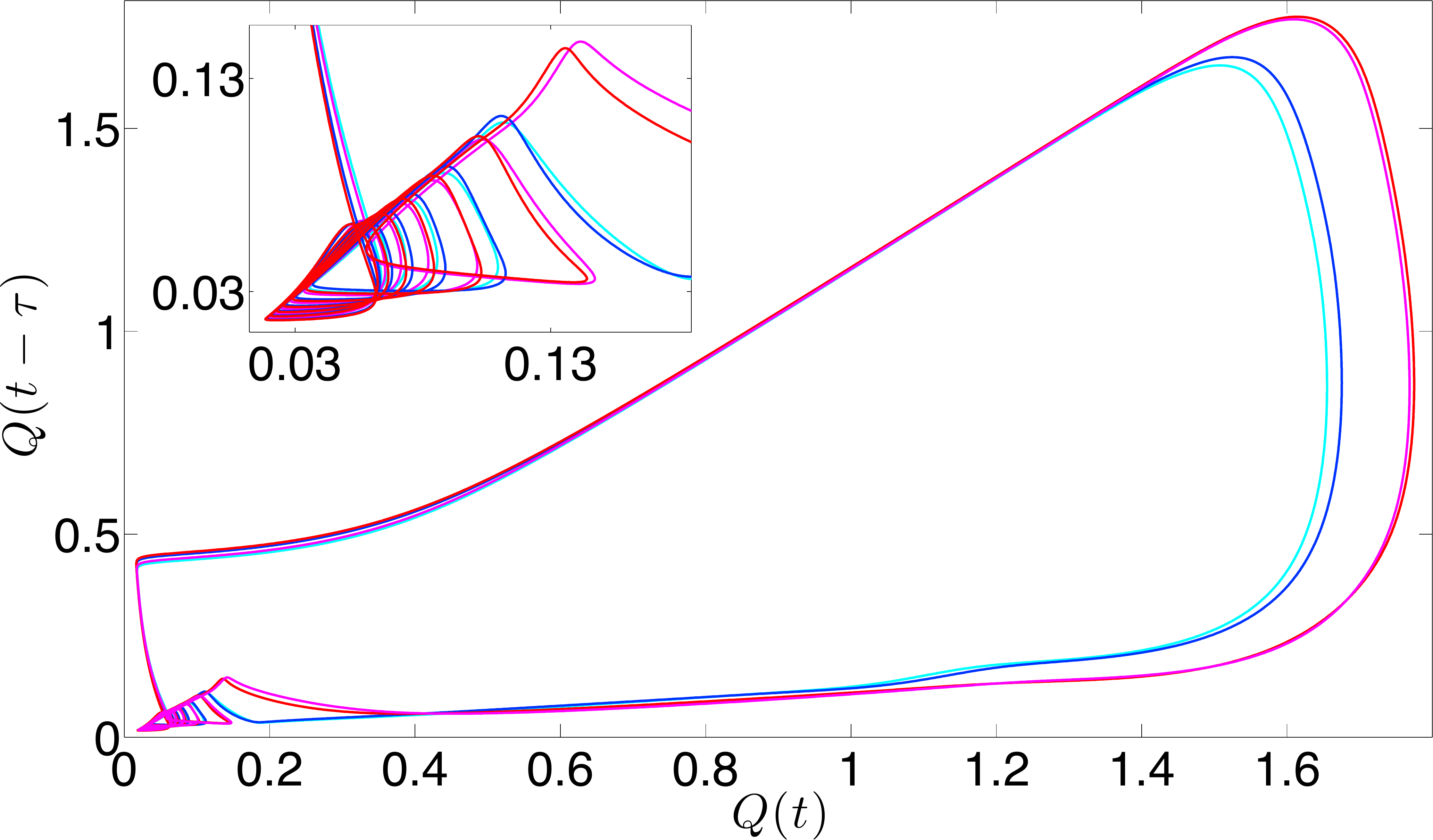}
\put(-245,2.5){\footnotesize{\textit{(ii)}}}
\vspace*{-2mm}
\caption{Four stable periodic orbits from neighbouring peaks and troughs of the ripples illustrate how they arise.
As seen in the inset of the time-embedding plot, the large period and amplitude orbits from the peaks of the ripples perform one more small amplitude oscillation before escaping to $Q(t)\gg Q^*$ compared to the smaller period and amplitude solution from the preceding trough in the ripples.
}
\label{fig:TauOrb}
\end{figure}

In Figure~\ref{fig:TauCont} we present the dynamics observed from applying one parameter continuation in the delay $\tau$, with the other parameters held at their values in Table~\ref{tab.model.par}. The steady state $Q^*$ is
again stable unless the delay $\tau$ is close to its upper bound,
and only unstable for $\tau$ between the pair of Hopf bifurcation points, which occur at $\tau=\tau_1^-$ and $\tau=\tau_1^+$, where $\tau^\pm$ are defined in \eqref{hstabtau}.
The left Hopf point is subcritical, leading to unstable orbits of period about $52$ days, growing to a period
of about $124$ days at a saddle-node bifurcation of periodic orbits with $\tau\approx5.72940$ days, where the periodic orbits become stable, creating an interval of bistability between the steady state and the stable periodic orbits.
Ripples are visible in the amplitude of the branch of stable periodic orbits, with the magnitude of these undulations decreasing to zero as $\tau$ approaches the right Hopf bifurcation point, as shown in the bottom right inset of Figure~\ref{fig:TauCont}(i). There are corresponding ripples in the period of the orbits, visible in the first inset of Figure~\ref{fig:TauCont}(ii).
The other insets show details of the branch of periodic orbits near the Hopf bifurcation points.

Figure~\ref{fig:TauOrb} illustrates stable periodic orbits from the left part of this branch. Although these orbits superficially resemble those of Figure~\ref{fig:GammaOrb}, with a single peak above $Q^*$ and a small amplitude oscillation close to $Q=0$, the periodic orbits seen in Figure~\ref{fig:TauOrb} have a quite different character to those seen
in Figure~\ref{fig:GammaOrb}.
Specifically the orbits have a growing oscillation close to $Q=0$ with a period close to the delay $\tau$. In contrast, the solutions seen in Figure~\ref{fig:GammaOrb} have a decaying oscillation near their minimum value, which is not particularly close to $Q=0$.
Figure~\ref{fig:TauOrb}(ii) shows the delay embedding of the solutions, from which we see that the ripples in the amplitude and period along the branch are related to the number of low amplitude oscillations in the solution, with the smaller amplitude and period solutions at the bottom of the ripples performing one less oscillation in the $(Q(t),Q(t-\tau))$ projection before escaping to $Q(t)\gg Q^*$.

The dynamics of the oscillations close to $Q(t)=0$ are easy to describe but harder to explain. The trivial steady state $Q=0$ has one positive real characteristic value ($\lambda=0.017605$ when $\tau=6$) and infinitely many complex conjugate characteristic values, three pairs of which have positive real part. Thus the steady state $Q=0$ of the DDE
\eqref{Qprime} has a seven-dimensional unstable manifold and an infinite-dimensional stable manifold in the full infinite dimensional phase space of the functional differential equation. However, complex characteristic values would give rise to oscillatory solutions about $Q=0$ which change sign. Since, by Theorem~\ref{theorem.bound}, solutions with positive initial conditions remain positive, oscillations about $Q=0$ will not arise with physiological initial history functions. (Remark: This implies that for any complex characteristic value $\lambda=\alpha\pm i\omega$ that
$|\omega|\geq\pi/\tau$, since otherwise $\tau$ would be less than half the period of the oscillation, and the positive half of the oscillation could be used to define an initial function $\varphi=-\epsilon e^{\alpha t}\sin(\omega t)$ for $t\in[-\tau,0]$ so the DDE that would have a solution close to $Q(t)=-\epsilon e^{\alpha t}\sin(\omega t)$ for $0<t\ll 1$ which would violate the positivity of solutions). Consequently, in the restricted phase space of positive solutions that we consider $Q=0$ has a one-dimensional unstable manifold and a trivial stable manifold. \label{positive}

In the inset of Figure~\ref{fig:TauOrb}(ii) we see that in the delay embedding $(Q(t),Q(t-\tau))$ that $Q(t)$ takes
its minimum value on the periodic orbit when $Q(t-\tau)$ is close to $0.5$ but decreasing. $Q(t)$ then increases
slightly before decreasing again to its next minimum which occurs very close to the minimum of $Q(t-\tau)$ on the
solution. This sets in train a clockwise oscillation in the $(Q(t),Q(t-\tau))$ projection close to $0$. Along this
oscillation the local minima of $Q(t)$ and $Q(t-\tau)$ occur very close to each other in time because the period of this low
amplitude oscillation (as seen in Figure~\ref{fig:TauOrb}(i)) is very close to the delay $\tau$. After the double local
minima the solution grows close to the local unstable manifold of $Q=0$ with $Q(t-\tau)\approx Q(t)e^{-\lambda\tau}$
for a time, until $Q(t-\tau)$ starts to decrease again (towards the previous local minima of $Q(t)$), after which
$Q'(t)$ becomes negative and $Q(t)$ decreases to its next local minima, completing one cycle. The amplitude of this
oscillation grows slightly with each subsequent cycle, until eventually the oscillation escapes to $Q(t)\gg Q^*$.

As $\tau$ increases across the branch of stable orbits the character of the periodic orbits changes
(not illustrated), with the growth rate of the small amplitude oscillations progressively decreasing and the period of the orbit increasing. For $\tau$ sufficiently large the small amplitude oscillations decay instead of grow, and thereafter the periodic orbits resemble the longest period orbit shown in Figure~\ref{fig:GammaOrb}(i).
The period of the orbit, but not the amplitude, continues to grow until the amplitude and period of the solutions
decreases abruptly just before the right bifurcation point, apparently in a canard explosion.
The period reaches its maximum value of $3281$ days for $\tau\approx 6.874295373$ and decreases dramatically to $182$ days while the value of $\tau$ remains constant to 10 significant digits
(see right inset of Figure~\ref{fig:TauCont}(ii)).

Although we do not find a homoclinic bifurcation, the longest period orbit is about 9 years, nearly $500$ times larger than the delay in the system, with the orbit close to homoclinic to the non-trivial steady state $Q^*$.
The solutions in this region are similar to the long-period orbits displayed in Figure~\ref{fig:GammaOrb}.



\subsection{Two-parameter continuation}
\label{sec.2p}
\begin{figure}[t]
\includegraphics[width=0.79\textwidth]{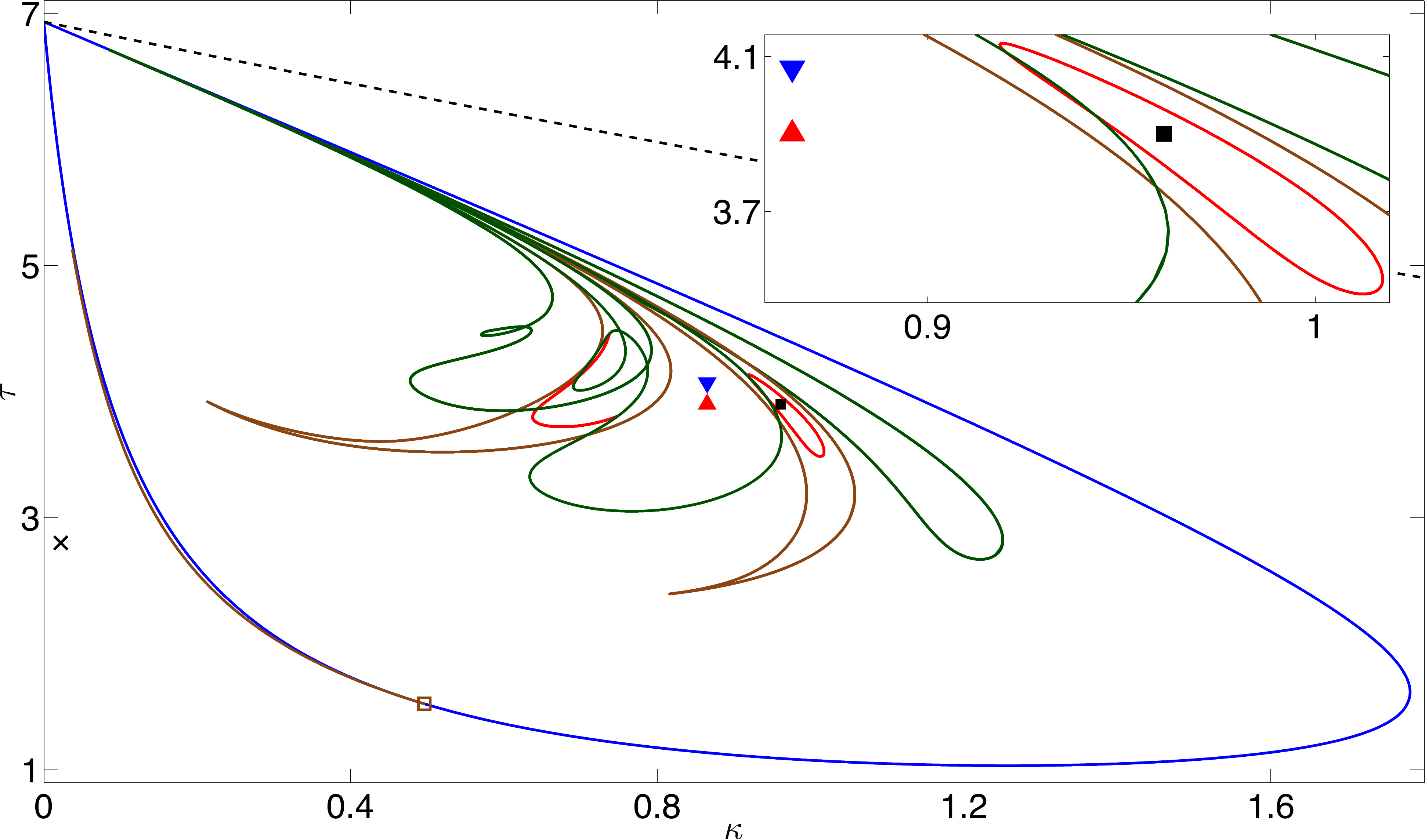}\;\raisebox{0.4cm}{\includegraphics[width=0.2\textwidth]{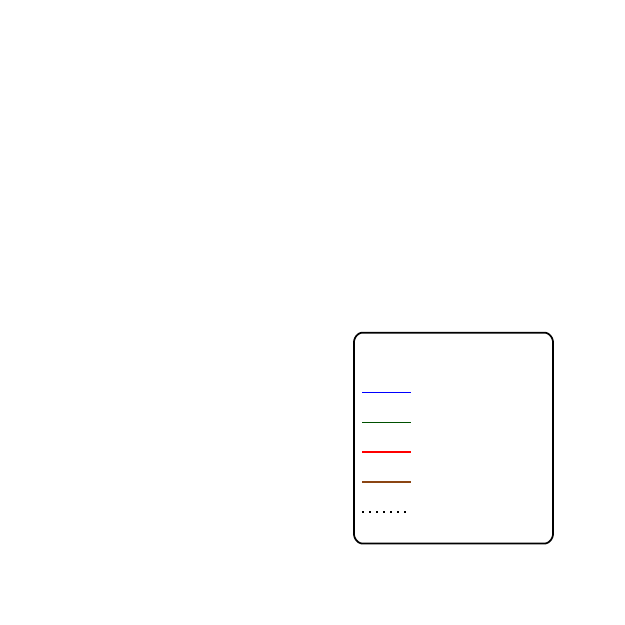}}
\put(-80,102){\footnotesize{\emph{Bifurcation Key}}}
\put(-63,86){\footnotesize{\emph{Hopf}}}
\put(-63,71){\scriptsize{\emph{period-doubling}}}
\put(-63,55){\footnotesize{\emph{torus}}}
\put(-63,42){\footnotesize{\emph{fold}}}
\put(-63,27){\footnotesize{\emph{steady state}}}
\vspace*{-3mm}
\caption{Two parameter $(\kappa,\tau)$-bifurcation diagram for the DDE~\eqref{Qprime} with other parameter values given by Table~\ref{tab.model.par}.
Each solid curve represents the locus of bifurcation of periodic orbits, as indicated in the key.
The brown square \textlarger[0]{$\textcolor{brown1}{\Box}$} denotes a Bautin bifurcation; see text.
The black cross marks the homeostasis values of the parameters. Within the inset, the black square indicates the parameter values of the torus investigated in Section~\ref{sec.torus}, and the red and blue triangles represent chaotic attractors seen in Section~\ref{sec.chaos}.}
\label{fig:KappaTau_2D_Cont}
\end{figure}
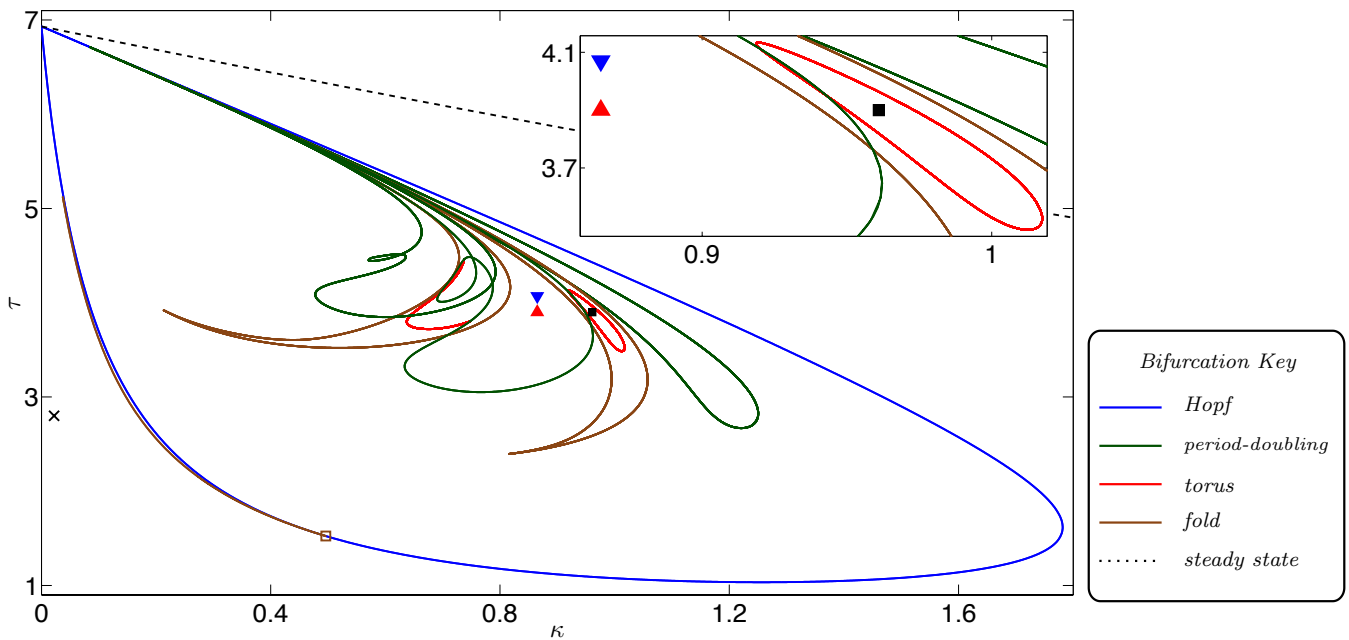

Recent versions of DDEBiftool~\cite{Engelborghs_2002,DDEBiftool15} have the facility to perform two-parameter continuation
of bifurcations of periodic orbits, and we used this to study the bifurcations of
the DDE~\eqref{Qprime} as the parameters $\kappa$, $\gamma$ and $\tau$ are varied pairwise.

In Figure~\ref{fig:KappaTau_2D_Cont} we present the two-parameter
bifurcation diagram as $\kappa$ and $\tau$ are varied, with all other parameters at their values in Table~\ref{tab.model.par}, which
reveals the curves of Hopf, period-doubling, saddle-node, torus and steady state bifurcation as this pair of parameters are varied. Taking a straight line through Figure~\ref{fig:KappaTau_2D_Cont} with $\tau=2.8$ or with $\kappa=0.022$ reveals the bifurcations found in Figures~\ref{fig:KappaCont} and~\ref{fig:TauCont} respectively. From
Figure~\ref{fig:KappaTau_2D_Cont} we see that the homeostasis parameters $(\kappa,\tau)=(0.022,2.8)$ are not particularly close to any bifurcations,
with the Hopf curve and an associated curve of saddle-node of limit-cycle bifurcations
being the only other bifurcations near to that part of parameter space.
The Hopf bifurcations are subcritical to the left of the
Bautin or generalised Hopf bifurcation
at $(\kappa,\tau)=(0.4960,1.525)$ and supercritical otherwise.
We already saw instances of the subcritical Hopf bifurcations in Figures~\ref{fig:KappaCont} and~\ref{fig:TauCont}; Bernard \textit{et al}~\cite{Bernard_2004} previously presented an example
with both Hopf bifurcations supercritical.

If the delay is small ($\tau<1$) there are no bifurcations at all, while the bifurcation structures become more complicated as $\tau$ is increased with two curves of fold bifurcations of periodic orbits created in a cusp bifurcation at
$(\kappa,\tau)\approx(0.81600,2.3956)$, and a further cusp bifurcation and torus bifurcation curves only occurring for $\tau>3$.
Figure~\ref{fig:KappaTau_2D_Cont} suggests that for one parameter continuation in $\kappa$, taking $\tau$ close to $4$ will lead to more complicated dynamics than was seen in Figure~\ref{fig:KappaCont} for $\tau=2.8$. Indeed, the curves seen in Figure~\ref{fig:KappaTau_2D_Cont} were seeded by performing a one parameter continuation in $\kappa$ with $\tau=3.9$ (see Figure~\ref{fig2_HSC_Biftool1D_KappaDA}) and consequently Figure~\ref{fig:KappaTau_2D_Cont} shows all the bifurcation curves that cross $\tau=3.9$. There may be other bifurcation curves that remain above $\tau=3.9$, but it appears from Figure~\ref{fig:KappaTau_2D_Cont} that they would be constrained to be near the right Hopf bifurcation.

\begin{figure}[t]
\includegraphics[width=0.79\textwidth]{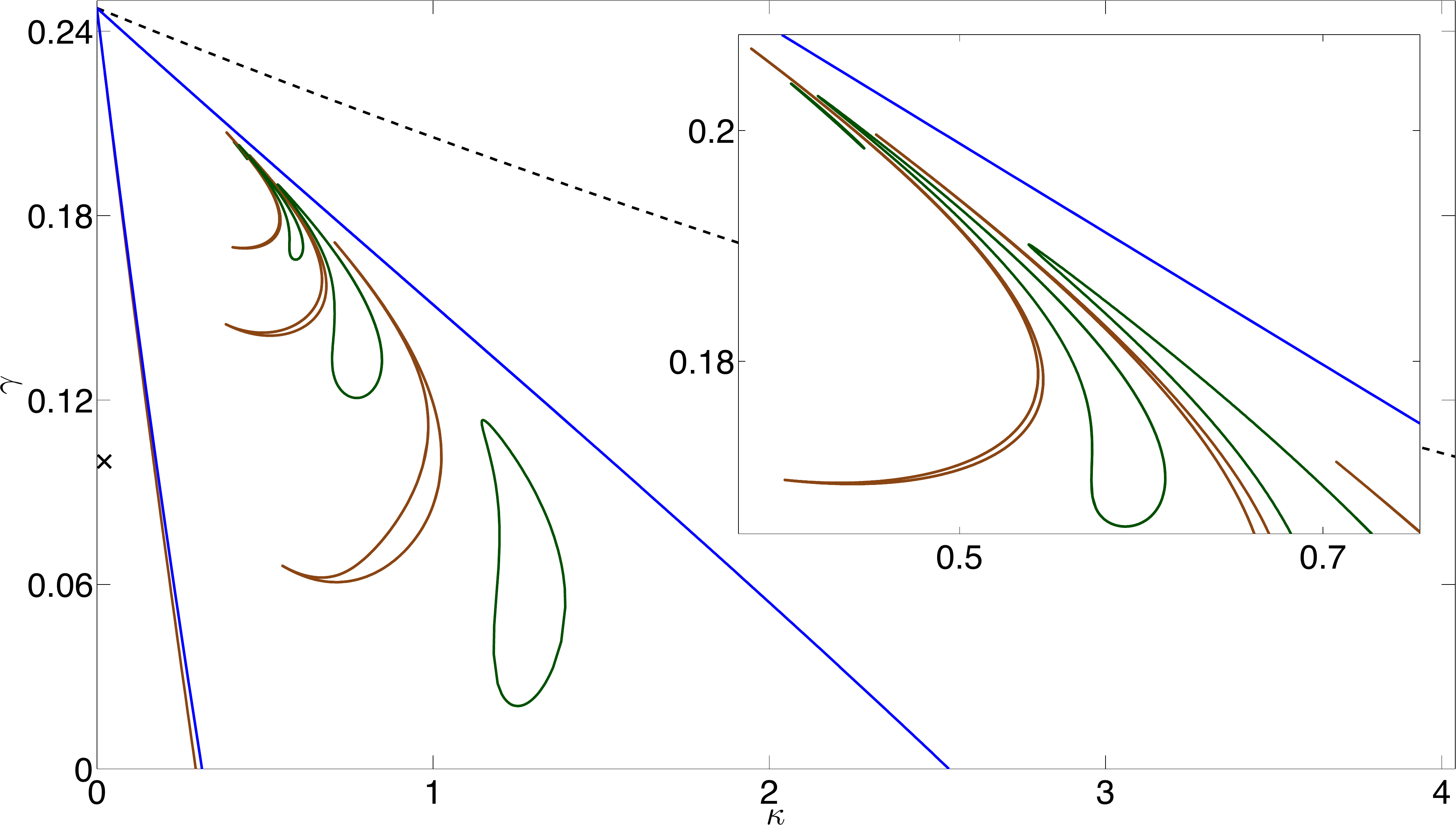}\;\raisebox{0.4cm}{\includegraphics[width=0.2\textwidth]{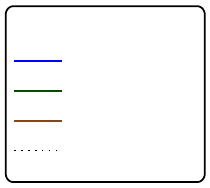}}
\put(-80,83){\footnotesize{\emph{Bifurcation Key}}}
\put(-61,70){\footnotesize{\emph{Hopf}}}
\put(-61,55){\scriptsize{\emph{period-doubling}}}
\put(-61,42){\footnotesize{\emph{fold}}}
\put(-61,27){\footnotesize{\emph{steady state}}}
\vspace*{-3mm}
\caption{Continuation in $(\kappa,\gamma)$ 
with other parameters at homeostasis (see Table~\ref{tab.model.par}).
Each solid curve represents the locus of a certain type of bifurcation of periodic orbits as specified in
the key.
}
\label{fig:GammaKappa_2D_Cont}
\end{figure}
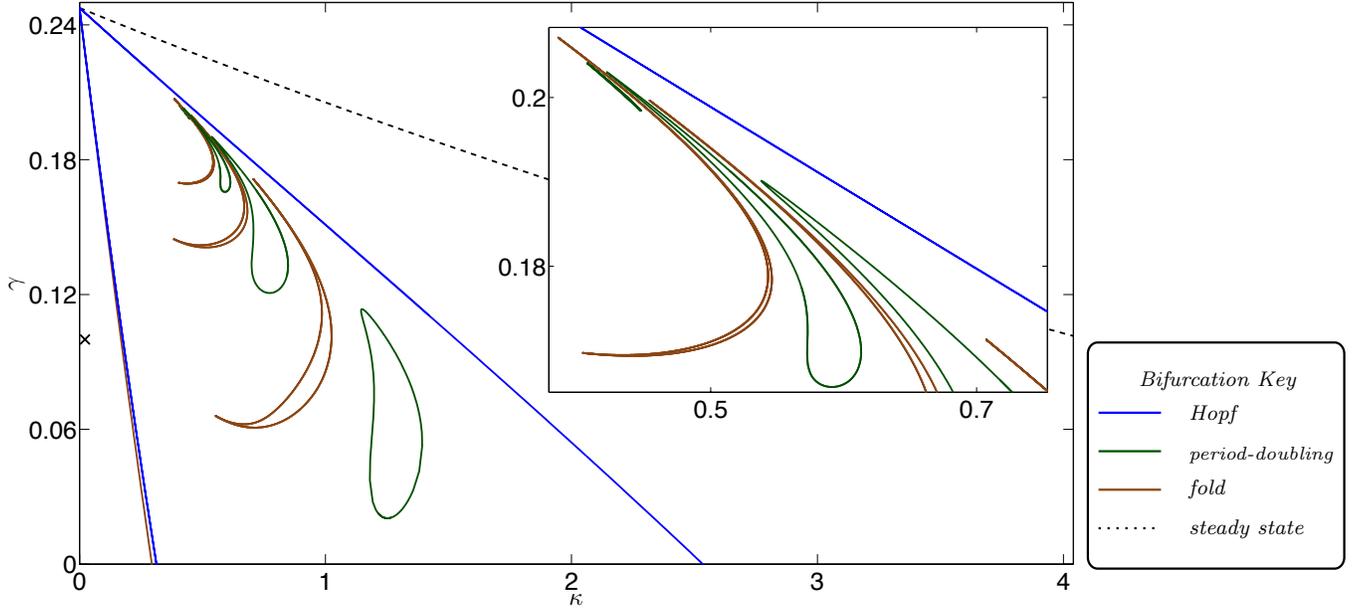

Figure~\ref{fig:GammaKappa_2D_Cont} shows the bifurcation curves found for two-parameter continuation in $(\kappa,\gamma)$ with all the other parameters taking their values from Table~\ref{tab.model.par}, revealing an alternating sequence of curves of period-doubling and fold bifurcations (of limit cycles), and associated cusp bifurcation of limit cycles.
The closed curves of bifurcations become shorter and narrower as parameters
approach the upper bound on $\gamma$ for periodic orbits to exist, and also progressively more delicate to compute numerically. There may be additional curves of bifurcations for $\gamma>0.2$ which we were not able to compute.
A straight line through Figure~\ref{fig:GammaKappa_2D_Cont} with $\gamma=0.1$ or with $\kappa=0.022$
reveals the bifurcations found in Figures~\ref{fig:KappaCont} and~\ref{fig:GammaCont}, respectively.
The inset reveals that the period-doubling and saddle-node loci do not overlap.



\begin{figure}[t]
\centering
\includegraphics[width=0.79\textwidth]{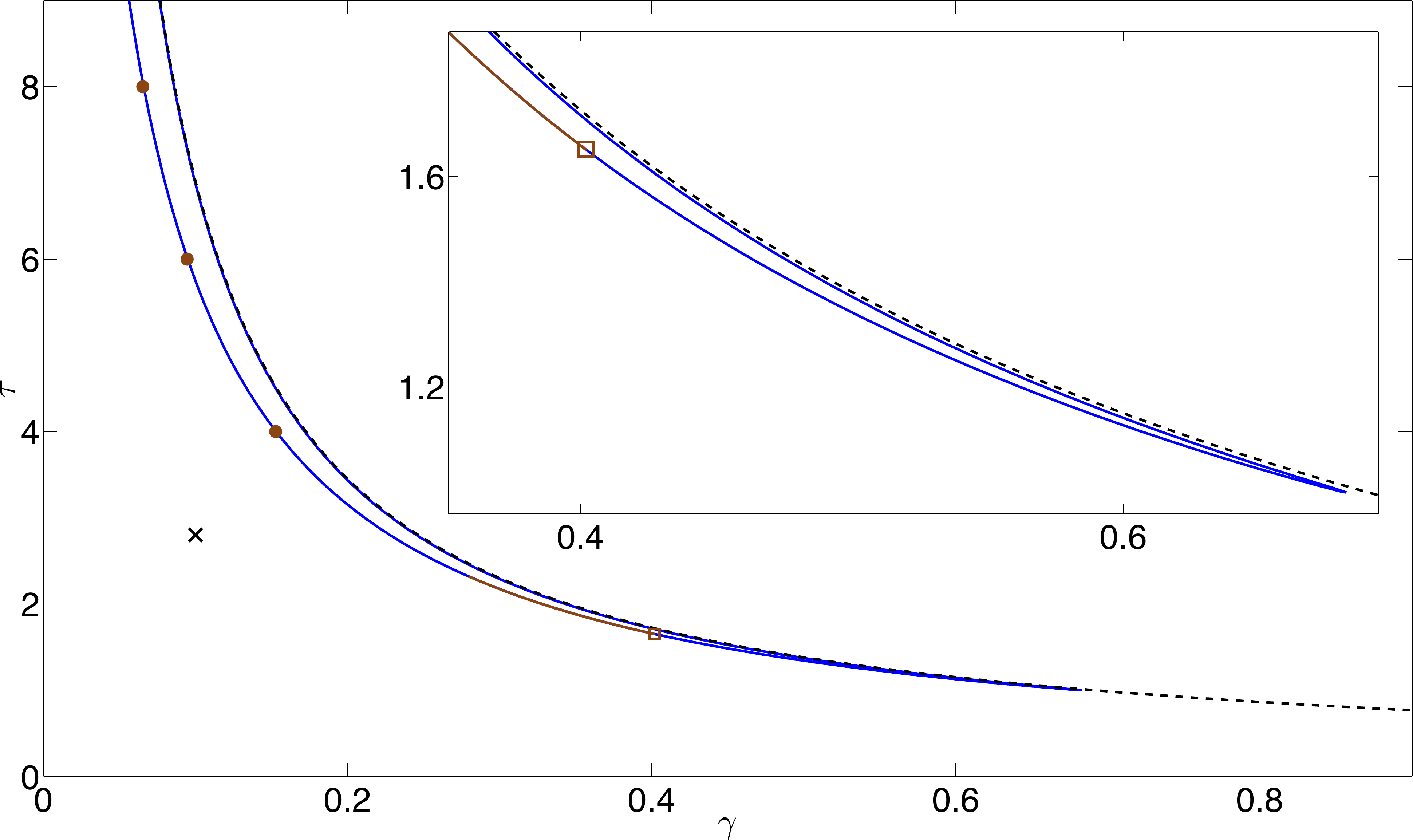}
\vspace*{-2mm}
\caption{Continuation in $(\gamma,\tau)$ 
with other parameters at homeostasis values.
Each solid curve represents the locus of a bifurcation of periodic orbits,
as specified in the key to Figure~\ref{fig:GammaKappa_2D_Cont}.
The brown square \textlarger[0]{$\textcolor{brown1}{\Box}$} denotes a Bautin bifurcation.
The brown dots and arc of brown curve emanating from the Bautin bifurcation form part of a curve of fold bifurcations of limit cycles, which DDEBiftool is only partially able to compute (see text).}
\label{fig:GammaTau_2D_Cont}
\end{figure}

The results of two-parameter continuation in $(\gamma,\tau)$ are shown in Figure~\ref{fig:GammaTau_2D_Cont}.
This reveals the locus of the Hopf bifurcations already observed in Figures~\ref{fig:GammaCont} and~\ref{fig:TauCont}. There is also a Bautin bifurcation at
$(\gamma,\tau)=(0.4020,1.651)$ and a branch of saddle-node bifurcations of limit cycles which emerges this point. This branch represents
the two-parameter continuation of the fold bifurcation seen in Figures~\ref{fig:GammaCont} and~\ref{fig:TauCont}. This bifurcation is very delicate for DDE-Biftool to compute and continue numerically, and we were not able to compute the full branch.
For $\tau=2.8$, $4$, $6$ and $8$ we performed one parameter continuation in $\gamma$ to confirm that the fold bifurcation persists for larger $\tau$ values and also to verify that there are not other bifurcation curves missing from the diagram. For the larger values of $\tau$, DDEBiftool is not able to identify the fold bifurcation. While we are able to find the fold from a one-parameter continuation by simply looking for the minimum value of $\gamma$ along the branch (and we added these points to Figure~\ref{fig:GammaTau_2D_Cont}), DDEBiftool computes and continues fold bifurcations of limit-cycles in two parameters by solving the defining equations for a fold bifurcation of periodic orbits~\cite{DDEBiftool15}, which is a considerably more complicated computation.


The two-parameter continuations in Figures~\ref{fig:KappaTau_2D_Cont},~\ref{fig:GammaKappa_2D_Cont} and~\ref{fig:GammaTau_2D_Cont} reveal that the non-trivial steady-state solution $Q^*>0$ remains
stable for all reasonably small perturbations from the
homeostasis parameter values of Table~\ref{tab.model.par}.
We also see from
Figures~\ref{fig:KappaTau_2D_Cont} and~\ref{fig:GammaTau_2D_Cont} that $Q^*$ is stable for all small delays $\tau<1$ (at least when the other parameters are varied one at a time). This suggests that an ODE model would not capture the instabilities driven by the delays. Since the cell-cycle time for stem cells is estimated to be much larger than one day ($2.8$ days in Craig~\cite{Craig_2016}) it is essential to include the delay in the DDE model~\eqref{Qprime} to properly capture the possible dynamics of the system.
That the two-parameter continuations in $(\kappa,\gamma)$ and $(\gamma,\tau)$ reveal less interesting bifurcation diagrams than for continuation in $(\kappa,\tau)$, is probably not intrinsic to the properties of the parameters in the model, but rather determined by the homeostasis value of the third parameter from Table~\ref{tab.model.par} when we perform two-parameter continuation. More complicated bifurcation diagrams can be generated by taking $(\kappa,\tau)$ close to
$(0.86,4)$, in the interesting part of Figure~\ref{fig:KappaTau_2D_Cont}, and then doing two-parameter continuation in any pair of these three parameters. For example, with $\kappa=0.68$, two-parameter continuation in $(\gamma,\tau)$ (not shown)
reveals torus and period-doubling curves, quite unlike anything seen in
Figure~\ref{fig:GammaTau_2D_Cont}. However, here we have based our continuation on using the homeostasis values of the parameters from Table~\ref{tab.model.par} to start one and two-parameter continuations. If we allowed all the parameters to vary its likely that we could find more exotic dynamics, but what the relationship, if any, that dynamics would have to
Burns-Tannock HSC model is not clear.



\section{Long Period Orbits and Canard Explosion}
\label{sec.longp.hsc}


A canard explosion is a dynamical phenomenon seen in fast-slow or singularly perturbed systems whereby over an
exponentially small range of the continuation parameter a periodic orbit is transformed into a long period relaxation oscillation.
For ODEs this requires at least two space dimensions, with classical examples being the van der Pol oscillator and Fitzhugh-Nagumo equations~\cite{Benoit1981,Wech13}.
Canard explosions have already been explored in DDEs~\cite{CSE09,KrupaTouboul2016}, but only in systems with at least two spatial dimensions
that incorporate a delay. However, since DDEs are inherently infinite dimensional there
is no reason why a canard explosion should not be seen in a scalar DDE such as~\eqref{Qprime}.

\begin{figure}[t]
\includegraphics[width=0.49\textwidth,height=44mm]{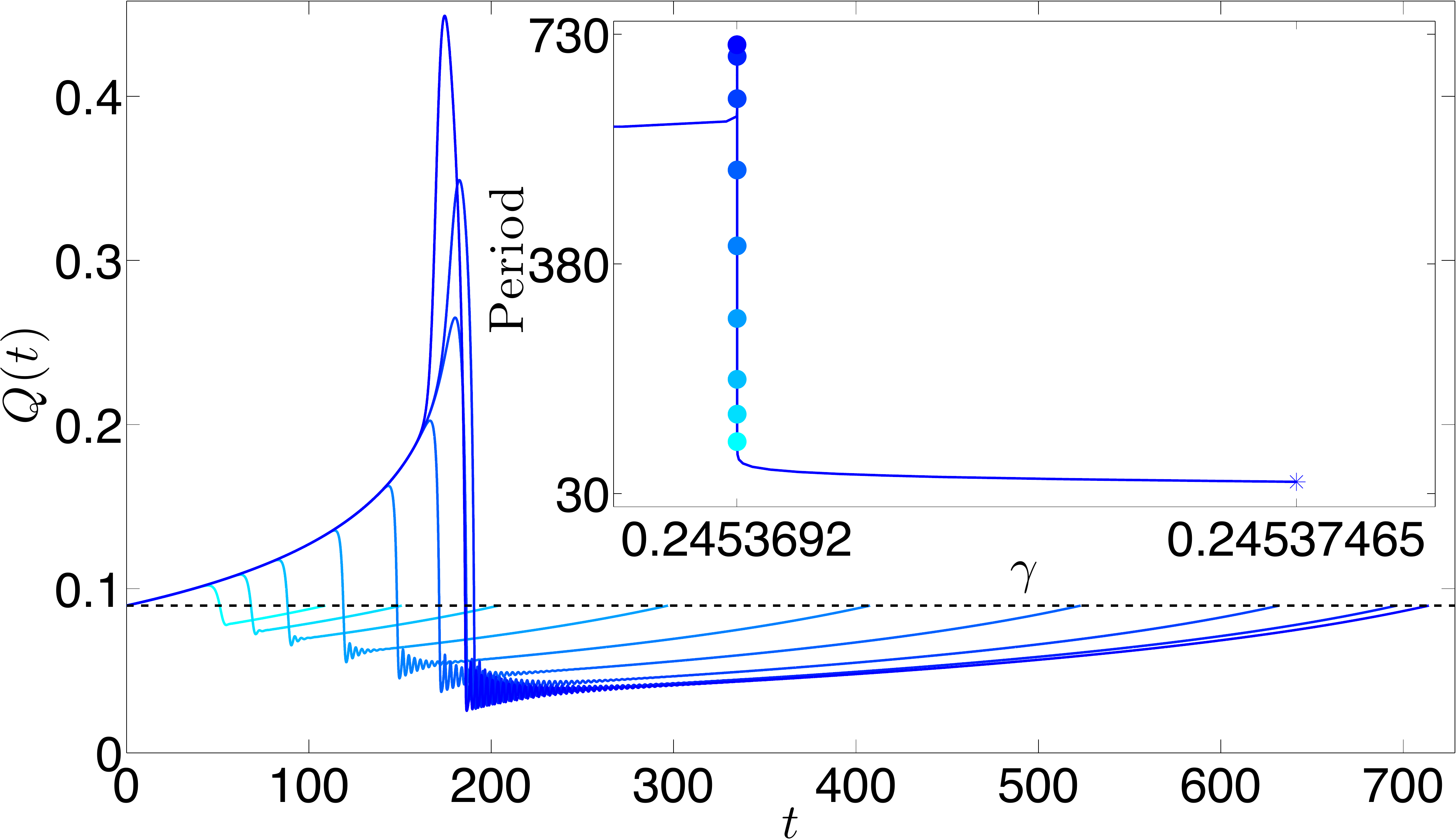}
\put(-245,2.5){\footnotesize{\textit{(i)}}}
\hspace*{1mm}
\includegraphics[width=0.4900\textwidth,height=44mm]{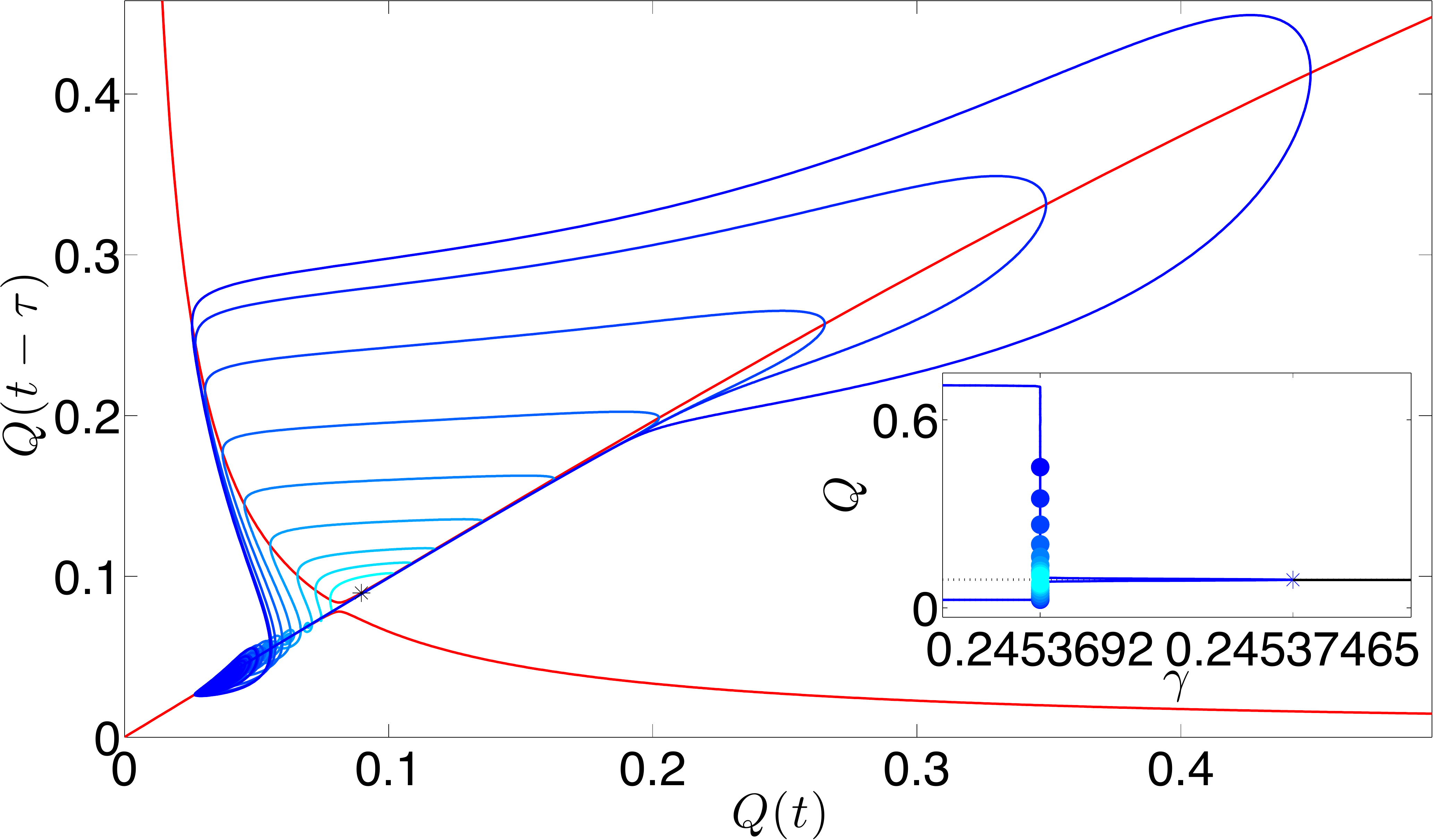}
\put(-245,2.5){\footnotesize{\textit{(ii)}}}

\vspace*{2mm}

\includegraphics[width=0.49\textwidth,height=44mm]{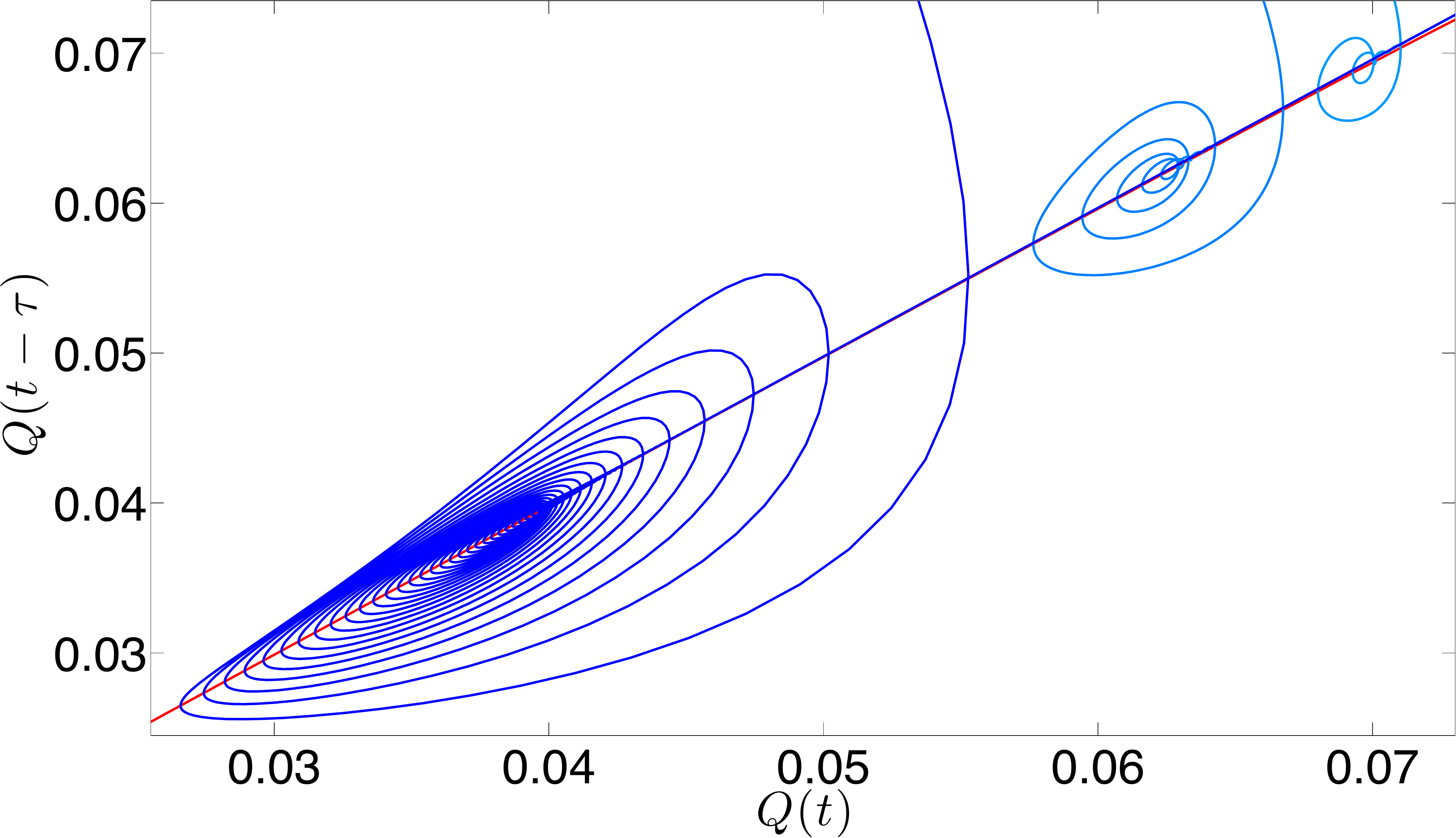}
\put(-245,2.5){\footnotesize{\textit{(iii)}}}
\hspace*{1mm}
\includegraphics[width=0.49\textwidth,height=44mm]{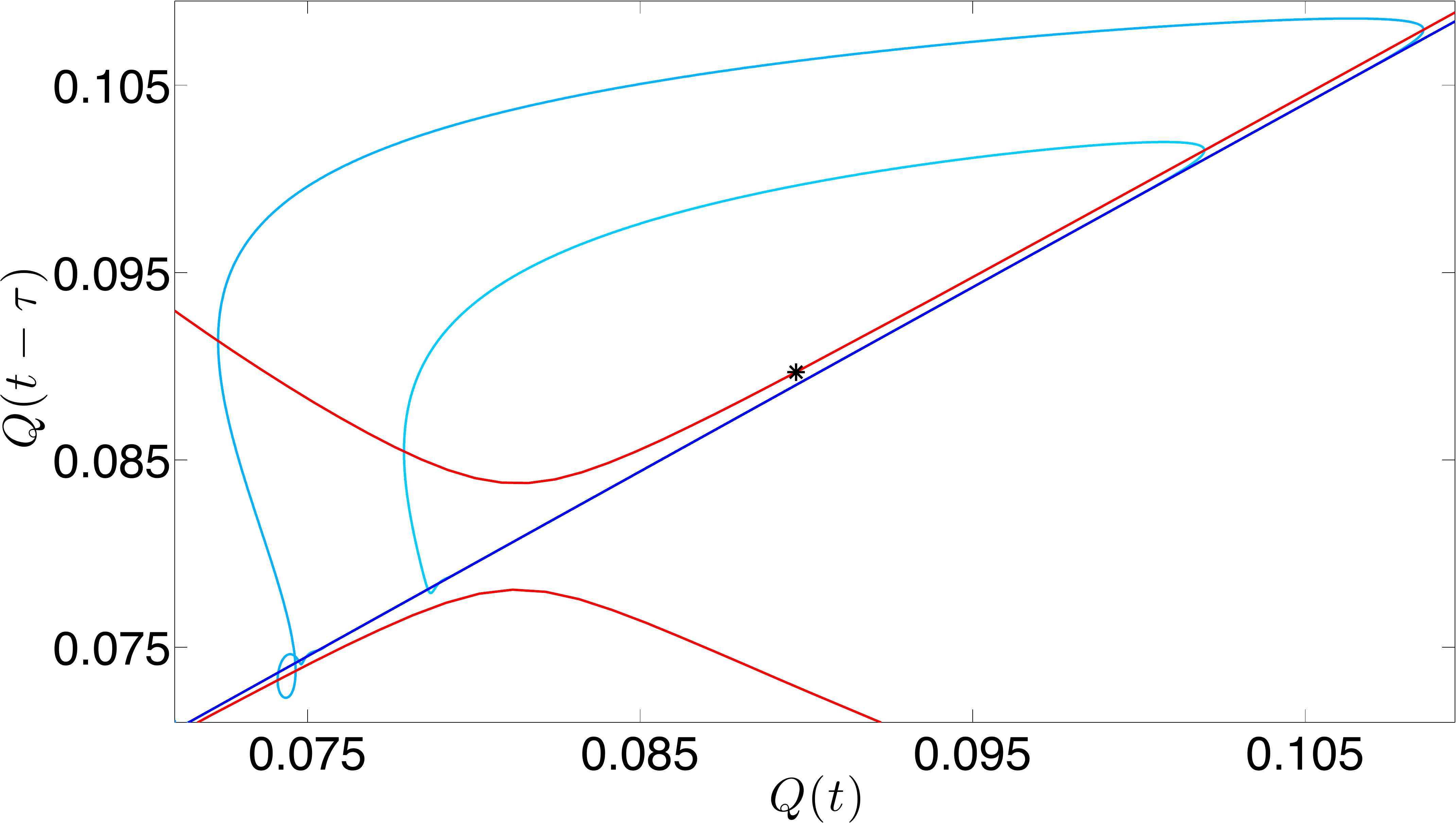}
\put(-245,2.5){\footnotesize{\textit{(iv)}}}
\vspace*{-2mm}
\caption{(i) Solution profiles and (ii) time-delay embedding $(Q(t),Q(t-\tau))$ of periodic solutions from the same $\gamma$ continuation shown in
Figure~\ref{fig:GammaCont}, show an apparent canard explosion at $\gamma\approx0.2453692$.
The insets indicate where the illustrated orbits lie on the branch. (iii) and (iv) shows parts of (ii) at a much larger scale. (iii) reveals the oscillatory convergence of the solutions onto a slow manifold for $Q<Q^*$. (iv) shows the  dynamics near to $Q^*$.
The nullcline $Q'(t)=0$ is indicated by the red curves in panels (ii)-(iv).}
\label{fig:GammaCont_Canard}
\end{figure}

Recalling the continuation in $\gamma$ shown in Figure~\ref{fig:GammaCont}, at the right Hopf bifurcation point
$\gamma\approx0.2453746$ a periodic orbit is born with period approximately $48$ days, but at
$\gamma\approx0.2453692$ the period increases dramatically to about $700$ days while
the value of $\gamma$ remains constant to 7 significant figures.
This would appear to be a canard explosion.
Figure~\ref{fig:GammaCont_Canard} illustrates orbits from this part of the branch as the period increases.
Comparing the time plot with the $(Q(t),Q(t-\tau))$ phase space projection, the slow manifold appears to be close
to $Q(t)=Q(t-\tau)$, with $Q'(t)$ gradually increasing along this curve, followed by a fast transition layer as $Q(t)$
decreases to close to its minimum value while $Q(t-\tau)$ remains close to its maximum. Then $Q(t-\tau)$ passes through
the transition layer to also be close to its minimum value after which there is a slowly decaying oscillation with a period of about $3$ days as the solution converges back to the slow manifold.
The largest period orbit illustrated has a period of about $701.3$ days, with $Q(t)$ crossing the steady state $Q^*$
once in each direction, with $Q(t)>Q^*$ for approximately $191.3$ days and $Q(t)<Q^*$ for the remaining $510.0$ days.
Here the delay $\tau=2.8$ days, so this is an example of a (very) slowly oscillating periodic solution.

Fast-slow systems in ODEs often have separate fast and slow variables which can be considered separately in the fast and slow subsystems. That separation of variables does not occur in the DDE~\eqref{Qprime}, for which we have only one variable. Nevertheless, relaxation oscillators with both fast and slow segments within the solution can arise and have been studied in DDEs, by tackling the fast and slow segments of the solution separately. In particular a relaxation oscillator for the HSC DDE~\eqref{Qprime} has been studied using singular perturbation analysis~\cite{Colijn2006,Fowler_Mackey_2002}. Of particular note is the extensive work of Mallet-Paret and Nussbaum studying slowly oscillating periodic solutions in singularly perturbed constant and state-dependent DDEs~\cite{JMPRN86,JMPRNIII,JMPRN11}.

A complete analysis of how the canard explosion arises in~\eqref{Qprime} will be beyond the scope of this work, but we will show that equation~\eqref{Qprime} can be considered as singular perturbation problem. We identify the critical manifold, and also investigate its persistence by approximating the resulting slow manifold and studying its stability.
We will show that a segment of this manifold for $Q<Q^*$ is stable with oscillatory convergence of nearby trajectories onto the manifold (see Figures~\ref{fig:GammaCont_Canard}(ii) and (iii))
while a segment for $Q>Q^*$ is unstable, leading to the divergence of trajectories from the manifold. In the current work, we will not study the fast dynamics in the transition layer.

For $\gamma=0.2453692$ and all other parameters taking their values from
Table~\ref{tab.model.par}, we notice that $\kappa\approx0$ and $A\approx1$, so we introduce the perturbation parameter $\epsilon\geq0$ and let
\begin{equation} \label{eq:epspars}
\epsilon=A-1=2e^{-\gamma\tau}-1,
\qquad
\kappa=\frac{\epsilon f}{C}.
\end{equation}
For the non-zero steady state $Q^*$ to exist the inequality \eqref{kgt} must hold; equivalently
the constant $C$ must satisfy $C>1$. Then $Q^*>0$ is given from \eqref{Q.star} by
\begin{equation} \label{Qstarsing}
Q^*=\theta(C-1)^{1/s},
\end{equation}
independent of the value of $\epsilon>0$.
For the parameters used in Figure~\ref{fig:GammaCont_Canard} we have
$\epsilon=6.132\times10^{-3}$ and $C=2.23$ with $Q^*=0.0896868$
when $\gamma=0.2453692$. The parameter definitions in \eqref{eq:epspars} could be applied to
the non-dimensionalised equation~\eqref{dQdt_hat} with $\hat{A}=\epsilon$ and $\hat{\kappa}=\epsilon/C$,
but we prefer to continue to study~\eqref{Qprime} directly.

Letting $h(Q)=Q\beta(Q)$, as in~\eqref{hx}, and
using \eqref{eq:epspars} we re-write \eqref{Qprime}  as
\begin{equation} \label{eq:epszeroman}
Q^{\prime}(t) = -\frac{\epsilon f}{C}Q(t)-h(Q(t))+(1+\epsilon)h(Q(t-\tau)).
\end{equation}
When $\epsilon=0$ this reduces to
\begin{equation} \label{eq:epszero}
Q^{\prime}(t) =-h(Q(t))+h(Q(t-\tau))=
-\frac{fQ(t)}{1+(Q(t)/\theta)^s}+\frac{fQ(t-\tau)}{1+(Q(t-\tau)/\theta)^s}.
\end{equation}
While equation \eqref{eq:epszeroman} has the unique positive steady state $Q^*$ given by \eqref{Qstarsing},
when $\epsilon=0$ equation \eqref{eq:epszero} has a line of equilibria with $Q$ being an arbitrary constant,
which is the critical manifold.
The linearisation of \eqref{eq:epszero} is given by \eqref{dQdtLin} with
$a=-h'(Q)$ and $b=-a$ and so the characteristic function \eqref{hayes} becomes
$$p(\lambda)= \lambda - a + a\mathnormal{e}^{-\lambda\tau}.$$
This satisfies $p(0)=0$ and $p'(0)=1-a\tau$, and has $\lambda=0$ as a solution for any value of $a$. There is an additional real negative root if $p'(0)>0$,  i.e. $a\tau<1$. This root crosses zero when $p'(0)=0$ and becomes positive for $a\tau>1$ when $p'(0)<0$. Thus the steady state stability changes when
$h'(Q)=-1/\tau$.
Using \eqref{beta2} and \eqref{hx} and the non-dimensionalised variables of \eqref{dQdt_hat}
the identity $h'(Q)=-1/\tau$ reduces to a quadratic equation for $\hat Q^s$:
$$\hat{Q}^{2s}+(2-(s-1)\hat f)\hat{Q}^{s}+1+\hat{f}=0.$$
Solving this with parameters corresponding to Figure~\ref{fig:GammaCont_Canard} we find that the stability on the critical manifold changes when $Q=\theta\hat Q=0.0893174$, very close to the value $Q^*$.

The critical manifold should persist where it is transversally hyperbolic as a slow manifold
following the theory of Fenichel \cite{Fenichel_1979}. However, that theory was developed for ODEs in multiple space dimensions with at least one fast and one slow variable. Likewise, the previous examples of canards in DDEs
considered systems with two spatial dimensions with one fast and one slow variable~\cite{CSE09,KrupaTouboul2016}.
For the scalar DDE \eqref{Qprime} there is not an obvious separation into fast and slow variables, and it is not apparent how to proceed rigourously. Nevertheless, it is apparent from
Figure~\ref{fig:GammaCont_Canard} that there is a slow manifold, and in the remainder of this section we will show how to approximate the slow manifold and determine its stability.

The slow manifold on which $Q'\approx0$ should be close to the nullcline $Q'(t)=0$, which
from~\eqref{Qprime} is given by
\begin{equation} \label{eq:critman}
0 = -\frac{\kappa}{f}Q(t)-\frac{Q(t)}{1+(Q(t)/\theta)^s}+2e^{-\gamma\tau}\frac{Q(t-\tau)}{1+(Q(t-\tau)/\theta)^s}.
\end{equation}
If one of $Q(t)$ or $Q(t-\tau)$ is fixed, then for $Q'(t)=0$ with $s=2$
from~\eqref{eq:critman} the value of the other one is defined by a cubic equation. The resulting nullcline is displayed as the two red curves in
Figure~\ref{fig:GammaCont_Canard}(ii), which are seen to be disjoint in
Figure~\ref{fig:GammaCont_Canard}(iv) which shows an expanded view near to $(Q^*,Q^*)$.
Typically, canard explosions are seen close to a bifurcation of the intersections of the nullclines of the slow and fast variables. Here we do not have separate fast and slow variables,
but we see that we are close to a bifurcation of the $Q'(t)=0$ nullcline itself, with the two disjoint parts coming very close to each other in Figure~\ref{fig:GammaCont_Canard}(iv). Figure~\ref{fig:GammaCont_Canard}(iv) also shows that
the periodic orbits that form the canard explosion appear to lie on a slow manifold between the branches of the nullcline and switch from following the lower branch to following the upper branch at the point close to $Q^*$ where the two curves are closest.
It is thus likely essential for the canard explosion that the parameter set is close to this bifurcation of the nullcline structure.

To obtain a simple approximation to the slow manifold, let $h(Q)=Q\beta(Q)$, as in~\eqref{hx}, so
\eqref{Qprime} becomes\begin{equation}\label{dQdt3}
Q^{\prime}(t)  = -\kappa Q(t)- h(Q(t))+Ah(Q(t-\tau)).
\end{equation}
Then use the approximation
\begin{equation} \label{hQtauapprox}
h(Q(t-\tau))\approx h(Q(t))-\tau \frac{d}{dt}h(Q(t))=h(Q(t))-\tau Q'(t) h'(Q(t)),
\end{equation}
to remove the delay from~\eqref{dQdt3}. Note that $\tau\gg0$ so~\eqref{hQtauapprox} is only useful if
$\tau^2\frac{d^2}{dt^2}h(Q(t))\ll1$. Substituting~\eqref{hQtauapprox} into~\eqref{dQdt3} and rearranging
we obtain
$$Q'(t)=\frac{-\kappa Q(t)+(A-1)h(Q(t))}{1+A\tau h'(Q(t))}.$$
On the slow manifold $Q(t-\tau)\approx Q(t)-\tau Q'(t)$ so we can approximate the manifold in the
delay embedding $(Q(t),Q(t-\tau))$ by the curve
$(Q,Q_\tau)$ where
\begin{equation} \label{eq:Qtaucanardslowman}
Q_\tau =
\frac{(1+\kappa\tau+A\tau h'(Q))Q-\tau(A-1)h(Q))}{1+A\tau h'(Q)}.
\end{equation}
This gives a good approximation to the slow manifold for $Q\in(0,0.23)$,
except for a very small interval $|Q-Q^*|<0.003$ about the non-trivial steady state.
This is shown in Figure~\ref{fig:linearslowman}(i).
The curve $(Q,Q_\tau)$ is also an approximation to the unstable manifold of $Q=0$; the trivial steady state has a single positive characteristic value, and infinitely many pairs of complex conjugate characteristic values with negative real part. The canard explosion is \emph{not} associated with a solution homoclinic to $Q=0$ though; no such homoclinic solution can exist in the space of non-negative solutions, as all solutions in the stable manifold of $Q=0$ will be oscillatory and violate the positivity of solutions (as already noted on page~\pageref{positive}).

\begin{figure}[t]
\includegraphics[width=0.49\textwidth,height=44mm]{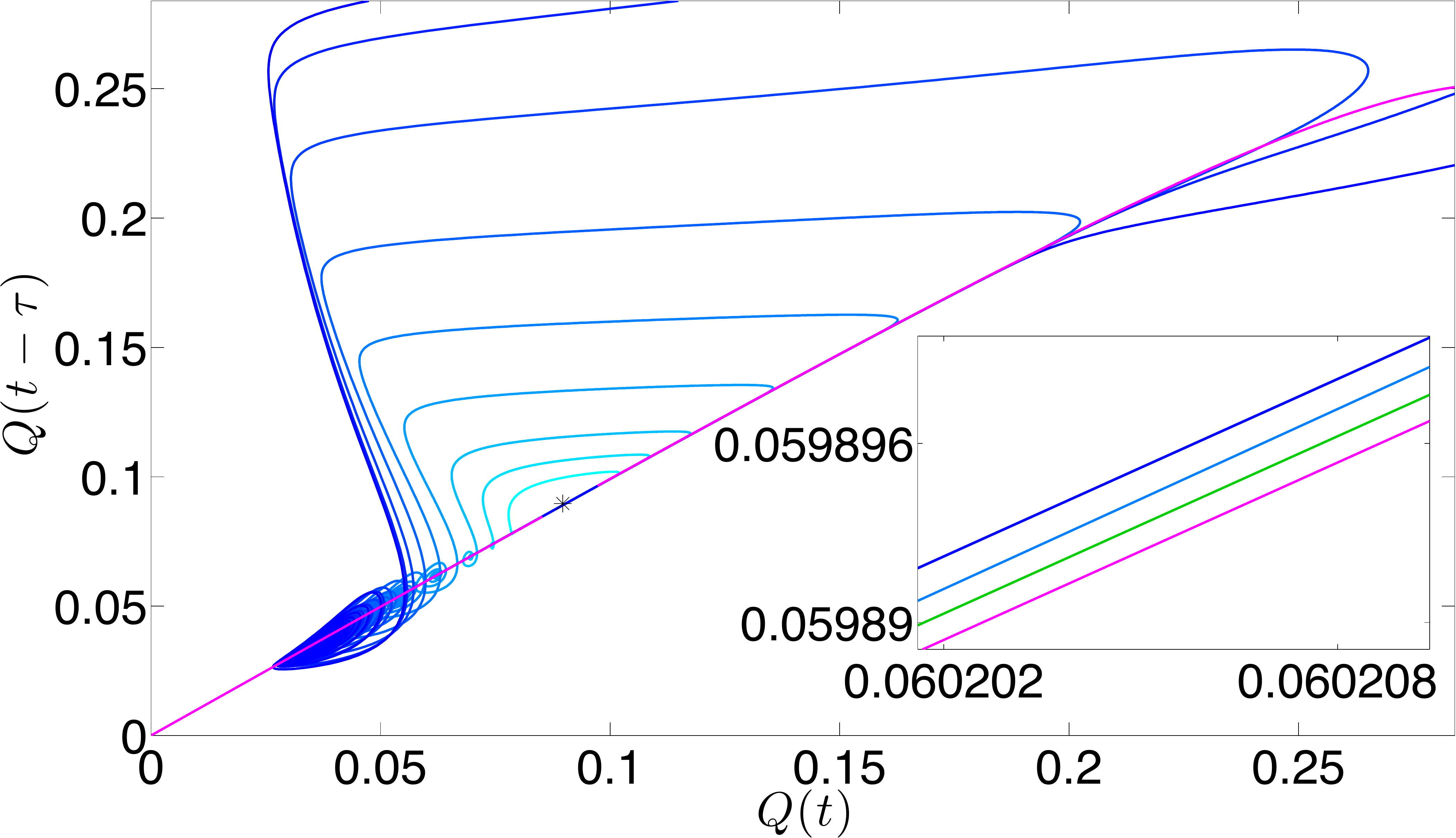}
\put(-245,2.5){\footnotesize{\textit{(i)}}}
\hspace*{1mm}
\includegraphics[width=0.49\textwidth,height=44mm]{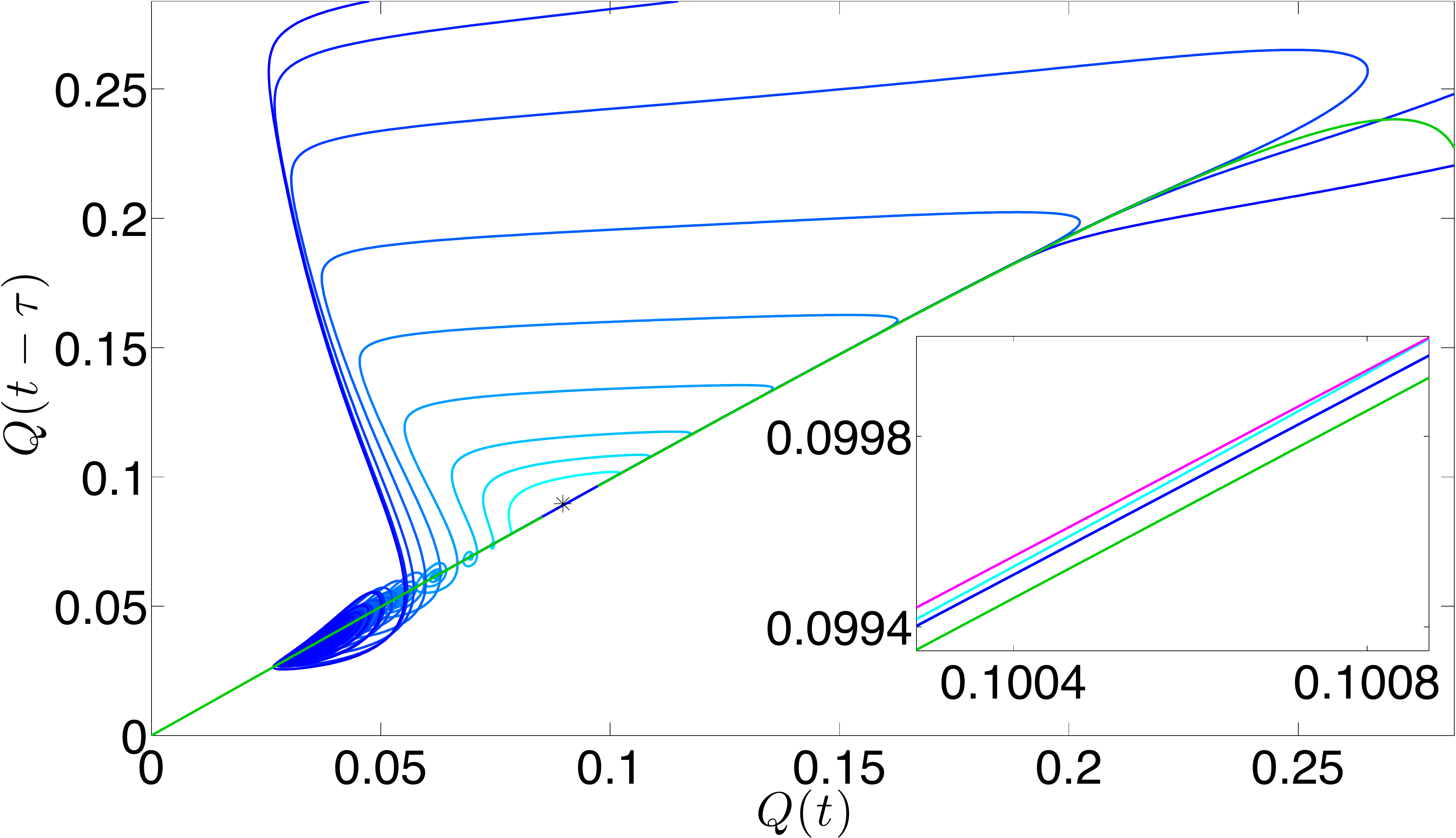}
\put(-245,2.5){\footnotesize{\textit{(ii)}}}
\vspace*{-2mm}
\caption{Approximations $(Q,Q_\tau)$ to the slow manifold of the canard solution where: (i) $Q_\tau$ is given by~\eqref{eq:Qtaucanardslowman} and shown in violet; (ii) $Q_\tau$ is given by~\eqref{eq:Qtauslowlin} with $Q_{r}$ replaced by $Q$ and shown in green. Both approximations are shown in the insets.}
\label{fig:linearslowman}
\end{figure}

To determine the dynamics close to the slow manifold, it is necessary to take proper account of the delayed term, 
and so we will derive another slow manifold approximation.
For this, linearise $h(Q)$
about $Q=Q_r$ for general $Q_r$ as
\begin{equation} \label{hlin}
h(Q)\approx h(Q_r)+h'(Q_r)(Q-Q_r).
\end{equation}
Then let $Q_s(t)$ be a solution on the slow manifold, and let $Q(t)$ be some other solution in a neighbourhood of this manifold, and let $w(t)=Q(t)-Q_s(t)$ be the difference between these two solutions, then using~\eqref{hlin} and~\eqref{dQdt3}
we find that
\begin{equation}\label{eq:linslddenonaut}
w'(t) = Q'(t)-Q_s'(t) \approx -(\kappa+h'(Q_s(t)))w(t)+Ah'(Q_s(t-\tau))w(t-\tau).
\end{equation}
The linearisation~\eqref{eq:linslddenonaut} is  actually valid as an approximation of the dynamics about any solution $Q_s(t)$ of the DDE~\eqref{Qprime}. However, it is problematical to use, even about solutions on the slow manifold,
since although~\eqref{eq:linslddenonaut} is linear it is non-autonomous and the time-dependent terms depend on the as yet unknown slow-manifold solution $Q_s(t)$.
As an alternative, rather than linearise about a particular solution, close to the slow manifold we can use
\eqref{hlin} to linearise $h$ about $Q=Q_r$, for some reference value of $Q$ and convert~\eqref{Qprime} into a
linear DDE for the dynamics near to $Q=Q_r$, with
\begin{equation}\label{DDEhlinderiv}
Q^{\prime}(t) \approx (A-1)[h(Q_r)-h^{\prime}(Q_r)Q_r] -(\kappa+h^{\prime}(Q_r))Q(t)+Ah^{\prime}(Q_r)Q(t-\tau).
\end{equation}
We rewrite~\eqref{DDEhlinderiv} as
\begin{equation} \label{DDEhlinnon}
Q^{\prime}(t)= aQ(t)+bQ(t-\tau)+c,
\end{equation}
where letting
\begin{equation}\label{fQsQ}
G(Q) := (A-1)\beta(Q)Q-\kappa Q=(A-1)h(Q)-\kappa Q,
\end{equation}
we see that the constants $a$, $b$ and $c$ satisfy
\begin{gather*}
a=-(\kappa+h^{\prime}(Q_r)),\qquad
b=Ah^{\prime}(Q_r),\\
c=(A-1)[h(Q_r)-h^{\prime}(Q_r)Q_r]=G(Q_r)+[\kappa-(A-1)h^{\prime}(Q_r)]Q_r=G(Q_r)-(a+b)Q_r.
\end{gather*}
Noting that
$$G'(Q_r)=(A-1)h'(Q_r)-\kappa=a+b,$$
we can rewrite $c$ as
$c=G(Q_r)-G'(Q_r)Q_r.$

The function $G$ is unimodal, and recalling~\eqref{dQdt.star} it follows that
$G(0)=G(Q^*)=0$, and $G$ has a unique maximum for
$Q=Q_f$ for some $Q_f\in(0,Q^*)$.
Thus $G'(Q_r)>0$ (resp. $G'(Q_r)<0$) for $Q<Q_f$ (resp. $Q>Q_f$).
Another value of $Q$ which will be relevant below
is $Q_h$, the value of $Q$ such that $h'(Q_h)=0$.
It follows easily that $Q_h>Q_f$. With $\gamma = 0.2453692$ and the other parameters
from Table~\ref{tab.model.par} we have
$$Q_f=0.042263 \quad < \quad Q_h=(s-1)^{-\frac1s}\theta=\theta \quad < \quad Q^*=0.089686.$$

Solutions to the nonhomogeneous linear DDE~\eqref{DDEhlinnon} consist of a particular solution of~\eqref{DDEhlinnon}
and any linear combination of solutions of the homogeneous linear DDE
\begin{equation} \label{DDEhlinaut}
Q^{\prime}(t)=  aQ(t)+bQ(t-\tau).
\end{equation}

The DDE~\eqref{DDEhlinaut} is of the same form as equation~\eqref{dQdtLin} and has the
same characteristic equation~\eqref{hayes}.
This has infinitely many complex roots, which would lead to oscillatory solutions of~\eqref{DDEhlinaut}. However, from above the slow manifold appears to be monotonic, hence we will seek a monotonic solution of~\eqref{DDEhlinderiv}, for which we require real roots of~\eqref{hayes}.
Since $p'(\lambda)=1+b\tau e^{-\lambda\tau}$, for $Q\leq Q_h$ we have $p'(\lambda)\geq1$ and equation~\eqref{hayes} has a unique real root.
Real roots of~\eqref{hayes} can be found using the Lambert-$W$ function~\cite{Corless1996}.
Rearranging~\eqref{hayes} we see that any root $\lambda$ satisfies $b\tau e^{-a\tau}=(\lambda-a)\tau e^{(\lambda-a)\tau}.$
Hence $W(b\tau e^{-a\tau})=(\lambda-a)\tau$, and so
\begin{equation} \label{lambdaW}
\lambda=a+\frac{1}{\tau}W(b\tau e^{-a\tau}).
\end{equation}

Consider first the case where $Q_r<Q_f$. Then $G'(Q_r)=a+b>0$ and $b>0>a$.
Since $p(0)=-(a+b)$ and $p'(\lambda)\geq1$ the characteristic equation~\eqref{hayes} has a
unique real root $\lambda_+=\lambda_+(Q_r)\in(0,G'(Q_r))$ given by~\eqref{lambdaW}. Equation~\eqref{DDEhlinnon}
then admits a constant solution, $Q(t)=k$, where
$$k=\frac{-c}{a+b}=\frac{-G(Q_r)+G'(Q_r)Q_r}{G'(Q_r)}=Q_r-\frac{G(Q_r)}{G'(Q_r)}.$$
Hence a monotonic solution of~\eqref{DDEhlinnon} passing through $Q(0)=Q_r$ is
\begin{equation} \label{QsollessQf}
Q(t)=Q_r+(e^{\lambda_+ t}-1)G(Q_r)/G'(Q_r).
\end{equation}
The general solution of~\eqref{DDEhlinnon} which describes
the behaviour of solutions in a neighbourhood of the monotonic solution is obtained by adding an arbitrary linear combination of
the solutions of~\eqref{DDEhlinaut}, defined by the roots of~\eqref{hayes}. But since $\lambda_+>0$, even without solving for the complex characteristic values, we already know that the monotonic solution is not stable for $Q<Q_f$.
Nevertheless, we can use~\eqref{QsollessQf} to approximate the slow manifold for $Q<Q_f$.
From~\eqref{QsollessQf}, when $Q=Q_r$ we have
$Q'=\lambda_+G(Q_r)/G'(Q_r)<G(Q_r)$, and hence assuming that $Q(t-\tau)\approx Q(t)-\tau Q'(t)$, we can approximate the slow manifold $Q_s(t)$ in the delay embedding $(Q(t),Q(t-\tau))$ for $Q<Q_f$ by the curve $(Q_r,Q_\tau)$ where
\begin{equation} \label{eq:Qtauslowlin}
Q_\tau= Q_r - \tau\lambda G(Q_r)/G'(Q_r)
\end{equation}
and $\lambda=\lambda_+(Q_r)$.



Next consider the case for which
$Q_r\in(Q_f,Q_h)$. Then $b>0>a+b=G'(Q_r)$.
Now, $p(0)=-(a+b)>0$ and $p'(\lambda)\geq1$, so the characteristic equation~\eqref{hayes} has a unique
real root which is negative, $\lambda_-=\lambda_-(Q_r)\in(G'(Q_r),0)$ given by~\eqref{lambdaW}.
Similarly to above, equation~\eqref{DDEhlinnon} then has a monotonic solution
\begin{equation} \label{QsolgtQf}
Q(t)=Q_r+(e^{\lambda_- t}-1)G(Q_r)/G'(Q_r),
\end{equation}
passing through $Q(0)=Q_r\in(Q_f,Q_h)$. The behaviour of nearby solutions is determined by the general solution of
\eqref{DDEhlinnon} which is
\begin{equation} \label{QgensolmoreQf}
Q(t)=Q_r+(e^{\lambda_- t}-1)\frac{G(Q_r)}{G'(Q_r)}+\beta_0e^{\lambda_-t}
+\sum_{j\geq1}e^{\alpha_jt}(\beta_j\cos(\omega_jt)+\gamma_j\sin(\omega_jt)),
\end{equation}
for constants $\beta_j$, $\gamma_j$ where $\lambda_j=\alpha_j\pm i\omega_j$ are the complex roots
of~\eqref{hayes}.
For this solution, not only is $\lambda_-<0$, but we can also show that all the complex characteristic values that solve~\eqref{hayes} also have strictly negative real part.
Taking real and imaginary parts of~\eqref{hayes} we find that $\lambda_j=\alpha_j\pm i\omega_j$ satisfies
$0=\alpha_j-a-be^{-\alpha_j\tau}\cos(\omega_j\tau)=\omega_j+be^{-\alpha_j\tau}\sin(\omega_j\tau),$
which implies that
$\omega_j^2=b^2e^{-2\alpha_j\tau}-(\alpha_j-a)^2.$
But for $Q_r\in(Q_f,Q_h)$ we have $a<0$ and $-a>b>0$, hence for a characteristic root with $\alpha_j\geq0$ we have
$\omega_j^2\leq b^2-a^2<0$, a contradiction, and so all characteristic values
have $Re(\lambda_j)=\alpha_j<0$.

For $Q\in(Q_f,Q_h)$ we can again approximate the slow manifold $Q_s(t)$ in the delay embedding $(Q(t),Q(t-\tau))$ by the curve $(Q_r,Q_\tau)$ where $Q_\tau$ is defined by~\eqref{eq:Qtauslowlin} with $\lambda=\lambda_-(Q_r)$.
Since $\alpha_j<0$ for all $j$ and $\lambda_-<0$ all the additional solution elements included in
\eqref{QgensolmoreQf} are decaying, and the solution defined by~\eqref{QsolgtQf} and the resulting slow manifold are attracting in this region of phase space.

\begin{table}[t]
\begin{center}
\begin{tabular}{|c|c|c|c|c|}\hline
$Q_r$ &  $\lambda_\pm$ & $Q'$ & $Q_\tau$ & $\lambda_1$ \\ \hline
$0.01$ & $1.10\times10^{-5}$ & $1.17\times10^{-5}$ & $9.96\times10^{-3}$ & $-0.0118 \pm 2.15i$ \\
$0.02$ & $9.56\times10^{-4}$ & $2.45\times10^{-5}$ & $1.99\times10^{-2}$ & $-0.0155 \pm 2.13i$ \\
$0.03$ & $6.68\times10^{-4}$ & $3.96\times10^{-5}$ & $2.98\times10^{-2}$ & $-0.0237 \pm 2.11i$ \\
$0.04$ & $1.60\times10^{-4}$ & $5.81\times10^{-5}$ & $3.98\times10^{-2}$ & $-0.0408 \pm 2.07i$ \\
$0.05$ & $-7.40\times10^{-4}$ & $8.13\times10^{-5}$ & $4.97\times10^{-2}$ & $-0.077 \pm 2.00i$ \\
$0.06$ & $-2.45\times10^{-3}$ &  $1.10\times10^{-4}$ & $5.96\times10^{-2}$ & $-0.157 \pm 1.90i$ \\
$0.07$ & $-6.28\times10^{-3}$ &  $1.47\times10^{-4}$ & $6.95\times10^{-2}$ & $-0.354 \pm 1.75i$ \\
$0.08$ & $-1.93\times10^{-2}$ &  $1.98\times10^{-4}$ & $7.94\times10^{-2}$ & $-1.35 \pm 1.42i$ \\ \hline
\end{tabular}
\end{center}
\vspace*{2mm}
\caption{Table showing the real characteristic value $\lambda_-$ or $\lambda_+$ of~\eqref{hayes} along with the first pair of complex characteristic values $\lambda_1$, and the derivative of the solution $Q'(t)$ given by~\eqref{QsollessQf} or~\eqref{QsolgtQf}, and the resulting approximation of $Q_\tau$ given by~\eqref{eq:Qtauslowlin} for a range of values of $Q_r$.}
\label{tab:lambdas}
\end{table}
%
The convergence onto the slow manifold is oscillatory, as seen in Figure~\ref{fig:GammaCont_Canard}. This is governed by the dominant complex characteristic value of~\eqref{hayes}, the value of which
is stated as $\lambda_1$ in Table~\ref{tab:lambdas}. We see from the table that $Re(\lambda_1)$ becomes more negative as $Q$ increases, implying that the slow manifold becomes more attractive as $Q$ increases towards $Q_h$. This is clearly visible in
Figure~\ref{fig:GammaCont_Canard}(iii) with progressively fewer  oscillations visible for the orbits converging onto the
slow manifold for larger values of $Q$. The period of these oscillations $2\pi/\omega_1$ also increases with $Q$ but not greatly, and is close to $3$ in the range of $Q$ values where the oscillations are most visible.

Figure~\ref{fig:canard_lindyn} illustrates how well our approximations perform in the region $Q\in(Q_f,Q_h)$
where the slow manifold is attracting. The blue curves in Figure~\ref{fig:canard_lindyn} show part of the limit cycle
of the nonlinear DDE~\eqref{Qprime}
with period $297$ days when $\gamma= 0.2453692$, which occurs in the canard explosion and was previously shown in Figure~\ref{fig:GammaCont_Canard}. Taking $Q_r=0.063224$ we find that
the rightmost characteristic value is $\lambda_{-}= -3.33\times10^{-3}$, then~\eqref{QsolgtQf} defines an approximation to the slow manifold which is shown as the black curve in Figure~\ref{fig:canard_lindyn}.

\begin{figure}[t]
\includegraphics[width=0.49\textwidth,height=44mm]{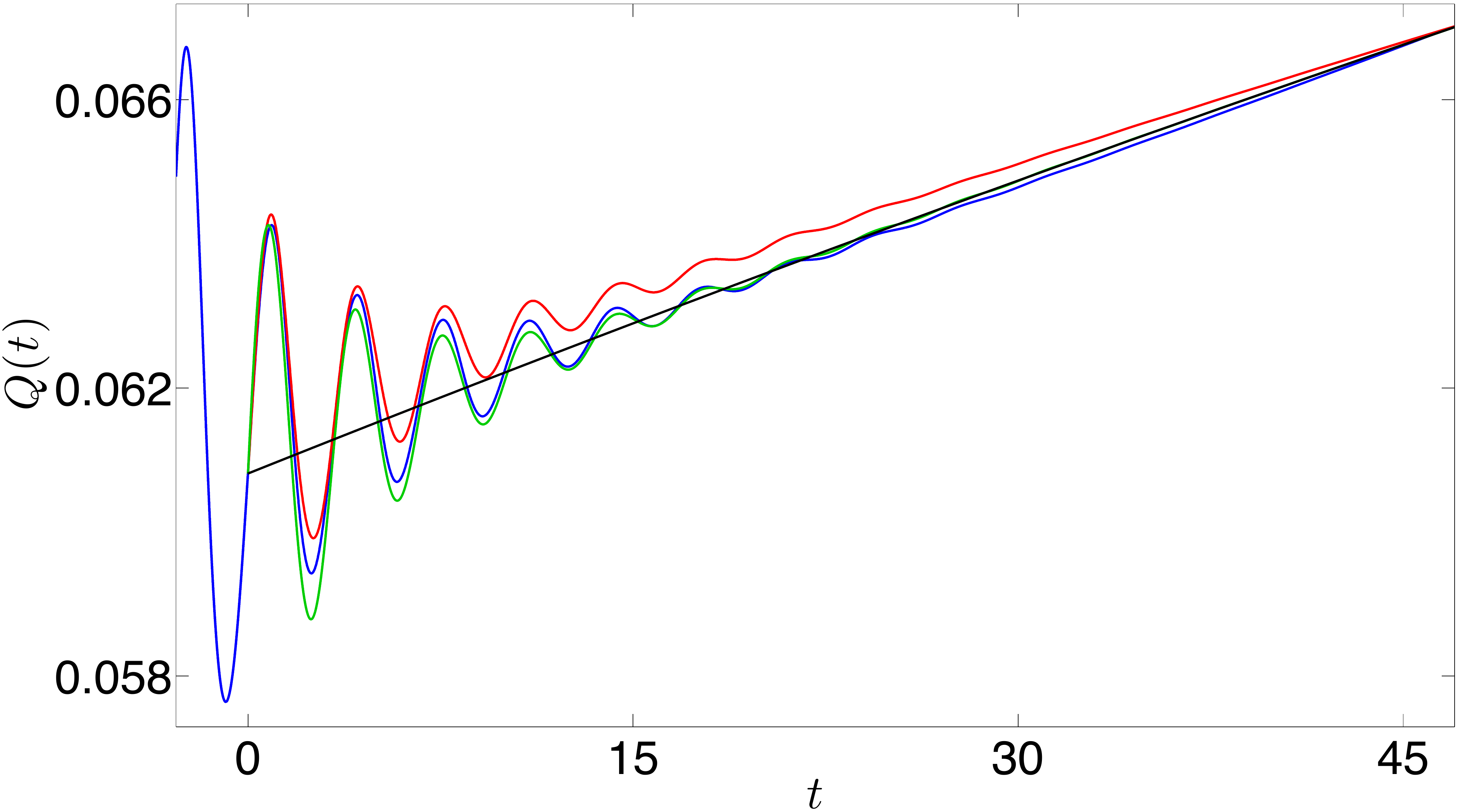}
\put(-245,2.5){\footnotesize{\textit{(i)}}}
\hspace*{1mm}
\includegraphics[width=0.49\textwidth,height=44mm]{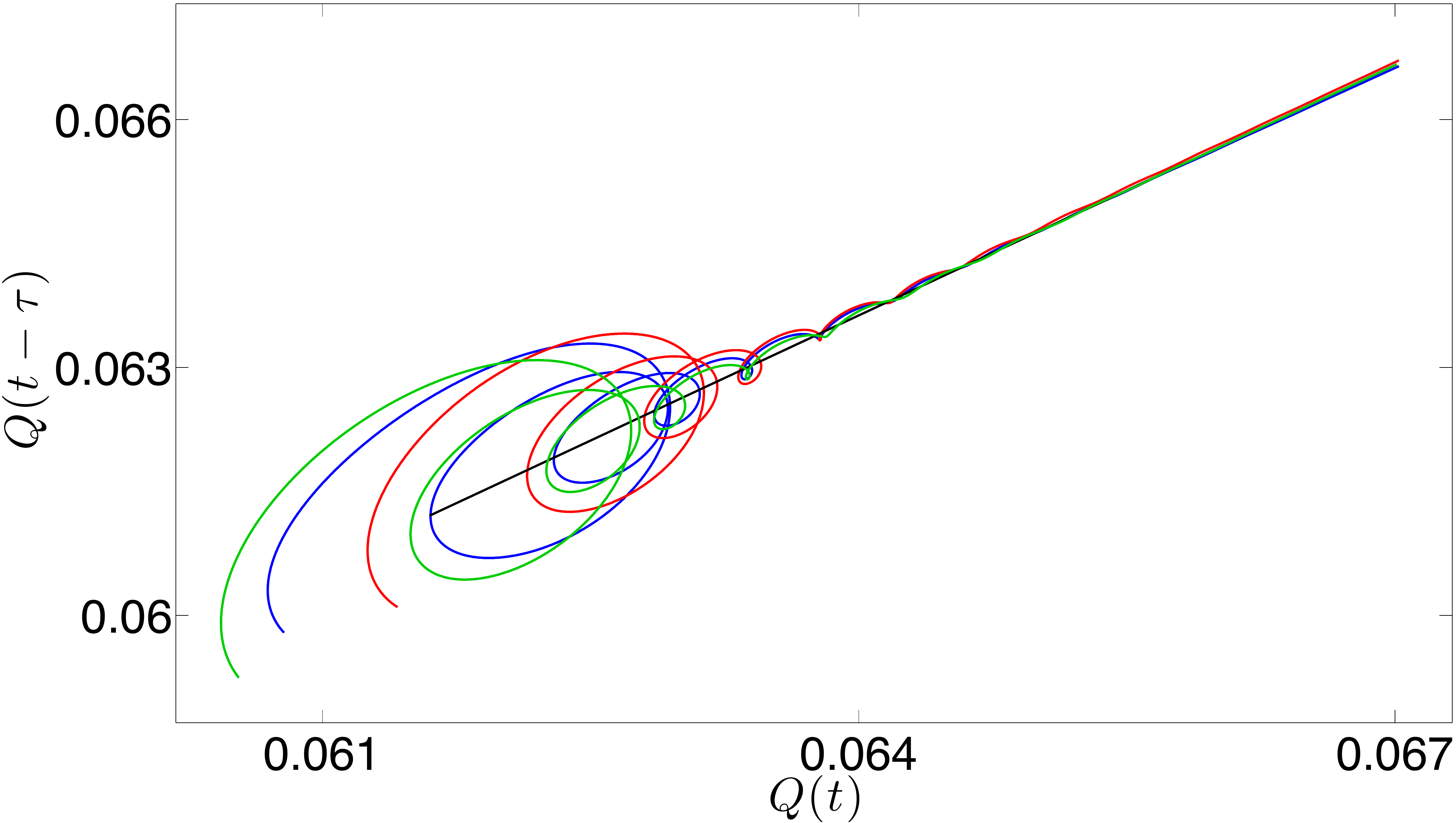}
\put(-245,2.5){\footnotesize{\textit{(ii)}}}
\vspace*{-2mm}
\caption{(i) Profile and (ii) time-delay embedding, for a periodic solution of the nonlinear DDE~\eqref{Qprime}, and several of its approximations.
For $\gamma = 0.2453692$ the DDE~\eqref{Qprime} has
a limit cycle with period $297$ days, part of which is shown here (blue curve). For the approximations we take
$Q_{r}=0.063224$, then the black curve shows the approximation~\eqref{QsolgtQf} to the slow manifold. The red curve shows a solution to~\eqref{DDEhlinnon} computed numerically using
a segment of the solution of the nonlinear DDE~\eqref{Qprime} for $t\in[-2.8,0]$ to define the initial function. The green curve shows a solution of~\eqref{DDEhlinnon} defined by~\eqref{QgensolmoreQf} with
$\beta_j=\gamma_j=0$ for all $j$, except $\gamma_1=0.003935$.}
\label{fig:canard_lindyn}
\end{figure}

The second-rightmost characteristic value $\lambda_{1} = \alpha_1+i\omega_1=-0.202 + 1.86i$ yields the approximate oscillation time of 3.37 days. To show that this characteristic value governs the convergence of solutions onto the slow manifold in
Figure~\ref{fig:canard_lindyn} we show as the green curve the solution~\eqref{QgensolmoreQf} of the linearised DDE~\eqref{DDEhlinnon} with $\beta_j=\gamma_j=0$ for all $j$, except for
$\gamma_1\ne0$, so that the only oscillatory mode included in the solution is defined by $\lambda_{1}$. Additionally, the red curve shows the solution of~\eqref{DDEhlinnon} incorporating all modes, computed by solving
\eqref{DDEhlinnon} numerically using part of the solution of~\eqref{Qprime} as the initial function. Both approximations have oscillations about the slow manifold with very similar period and decay rate as for the solution of the full nonlinear DDE~\eqref{Qprime}, demonstrating the validity of our approximations.


In the current work we will not describe the passage of the slow manifold past the steady state $Q^*$, but note that the behaviour of the solutions of~\eqref{hayes} changes when $Q_r$ approaches $Q^*$.
For $Q>Q_h$ we have $b<0$ and the different branches of $W(x)$ in~\eqref{lambdaW} for $x<0$ can lead to zero or two real solutions for $\lambda$. There are two values of $Q$, $Q_{h'}^-<Q^*<Q_{h'}^+$,
such that when $Q=Q_{h'}^\pm$, we have $b\tau e^{-a\tau}=-e^{-1}$ and the two branches of the Lambert-$W$ function coalesce.
These points can be computed from the solution on each branch of $h'(Q_{h'}^\pm)=W(-e^{-1-\kappa\tau}/A)/\tau$,
which leads to $Q_{h'}^-=0.08626$ and $Q_{h'}^+=0.09389$.
For $Q\in(Q_{h'}^-,Q_{h'}^+)$ equation~\eqref{hayes} has no real roots. At the boundaries, $Q_{h'}^\pm$, of this interval a pair of complex conjugate characteristic roots coalesce, and for
$Q<Q_{h'}^-$ or $Q>Q_{h'}^+$ there are two real  characteristic roots.
At the steady state $Q^*\in(Q_{h'}^-,Q_{h'}^+)$, there is a single pair of characteristic roots with positive real part and leading characteristic roots of \eqref{hayes} are $\lambda_1=0.0070 \pm 0.1303i$ and $\lambda_2=-0.73\pm2.67i$.



For $Q>Q_{h'}^+$, the function $p(\lambda)$ in~\eqref{hayes} is convex with $p(0)=-(a+b)>0$, and $p(\lambda)>0$ for $\lambda\geq a$. With the other parameters as stated,~\eqref{hayes}  has two positive solutions provided
$Q_r\leq 0.28577$. Using the smaller of these two roots, we obtain a monotonic solution of the same form
as~\eqref{QsollessQf} and~\eqref{QsolgtQf}, which can be similarly used to construct an approximation to the slow manifold for $Q>Q_{h'}^+$, as shown in Figure~\ref{fig:linearslowman}(ii).
Because of the two positive characteristic roots, this part of the slow manifold is unstable, as seen in the dynamics where the
periodic orbits of different amplitudes and periods are seen in Figure~\ref{fig:GammaCont_Canard}(ii) to peel away from each other sooner or later depending on their amplitude and period. Thus we have approximated the attracting and repelling parts of the slow manifold either side of the steady state $Q^*$. A complete analysis of the canard explosion would require the dynamics that join these segments of the slow manifold, both near to the steady state, and also the fast dynamics when the solution is far from the slow manifold.

\section{Non-periodic and Chaotic Dynamics}
\label{sec.chaos.hsc}

\sloppy{The period-doubling and torus (Neimark-Sacker) bifurcations that we found in Section~\ref{sec.bifurc.hsc}
suggest that the DDE~\eqref{Qprime} can generate quasi-periodic and chaotic dynamics. The software package
DDEBiftool cannot be used to directly find such dynamics, but we can use the DDEBiftool bifurcation studies of
Section~\ref{sec.bifurc.hsc} to determine parameter regions where non-periodic dynamics should arise.
Long time numerical simulations of the DDE, can then be performed
using the MATLAB \texttt{dde23} routine~\cite{Matlab} with suitable initial history functions
to study the dynamics in these parameter regimes.}

\subsection{Quasi-Periodic Dynamics}
\label{sec.torus}

It is somewhat surprising to find torus bifurcations for the HSC model~\eqref{Qprime}, because as we already noted in Section~\ref{sec.decoup.hsc}, its linearisation about a steady state is equation~\eqref{dQdtLin}
which does not admit any double Hopf bifurcations. Double Hopf or Hopf-Hopf bifurcations are a standard mechanism for generating tori and curves of torus bifurcations~\cite{Kuznetsov_2004}, and arise frequently in systems with coupled oscillators and systems with multiple delays~\cite{Calleja_2017}. However, the HSC model~\eqref{Qprime} is scalar, with a single delay, but nevertheless torus bifurcations do occur, as seen in Section~\ref{sec.2p}. So here we will investigate the existence of invariant tori for~\eqref{Qprime}.

\begin{figure}[t]
\includegraphics[width=0.49\textwidth,height=44mm]{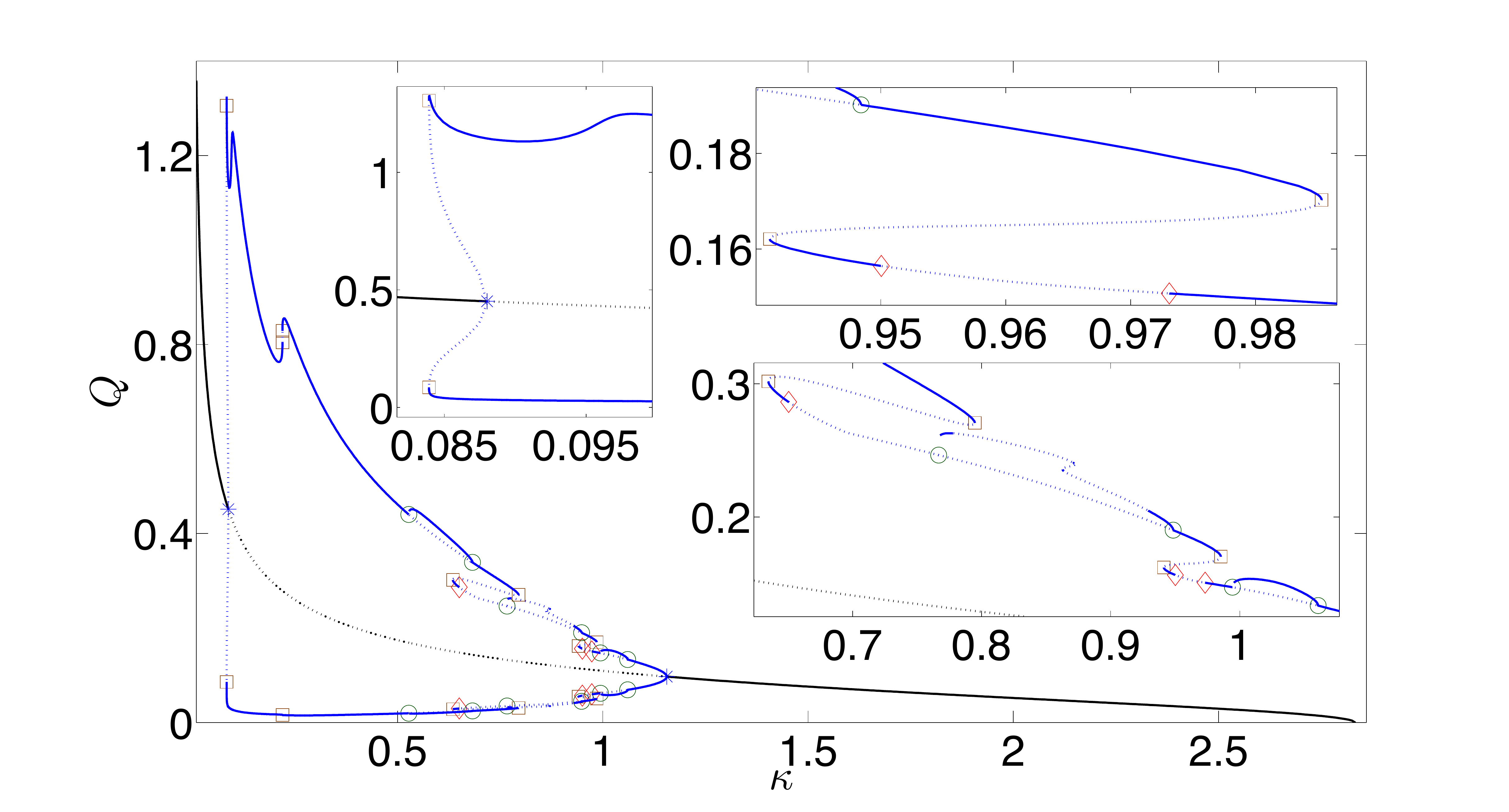}
\put(-245,2.5){\footnotesize{\textit{(i)}}}
\hspace*{1mm}
\includegraphics[width=0.49\textwidth,height=44mm]{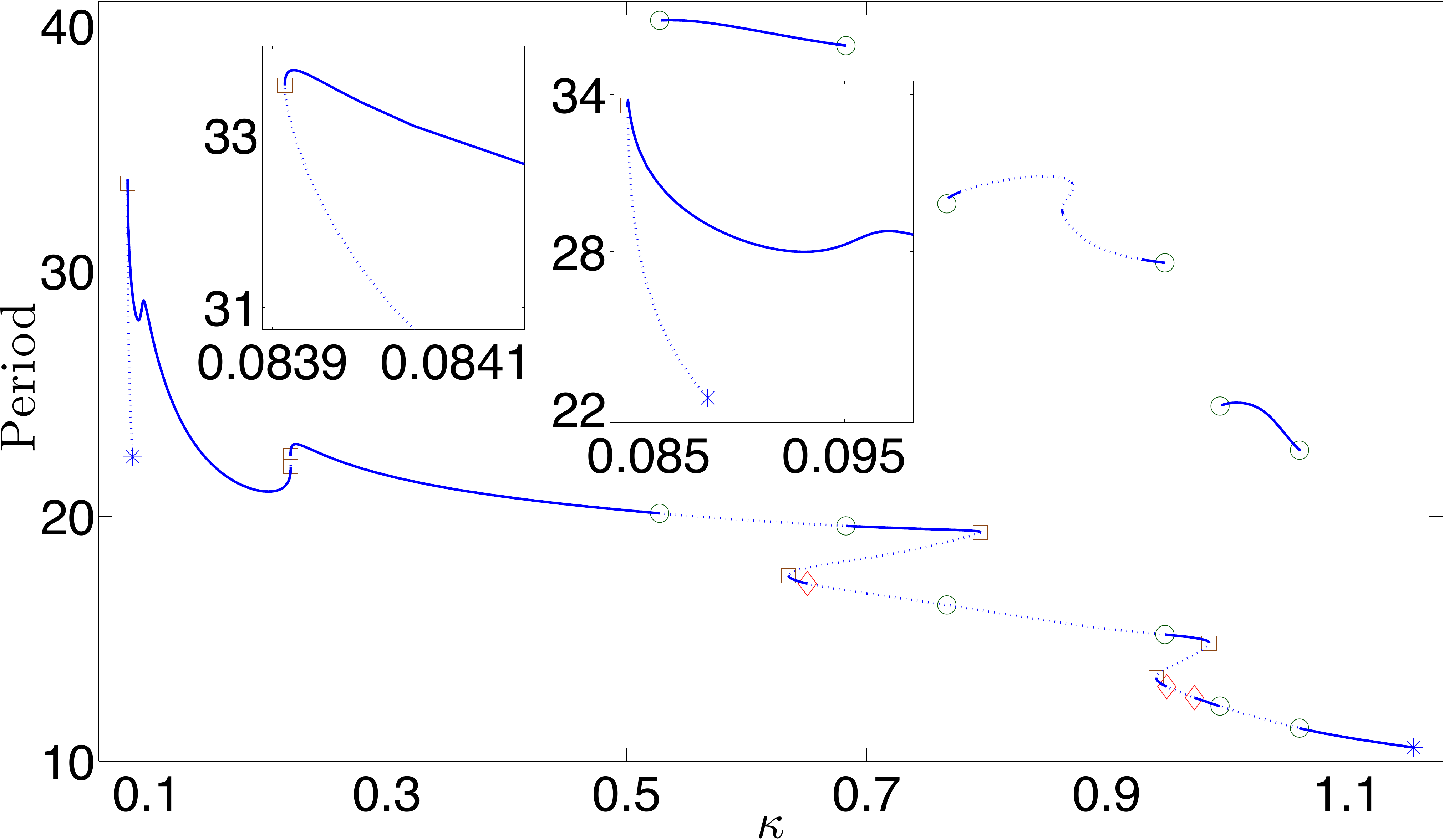}
\put(-247,2){\footnotesize{\textit{(ii)}}}
\vspace*{-2mm}
\caption{(i) Bifurcation diagram for one-parameter continuation in $\kappa$ with $\tau=3.9$ and other parameters taking homeostasis values from Table~\ref{tab.model.par}.
Hopf bifurcation points \textlarger[2]{$\textcolor{blue}{*}$}, saddle-node bifurcation of limit cycles points \textlarger[0]{$\textcolor{brown1}{\Box}$}, period-doubling bifurcation points \textlarger[-1]{$\textcolor{green1}{\Circle}$},
and torus bifurcation points \textlarger[2]{$\textcolor{red}{\diamond}$}
are indicated and highlighted in insets.
(ii) Period of the limit cycles displayed in (i).}
\label{fig2_HSC_Biftool1D_KappaDA}
\end{figure}

The two-parameter continuation in $(\kappa,\tau)$ shown in
Figure~\ref{fig:KappaTau_2D_Cont} reveals an isola of torus bifurcations
for $\kappa\in(0.91859,1.0174)$ and $\tau\in(3.4857,4.1342)$.
In a one parameter continuation, as $\kappa$ is varied with $\tau=3.9$ fixed,
Figure~\ref{fig2_HSC_Biftool1D_KappaDA} reveals that the main branch of periodic
orbits loses stability for $\kappa\in(0.95004,0.97309)$
at a pair of torus or Neimark-Sacker bifurcations, corresponding to the
points where the continuation crosses the isola found in the two-parameter continuation.

The simplest explanation is that there should be a stable torus at $\kappa$ values between the two Neimark-Sacker bifurcations. Although DDEBiftool~\cite{DDEBiftool15} cannot be used to find tori directly, a stable torus can be found by direct numerical simulation
if a suitable initial function is chosen in the basin of attraction of the torus.
Some care needs to be taken, because as the top right inset in
Figure~\ref{fig2_HSC_Biftool1D_KappaDA}(i) reveals, folds on the main branch of periodic solutions result in a stable (as well as two unstable)
periodic orbits existing for $\kappa\in(0.088007,1.1556)$, so if a stable torus exists it will co-exist with a stable periodic orbit.

\begin{figure}[t]
\includegraphics[width=0.49\textwidth,height=44mm]{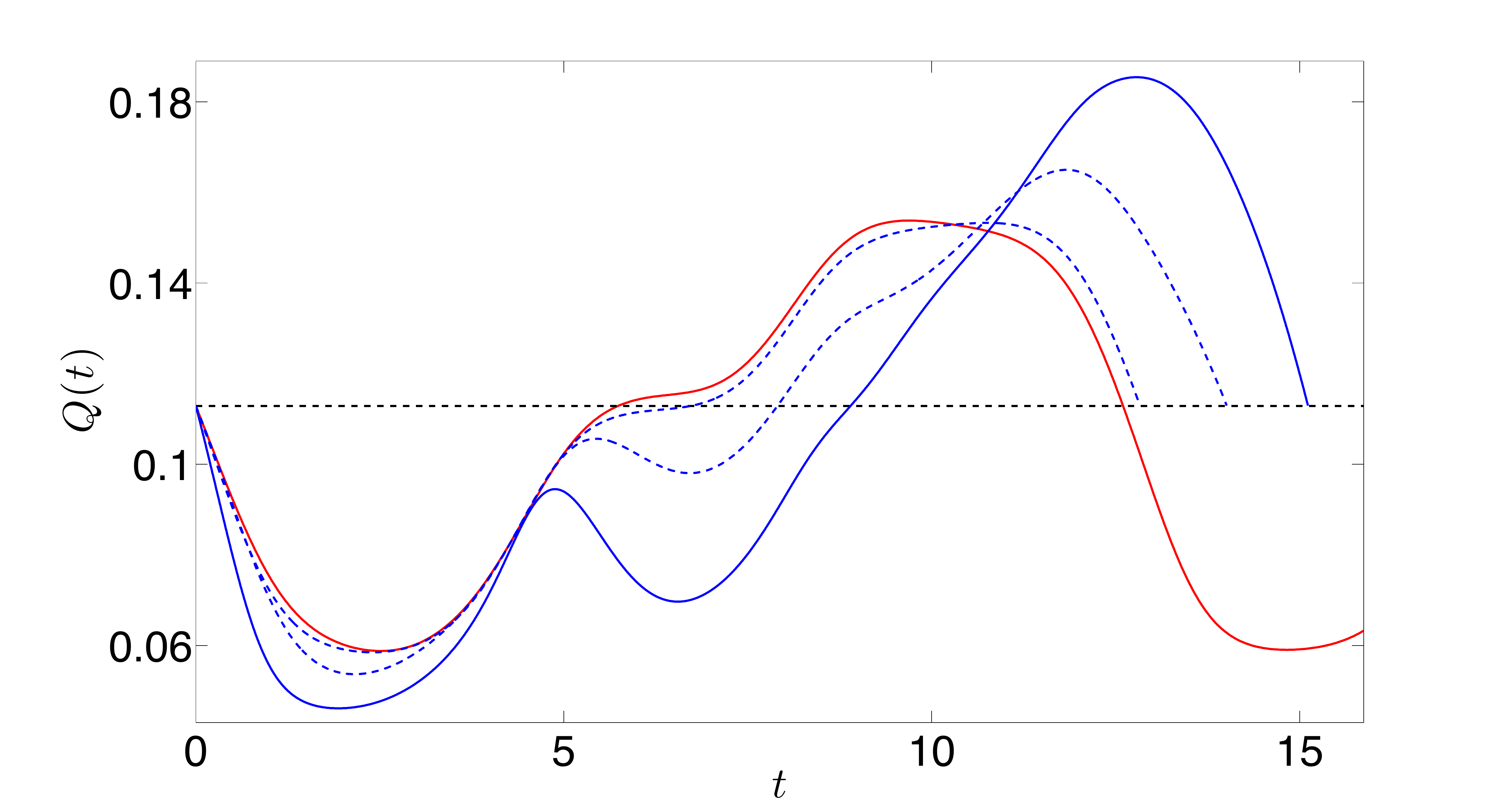}
\put(-245,2.5){\footnotesize{\textit{(i)}}}
\hspace*{1mm}
\includegraphics[width=0.49\textwidth,height=44mm]{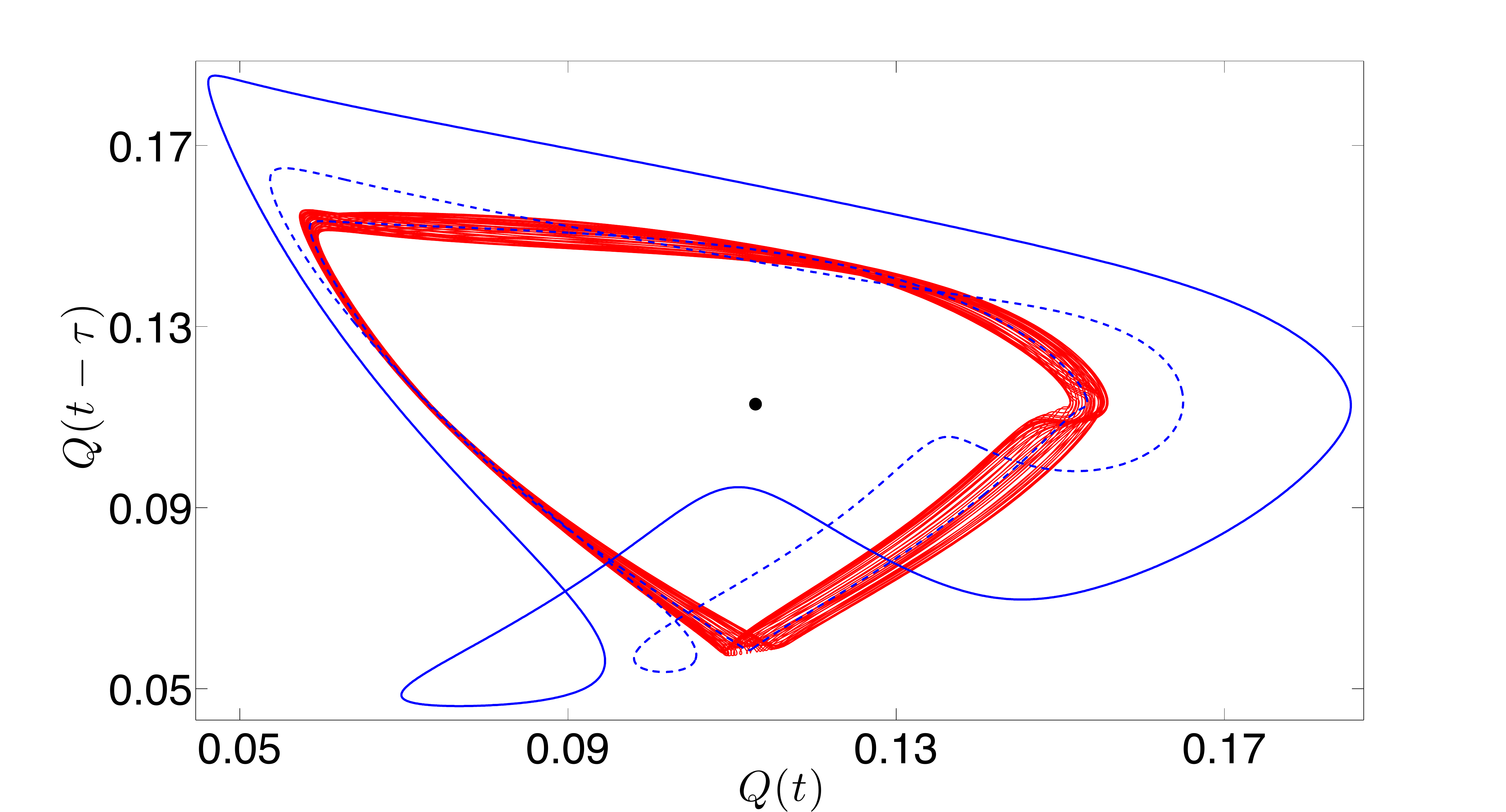}
\put(-245,2.5){\footnotesize{\textit{(ii)}}}

\vspace*{2mm}

\includegraphics[width=0.49\textwidth,height=44mm]{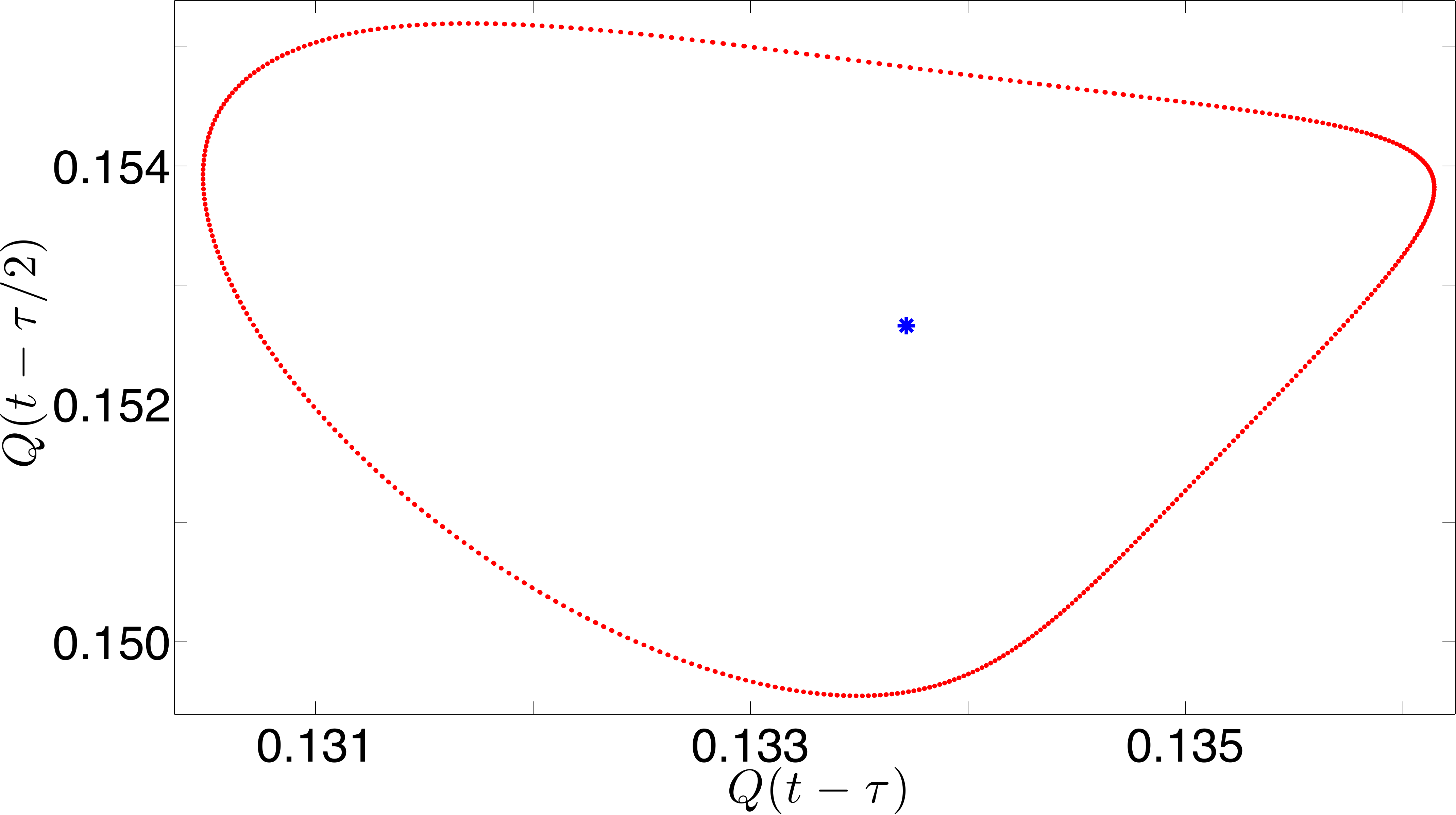}
\put(-245,2.5){\footnotesize{\textit{(iii)}}}
\hspace*{1mm}
\includegraphics[width=0.49\textwidth,height=44mm]{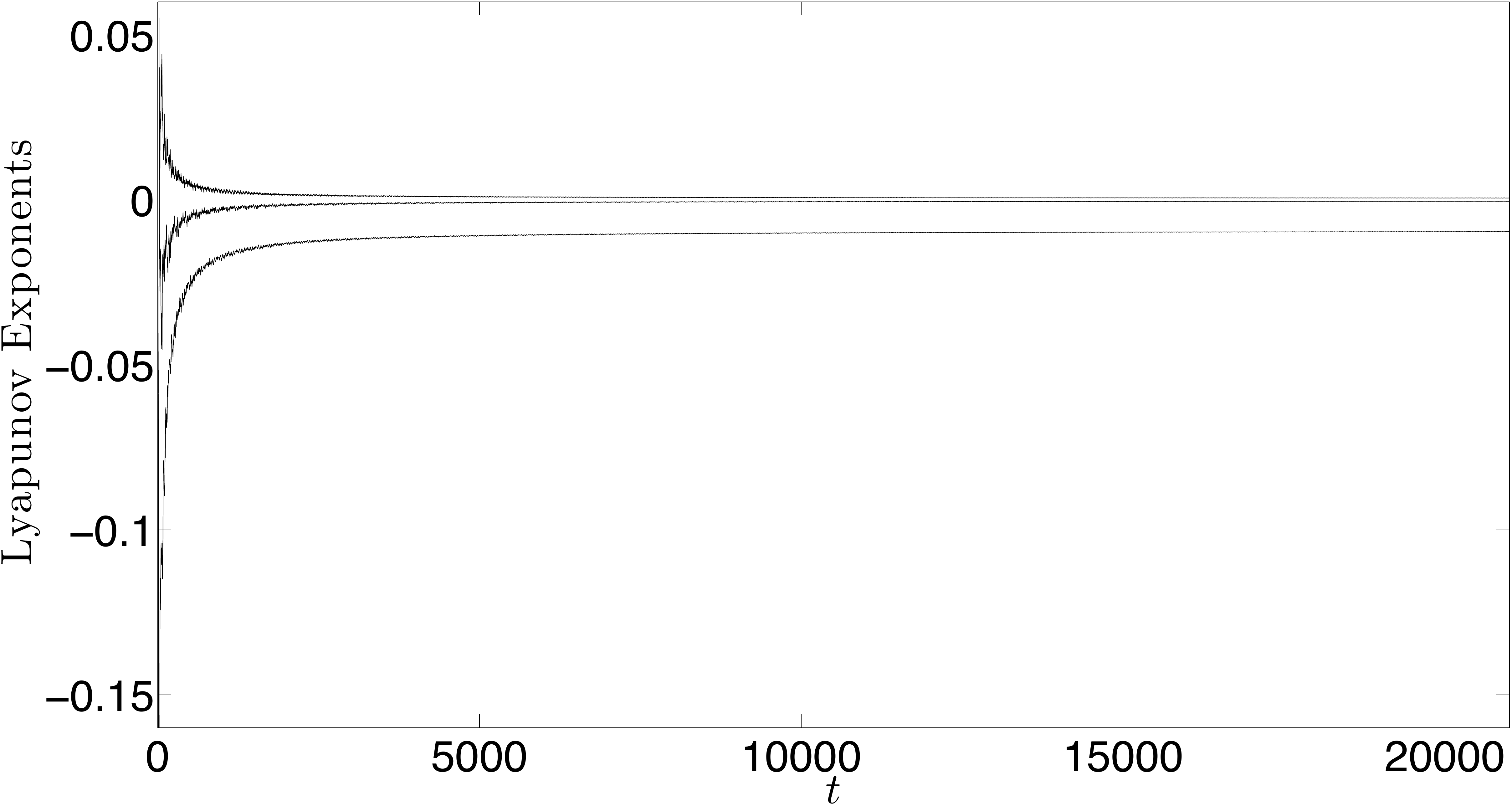}
\put(-245,2.5){\footnotesize{\textit{(iv)}}}
\vspace*{-2mm}
\caption{For $\tau=3.9$ and $\kappa=0.961$. (i) Time series and (ii) projected solution space $(Q(t),Q(t-\tau))$ for
a quasi-periodic orbit (red curve) and the
three periodic solutions (blue curves) seen in the top right inset of Figure~\ref{fig2_HSC_Biftool1D_KappaDA}(i)
for $\kappa=0.961$.
(iii) Projected Poincar\'{e} section of the quasi-periodic orbit (red) and the unstable orbit (blue) that it envelops
onto the plane $(Q(t-\tau),Q(t-\tau/2))$ for crossing of the Poincar\'{e} section $Q(t)=c=0.14$ with $Q'(t)>0$.
(iv) Initial convergence of the first three Lyapunov exponents.
}
\label{fig1_DdeStemKappaQTori01}
\end{figure}

To confirm the existence of a stable torus we performed a long time integration of the DDE
using the MATLAB \texttt{dde23} routine~\cite{Matlab}
with initial history function very close to the
unstable periodic orbit on the main branch of solutions.
For $\kappa=0.961$ with $\tau=3.9$ and all other parameters taking their values from Table~\ref{tab.model.par}
(this parameter combination is indicated by the black square in the inset within Figure~\ref{fig:KappaTau_2D_Cont})
we found a quasi-periodic torus which envelopes the unstable periodic orbit,
as illustrated in Figure~\ref{fig1_DdeStemKappaQTori01}.
The existence of the quasi periodic torus was confirmed numerically both by plotting the Poincar\'e section and by computing the Lyapunov exponents.

Recall from Section~\ref{sec.decoup.hsc} that
the DDE~\eqref{Qprime} has the infinite dimensional phase space $C=C([-\tau,0],\mathbb{R})$, consequently a hyperplane defined by a Poincar\'e section is also infinite dimensional. For $\alpha\in[0,\tau]$ and some constant $c\in\mathbb{R}$ we define the Poincar\'e section
$\cP_\alpha:=\{u_t\in C: u_t(-\alpha)=c, \; u_t'(-\alpha)>0\}$. For $\alpha=0$ this is equivalent to looking for the points $\hat t$ along the solution trajectory such that
$u(\hat t)=c$ and $u'(\hat t)>0$ and taking as the
corresponding element of the Poincar\'e section the function segment $u(t)$ for $t\in[\hat t-\tau,\hat t]$, so that $u$ is equal to $c$ at the right-hand end of the function segment. Other choices of $\alpha$ are also possible, so for example with $\alpha=\tau$ the function $u$ will be equal to $c$ at the left-hand end of the interval.

For the Poincar\'e section $\cP_0$ to be useful we need to project it into finite dimensions. The simplest way to do this is to take the value of the solution $u(t)$ at a finite set of points in $[t-\tau,t]$. Since the choice $\alpha=0$ fixes $u(t)=c$, we choose the time points $t-\tau/2$ and $t-\tau$ and project the Poincar\'e section into $\mathbb{R}^2$ by
plotting $u(t-\tau/2)$ against $u(t-\tau)$ for values of $t$ such that $u(t)=c$.
This is equivalent to the projection
$P:\cP_0\to\mathbb{R}^2$ defined by $P(u_t)=(u(t-\tau),u(t-\tau/2))$.
Figure~\ref{fig1_DdeStemKappaQTori01}(iii) reveals the results of doing this with $c=0.14$
for both the putative torus and the unstable periodic orbit that gave rise to it. This reveals the expected torus structure with the points representing the function segments in $\cP_0$ lying on a closed curve that encloses the point representing the periodic orbit in the two-dimensional projection. Since each of the red points represents separate intersections of the same orbit with the Poincar\'e section, the orbit is either quasi-periodic or of period longer than a human adult lifespan (the time integration was $30000$ days, which is longer than $82$ years).

We computed the Lyapunov exponents of the quasi-periodic orbit on the torus using the method of Breda and Van Vleck~\cite{Breda_2014}. 
Figure~\ref{fig1_DdeStemKappaQTori01}(iv) shows
the initial convergence of the numerical estimates for the three largest Lyapunov exponents.
After $3\times10^4$ days, the six largest Lyapunov exponents are estimated to be
$0.00052$, $-0.00066$, $-0.0093$, $-0.18$, $-0.29$ and $-0.29$. This reveals that up to the numerical accuracy the first two exponents are both zero, and the rest are negative, as is characteristic for a quasi-periodic two-torus.

\begin{figure}[t]
\includegraphics[width=0.49\textwidth,height=44mm]{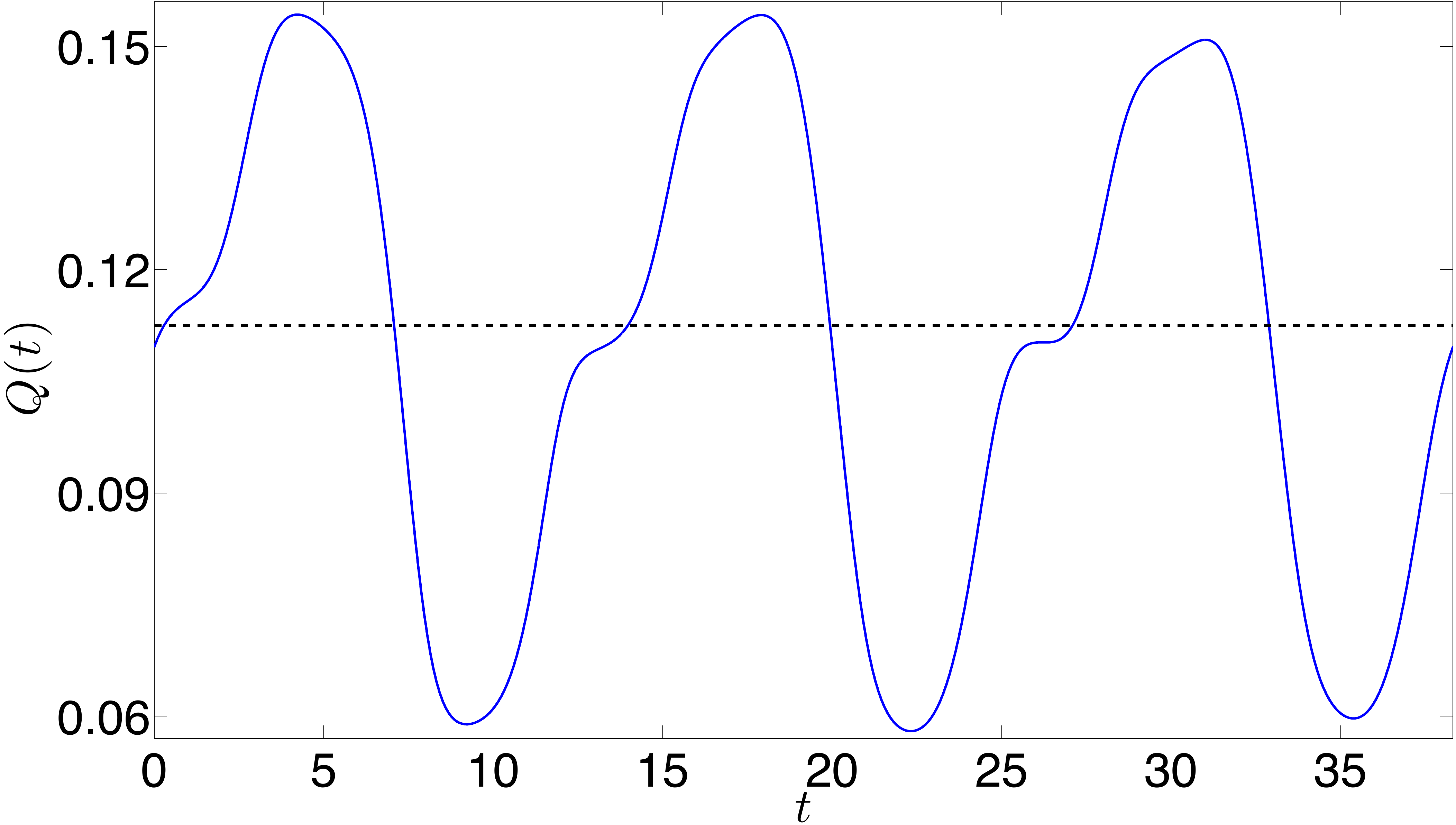}
\put(-245,2.5){\footnotesize{\textit{(i)}}}
\hspace*{1mm}
\includegraphics[width=0.49\textwidth,height=44mm]{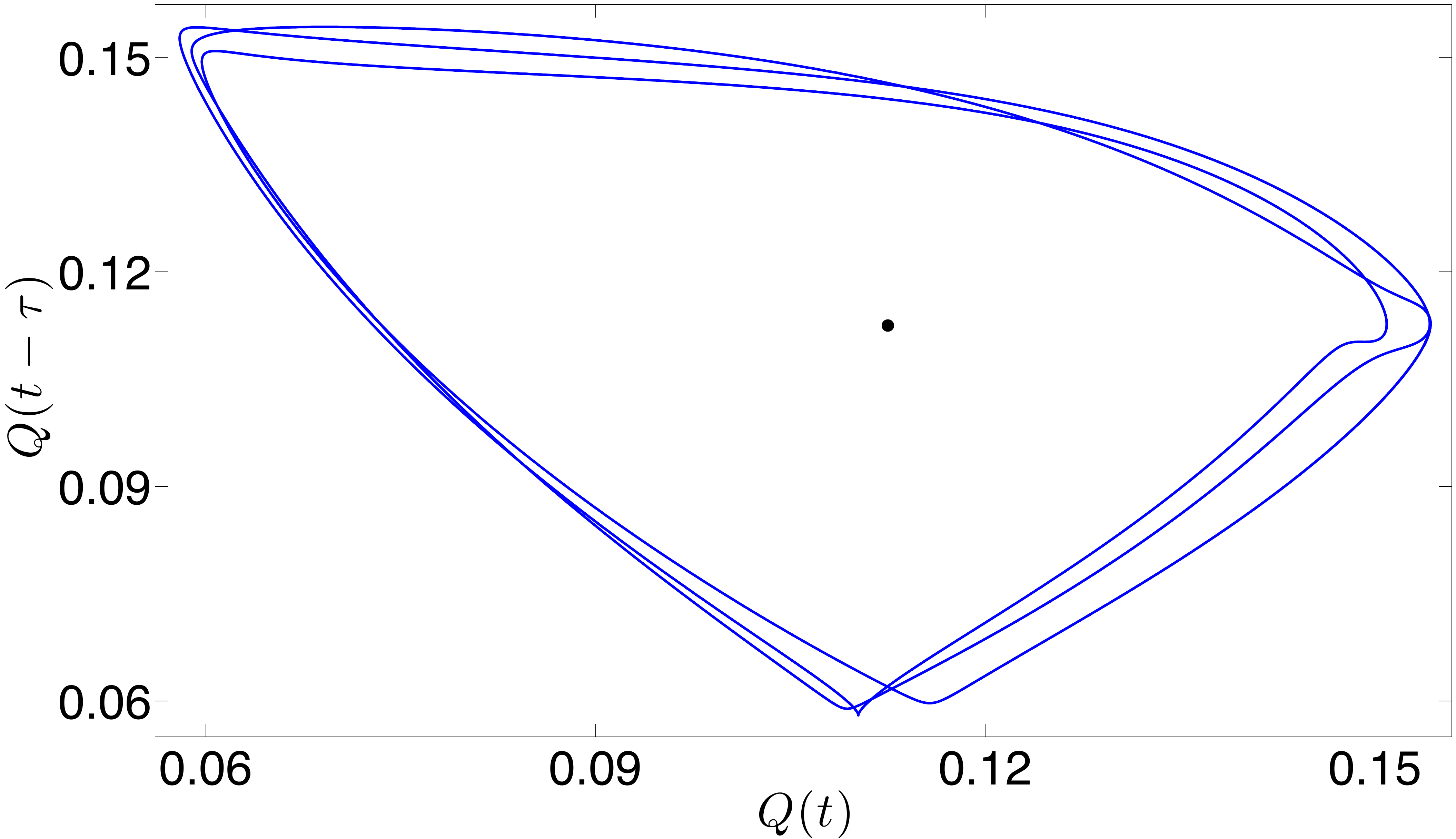}
\put(-245,2.5){\footnotesize{\textit{(ii)}}}

\vspace*{2mm}

\includegraphics[width=0.49\textwidth,height=44mm]{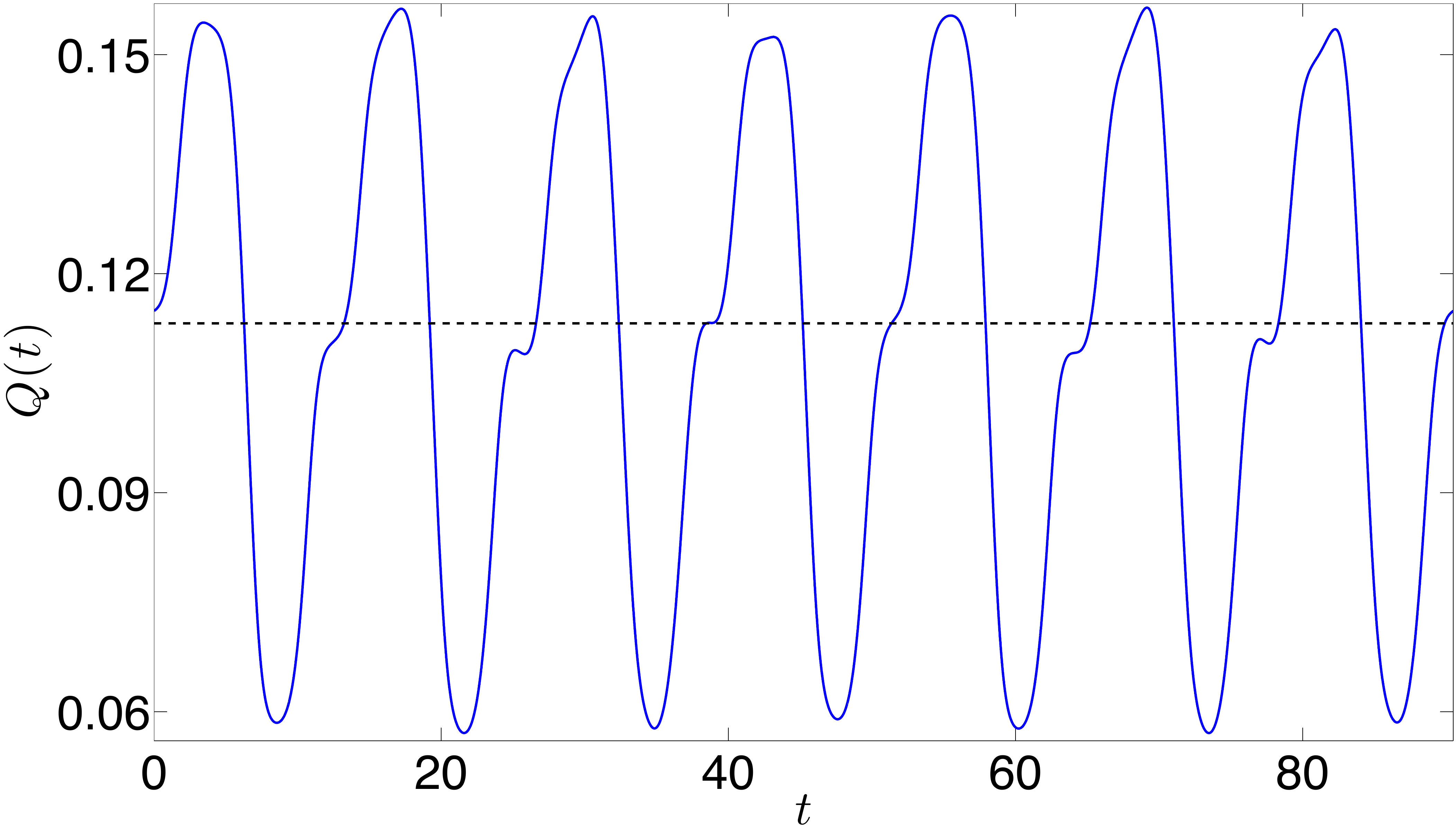}
\put(-245,2.5){\footnotesize{\textit{(iii)}}}
\hspace*{1mm}
\includegraphics[width=0.49\textwidth,height=44mm]{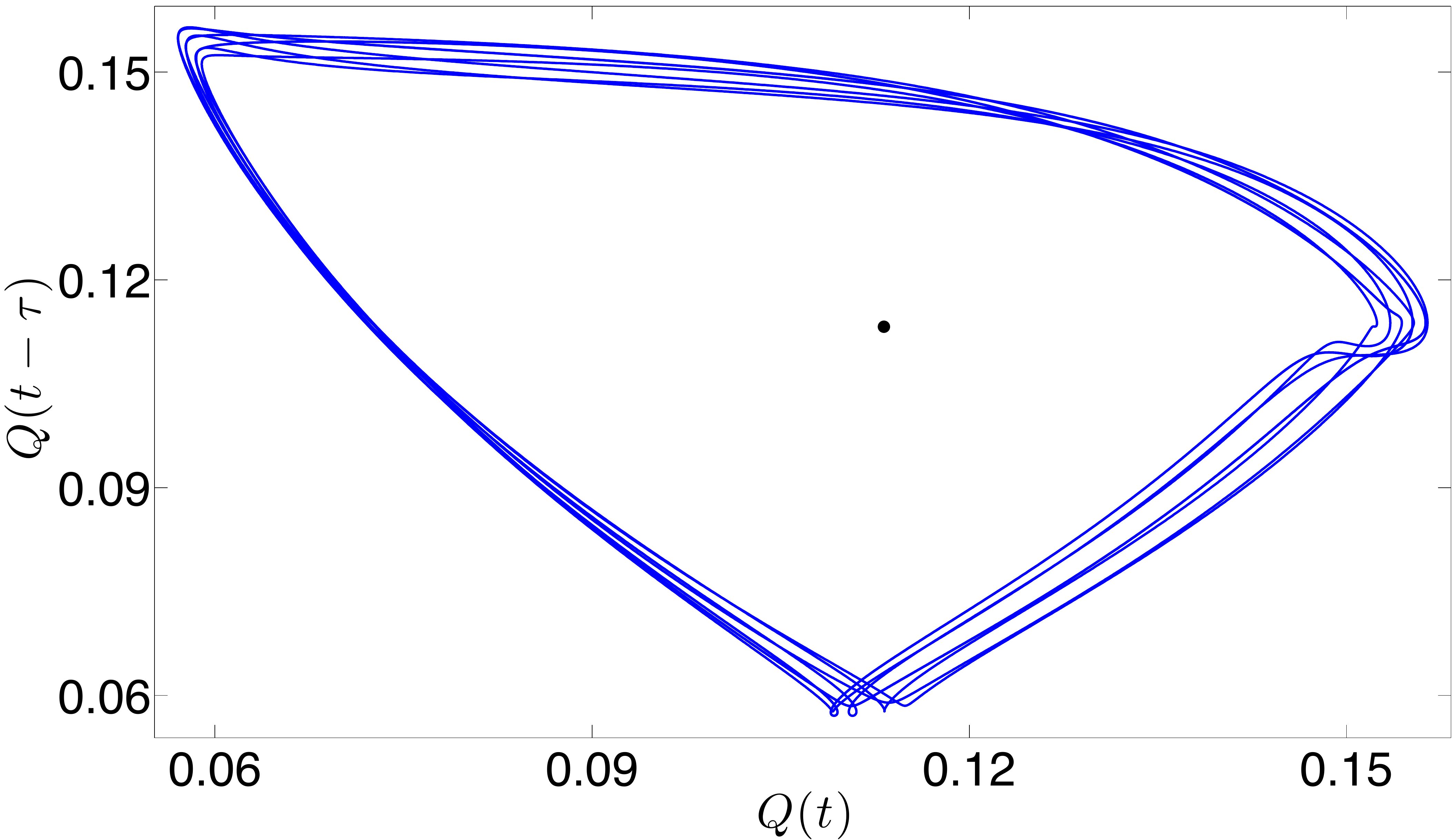}
\put(-245,2.5){\footnotesize{\textit{(iv)}}}
\vspace*{-2mm}
\caption{Time series over one period and the corresponding time-delay embedding $(Q(t),Q(t-\tau))$
for stable periodic orbits on the phase-locked torus with
$\tau=3.9$ for (i-ii) $\kappa=0.965$, and, (iii-iv) $\kappa=0.957$, which intersect the
Poincar\'e section $\cP_0$ three and seven times respectively.}
\label{fig_phase_locking}
\end{figure}

As is well known in torus dynamics~\cite{Broer_1996}, perturbing parameters in the system will change the dynamics on the torus, with parameter regions of phase locking, where there is a stable periodic orbit on the torus, interspersed with parameter sets for which the dynamics are truly quasi-periodic. So, although we cannot prove that there exists a quasi-periodic torus for
exactly the parameters illustrated in Figure~\ref{fig1_DdeStemKappaQTori01}, there will be for nearby parameter values. Equally, there will be parameter sets for which phase locking occurs on the torus, leading to stable periodic orbits of large period.
In Figure~\ref{fig_phase_locking} we show examples with $\kappa=0.957$ and $\kappa=0.965$ and all the other parameters at their values for the example of Figure~\ref{fig1_DdeStemKappaQTori01}
(in the $(\kappa,\tau)$ space of Figure~\ref{fig:KappaTau_2D_Cont} both these parameter sets are inside the torus curve close to the black square). These
show stable periodic orbits which close after going around the torus $7$ and $3$ times, leading to periodic orbits of periods approximately $90.5$ and $38.3$ days. These orbits intersect the Poincar\'e section $\cP_0$ and its projection into $\mathbb{R}^2$ seven and three times respectively.

The phase locked orbits exist over a parameter region called an Arnold tongue. DDEBiftool can be used to find the edges of these Arnold tongues (which are bounded by a fold bifurcation of periodic orbits between the interleaved stable and unstable orbits that lie on the torus inside the parameter region of the Arnold tongue). These Arnold tongues will lie in the small parameter region indicated in the inset of Figure~\ref{fig:KappaTau_2D_Cont} where torus bifurcations occur. We will not pursue the Arnold tongue structure in this work; Arnold tongues have previously been computed for DDEs, even in the state-dependent delay case~\cite{Calleja_2017}.

Another curve of torus bifurcations is visible in Figure~\ref{fig:KappaTau_2D_Cont} close to $(\kappa,\tau)=(0.7,4)$. The corresponding torus bifurcation can be seen in Figure~\ref{fig2_HSC_Biftool1D_KappaDA} at $\kappa\approx0.65046$
(with $\tau=3.9$), where the periodic orbit loses stability in a torus bifurcation. In this case there is not a second corresponding torus bifurcation where the periodic orbit regains stability. Instead there is a period-doubling bifurcation (which is a resonant torus bifurcation) near $\kappa=0.767$, but the periodic orbit on the principal branch does not regain stability at this point. That this period-doubling bifurcation is associated with the neighbouring torus bifurcation can be surmised from Figure~\ref{fig:KappaTau_2D_Cont} where we see that the endpoints of the curve of torus bifurcations lie on the period-doubling bifurcation curve. The torus dynamics are likely to be more complicated in this case, but we did not explore them.


\subsection{Chaotic Dynamics}
\label{sec.chaos}

Having found long period and quasi-periodic orbits, it is natural to also ask whether~\eqref{Qprime} admits chaotic solutions.
Kaplan and Yorke~\cite{KaplanYorke1979}
defined an attractor dimension, now known as the Lyapunov dimension, to be
\begin{equation} \label{lyapdim}
d=k-\frac{1}{\lambda_{k+1}}\sum_{j=1}^k\lambda_j
\end{equation}
where the Lyapunov exponents are ordered so $\lambda_1\geq \lambda_2\geq\ldots$, and $k$ is the largest integer so that the sum of the first $k$ exponents is non-negative, thus necessarily $\lambda_{k+1}<0$, and $d\in[k,k+1)$. For the torus seen in the previous section with $\lambda_1=\lambda_2=0>\lambda_3$ equation~\eqref{lyapdim} gives a dimension of $d=2$, as expected for a torus.

One generally accepted indication of chaos is the presence of a positive Lyapunov exponent, in which case the Lyapunov dimension will be larger than two.
We will investigate the existence of chaotic solutions for~\eqref{Qprime} by
numerically computing the Lyapunov exponents, again using the method
of Breda and Van Vleck~\cite{Breda_2014}.
\begin{figure}[t]
\hspace*{-2mm}
\begin{minipage}{0.728\columnwidth}
\includegraphics[width=\columnwidth]{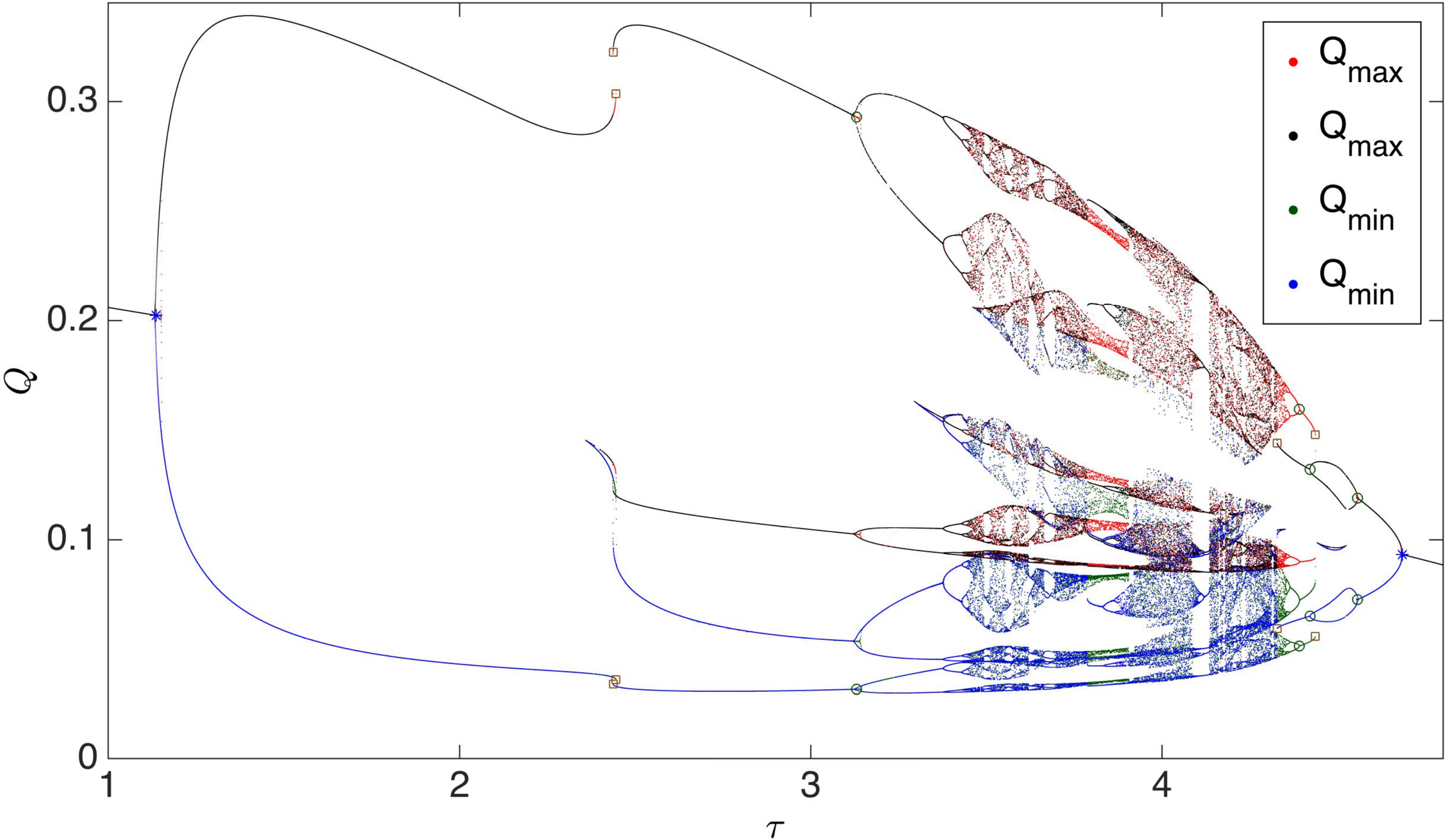}
\hspace*{-2mm}
\end{minipage}
\begin{minipage}{0.25\columnwidth}
\vspace*{-0mm}
\includegraphics[width=\columnwidth]{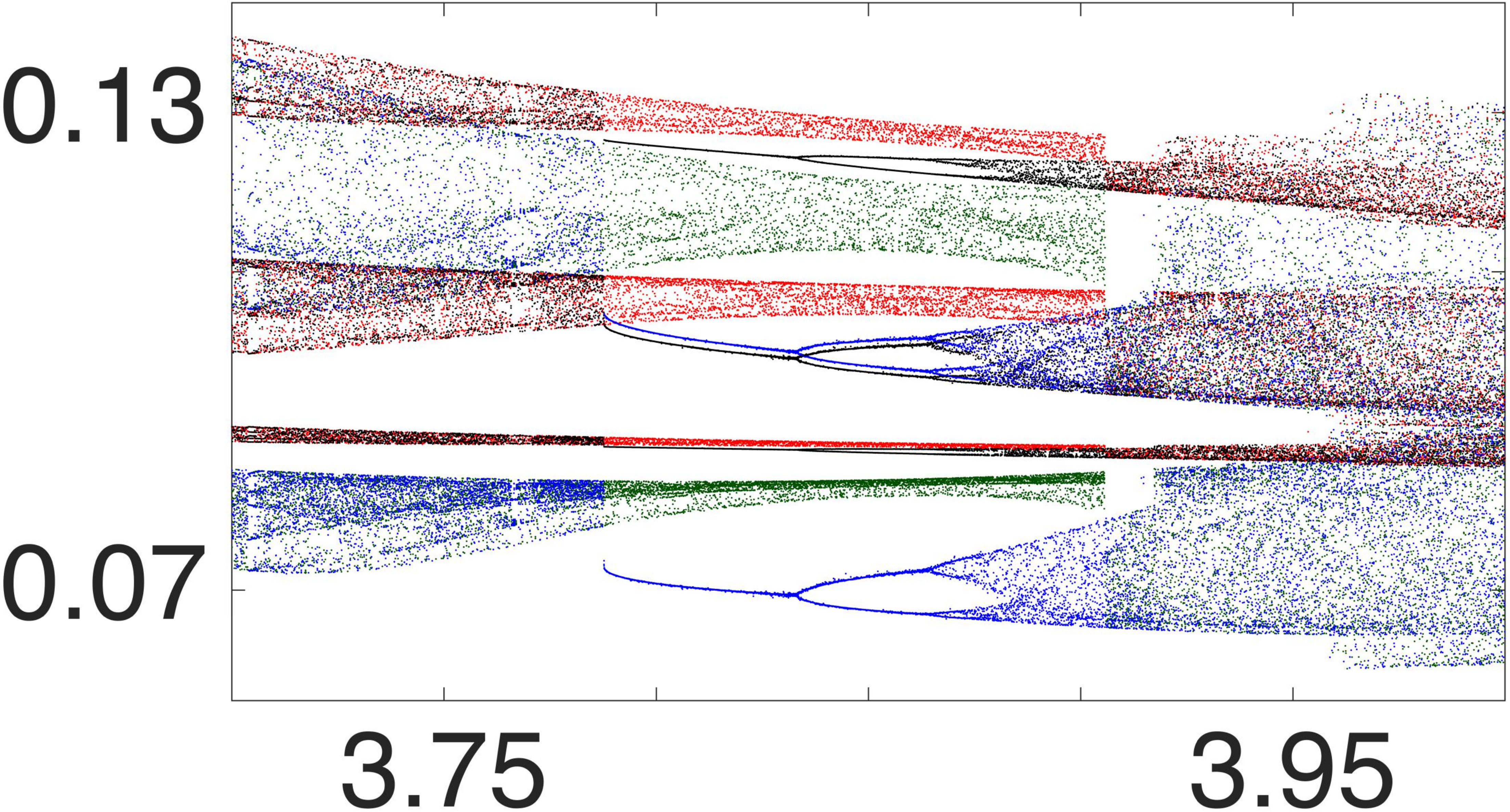}
\\[1mm]
\includegraphics[width=\columnwidth]{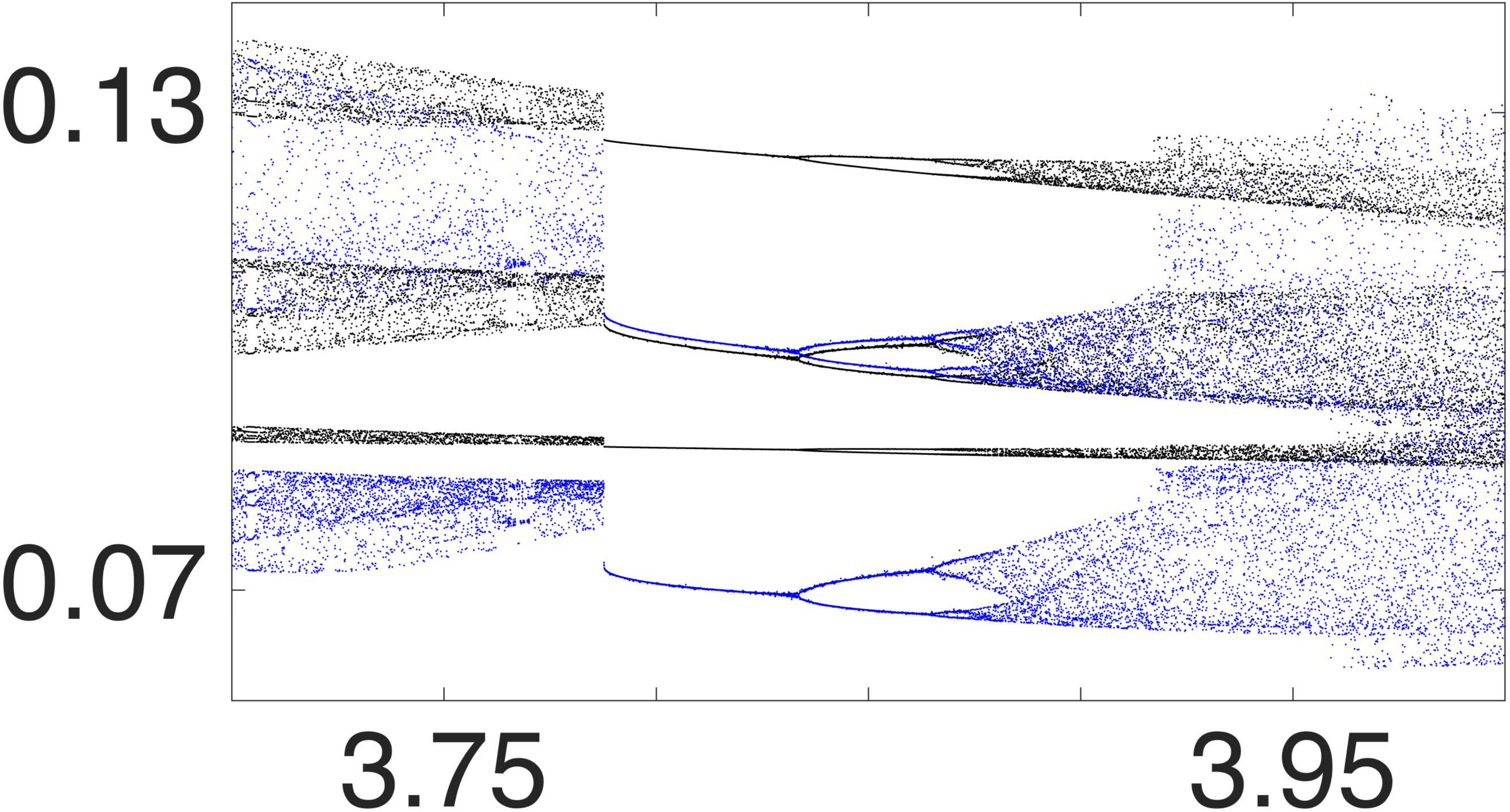}
\\[1mm]
\includegraphics[width=\columnwidth]{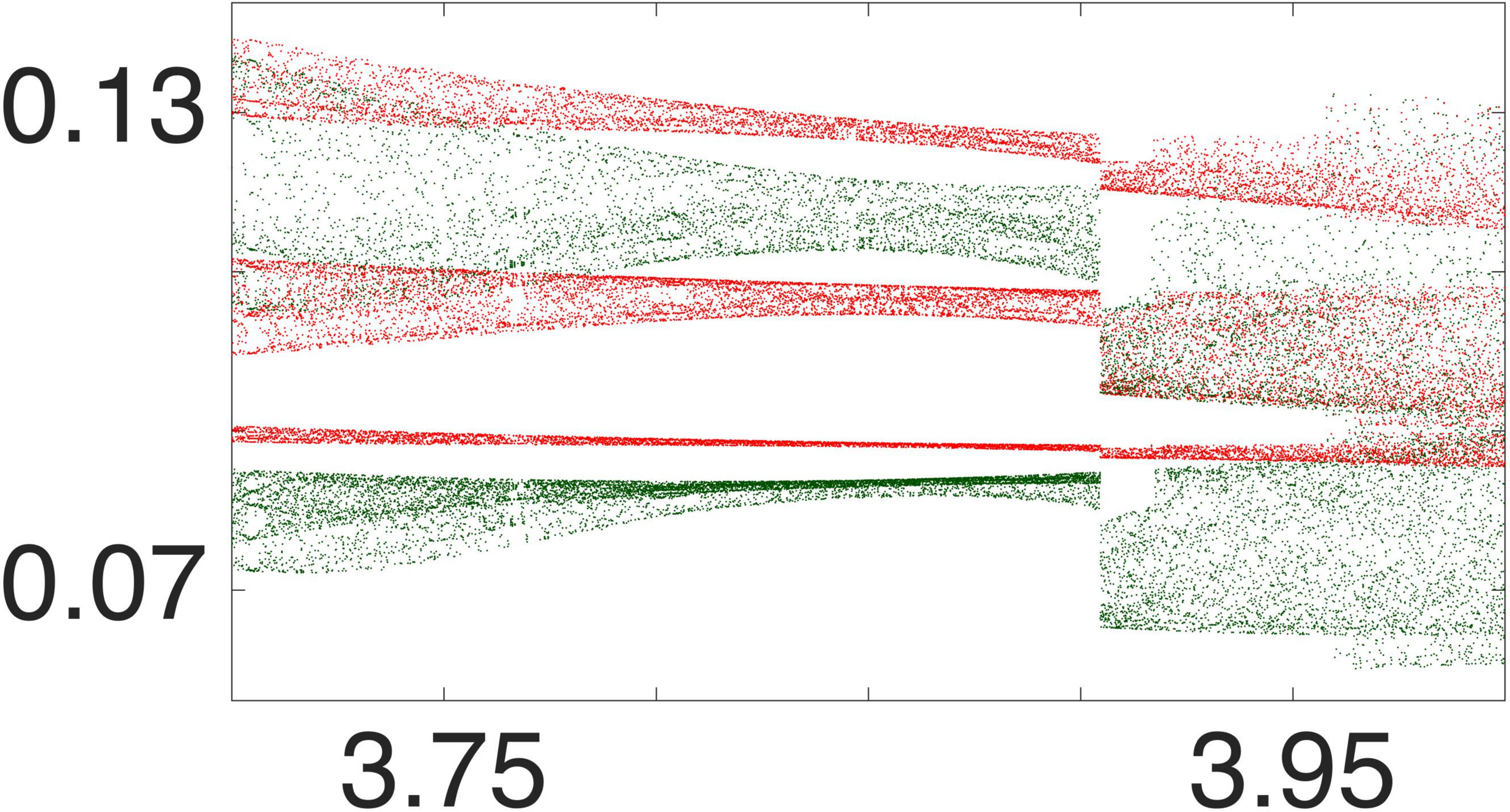}
\end{minipage}
\vspace*{-2mm}
\caption{Orbit diagram showing local maxima and minima of solutions of \eqref{Qprime} as a function of the delay $\tau$,
with $\kappa=0.865$ and other parameters taking their values from Table~\ref{tab.model.par}.
The red and green dots denote respectively the local maxima and minima computed
along a mesh with 30400 points for increasing $\tau$,
while the black and blue dots denote the local maxima and minima
computed by decreasing $\tau$.
The bifurcation points of primary branch are identified using the same symbols as in Figures~\ref{fig:KappaCont} and~\ref{fig2_HSC_Biftool1D_KappaDA}.
The upper side panel shows a detail from the main panel, while the other two side panels show just the
decreasing and increasing parameter scans, illustrating bistability and hysteresis in the system.}
\label{fig.chaos0}
\end{figure}
\begin{figure}[th]
\includegraphics[width=0.49\textwidth,height=44mm]{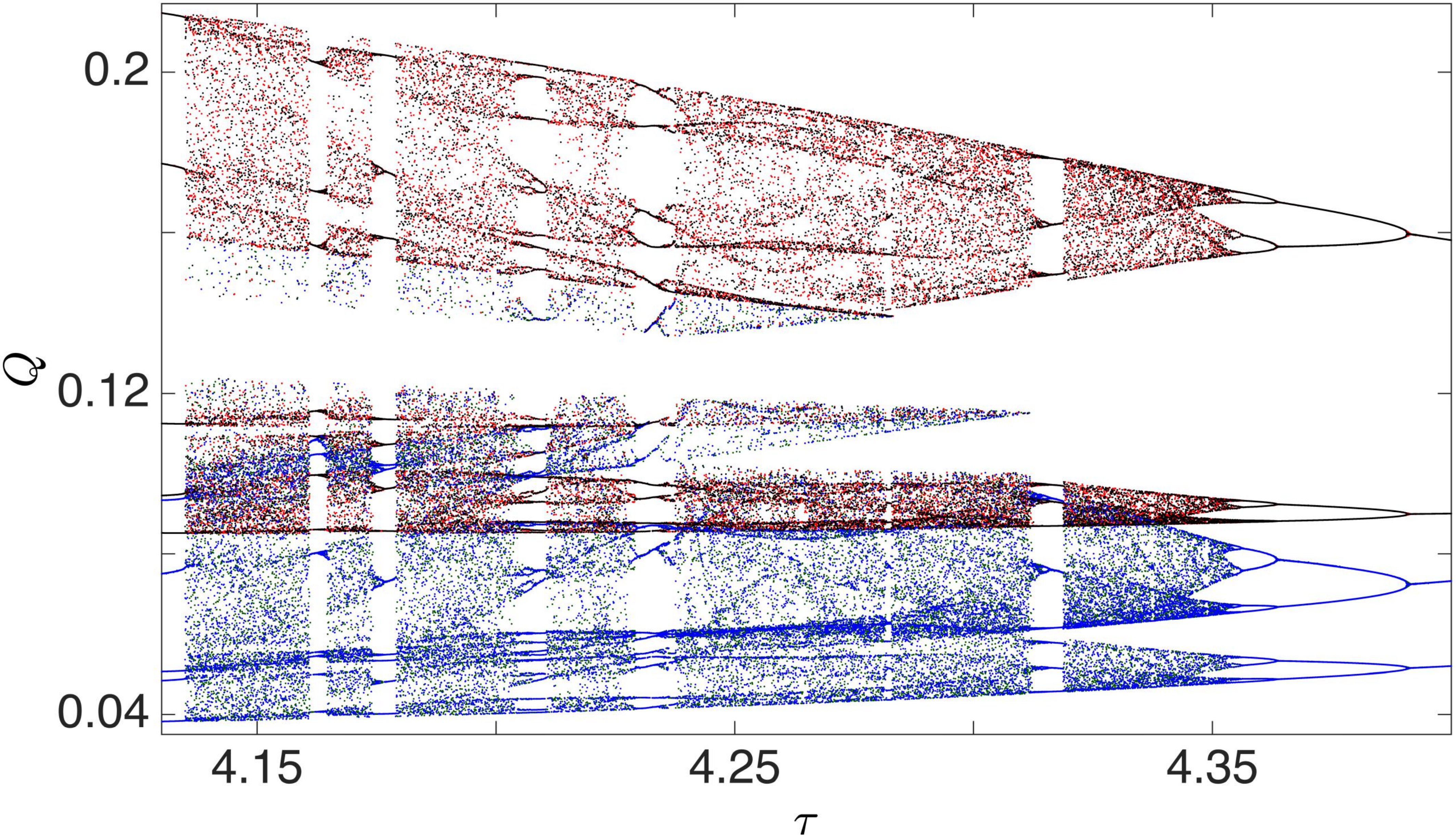}
\put(-245,2.5){\footnotesize{\textit{(i)}}}
\hspace*{1mm}
\includegraphics[width=0.49\textwidth,height=44mm]{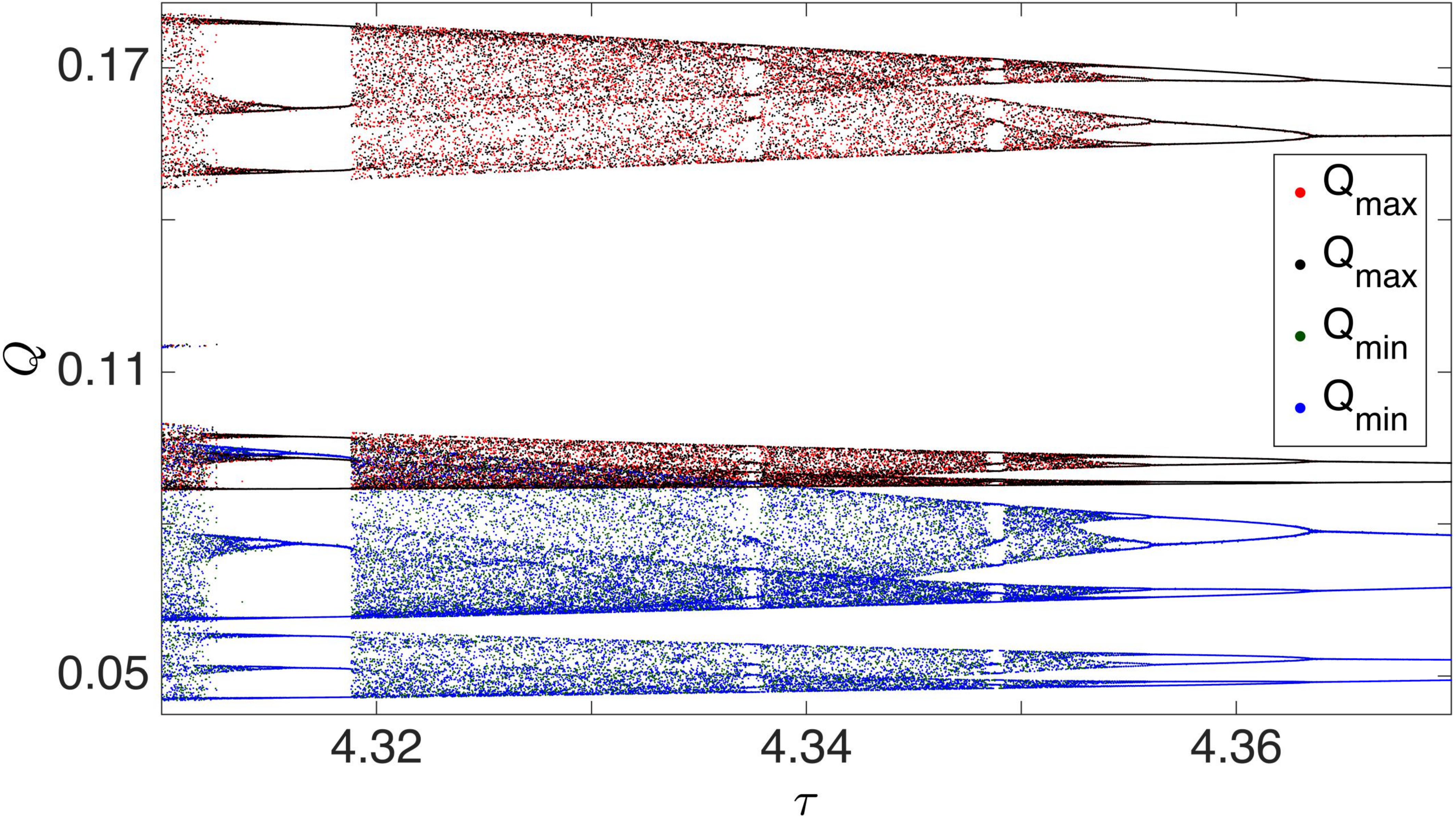}
\put(-245,2.5){\footnotesize{\textit{(ii)}}}
\vspace*{-2mm}
\caption{A sequence of windows of periodic dynamics with parameter intervals of apparent chaotic dynamics.
Panel (i) shows a zoom of part of the Figure~\ref{fig.chaos0}, while panel (ii) shows a zoom of part of panel (i).
}
\label{fig.chaosB}
\end{figure}

In Figure~\ref{fig2_HSC_Biftool1D_KappaDA} we see that for $\kappa\in(0.795,0.929)$ both steady states and the periodic orbits on the main branch and on the period doubled branch are all unstable. However from Theorem~\ref{theorem.bound} we know that the dynamics must remain bounded, and so there must be a global attractor for these parameters. This parameter interval of unstable solutions for the $\kappa$ continuation lies between two period doubling bifurcations in Figure~\ref{fig2_HSC_Biftool1D_KappaDA}, which is also inside the lobe of period-doubling bifurcations depicted in the two-parameter continuation in $\kappa$ and $\tau$ in Figure~\ref{fig:KappaTau_2D_Cont}, and we investigate the dynamics within this region.

In Figure~\ref{fig.chaos0} we present an orbit diagram for~\eqref{Qprime} as $\tau$ is varied across this region with $\kappa=0.865$. Orbit diagrams are usually produced for maps, and we reduce the solution of~\eqref{Qprime} to a map by considering the crossings of a Poincar\'e section. Previously, we considered Poincar\'e sections with $Q(t)$ constant, which would not work so well in this case because the value of $Q^*$ changes
as $\tau$ is varied,
and we would need to vary the constant to ensure that the orbits cross the Poincar\'e section. Instead, we consider the local maxima and minima of $Q(t)$ along the solution, or equivalently the points where $Q'(t)=0$ with $Q''(t)<0$ or $Q''(t)>0$ (respectively).
For each value of $\tau$ using the MATLAB \texttt{dde23} routine~\cite{Matlab} we integrate through a time interval of $50\tau$ days, then plot the value of $Q$ at its last local maxima and minima.
Since the dynamics are more interesting for some $\tau$ values then others,
we defined a $\tau$ mesh with $9121$ points from $1$ to $3.4$, $19000$ points from $3.4$ to $4.4$, and
$2281$ points from $4.4$ to $5$.
These three meshes were combined to form a mesh of $30400$ points
from 1 to 4 for increasing $\tau$.
A second mesh with 30399 points interleaved with the previous mesh was used for decreasing $\tau$.
For each mesh point the last $\tau$ time units of the solution was used as the initial function to compute the solution at the next mesh point.
The results displayed in Figure~\ref{fig.chaos0} clearly reveal the bifurcations already shown in Figure~\ref{fig:KappaTau_2D_Cont} including the Hopf bifurcations at $\tau=1.1364$ and $4.6841$,
the fold bifurcations near $\tau= 2.4379$, $2.4451$, $4.3281$ and $4.4364$ and the period doubling bifurcations at $\tau= 3.1303$, $4.3909$, $4.4215$ and $4.5575$. Between
those period doubling bifurcations, Figure~\ref{fig.chaos0} reveals numerous period doubling cascades and several parameter intervals of apparent chaotic dynamics with windows of periodic dynamics.

For some intervals of parameter values the results of sweeping left to right and right to left are significantly different, revealing the bistability of attracting states and hysteresis between them. The side panels to Figure~\ref{fig.chaos0} illustrate this for $\tau\approx3.85$, where increasing $\tau$ sequentially appears to reveal chaotic dynamics, but decreasing $\tau$ reveals a stable periodic-orbit which appears to undergo a period doubling cascade leading to a small interval of parameter values for $\tau\approx3.9$ for which there are apparently co-existing chaotic attractors.

Figure~\ref{fig.chaosB} shows two successive magnifications from a small region of Figure~\ref{fig.chaos0}. To reveal the finer structure, we recomputed the orbit diagram for each of these $\tau$ intervals for $30000$ equally spaced increasing $\tau$ values, and a second interleaved mesh with one fewer point
with decreasing $\tau$ values.
Figures~\ref{fig.chaos0} and~\ref{fig.chaosB}(i)-(ii) together suggest a self-similarity of the structure with sequences of windows of periodic dynamics separated by intervals of apparent chaotic dynamics, on ever smaller parameter intervals.
While it would be interesting to study the scaling in the period doubling cascades, this is very difficult to do because the mapping is only implicitly defined, and requires that we numerically solve the DDE~\eqref{Qprime} between each extremum of $Q(t)$. Instead here, we will investigate the nature of the chaotic solutions.
%
%
\begin{figure}[t]
\includegraphics[width=0.49\textwidth,height=44mm]{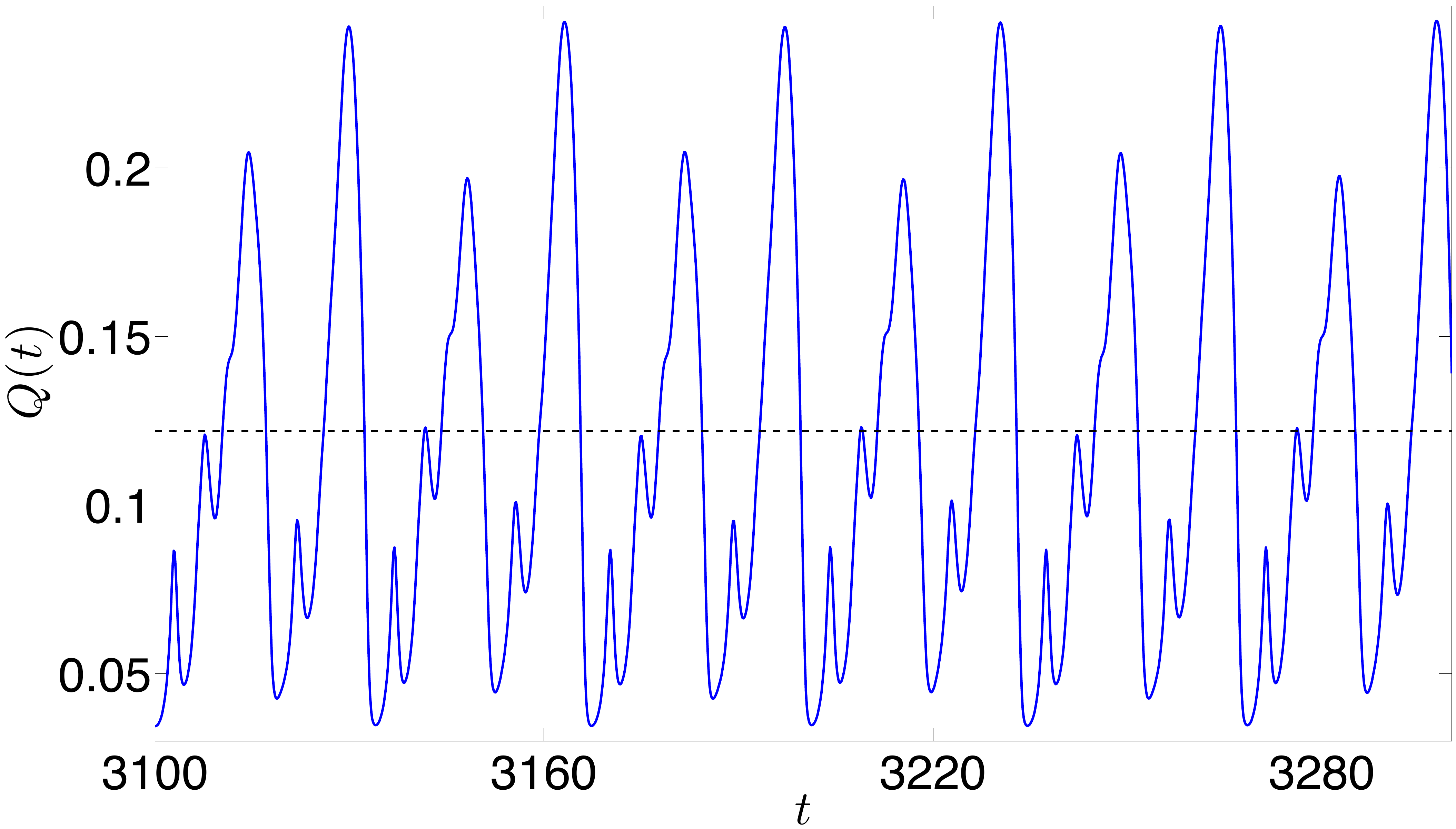}
\put(-245,2.5){\footnotesize{\textit{(i)}}}
\hspace*{1mm}
\includegraphics[width=0.49\textwidth,height=44mm]{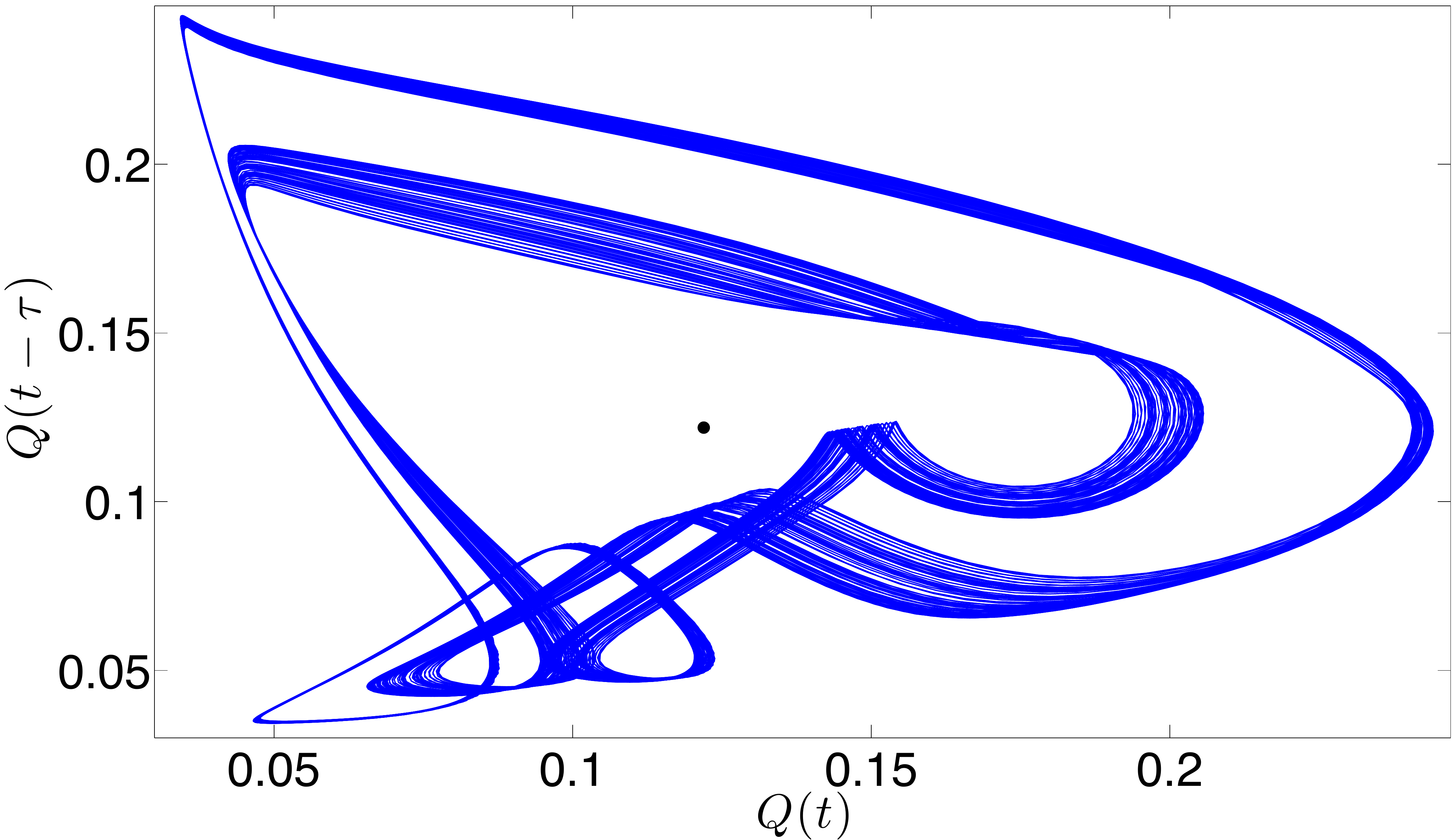}
\put(-245,2.5){\footnotesize{\textit{(ii)}}}

\vspace*{2mm}

\includegraphics[width=0.49\textwidth,height=44mm]{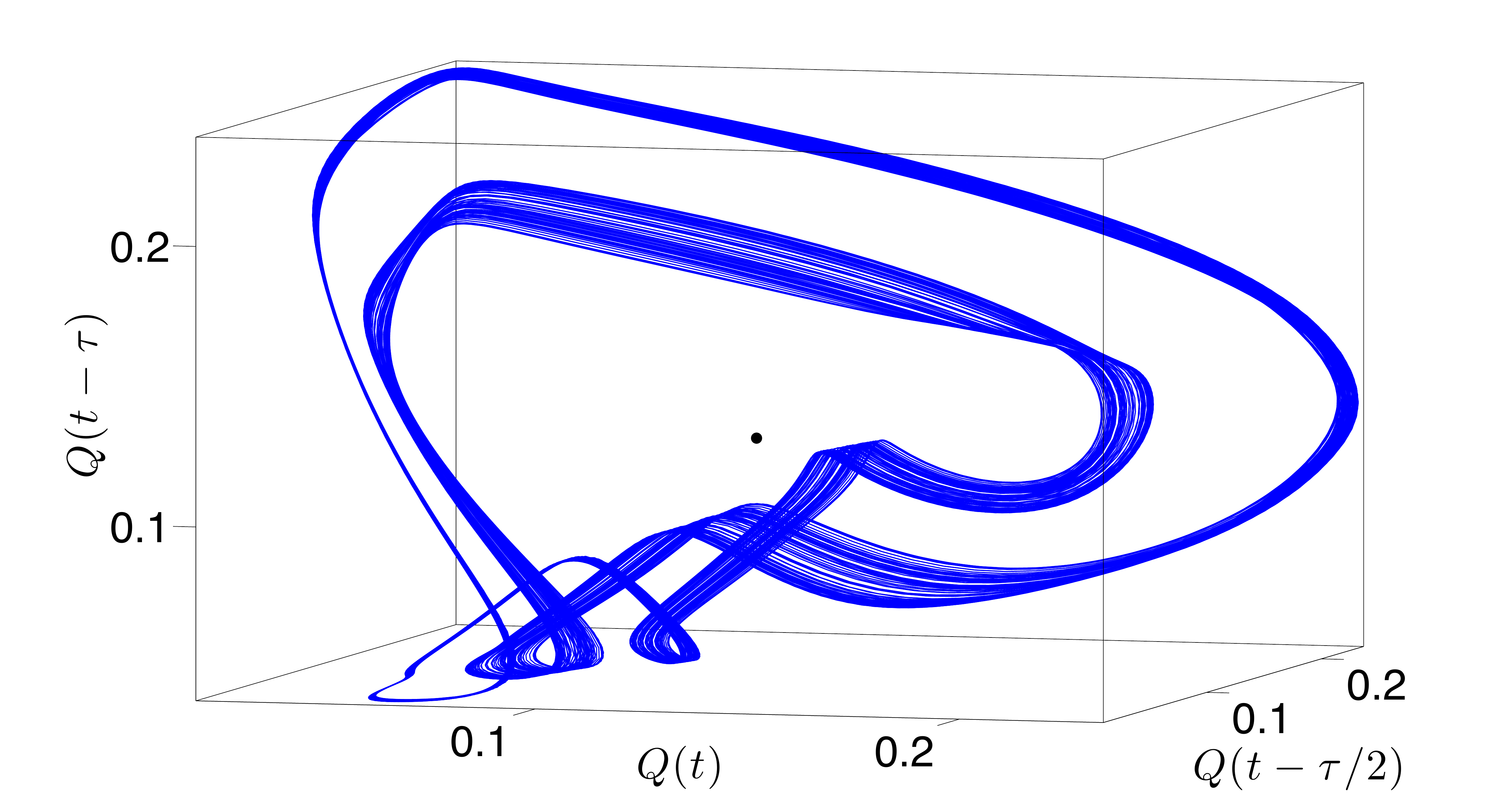}
\put(-245,2.5){\footnotesize{\textit{(iii)}}}
\hspace*{1mm}
\includegraphics[width=0.49\textwidth,height=44mm]{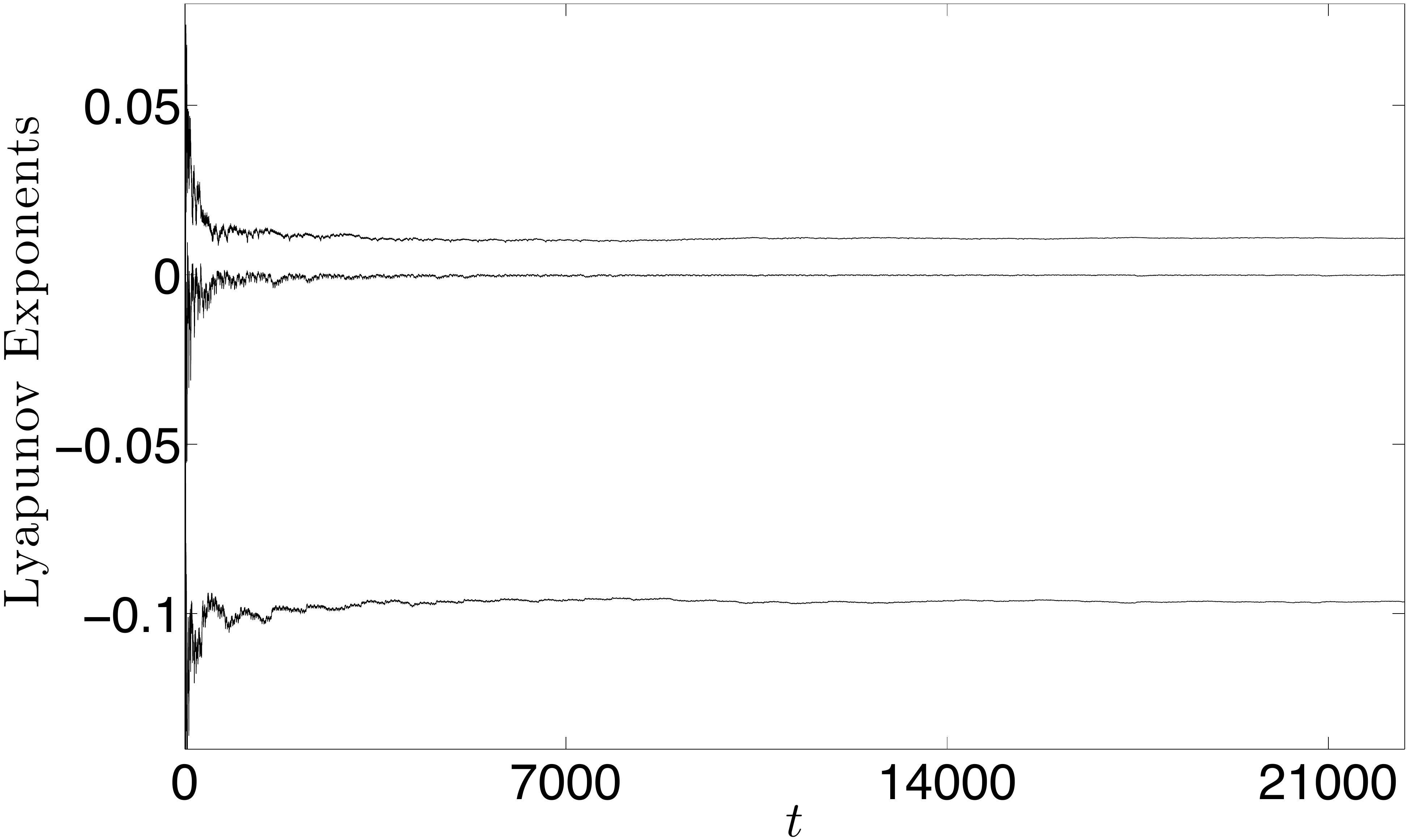}
\put(-245,2.5){\footnotesize{\textit{(iv)}}}
\vspace*{-2mm}
\caption{Chaotic orbit for $\tau=3.9$ and $\kappa=0.865$.
(This parameter set is indicated by the red triangle in Figure~\ref{fig:KappaTau_2D_Cont} inset)
(i) Orbit segment.
(ii) Time-delay embedding  $(Q(t),Q(t-\tau))$ and (iii) solution space $(Q(t),Q(t-\tau/2),Q(t-\tau))$  of the solution for $t\in[3000,6000]$.
The black line and black dots in (i)-(ii) represent the unstable steady state $Q^*$.
(iv) Convergence of the first three Lyapunov exponents.
}
\label{fig.chaos1}
\end{figure}

With $\kappa=0.865$ and $\tau=3.9$ the orbit diagram suggests that the dynamics should be chaotic, and this case is illustrated in Figure~\ref{fig.chaos1}. At first glance the time series in panel (i) resembles a period-doubled solution, but the maxima close to $Q(t)=0.2$ actually alternate in height, so the solution is closer to a period-quadrupled solution. However, the time-delay embeddings in panel (ii) and (iii) appear to show that the orbit is not periodic but that there is a very structured low-dimensional attractor.
The solutions were computed by taking a constant initial history function close to $Q^*$ and integrating with \texttt{dde23} through the transient dynamics until the orbit converges to the attractor. The segment of the solution trajectory that is displayed in Figure~\ref{fig.chaos1}(ii) and (iii) spans 3000 days.
The initial convergence of the first three Lyapunov exponents is illustrated
in Figure~\ref{fig.chaos1}(iv), but the full computation
of the exponents,
using the method
of Breda and Van Vleck~\cite{Breda_2014},
is over a time interval of $30000$ days, or 82 years.
The dynamics are not periodic over this time interval and the
leading Lyapunov exponents are computed numerically to be
$0.0107$, $-0.0002$, $-0.0966$ and $-0.1577$. The second Lyapunov exponent here is $0$ to numerical accuracy, and the presence of a positive Lyapunov exponent indicates chaos. The appearance of the orbit being close to a period-quadrupled orbit is most likely just due to the provenance of the chaotic orbit being created through a period-doubling cascade. Had we only looked at the time-series we could have been wrongly led to conclude that the dynamics was not chaotic; the time-series of the solution alone is very rarely sufficient to determine the nature of the dynamics in the interesting cases.
For the attractor shown in Figure~\ref{fig.chaos1} the Lyapunov dimension is computed from~\eqref{lyapdim} to be $d=2.11$.

\begin{figure}[t]
\includegraphics[width=0.49\textwidth,height=44mm]{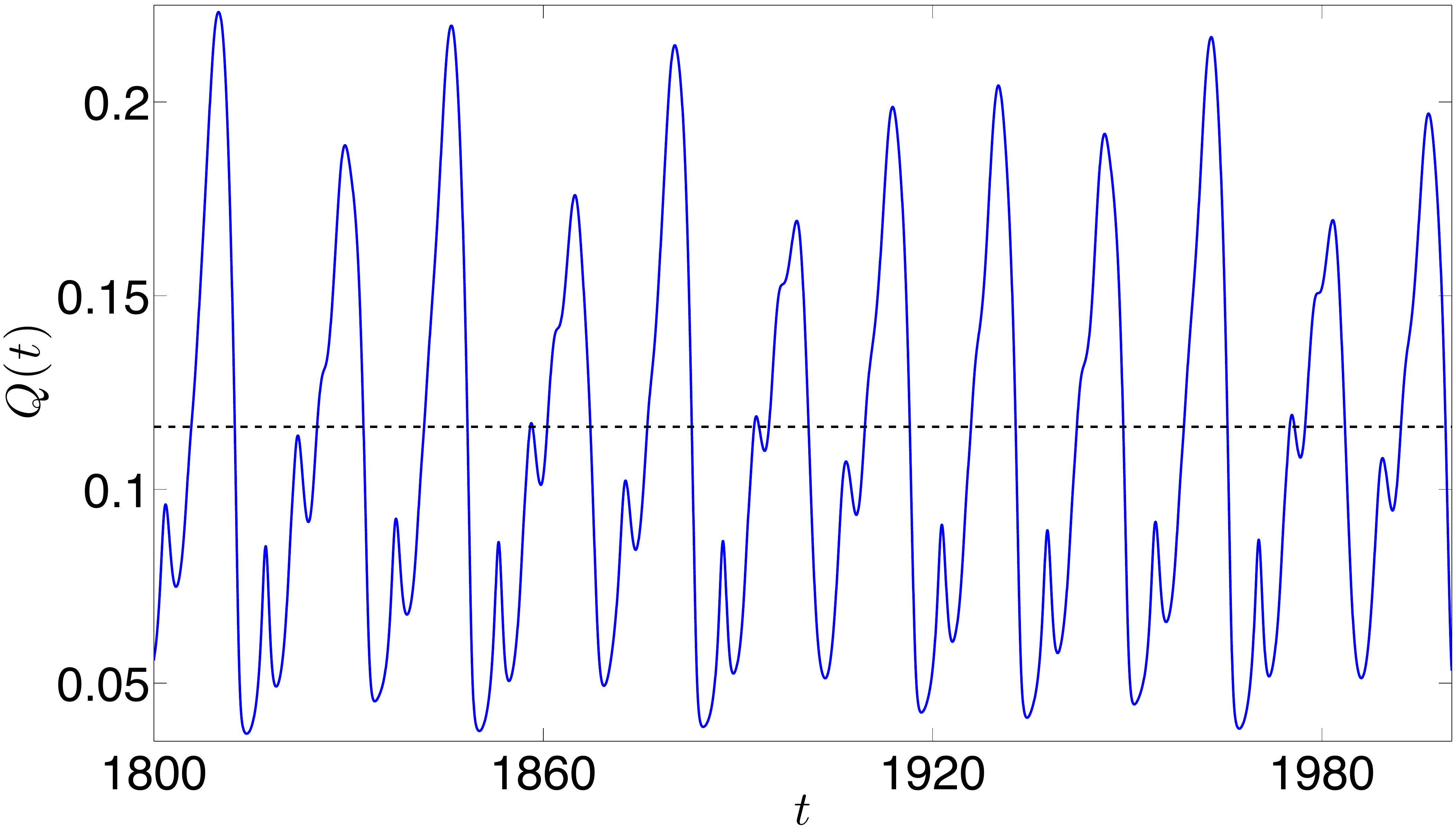}
\put(-245,2.5){\footnotesize{\textit{(i)}}}
\hspace*{1mm}
\includegraphics[width=0.49\textwidth,height=44mm]{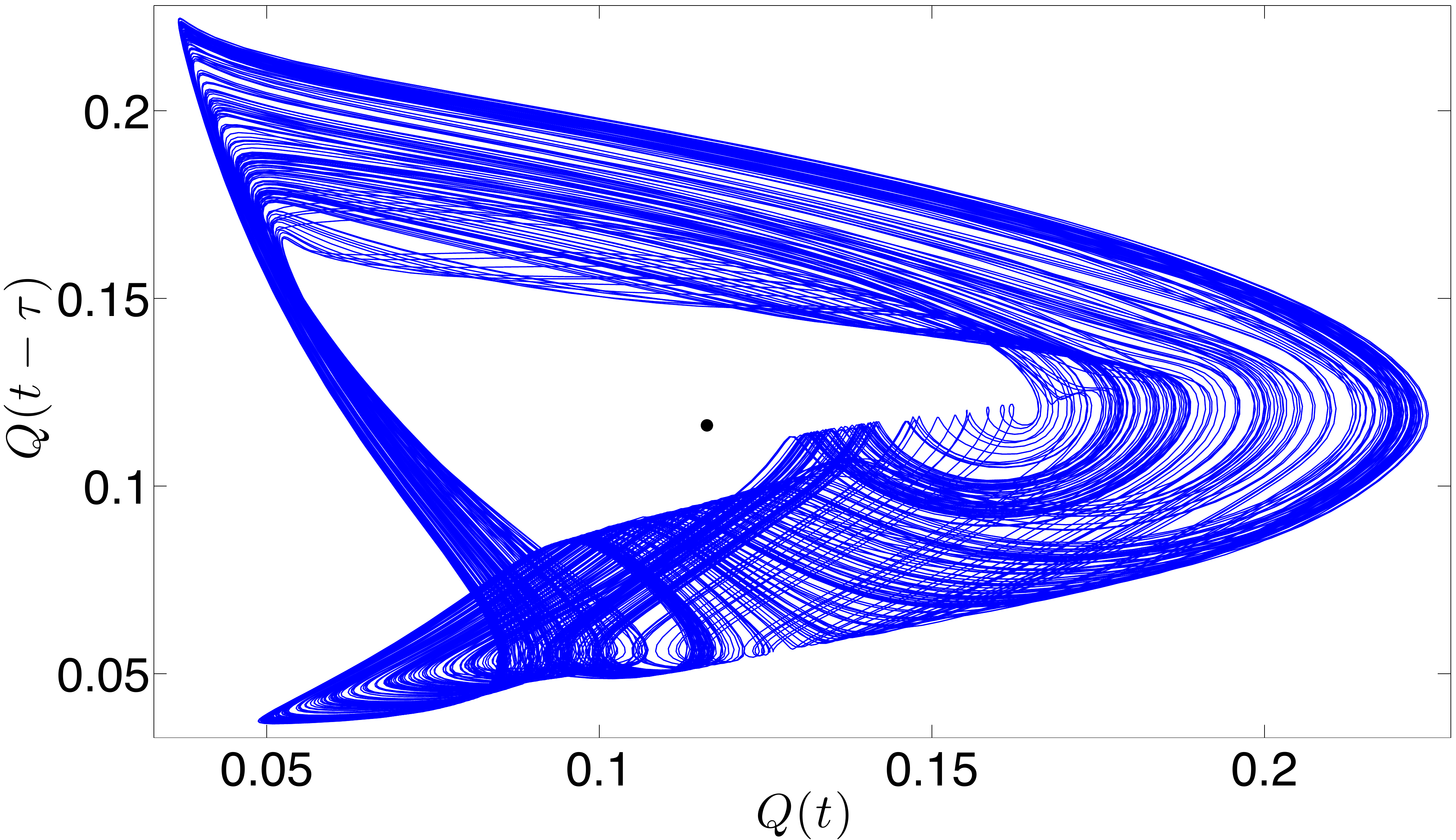}
\put(-242,2.5){\footnotesize{\textit{(ii)}}}

\vspace*{2mm}

\includegraphics[width=0.49\textwidth,height=44mm]{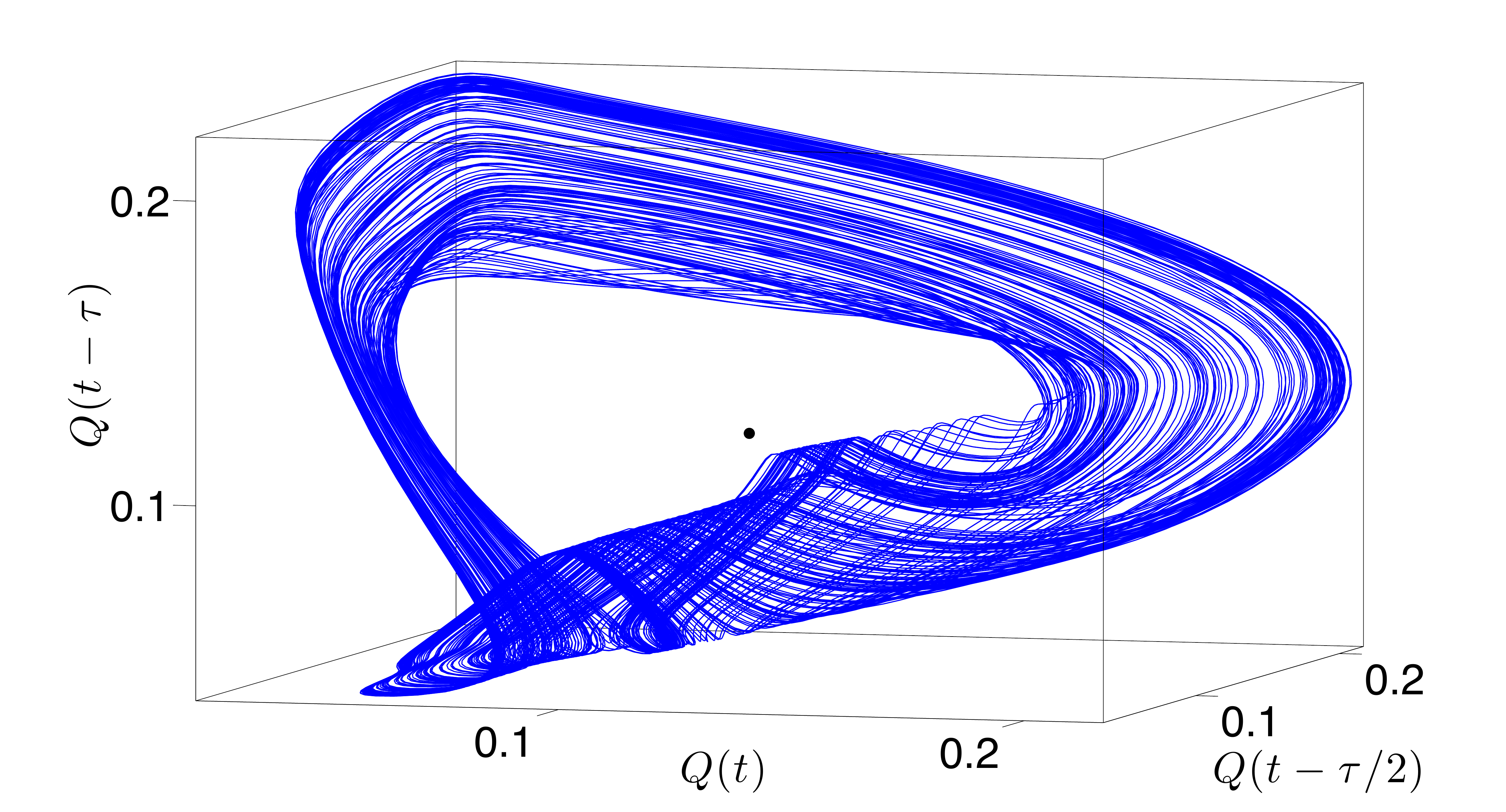}
\put(-240,2.5){\footnotesize{\textit{(iii)}}}
\hspace*{1mm}
\includegraphics[width=0.49\textwidth,height=44mm]{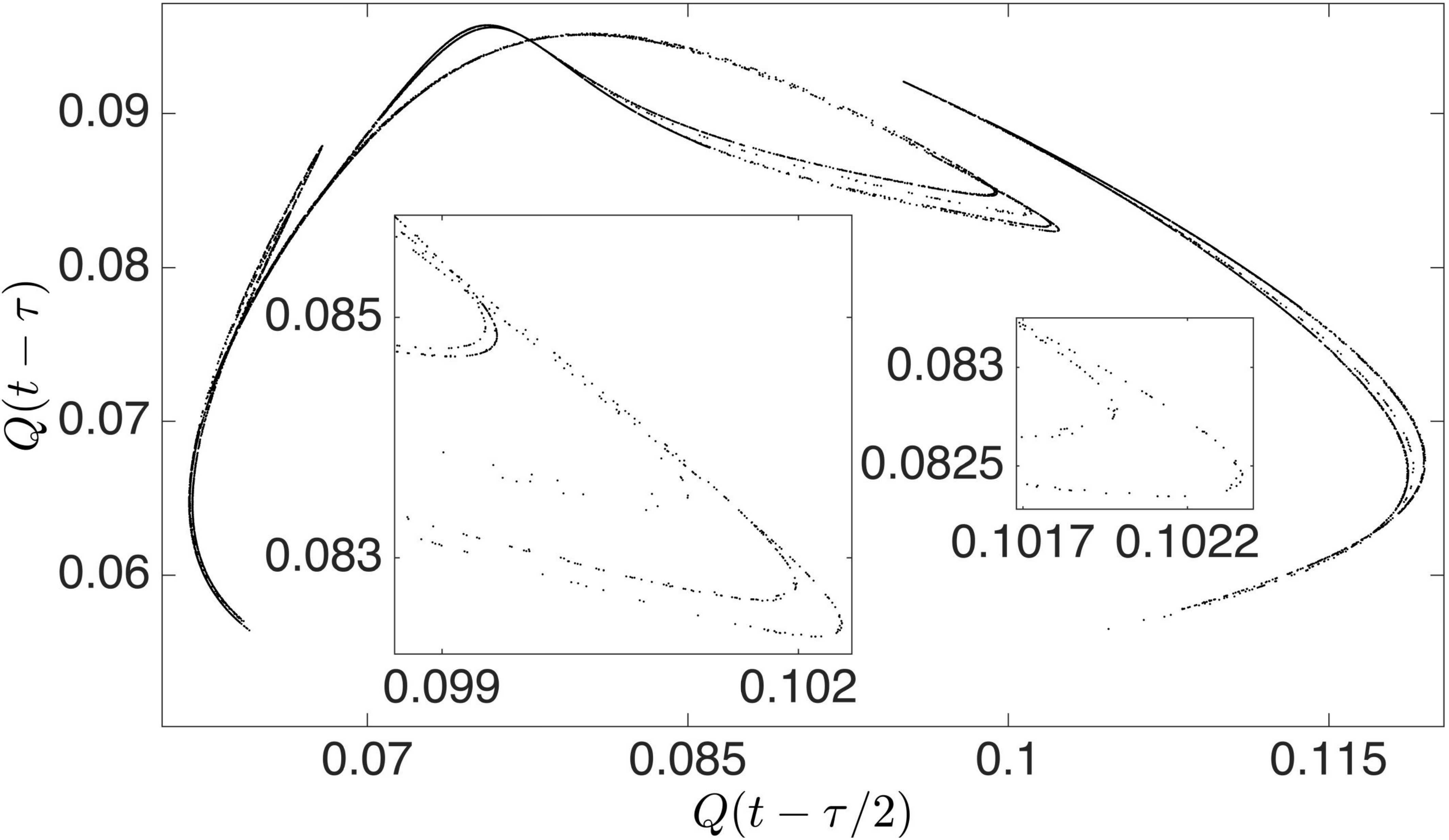}
\put(-240,2.5){\footnotesize{\textit{(iv)}}}
\vspace*{-2mm}
\caption{Chaotic orbit for $(\kappa,\tau)=(0.865,4.07)$.
(i) Segment of solution time series on the attractor.
Time-Delay embeddings (ii) $(Q(t),Q(t-\tau))$, and, (iii) $(Q(t),Q(t-\tau/2),Q(t-\tau))$ for $t\in[3000,6000]$.
The black line and black dots in (i)-(iii) represent the unstable steady state.
(iv) Projection $P(u_t)=(u_t(-\tau/2),u_t(-\tau))$ of function elements $u_t$ in the Poincar\'e section $\cP_\tau=\{u_t: u_t(0)=Q^{*}, u_t'(0)<0\}$.
}
\label{fig5_DdeStemKappaQ0865Torus05}
\end{figure}

As can be seen from Figure~\ref{fig.chaos0} the character of the chaotic dynamics is very sensitive to changes in the parameter values.
Changing $\tau$ from $3.9$ to $4.07$ while keeping $\kappa$ and all the other parameters at their values in Figure~\ref{fig.chaos1}
(see the blue triangle in Figure~\ref{fig:KappaTau_2D_Cont} inset)
the dynamics becomes as shown in Figure~\ref{fig5_DdeStemKappaQ0865Torus05}.
Now the time series in Figure~\ref{fig.chaos1}(i) is visually non-periodic, and the time-delay embedding in Figure~\ref{fig.chaos1}(ii) and (iii) appear to fill more of phase space. This is reflected in the Lyapunov dimension. Computing out to $1.8\times10^5$ days (about $500$ years) the leading Lyapunov exponents are estimated to be $0.03027$, $-0.00009$, $-0.11271$ and $-0.153236$. Using~\eqref{lyapdim} the Lyapunov dimension is computed to be $d=2.268$, larger than in the previous example.

\begin{figure}[t]
\includegraphics[height=36mm,width=1\textwidth]{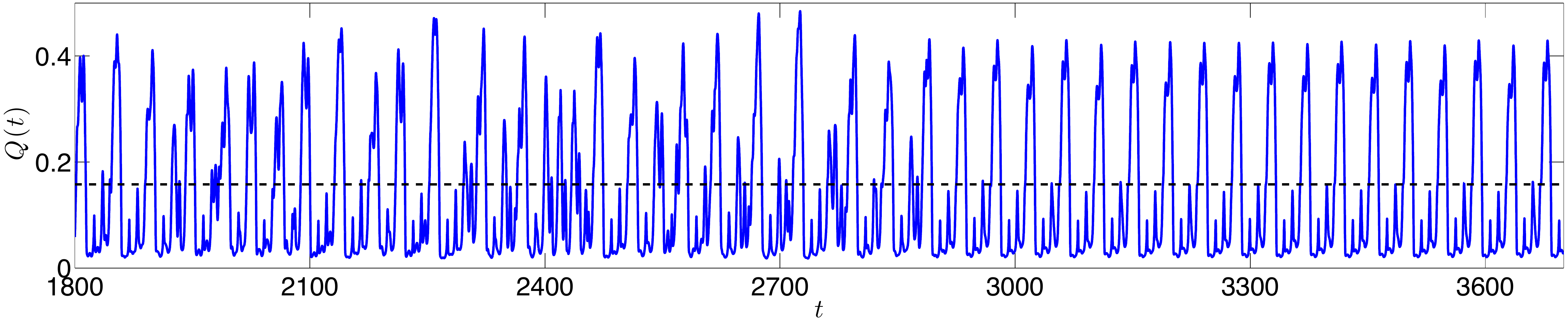}
\put(-503,2.5){\footnotesize{\textit{(i)}}}

\vspace*{2mm}

\includegraphics[width=0.49\textwidth,height=44mm]{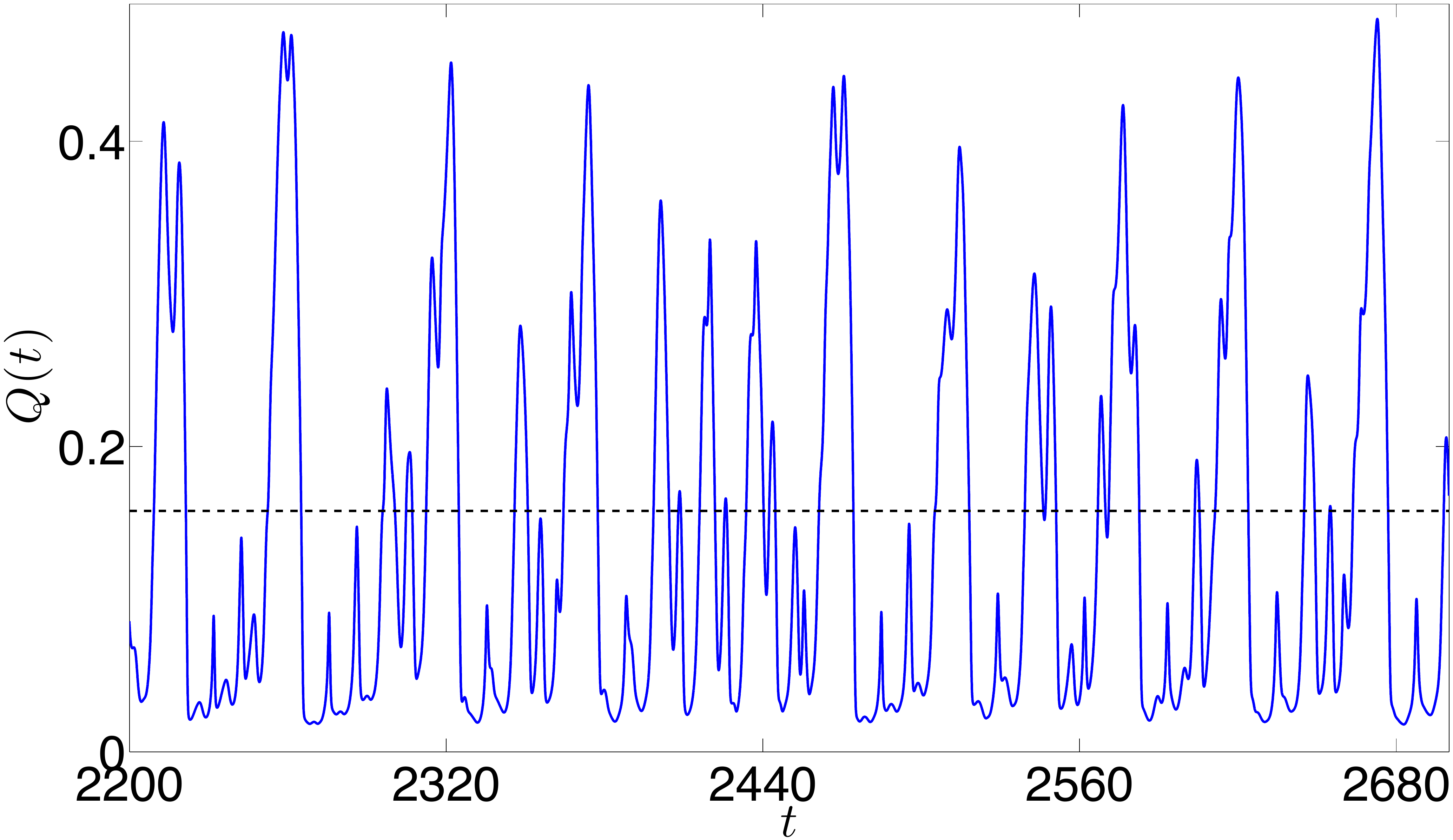}
\put(-251,2.5){\footnotesize{\textit{(ii)}}}
\hspace*{1mm}
\includegraphics[width=0.49\textwidth,height=44mm]{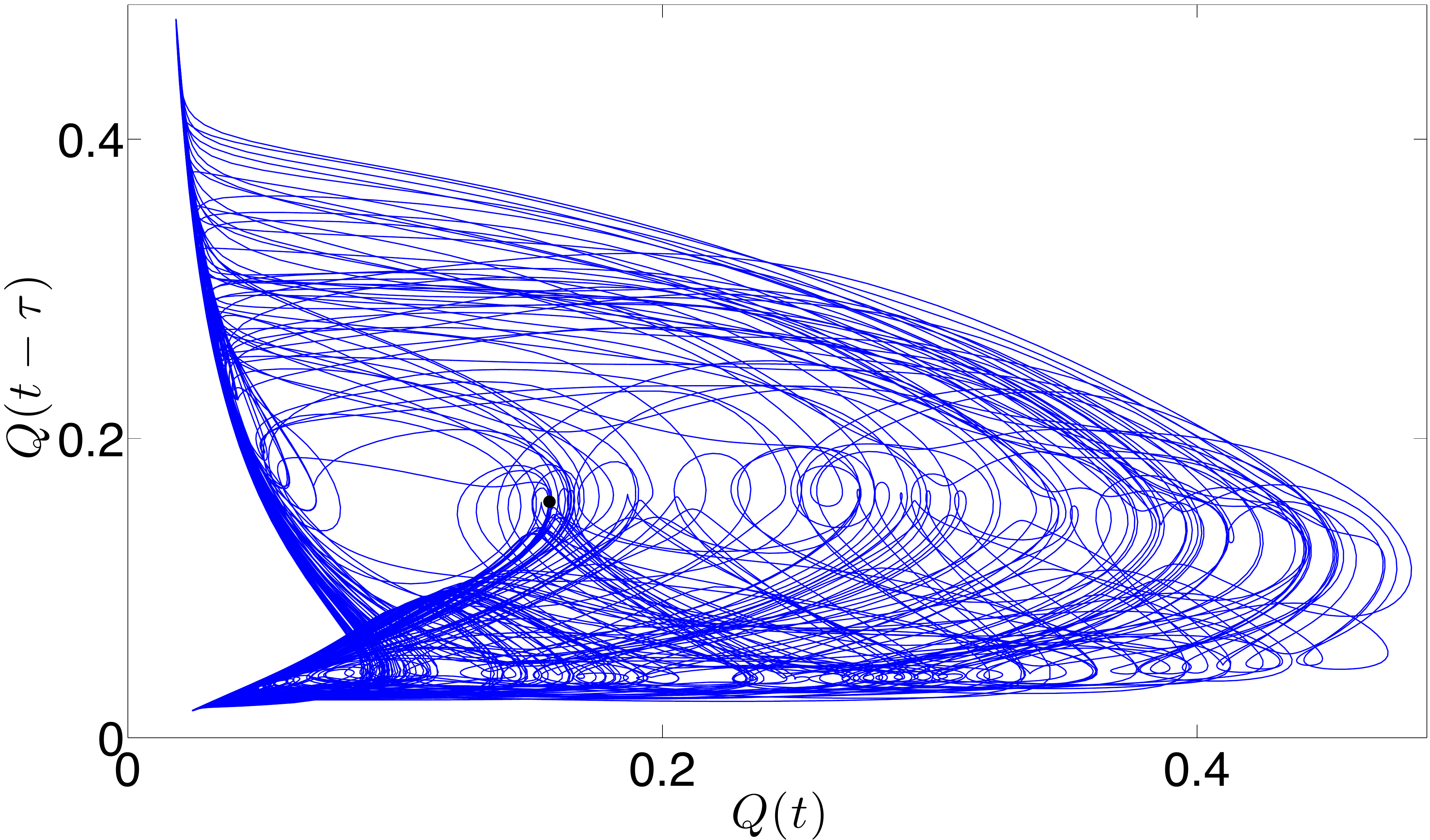}
\put(-247,2.5){\footnotesize{\textit{(iii)}}}

\vspace*{2mm}

\includegraphics[width=0.49\textwidth,height=44mm]{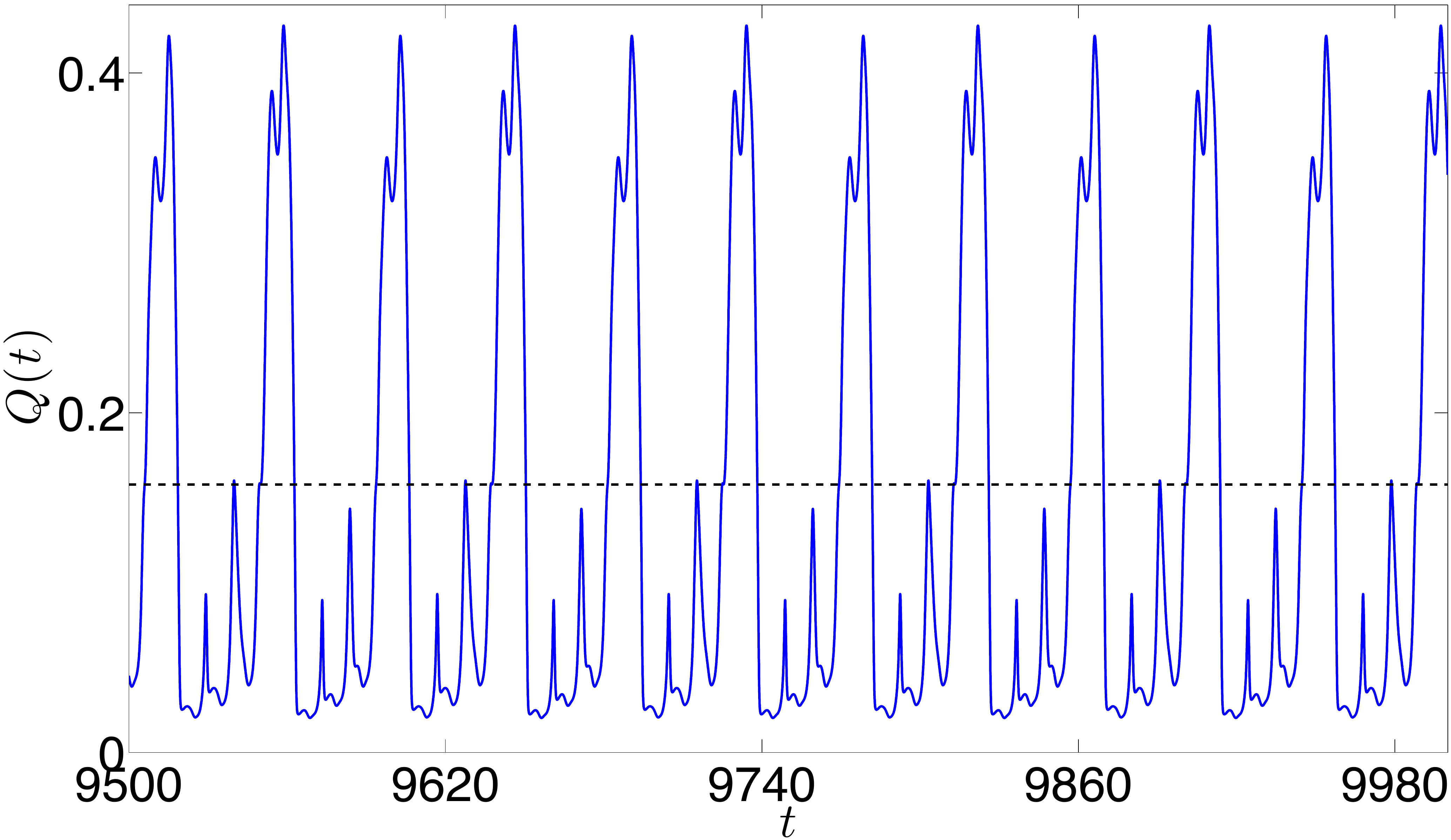}
\put(-251,2.5){\footnotesize{\textit{(iv)}}}
\hspace*{1mm}
\includegraphics[width=0.49\textwidth,height=44mm]{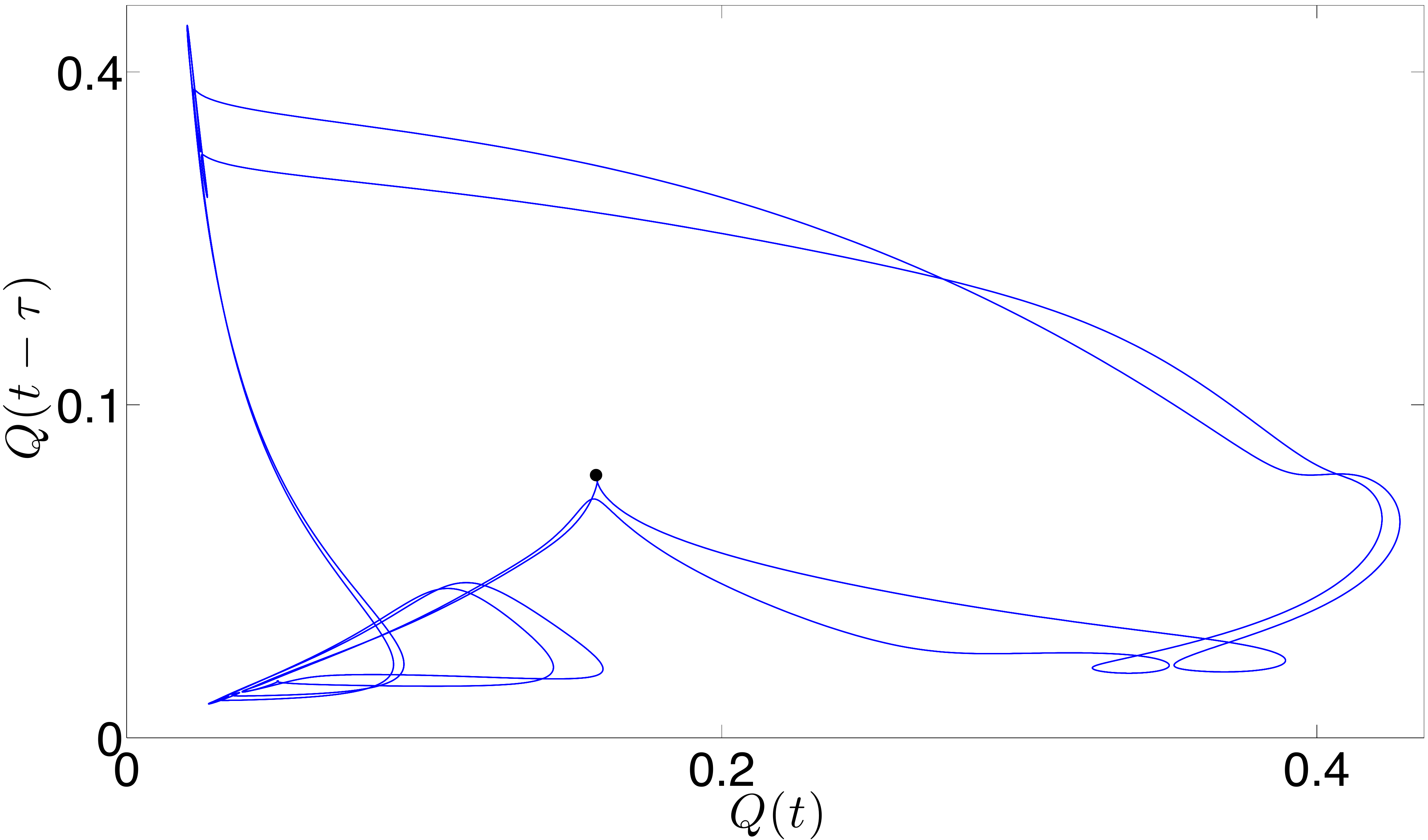}
\put(-245,2.5){\footnotesize{\textit{(v)}}}
\vspace*{-2mm}
\caption{For $\kappa=0.68$, $\gamma=0.0354608$ and $\tau=9.88888$ an orbit which appears to display transient chaos. (i) The transition from non-periodic to periodic motion.
(ii)-(iii) Time series $Q(t)$ and delay embedding $(Q(t-\tau),Q(t))$ for
the non-periodic part of the orbit, and (iv)-(v) the periodic orbit.
}
\label{fig_Transition}
\end{figure}

The Lyapunov exponents could have been obtained with a shorter integration interval; the reason to integrate out to $500$ years was to obtain many crossings of the Poincar\'e section $Q(t)=Q^*$ in order to try to reveal the fractal structure of the attractor. This is difficult to achieve because the mapping between the intersections with the Poincar\'e section is only implicitly defined by the solution of the DDE~\eqref{Qprime} which has to be solved numerically. Nevertheless Figure~\ref{fig5_DdeStemKappaQ0865Torus05}(iv) shows a projection of the crossing of the Poincar\'e section, with insets which reveal some of the fractal structure of the attractor.


If we vary all three of $\kappa$, $\gamma$ and $\tau$, while still holding all the other parameters at their homeostasis values from Table~\ref{tab.model.par}, further interesting chaotic solutions can be found. Figure~\ref{fig_Transition} shows an orbit that appears to display transient chaos.
We interpret this as co-existence of a chaotic invariant set which is not asymptotically stable along with a periodic orbit which is stable. The orbit initially appears to be chaotic with a high-dimensional attractor (see panels (ii) and (iii)) but after about 2850 days transitions to the stable period-doubled periodic orbit which has a period of about $87.75$ days.
This orbit was found by taking parameters close to a point where two period-doubling bifurcation branches cross each other in a  bifurcation diagram on parameter space $(\gamma,\tau)$ (not shown), similar to the diagram from Figure~\ref{fig:GammaTau_2D_Cont} but with $\kappa$ not at its homeostasis value.
If the value of $\kappa$ is changed to $\kappa=0.662$, but all the other parameters are held at their values from Figure~\ref{fig_Transition}, then the chaos becomes persistent. The attractor (not shown) looks very similar to
Figure~\ref{fig_Transition}(iii), but for $\kappa=0.662$ the chaos persists through at least $3\times10^4$ days.
That the attractor is of higher dimension than the previous examples can be inferred by comparing how disordered the two-dimensional projection seen in Figure~\ref{fig_Transition}(ii) looks compared to the previous examples.
The first six Lyapunov exponents are computed numerically to be $+0.02444$, $+0.008055$, $-0.00004119$, $-0.006071$,
$-0.01771$, and $-0.02882$. So for this example there are two positive Lyapunov exponents, the sum of the first five exponents is positive, and the Lyapunov dimension of the attractor is $d=5.3$. This dimension is relatively high compared to our previous examples and many of the classical examples of chaotic attractors in ODEs,
such as the Lorenz attractor~\cite{Lorenz_1963},
for which the dimension is often between $2$ and $3$. However, DDEs define
infinite-dimensional dynamical systems, and it is well-known that they can generate high-dimensional chaotic
attractors~\cite{Longtin1998}.

\subsection{Snaking branch}
\label{sec.snaking}

Continuation in $\tau$ with all the other parameters at their values from Table~\ref{tab.model.par} was illustrated in Section~\ref{sec.1p} (see Figures~\ref{fig:TauCont}-\ref{fig:TauOrb}), and appears to show a canard explosion,
similar to the canard explosion for $\gamma$ continuation, described in Section~\ref{sec.longp.hsc}. Different behaviour is observed
if we vary all three parameters $\gamma$, $\kappa$ and $\tau$.

\begin{figure}[t]
\includegraphics[width=0.49\textwidth,height=44mm]{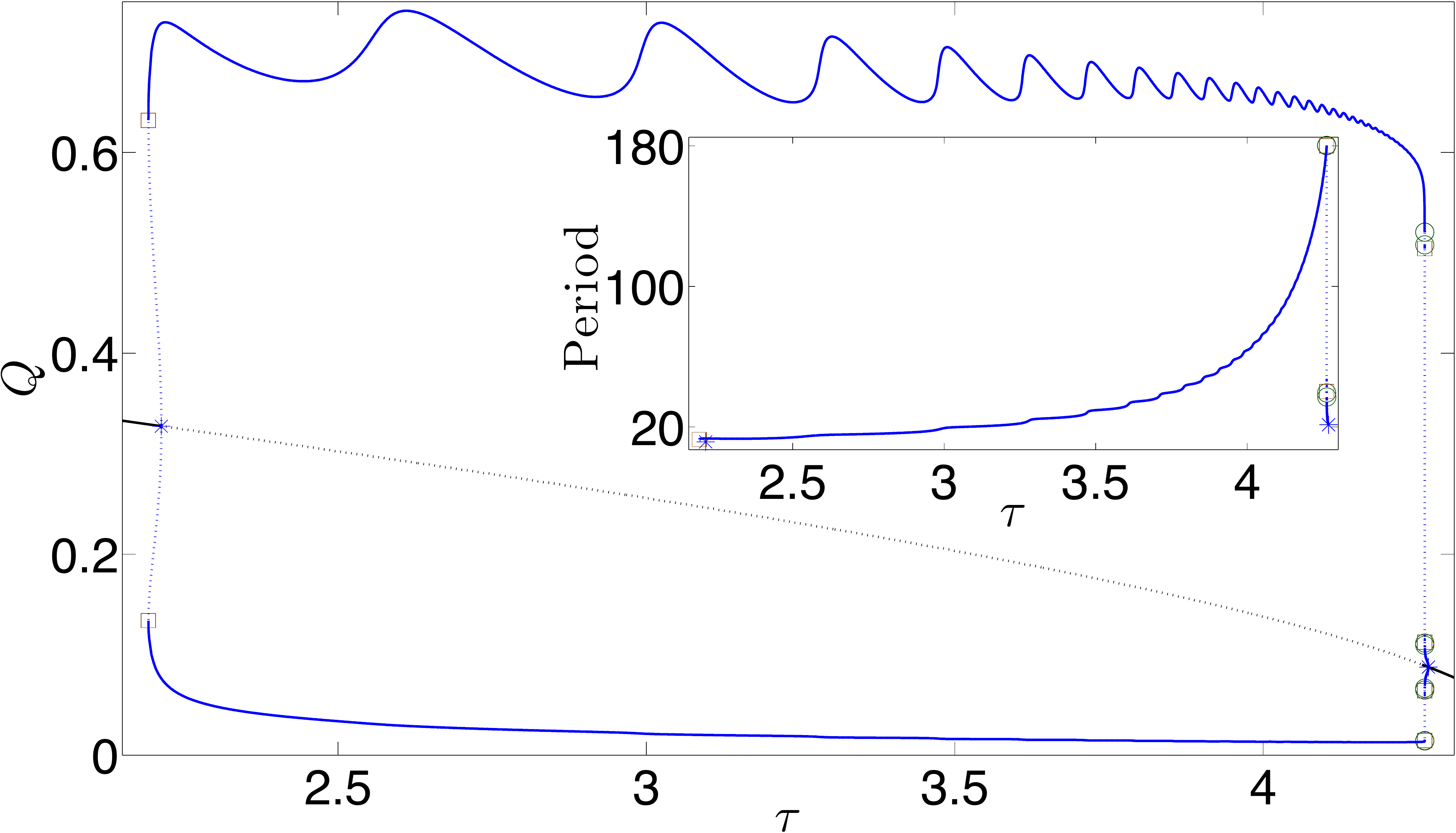}
\put(-245,2.5){\footnotesize{\textit{(i)}}}
\hspace*{1mm}
\includegraphics[width=0.49\textwidth,height=44mm]{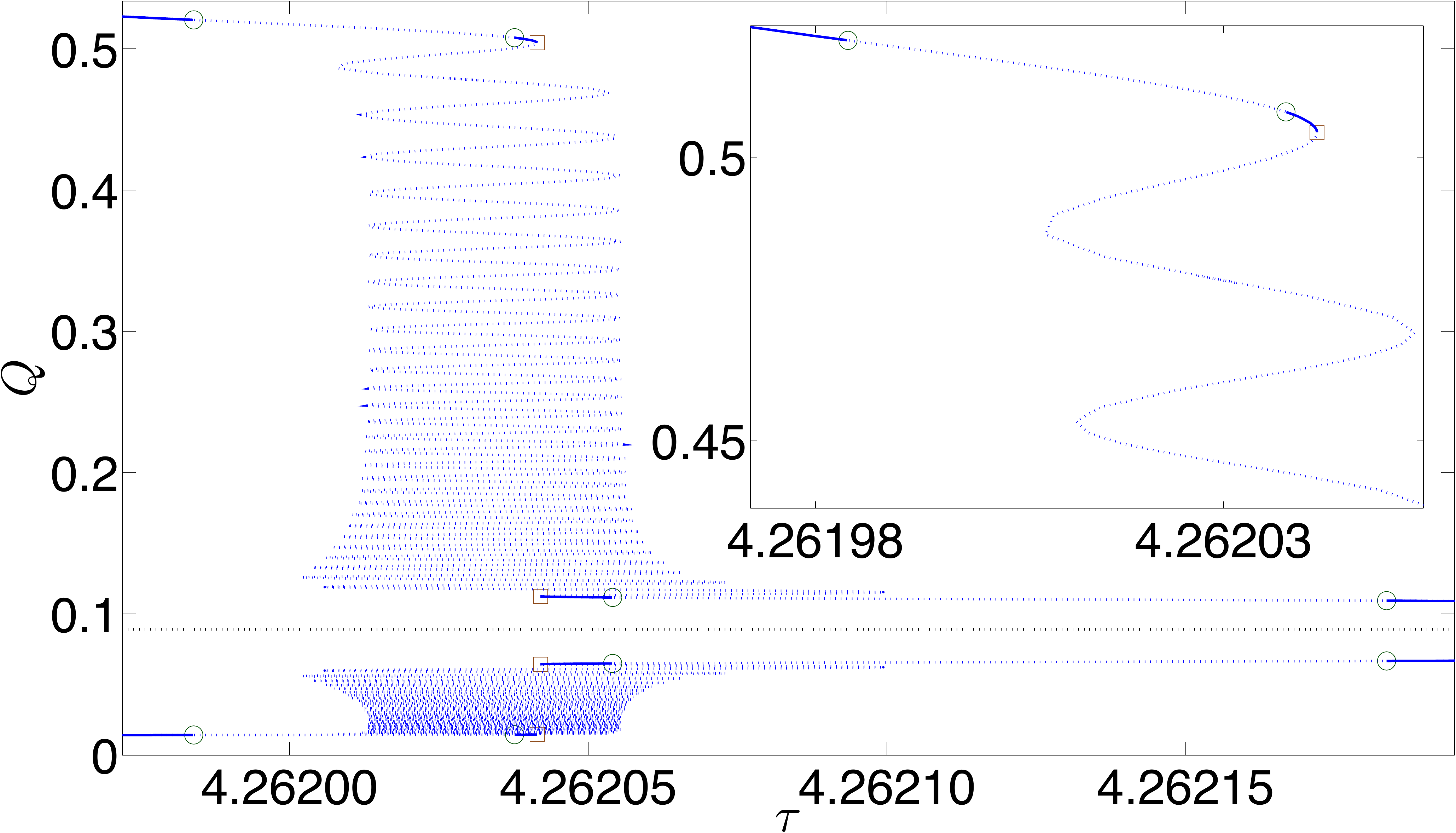}
\put(-245,2.5){\footnotesize{\textit{(ii)}}}

\vspace*{2mm}

\includegraphics[width=0.49\textwidth,height=44mm]{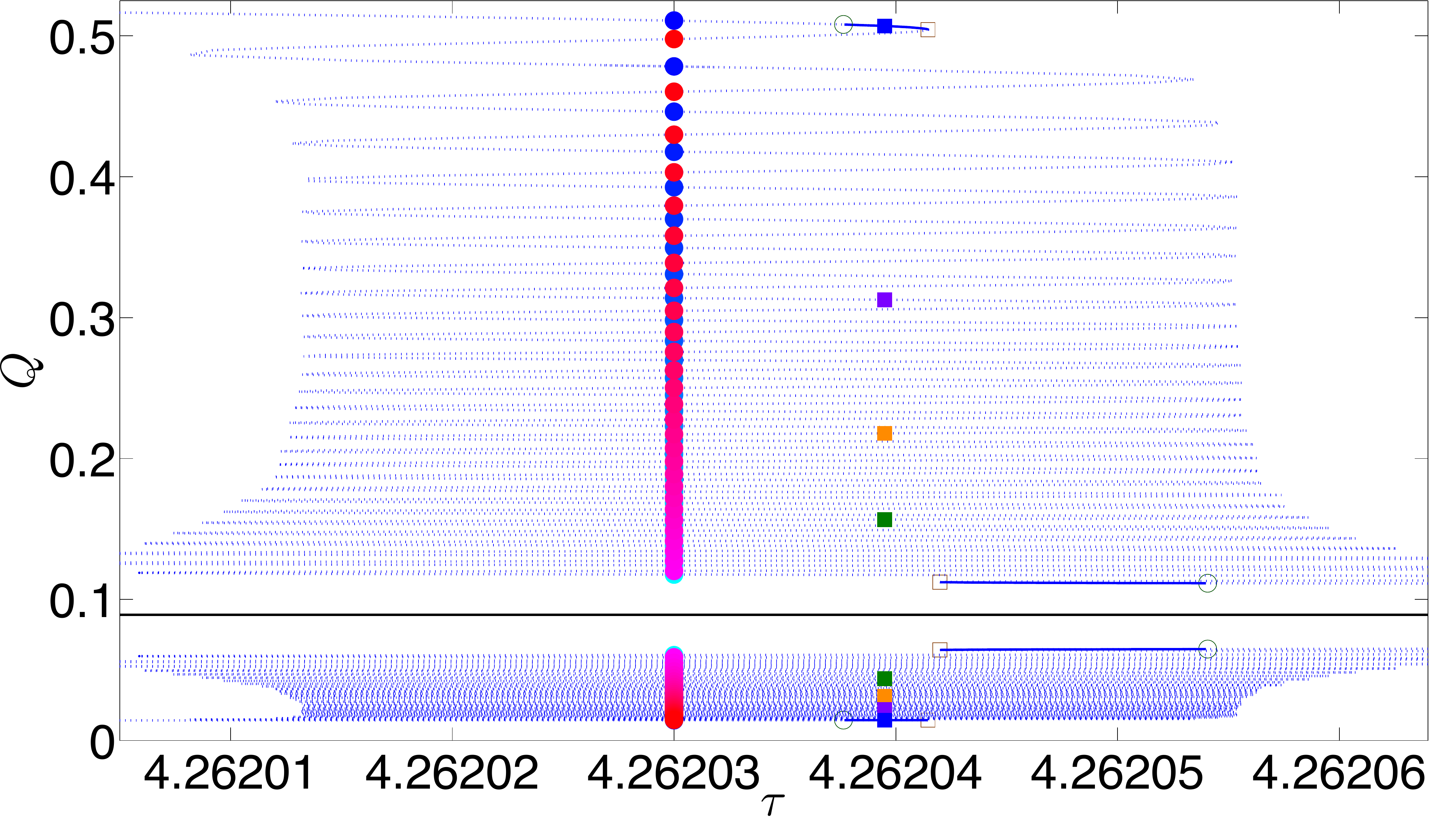}
\put(-245,2.5){\footnotesize{\textit{(iii)}}}
\hspace*{1mm}
\includegraphics[width=0.49\textwidth,height=44mm]{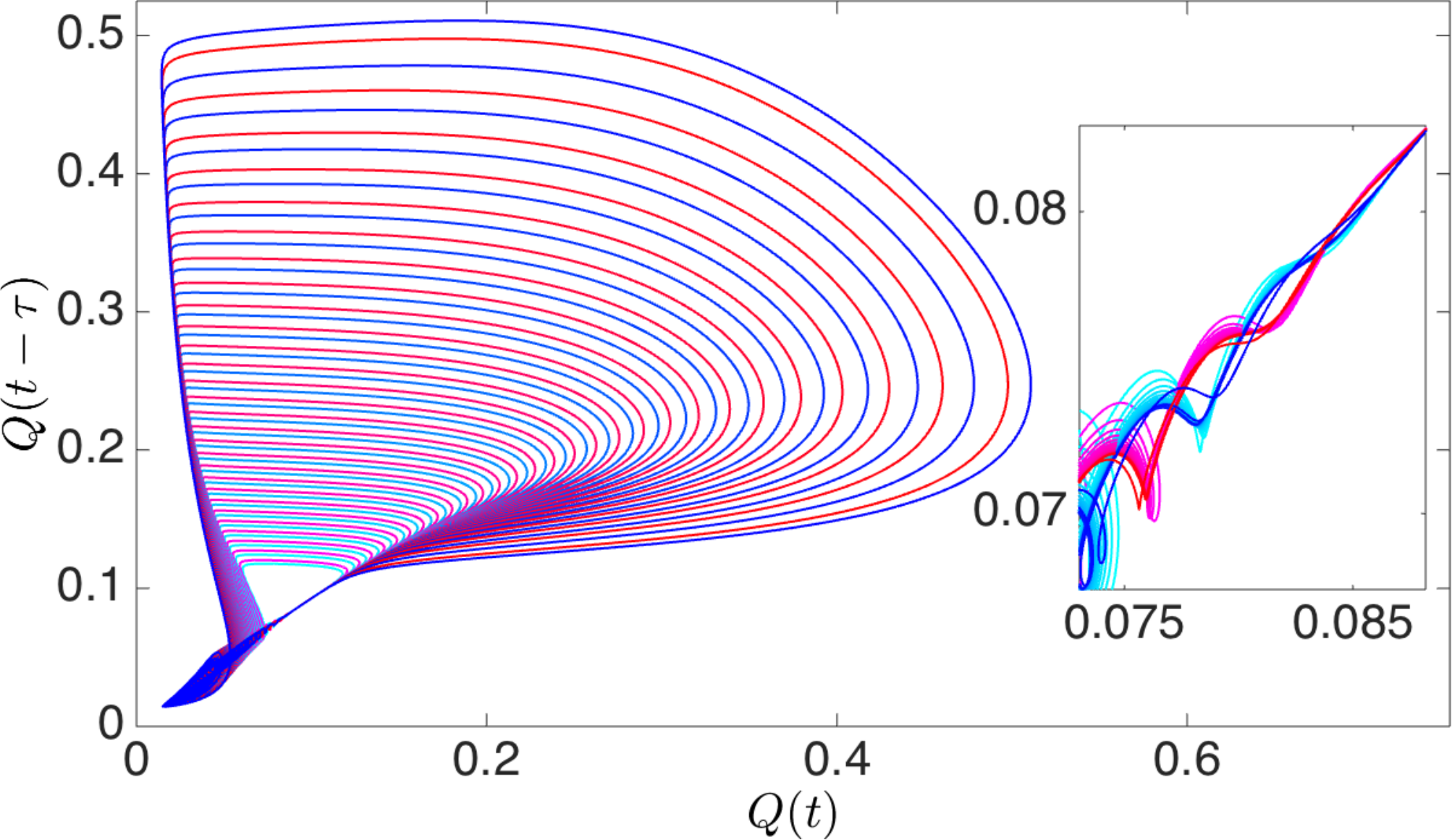}
\put(-245,2.5){\footnotesize{\textit{(iv)}}}

\vspace*{2mm}

\includegraphics[width=0.49\textwidth,height=44mm]{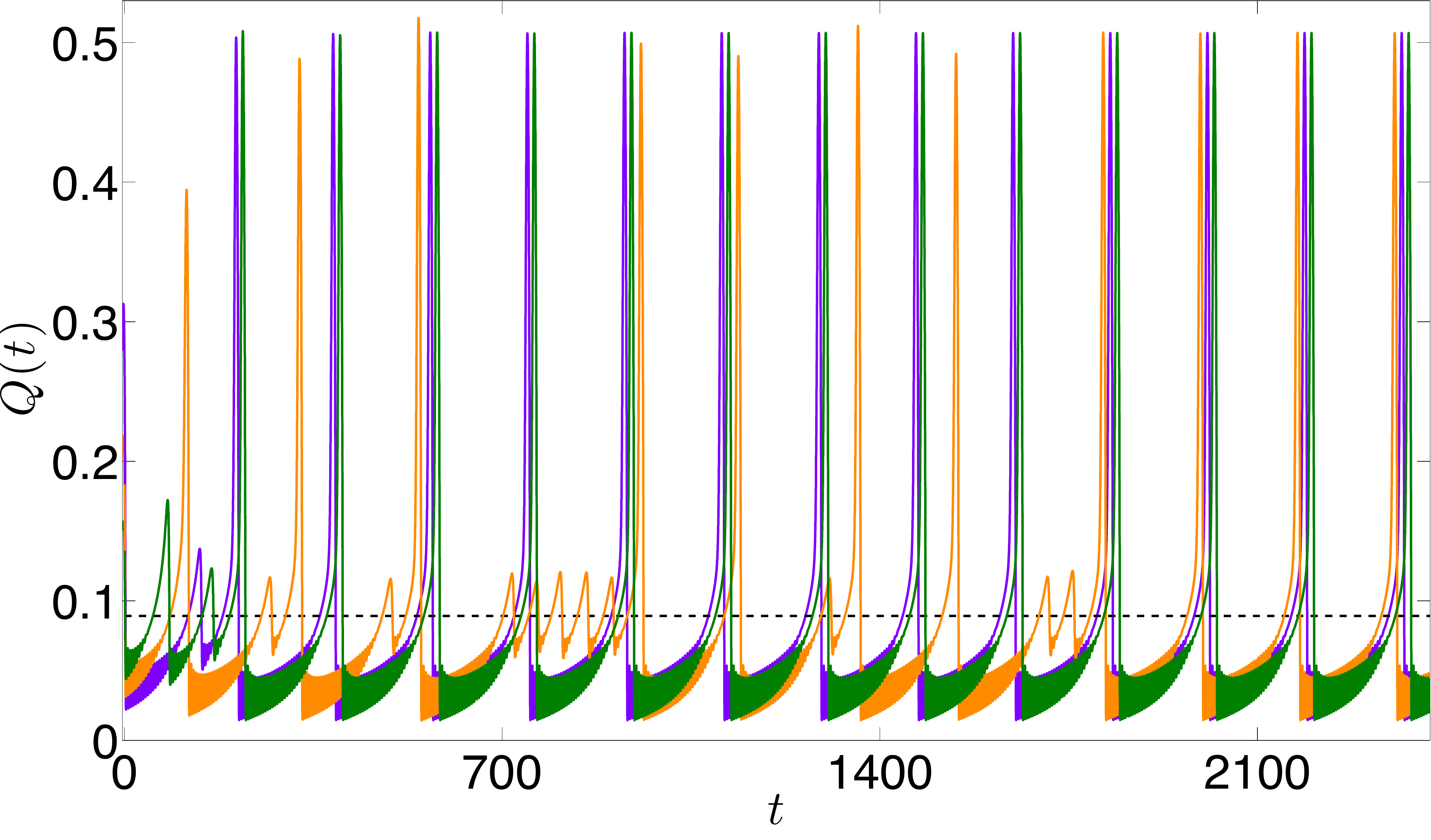}
\put(-245,2.5){\footnotesize{\textit{(v)}}}
\hspace*{1mm}
\includegraphics[width=0.49\textwidth,height=44mm]{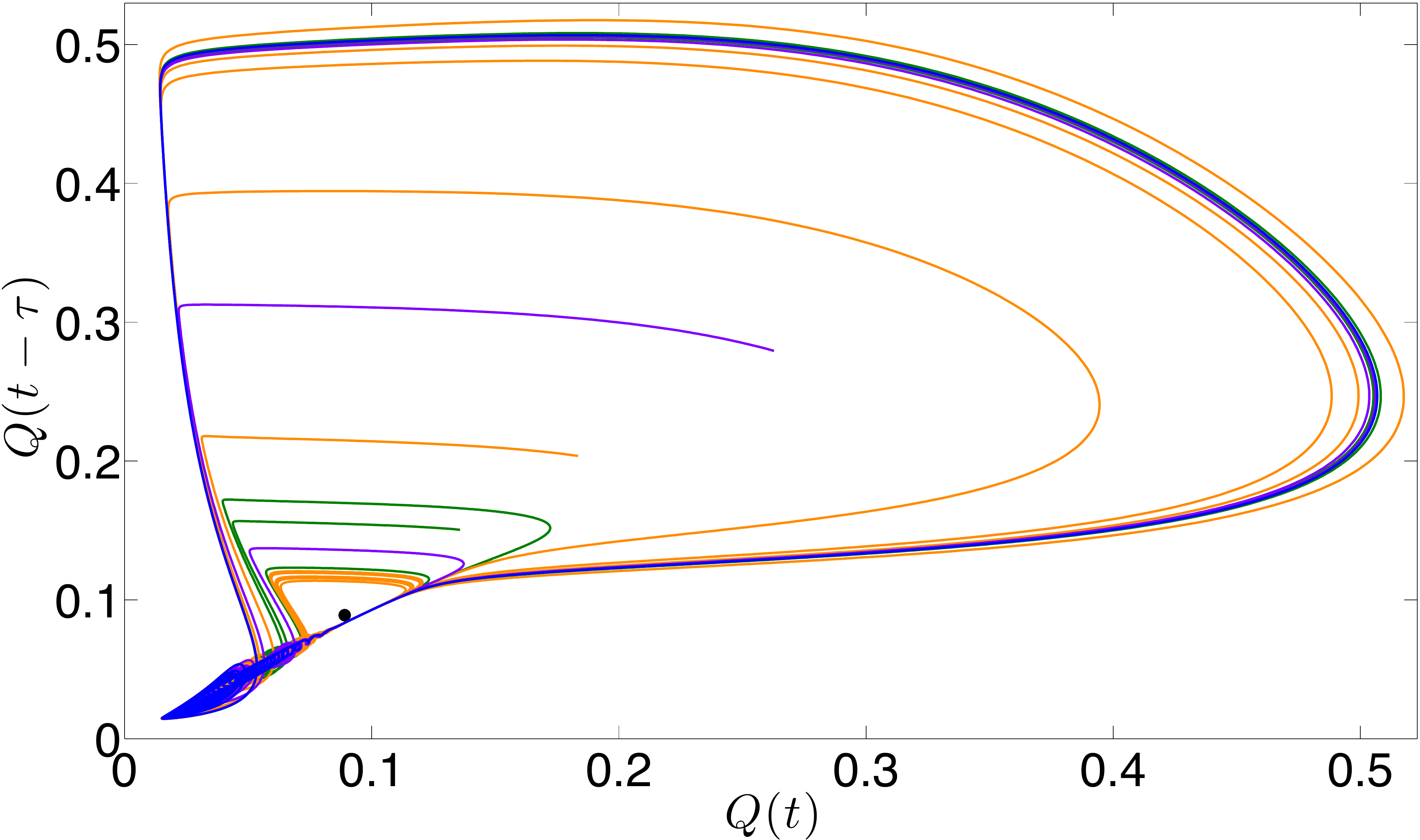}
\put(-245,2.5){\footnotesize{\textit{(vi)}}}
\vspace*{-2mm}
\caption{Parameter continuation in $\tau$ for periodic orbits with $\gamma=0.15$, $\kappa=0.2$.
(i) Bifurcation diagram showing Hopf bifurcations at $\tau=2.21327$ and $4.26817$, fold bifurcation of periodic orbits at $\tau=2.19228$ and $4.262041$, and period-doubling bifurcation of periodic orbits at $4.261983$, $4.262037$, $4.262054$ and $4.262183$.
The inset shows the period of the orbits.
(ii) Details of the snaking-branch region of the bifurcation diagram from panel (i).
(iii) Examples of co-existing periodic orbits for $\tau=4.26203$ and $\tau=4.2620395$.
(iv) Delay embeddings of co-existing periodic orbits for $\tau=4.26203$ and inset showing that the orbits are out of phase.
(v) Three solutions of Eq.~\eqref{Qprime} for $\tau=4.2620395$ computed using the MATLAB \texttt{dde23} routine~\cite{Matlab} with initial functions given by DDEBiftool solutions for the corresponding coloured dots shown in panel (iii). All three orbits converge to the large amplitude stable limit cycle seen in the
bifurcation diagram in panel (iii).
(vi) Delay embeddings of the three orbits shown in panel (v) along with the stable limit cycle to which they converge show in blue, also denoted by the blue square in panel (iii).}
\label{figSnakCont}
\end{figure}

In Figure~\ref{figSnakCont} we present the results of one-parameter continuation in $\tau$
with $\gamma=0.15$, $\kappa=0.2$, and the other parameters at their values from Table~\ref{tab.model.par}.
The bifurcation diagram in Figure~\ref{figSnakCont}(i) appears to show similar behaviour to the earlier $\tau$ continuation, with the steady state stable except between a pair of Hopf bifurcations. There is again a subcritical Hopf bifurcation leading to an interval of bistability between the steady state and a stable limit cycle, and there are again ripples in the amplitude and period of solutions along the branch of stable periodic orbits. The period of the orbits but not the amplitude increases significantly to reach $180$ days just before the period collapses to $21.2$ days at the Hopf bifurcation.

As illustrated in Figure~\ref{figSnakCont}(ii) there is \emph{not} a canard this time.
Instead the bifurcation branch snakes about $28$ times across $\tau=4.262041$ creating a small interval of $\tau$ values for which there are 57 co-existing periodic orbits.
If the periodic orbits had been computed just by simulating to only find the stable solutions, it would appear that the amplitude and period both suddenly increase as $\tau$ is decreased through $4.262041$, suggesting the possibility of a canard explosion. But the DDEBiftool computations, which allow us to compute unstable periodic orbits just as well as stable ones, show this not to be the case.


At the top and bottom of the snake there is a pair
of fold bifurcations of periodic orbits both at $\tau\approx4.262041$, with the $\tau$ values of the bifurcations points agreeing to at least $7$ significant figures. The large amplitude orbit at the top of the snaking branch is stable for very small interval of $\tau$ values ($\tau\in(4.262037,4.262041)$), before losing stability in a period doubling bifurcation at $\tau\approx4.262037$. The small amplitude orbit at the bottom of the snake is stable for $\tau\in(4.262037,4.262054)$, before also losing stability in a period doubling bifurcation at $\tau\approx4.262054$.
We will come back to the dynamics resulting from these period doublings at the end of this section.
On the snaking branch between the two fold bifurcations at $\tau\approx4.262041$ all the periodic orbits are unstable.

Figure~\ref{figSnakCont}(iv) shows the delay embeddings for the 57 unstable limit cycles
that co-exist when $\tau=4.26203$. The positions of these orbits on the snaking branch are indicated on
Figure~\ref{figSnakCont}(iii), where we use shades of pink to red to indicate orbits which are on the legs of the snake for which the amplitude increases as $\tau$ increases, and shades of cyan to blue for orbits on the legs of the branch where the amplitude decreases as $\tau$ increases.
Although Figure~\ref{figSnakCont}(iv) is very reminiscent of Figure~\ref{fig:linearslowman}, there are crucial differences between the dynamics. In particular the orbits shown in Figure~\ref{figSnakCont}(iv) are all unstable and all co-exist, whereas those of Figure~\ref{fig:linearslowman} are stable and exist over an exponentially small parameter interval, with a unique orbit existing for each of the parameter value. Nevertheless, there are significant similarities between the dynamics in the two cases with
Figure~\ref{figSnakCont}(iv) also appearing to indicate the presence of a slow manifold which is stable for a certain range of $Q$ values, with the orbits appearing to spiral onto the slow manifold.
The inset in Figure~\ref{figSnakCont}(iv) shows that the phase of this convergence is different on the two legs of the snaking branch.

In Figure~\ref{figSnakCont}(v)-(vi) we illustrate the dynamics with $\tau= 4.2620395$
when the large amplitude orbit (indicated by the blue dot on Figure~\ref{figSnakCont}(iii)) is stable.
For three different initial functions corresponding to unstable periodic orbits on the snaking branch (also indicated by coloured dots on panel (iii))
we take a part of the periodic orbit
generated by DDEBiftool as the initial function, then use the  MATLAB \texttt{dde23} routine~\cite{Matlab}
to simulate the solution. All three orbits are seen to converge to the stable large amplitude limit cycle, with period about $180$ days, but the nature of that convergence is not simple to explain. All of the periodic orbits along with their unstable manifolds are squeezed very close together when the orbits follow the slow manifold before diverging from each other again when the slow manifold becomes unstable, and probably as a consequence of this the connecting orbits between the limit cycles do not appear to have a simple structure. In the figure we see that the orbit shown in orange passes close to the slow manifold many times before approaching the stable limit cycle, around $t=2000$ days, while the other two initial functions lead to solutions which converge to the stable periodic orbit relatively quickly.

\begin{figure}[t]
\includegraphics[width=0.49\textwidth,height=44mm]{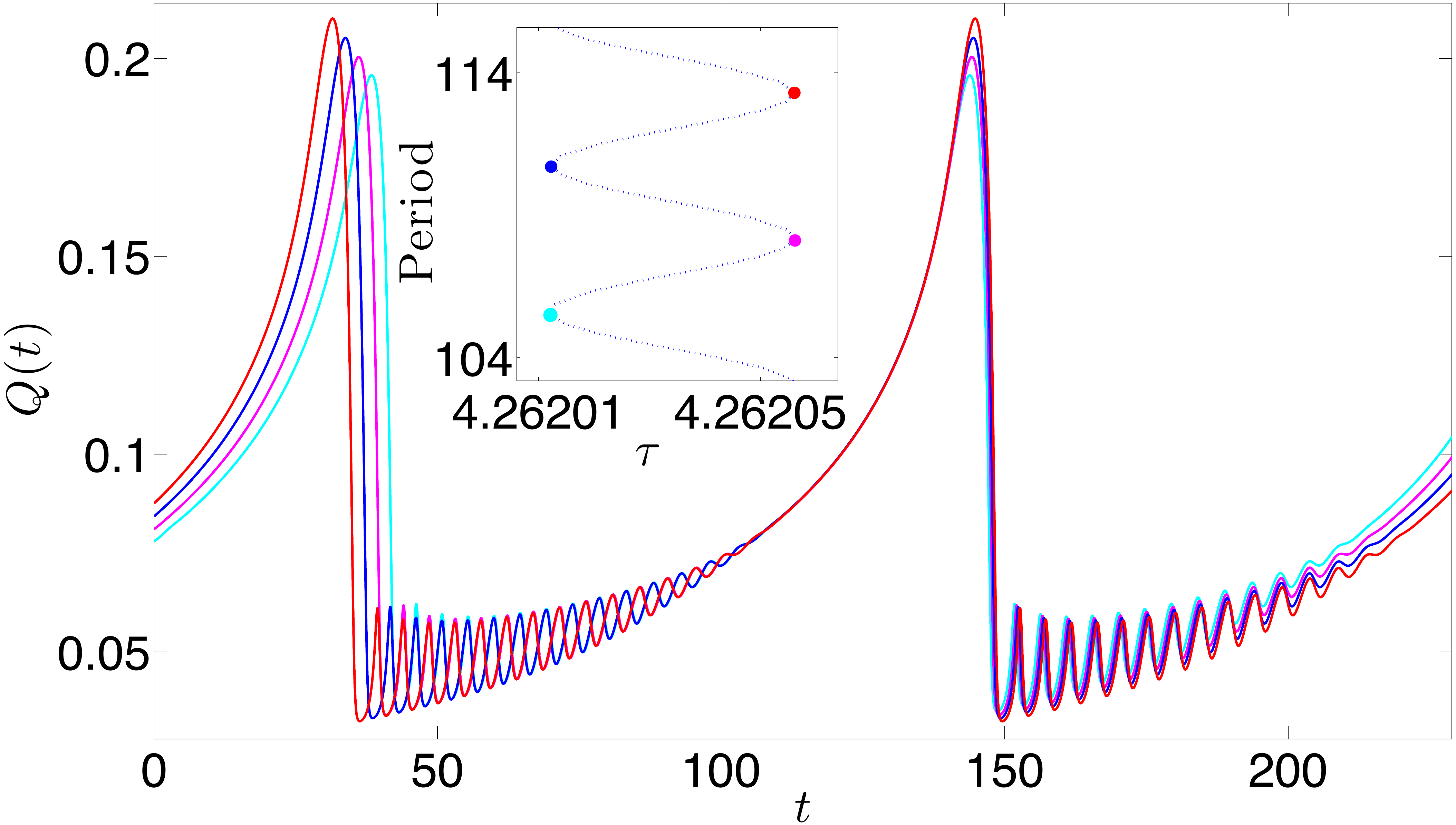}
\put(-245,2.5){\footnotesize{\textit{(i)}}}
\hspace*{1mm}
\includegraphics[width=0.49\textwidth,height=44mm]{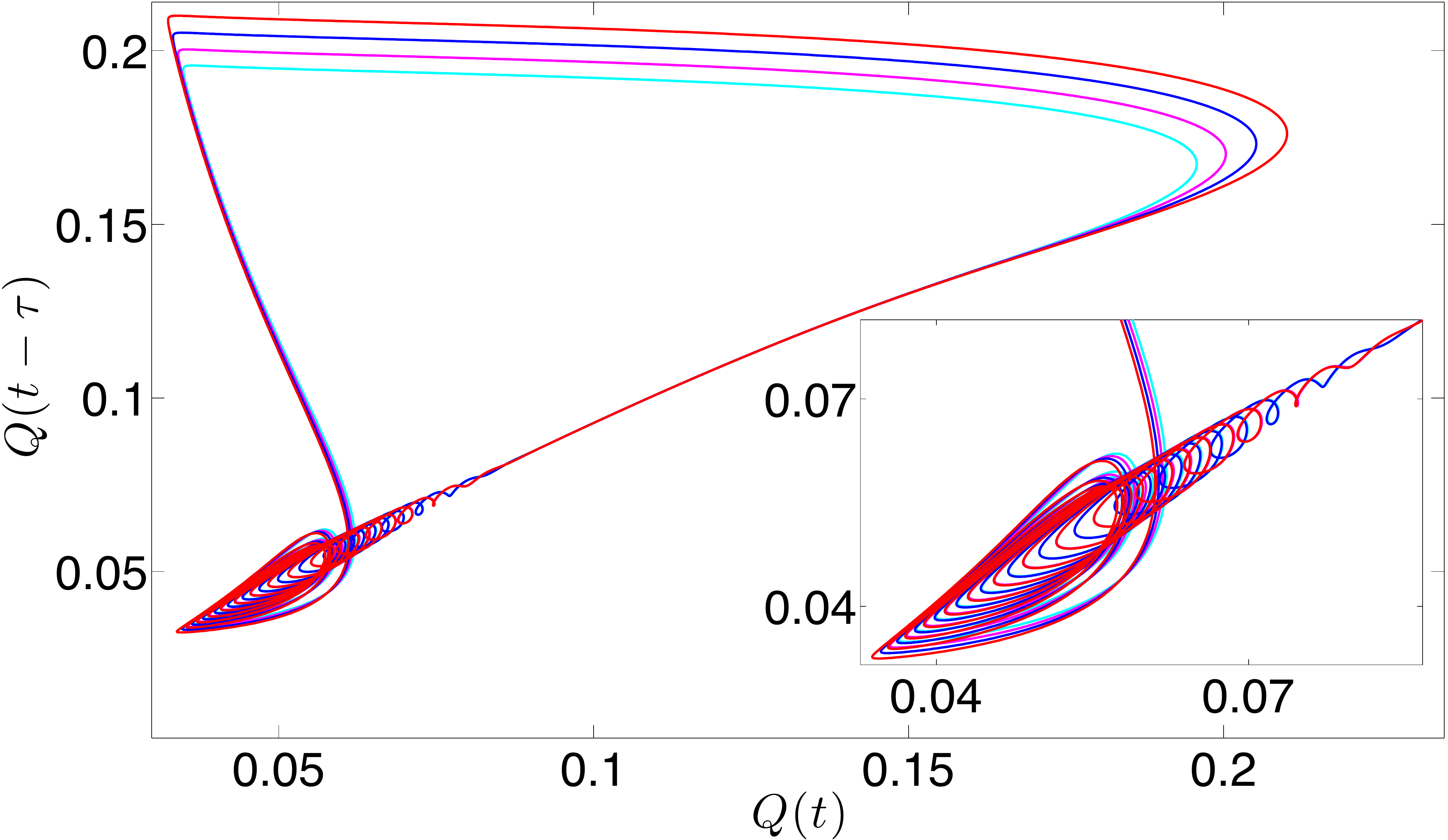}
\put(-245,2.5){\footnotesize{\textit{(ii)}}}
\vspace*{-2mm}
\caption{Example periodic orbits and their delay embedding from the snaking branch shown in Figure~\ref{figSnakCont}. The delay embedding shows that oscillations in the $(Q(t),Q(t-\tau))$ projection appear to be in anti-phase between the left and right sides of the snaking branch. The orbit profiles show the same behaviour when they are plotted with final time points and $Q$ values equal.}
\label{figSnakExtr}
\end{figure}

\def\big{\includegraphics[width=0.87\textwidth]{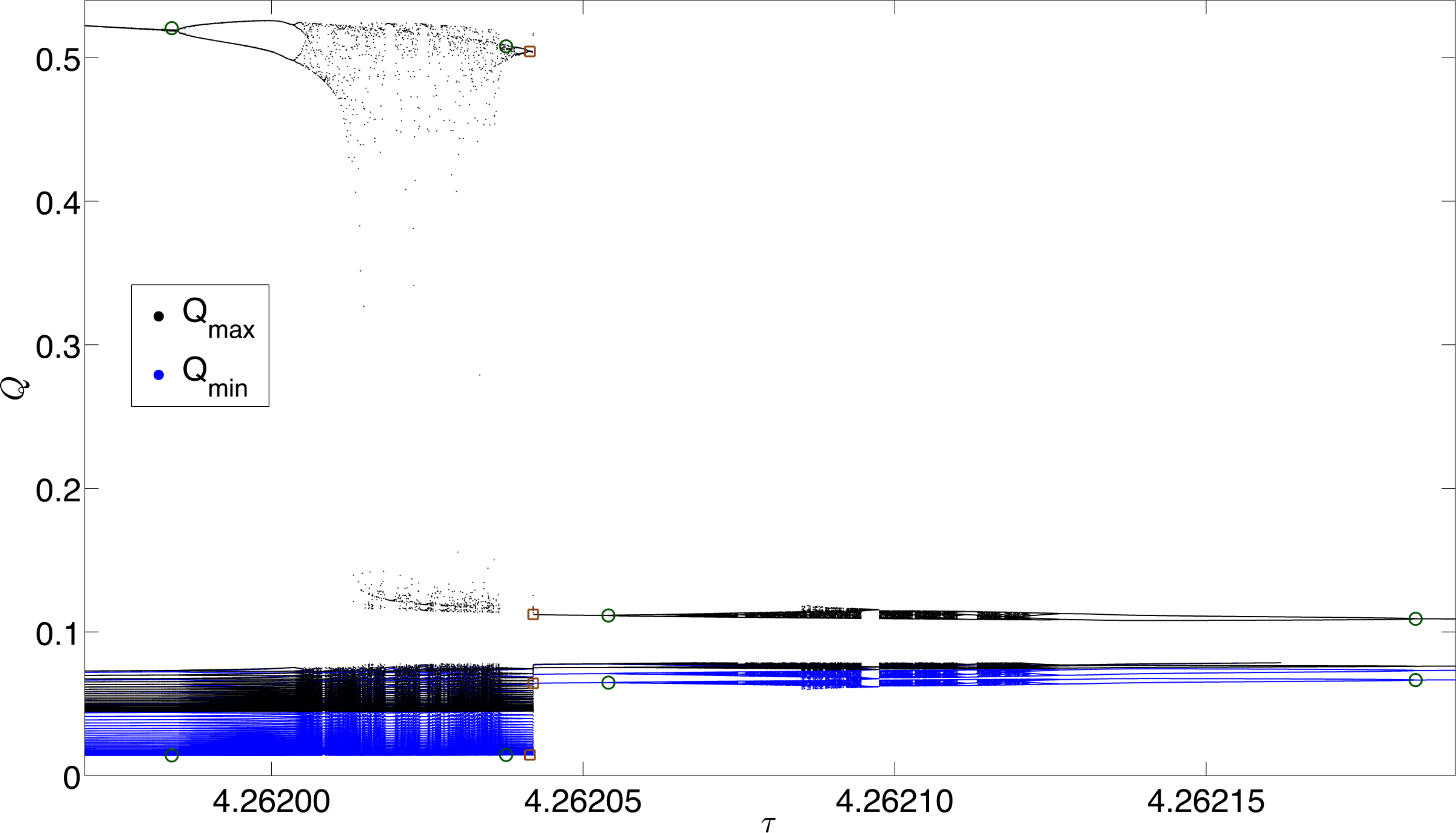}}
\def\little{\includegraphics[height=5.4cm]{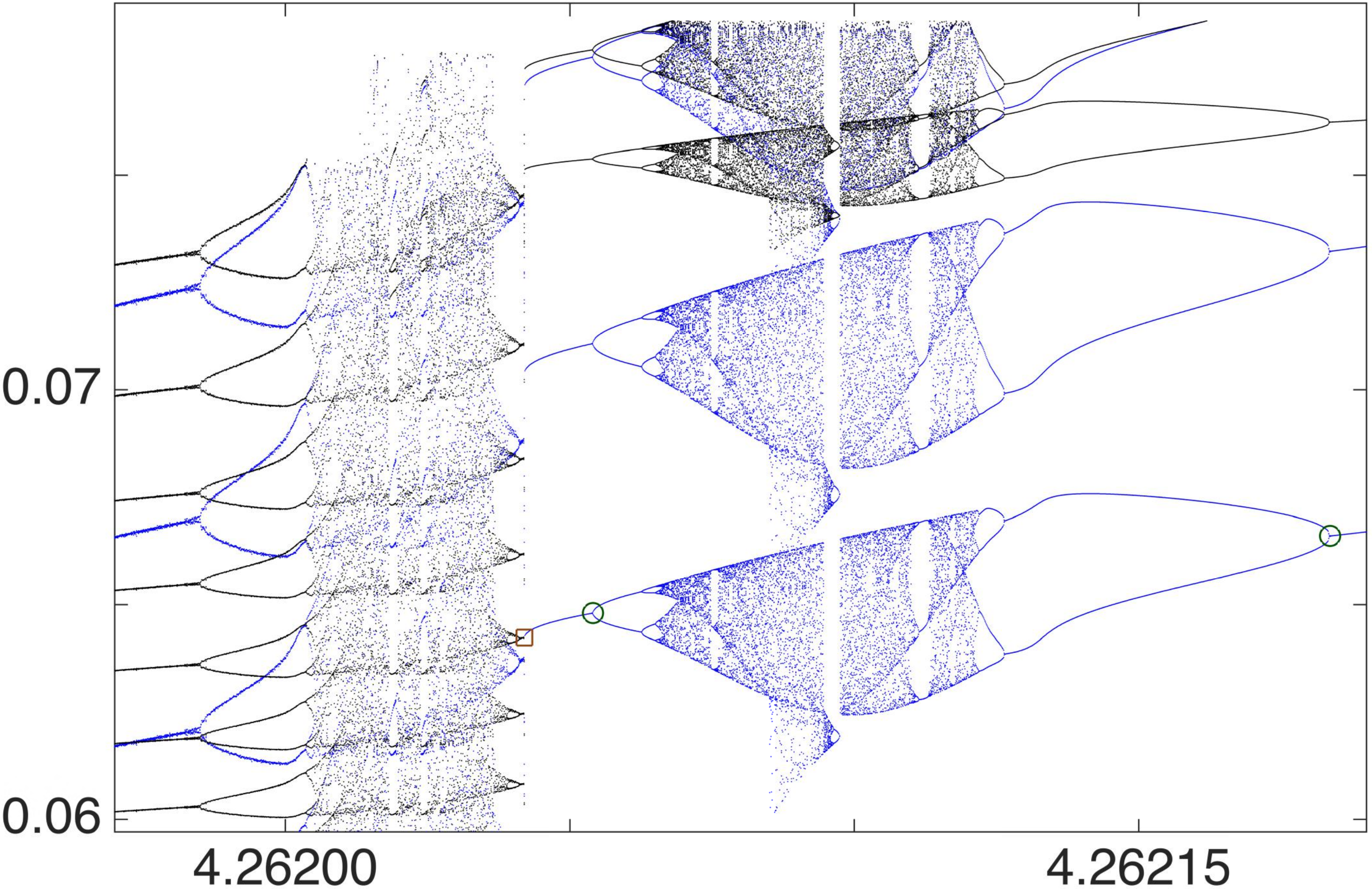}}
\begin{figure}[t]
\centering
\topinset{\little}{\big}{1.5mm}{24.8mm}
\vspace*{-3mm}
\caption{Orbit diagram showing local maxima and minima of solutions segments~\eqref{Qprime} as function of the delay $\tau$,
with $\gamma=0.15$, $\kappa=0.2$ and other parameters taking their values from Table~\ref{tab.model.par}.
Black and blue dots represent respectively the local maxima and minima for a decreasing sequence of $\tau$ values.
The bifurcation points of primary branch are identified using the same symbols as in Figure~\ref{fig:KappaCont} and~\ref{fig2_HSC_Biftool1D_KappaDA}. The inset shows the same interval of $\tau$ values as the main figure but for a restricted range of $Q$ values revealing details of the bifurcations.}
\label{figSnakBif}
\end{figure}

In Figure~\ref{figSnakExtr} we show 4 orbits located at adjacent local extrema of $\tau$ on the snaking branch, as shown in the inset of panel (i). The profiles in panel (i) and delay embeddings in panel (ii) illustrate how the periodic orbit changes along the snaking curve of solutions as the amplitude increases.
In Figure~\ref{figSnakExtr}(i) the phase of the orbits is aligned so that they all have the global maximum and minimum aligned (close to $t=150$). Looking back one period to the previous occurrence of the global maxima and minima, we
see that the position on the snaking branch of solutions is related to the number of short period oscillations seen as the solution converges onto the slow manifold.
Crossing each leg of the snaking branch with increasing amplitude corresponds to adding half a short period oscillation to the whole periodic orbit. So the points on the snaking branch at minima of $\tau$ display one less/more short oscillation than seen at the next minima of $\tau$ directly above/below them on the branch, and half a short period oscillation less/more than seen at the adjacent maxima of $\tau$ with
larger/smaller amplitude.  The time-delay embeddings in Figure~\ref{figSnakExtr}(ii) and its inset reveal that orbits located in the same extrema of the snaking branch converge to the slow manifold in phase with each other,
and in antiphase to orbits located in opposite extrema. Although we have seen how the solution changes along the snaking branch, this does not explain why the branch itself snakes; similar solution behaviour
but without branch snaking
was observed in Section~\ref{sec.longp.hsc} for the canard explosion.

As noted near the beginning of the section, either side of the fold points at $\tau=4.262041$ there are period doubling bifurcations. These bifurcations actually come in pairs, resulting in two separate intervals, one each side of the snaking part of the branch, for which the periodic solutions on the principal branch are unstable. For large amplitude solutions this occurs for
$\tau\in(4.261983,4.262037)$ with a period doubling bifurcation at each end of this interval.
For small amplitude solutions, the unstable part of the branch between the period-doubling bifurcations
is for $\tau\in(4.262054,4.262183)$. To explore the dynamics as $\tau$ is varied over these parameter intervals in Figure~\ref{figSnakBif} we present an orbit diagram showing the local maxima and minima of $Q(t)$ along the solutions of~\eqref{Qprime} as $\tau$ is varied across this region with $\gamma=0.15$, $\kappa=0.2$. This is computed similarly to Figure~\ref{fig.chaos0}, but this time integrating through a transient of $1530\tau$ days, then plotting all the maxima and minima that occur over the next $170\tau$ days.
A mesh of two thousand equally spaced points for $\tau\in[4.26197,4.26219]$ was used for decreasing $\tau$. In this case we did not observe any noticeable hysteresis effects.
For each mesh the solution over the last $\tau$ days was used as the initial history to start the transient computation for the next adjacent $\tau$ value.

The results displayed in Figure~\ref{figSnakBif} clearly reveal the bifurcations already shown in Figure~\ref{figSnakCont}(ii) including the fold bifurcation near
$\tau= 4.262041$ and the period doubling bifurcations at $\tau= 4.261984$, $4.262038$, $4.262054$ and $4.262184$. Between the pairs of period-doublings, much richer dynamics are displayed than we had expected.
Several period-doubling cascades are clearly visible (in the inset to the figure), leading to several intervals of apparently stable chaotic dynamics. There are also windows of stable periodic dynamics, including a period-3 window, which suggests the possibility of unstable chaotic dynamics (period-3 implies chaos only for one-dimensional maps).

\section{Dynamical Diseases}\label{sec.physiol}

In dynamic hematological diseases oscillations are observed
in the circulating concentrations of one or more of the cell lines \cite{Foley_2009a}.
Mathematical interest has often focused on what have been termed periodic hematological disorders, including cyclic neutropenia (CN), cyclic thrombocytopenia (CT) and
periodic chronic myelogenous leukemia (PCML).

CN is one of the most studied of these periodic diseases, with the concentration of circulating neutrophils varying from very low to normal or high levels with a period of about 19 to 21 days~\cite{Bernard_2003,Colijn_2005b}. Patients experience a bout of neutropenia (abnormally low neutrophil concentrations) each period, during which time the immune system is impaired and they are more susceptible to infection~\cite{Dale_2016}.
For patients with CT, oscillations in platelet counts from normal to very low values are observed with periods between 20 to 40 days~\cite{Haurie_1998}.
For patients with PCML, cycling in white blood cells from normal to high levels with periods from approximately 30 to 100 days~\cite{Haurie_1998} and 40 to 80 days~\cite{Pujo_Menjouet_2005} is reported.

Many mathematical models of hematopoiesis have been developed in an effort to understand these diseases and the origins of the oscillatory dynamics \cite{Bernard_2003,Pujo_Menjouet_2004,ZhugeMackeyLeiJTB2019}.
Efforts have often focused on deriving mathematical models and associated parameter sets for which the model has a stable limit cycle with a period commensurate with a particular
disease under consideration \cite{Colijn_2005a,Colijn_2005b,Langlois2017}.
Clinical efforts focus in entirely different directions, typically concentrating on alleviating the cytopenia (dangerously low blood cell concentrations) either by raising the concentration nadir or by decreasing the time interval that concentrations are below the recognised cytopenia threshold. Periodic oscillations in the strict mathematical sense are of limited clinical interest, and in the clinical literature the terms periodic and cyclic are often used as synonyms for episodic, and it is not
implied that the time intervals between episodes are fixed. Consequently, there are many other hematological disorders which at least for some subjects display dynamics with a periodic signature,
but for which there are only scattered case reports of the periodicity.
Examples include cyclic 100-day pancytopenia \cite{Birgens_BJH93},
cyclic (approximately 60 day) bicytopenia with Shapiro syndrome \cite{Roeker_CyclicBicyt_CRH2013},
and Polycyth{\ae}mia Vera \cite{Morley_AAM69} with approximately 28 day cycling.

%
%
%
%
%
%

CT typically involves oscillations of just the platelets~\cite{Foley_2009a}, though one case of multi-lineage CT has recently been reported \cite{Langlois_CCR2018}, while for
CN and PCML oscillations of all of the major blood cell groups are observed~\cite{Foley_2009a}.
This suggests that for CN and PCML the cycling in all cell lineages may be due to a dynamic destabilization at the stem cell level~\cite{Foley_2009a}. This destabilization occurs through different mechanisms in these two diseases with leukemic HSCs typically presenting a chromosome abnormality in PCML~\cite{Pujo_Menjouet_2005}. In CN a mutation in the ELANE gene leads to increased apoptosis in the neutrophil progenitor cells during mitosis~\cite{Dale_2002}, and the destabilization of the HSCs appears to be caused by a feedback mechanism from the neutrophil lineage.

Considerable variation in the oscillatory periods is observed within and between these disorders.
A Lomb periodogram~\cite{Langlois2017} is typically used to extract a periodic signature from the data,
but the data itself is never truly periodic.
There can be many reasons for this including data sampling, measurement error, intrinsic stochasticity of cell proliferation and differentiation, environmental variation, adaptation of the model parameters, or simply that the actual dynamics are not
periodic.

In Table~\ref{tab.model.par} we gave specific values of the model parameters from which we start our bifurcation analysis. Other authors use somewhat different values, or more correctly report ranges for the parameter values~\cite{Bernard_2003,Colijn_2005a}. Through inter-individual variability we should
expect that a single parameter set will not be suitable for all subjects.
However, as seen in Section~\ref{sec.bifurc.hsc} there are no bifurcations near to the stated homeostasis parameters. Hence, using other similar parameter values in the model will also lead to an asymptotically stable steady state.

To provoke a qualitative change in the dynamics of \eqref{Qprime} requires a
large change in the parameters. This situation was already envisioned by
Glass and Mackey \cite{Glass_Mackey_1979}
who coined the term dynamical disease to describe physiological systems where the control system itself is intact, but operating in a parameter range leading to abnormal dynamics.
With significant changes to one or more parameters we do observe non-trivial dynamics.
These dynamics only become of physiological, rather than mathematical, interest when they produce oscillations with characteristics similar to the reported diseases, and we do observe behaviour
reminiscent of CN, PCML and CT.

An increased apoptosis rate $\gamma$ during the cell cycle, as illustrated in
Figures~\ref{fig:GammaCont} and~\ref{fig:GammaOrb} results in stable oscillations in the HSCs of period between about 75 and 100 days for $\gamma\in[0.2278,0.24]$.
The shortest period orbit illustrated in  Figure~\ref{fig:GammaOrb}(i) is of interest.
This has a maximal value of $Q(t)$ greater than 70\% of $Q^h$,
and hence maximum differentiation rate $\kappa Q(t)$ to peripheral blood cell precursors
above 70\% of the homeostatic rate,
while the interval of severely reduced HSC numbers is relatively short (below 4 weeks).
Such cycling in the HSCs would naturally result in pancytopenia in a full model of the hematopoietic system. If the apoptosis rate $\gamma$ is increased slightly above $0.24$ longer periodic orbits result, but as seen from Figure~\ref{fig:GammaOrb}(ii) these have severely reduced HSC numbers for intervals of hundreds of days, which is much longer than the lifespan of circulating erythrocytes and which would induce a fatal anemia. Still higher values of $\gamma$ result in complete depletion of the HSCs with $Q=0$ becoming the globally attracting stable steady state.

In Figures~\ref{fig:KappaCont} and~\ref{KappaCont_Orbits_Bistab_kappa_plots}
we illustrate periodic dynamics of the HSCs for increased values of the differentiation rate $\kappa$. The stable periodic orbits in Figure~\ref{fig:KappaCont} exist when the rate constant $\kappa$ is $6$ or more times its homeostatic value, meaning that in this scenario the rate $\kappa Q(t)$ at which HSCs differentiate to precursors of circulating hematopoietic cells can be elevated compared to the homeostasis value, even when the number of HSCs $Q(t)$ is less than $Q^h$. While the periods observed in
Figure~\ref{fig:KappaCont} are too short for PCML, longer periods of
30-100 days consistent with PCML can be obtained by also increasing the cell cycle time $\tau$ as seen in Figure~\ref{fig2_HSC_Biftool1D_KappaDA}(i)
and  Figure~\ref{figSnakCont}(i)-inset.


The largest periods seen in Figure~\ref{fig:KappaCont}(ii), corresponding to the largest amplitude orbits on the main branch, and also the period-doubled orbits have period about 17 days, which is close to but a little shorter than typical periods for CN. Varying three parameters in
\eqref{Qprime} it is possible to find periodic orbits with periods typical of CN, for example
$(\kappa,\gamma,\tau)\approx(0.09,0.28,3)$ results in stable limit cycles with period between 19 and 21 days.
Stable large amplitude limit cycles are observed in Figure~\ref{fig2_HSC_Biftool1D_KappaDA}
with periods in the 20 to 40 day range typical of CT.

We observed numerous instances of bistability, which allows for the possibility that a
therapeutic intervention or some other outside affect on the
hematopoietic system could cause it to flip between different stable states.
This has been observed in practice, where for example
G-CSF can induce neutrophil oscillations with a period of about 7 to 15 days for neutropenic individuals~\cite{Haurie_1998}.

In Section~\ref{sec.longp.hsc} we explored a canard explosion.
The very long period orbits that we found
are likely not physiologically relevant, as they
include long time intervals during which the HSCs are severely depleted. During these intervals the production of peripheral blood cells would be so severely compromised, that a fatal cytopenia would likely result. Although we do not rule out the possibility that a canard explosion with other parameters
might lead to physiologically feasible long period orbits, the singular parameter $\epsilon$ suggests this is unlikely. We see from \eqref{eq:epspars} and \eqref{Qstarsing} that
$\kappa Q^*\sim\epsilon$, so in the parameter regime $0<\epsilon\ll 1$ where we might expect the canard to exist the differentiation of HSCs towards mature blood cell lines will be severely comprised.

The quasi-periodic and chaotic solutions observed in Section~\ref{sec.chaos.hsc} may be of more physiological relevance for two reasons. Firstly, although these solutions all have significantly reduced HSC concentrations compared to homeostasis, they are found in parameter regions where the differentiation rate $\kappa$ is significantly increased, so that the differentiation $\kappa Q(t)$ out of the HSC compartment is at or above the homeostatic rate when $Q(t)$ is close to a local maxima. These HSC dynamics would likely lead to episodic pancytopenia in a full model of the hematopoietic system, which could be an interesting topic for follow up study. A second reason why these dynamics are of physiological relevance is that they show the system generating non-constant non-periodic dynamics which is more akin to what seen in real data than the purely periodic solutions that we investigated earlier.

In this section we highlighted some of the solutions that we observed with periods in ranges characteristic of dynamical diseases. The two-parameter continuations of Section~\ref{sec.2p}
could be used as a starting point for an extended study to find additional parameter regions with periodic solutions commensurate with dynamical diseases. Although it would be tractable to do that
for the HSC model \eqref{Qprime}, such a study would be more interesting in a
model of the hematopoietic system that incorporates multiple mature cell lines. We have clearly shown that our HSC model can demonstrate the oscillatory dynamics characteristic of dynamical diseases, without the need for any feedback loops from more mature cell lines.
However, many of the solutions with interesting dynamics are associated with an increased differentiation rate $\kappa$. It remains an open question in particular dynamical diseases whether the differentiation rate is actually raised, and if so whether this is intrinsic to the disease-state HSC dynamics, or caused
by feedback from the peripheral blood cell dynamics.


\section{Discussion and Conclusions}\label{sec.conc}

We set out to show that the HSC model~\eqref{Qprime} could generate limit cycles of periods typical in dynamical diseases, simply by changing some of the parameter values in the model. Long period orbits had previously been observed
by varying $s$~\cite{Pujo_Menjouet_2006,Pujo_Menjouet_2005,Pujo_Menjouet_2004}.
We varied the parameters $\gamma$, $\kappa$ and $\tau$ and
found periodic orbits of periods from about one week up to 9 years, encompassing the
19-21 days typical of CN, 20-40 days of CT and the 30-100 days of PCML.
Whereas the model~\eqref{Qprime} treats the HSCs as a single homogeneous population, more recent mathematical models couple multiple copies of~\eqref{Qprime} together~\cite{Adimy_2006a,Qu_2010} to represent the different maturity levels of HSCs, and should be able to generate similar dynamics.

We also observed a plethora of more exotic dynamics including mixed mode oscillations, period-doubling cascades and chaotic solutions.
In Section~\ref{sec.torus} we showed that the DDE~\eqref{Qprime} admits stable torus solutions.
Elsewhere, in Section~\ref{sec.longp.hsc} we studied a putative canard explosion, identified the singular variable, and constructed an approximation to the slow manifold and nearby dynamics. We showed that the
local stability of the slow manifold changes very close to the point where the stability of the critical manifold changes. Our analysis of the canard explosion is incomplete. Established analysis and constructions rely on separating the slow and fast variables \cite{Wech13}. In contrast, equation~\eqref{Qprime} is scalar, and does not have a simple natural separation into fast and slow subsystems. We believe this to be the first demonstration of canard-like behaviour in a scalar system, and a full analysis will require an extension to current theory. In the current work we present a detailed numerical investigation of the phenomenon, with the hope that it will intrigue the theoreticians to complete the analysis.

Equation~\eqref{Qprime} clearly displays mixed mode oscillations
(see the time plots in Figures~\ref{fig:GammaOrb},~\ref{fig:TauOrb},~\ref{fig:GammaCont_Canard}).
Such dynamics are usually associated with slow-fast systems and coupled oscillators, and it is rather curious to see these phenomena in the scalar DDE~\eqref{Qprime}. It is well-known that such dynamics are possible when there are multiple delays, and in the case of two state-dependent delays no other nonlinearity is required other than the state-dependency of the delays~\cite{Calleja_2017}. In that case it seems that essentially the two delay terms interact as if they are coupled oscillators. However, equation~\eqref{Qprime} is scalar with only one delay, and has no Hopf-Hopf bifurcations. Equation~\eqref{Qprime} is in the general class of problems
$$\dot{u}(t)=-u(t) - \alpha h(u(t))+Ah(u(t-\tau),$$
where $h(u)$ is a unimodal function. Problems of this form, have been studied in the case $\alpha=0$, but we are not aware of systemic theoretical studies of the more general case with $\alpha\ne0$. It seems likely to us that the dynamics reminiscent of coupled oscillators are generated by an interaction between the two instances of the nonlinearity evaluated at the current time $t$ and the delayed time $t-\tau$.

We found many examples of bistability in \eqref{Qprime}. These include bistability between pairs of periodic orbits (Figures~\ref{fig:KappaCont},~\ref{fig2_HSC_Biftool1D_KappaDA}),
a periodic orbit and a stable steady state (Figures~\ref{fig:KappaCont},~\ref{fig:GammaCont},~\ref{fig:TauCont},~\ref{fig2_HSC_Biftool1D_KappaDA}),
a periodic orbit and a torus (Figures~\ref{fig2_HSC_Biftool1D_KappaDA},~\ref{fig1_DdeStemKappaQTori01}),
as well as bistability between chaotic and nonchaotic solutions (Figures~\ref{fig.chaos0}).
Bistability of periodic orbits is caused by pairs of fold bifurcations of limit cycle which originate in cusp bifurcations (seen in
Figures~\ref{fig:KappaTau_2D_Cont},~\ref{fig:GammaKappa_2D_Cont}).

We found both subcritical and supercritical Hopf bifurcations and
bistability between a stable limit cycle and a steady state is associated with
the Bautin or generalised Hopf bifurcation (denoted in Figures~\ref{fig:KappaTau_2D_Cont},~\ref{fig:GammaTau_2D_Cont}) where the criticality of the Hopf bifurcation changes. A curve of fold bifurcations of limit cycles emerges from this point which
results in the interval of bistability seen in the one-parameter continuations between the fold and the subcritical Hopf bifurcation.
Few previous studies have been sufficiently detailed to detect the criticality of the Hopf bifurcations, but those that were only found supercritical Hopf bifurcations~\cite{Bernard_2003,Bernard_2004,Milton_1989}, though
Bernard \textit{et al}~\cite{Bernard_2003} did find Hopf bifurcations which were close to a criticality change.

Mathematical studies of differential equations that model hematopoiesis frequently focus on existence
and stability of a nontrivial solution. Once a Hopf bifurcation is found, the steady state becomes unstable, and secondary bifurcations to more complex dynamical structures are often not pursued. On the other hand, peripheral blood samples are often only taken for a few days at a time during a hospital stay, and are otherwise not taken, or taken at widely and irregular spaced intervals. The data, even when well sampled, frequently appears noisy, and it is unheard of to see solutions that are exactly periodic. Often, a periodic signature is only revealed by a frequency test, such as the Lomb periodogram. In this context, the bistable, long period, quasi-periodic, and transient and persistent chaotic dynamics that we find are very interesting. The same individual can present very different looking dynamics during different sampling intervals. It might be that the dynamics are actually periodic, but the period is so long that
different parts of the periodic solution are revealed by different sampling intervals. Another possibility is that the dynamics are actually chaotic (but not random), and different parts of the chaotic attractor are revealed at different times. As we saw in Figure~\ref{fig.chaos1} a time series of chaotic dynamics can appear to be surprisingly regular, while in Figure~\ref{fig.chaos0} where the time series of the dynamics was clearly not regular, the system actually spends significant time near to the attractor of the first example.
In a period doubling cascade to chaos the initial seed orbit and its low order period-doublings continue
to exist after they become unstable, and can be expected to have some organising influence on the structure of the dynamics. Thus it is natural to
expect there to be some periodic signal contained in the time series, even of a chaotic solution, and
it is very unlikely that sufficient blood measurements would be taken from a single subject to discern genuinely chaotic dynamics. A widely used strategy  for determining perturbed parameters associated with dynamical diseases
is to try to find parameters which generate a periodic solution which is closest to the data~\cite{Langlois2017}.
Given the difficulty in distinguishing between chaotic and periodic dynamics, and the ability of the mathematical models to generate both, this strategy may not be optimal, and we should consider that the disease dynamics may not generate a simple periodic orbit, but that  there may be multiple bifurcations between the homeostasis and diseased states, leading to more complex dynamics.
%
\section*{Acknowledgments}
We are grateful to Mike Mackey for useful discussions and feedback on a draft of this work. We also thank John Mitry for our discussions on canards.
We are grateful to Dimitri Breda for sharing his code for the computation of Lyapunov exponents in DDEs.  Finally, we wish to thank the referees and editor for their many constructive suggestions.

DCS was supported by National Council for Scientific and Technological
Development of Brazil (CNPq) postdoctoral fellowship 201105/2014-4.
ARH is supported by a Discovery Grant from the Natural Science and Engineering Research Council (NSERC), Canada.
%
\section*{References}

\begin{thebibliography}{47}
\providecommand{\natexlab}[1]{#1}
\providecommand{\url}[1]{\texttt{#1}}
\expandafter\ifx\csname urlstyle\endcsname\relax
  \providecommand{\doi}[1]{doi: #1}\else
  \providecommand{\doi}{doi: \begingroup \urlstyle{rm}\Url}\fi

\bibitem{Adimy_2006c}
{\sc M.~Adimy, F.~Crauste, and A.~El~Abdllaoui}, {\em Asymptotic behavior of a
  discrete maturity structured system of hematopoietic stem cell dynamics with
  several delays}, Math. Model. Nat. Phenom., 1 (2006), pp.~1--22,
  \url{https://doi.org/10.1051/mmnp:2008001}.

\bibitem{Adimy_2006a}
{\sc M.~Adimy, F.~Crauste, and S.~Ruan}, {\em Periodic oscillations in
  leukopoiesis models with two delays}, J. Theor. Biol., 242 (2006),
  pp.~288--299, \url{https://doi.org/10.1016/j.jtbi.2006.02.020}.

\bibitem{Andersen_2001}
{\sc L.~K. Andersen and M.~C. Mackey}, {\em Resonance in periodic chemotherapy:
  A case study of acute myelogenous leukemia}, J. Theor. Biol., 209 (2001),
  pp.~113--130, \url{https://doi.org/10.1006/jtbi.2000.2255}.

\bibitem{Benoit1981}
{\sc E.~Beno\^it, J.~L. Callot, F.~Diener, and M.~Diener}, {\em Chasse au
  canard}, Collect. Math., 32 (1981), pp.~37--119.

\bibitem{Bernard_2003}
{\sc S.~Bernard, J.~B\'{e}lair, and M.~C. Mackey}, {\em Oscillations in
  cyclical neutropenia: {N}ew evidence based on mathematical modeling}, J.
  Theor. Biol., 223 (2003), pp.~283--298,
  \url{https://doi.org/10.1016/S0022-5193(03)00090-0}.

\bibitem{Bernard_2004}
{\sc S.~Bernard, J.~B\'{e}lair, and M.~C. Mackey}, {\em Bifurcations in a
  white-blood-cell production model}, C. R. Biol., 327 (2004), pp.~201--210,
  \url{https://doi.org/10.1016/j.crvi.2003.05.005}.

\bibitem{Birgens_BJH93}
{\sc H.~S. Birgens and H.~Karle}, {\em Reversible adult-onset cyclic
  haematopoiesis with a cycle length of 100 days}, Br. J. Haematol., 83 (1993),
  pp.~181--186, \url{https://doi.org/10.1111/j.1365-2141.1993.tb08269.x}.

\bibitem{Breda_2015}
{\sc D.~Breda, S.~Maset, and R.~Vermiglio}, {\em Stability of Linear Delay
  Differential Equations. A Numerical Approach with MATLAB}, Springer, 2015.

\bibitem{Breda_2014}
{\sc D.~Breda and E.~S. Van~Vleck}, {\em Approximating {Lyapunov} exponents and
  {Sacker--Sell} spectrum for retarded functional differential equations},
  Numer. Math., 126 (2014), pp.~225--257,
  \url{https://doi.org/10.1007/s00211-013-0565-1}.

\bibitem{Broer_1996}
{\sc H.~W. Broer, G.~B. Huitema, and M.~B. Sevryuk}, {\em Quasi-Periodic
  Motions in Families of Dynamical Systems. Order amidst Chaos}, vol.~1645 of
  Lecture Notes in Mathematics, Springer, 1996.

\bibitem{Burns_1970}
{\sc F.~J. Burns and J.~F. Tannock}, {\em On the existence of a
  $\mathrm{G_{0}}$-phase in the cell cycle}, Cell Tissue Kinet., 3 (1970),
  pp.~321--334, \url{https://doi.org/10.1111/j.1365-2184.1970.tb00340.x}.

\bibitem{Calleja_2017}
{\sc R.~C. Calleja, A.~R. Humphries, and B.~Krauskopf}, {\em Resonance
  phenomena in a scalar delay differential equation with two state-dependent
  delays}, SIAM J. Appl. Dyn. Syst., 16 (2017), pp.~1474--1513,
  \url{https://doi.org/10.1137/16M1087655}.

\bibitem{CSE09}
{\sc S.~A. Campbell, E.~Stone, and T.~Erneux}, {\em Delay induced canards in a
  model of high speed machining}, Dyn. Syst., 24 (2009), pp.~373--392,
  \url{https://doi.org/10.1080/14689360902852547}.

\bibitem{Colijn2006}
{\sc C.~Colijn, A.~C. Fowler, and M.~C. Mackey}, {\em High frequency spikes in
  long period blood cell oscillations}, J. Math. Biol., 53 (2006),
  pp.~499--519, \url{https://doi.org/10.1007/s00285-006-0027-9}.

\bibitem{Colijn_2005a}
{\sc C.~Colijn and M.~C. Mackey}, {\em A mathematical model of hematopoiesis:
  {I}. {P}eriodic chronic mylogenous leukemia}, J. Theor. Biol., 237 (2005),
  pp.~117--132, \url{https://doi.org/10.1016/j.jtbi.2005.03.033}.

\bibitem{Colijn_2005b}
{\sc C.~Colijn and M.~C. Mackey}, {\em A mathematical model of hematopoiesis:
  {II}. {C}yclical neutropenia}, J. Theor. Biol., 237 (2005), pp.~133--146,
  \url{https://doi.org/10.1016/j.jtbi.2005.03.034}.

\bibitem{Corless1996}
{\sc R.~M. Corless, G.~H. Gonnet, D.~E.~G. Hare, D.~J. Jeffrey, and D.~E.
  Knuth}, {\em On the {Lambert} {W} function}, Adv. Comput. Math., 5 (1996),
  pp.~329--359, \url{https://doi.org/10.1007/BF02124750}.

\bibitem{Craig_2016}
{\sc M.~Craig, A.~R. Humphries, and M.~C. Mackey}, {\em A mathematical model of
  granulopoiesis incorporating the negative feedback dynamics and kinetics of
  $\mathrm{G\mbox{-}CSF}$/neutrophil binding and internalization}, Bull. Math.
  Biol., 78 (2016), pp.~2304--2357,
  \url{https://doi.org/10.1007/s11538-016-0179-8}.

\bibitem{Dale_2002}
{\sc D.~C. Dale, A.~A. Bolyard, and A.~Aprikyan}, {\em Cyclic neutropenia},
  Semin. Hematol., 39 (2002), pp.~89--94,
  \url{https://doi.org/10.1053/shem.2002.31917}.

\bibitem{Dale_2016}
{\sc D.~C. Dale and K.~Welte}, {\em Neutropenia and neutrophilia}, in Williams
  {H}ematology, K.~Kaushansky, M.~Lichtman, J.~Prchal, et~al., eds.,
  Mc{G}raw-{H}ill, 9th~ed., 2016.

\bibitem{Engelborghs_2002}
{\sc K.~Engelborghs, T.~Luzyanina, and D.~Roose}, {\em Numerical bifurcation
  analysis of delay differential equations using {DDE-BIFTOOL}}, ACM Trans.
  Math. Soft., 28 (2002), pp.~1--21,
  \url{https://doi.org/10.1145/513001.513002}.

\bibitem{Fenichel_1979}
{\sc N.~Fenichel}, {\em Geometric singular perturbation theory for ordinary
  differential equations}, J. Diff. Eqns, 31 (1979), pp.~53--98,
  \url{https://doi.org/https://doi.org/10.1016/0022-0396(79)90152-9}.

\bibitem{Foley_2009a}
{\sc C.~Foley and M.~C. Mackey}, {\em Dynamic hematological disease: {A}
  review}, J. Math. Biol., 58 (2009), pp.~285--322,
  \url{https://doi.org/10.1007/s00285-008-0165-3}.

\bibitem{Fowler_Mackey_2002}
{\sc A.~C. Fowler and M.~C. Mackey}, {\em Relaxation oscillations in a class of
  delay differential equations}, SIAM J. Appl. Math., 63 (2002), pp.~299--323,
  \url{https://doi.org/10.1137/S0036139901393512}.

\bibitem{Glass_Mackey_1979}
{\sc L.~Glass and M.~C. Mackey}, {\em Pathological conditions resulting from
  instabilities in physiological control systems}, Ann. N. Y. Acad. Sci., 316
  (1979), pp.~214--235,
  \url{https://doi.org/10.1111/j.1749-6632.1979.tb29471.x}.

\bibitem{Hale_1993}
{\sc J.~K. Hale and S.~M. Verduyn~Lunel}, {\em Introduction to Functional
  Differential Equations}, vol.~99 of Applied Mathematical Sciences, Springer,
  1993.

\bibitem{Haurie_1998}
{\sc C.~Haurie, D.~C. Dale, and M.~C. Mackey}, {\em Cyclical neutropenia and
  other periodic hematological disorders: A review of mechanisms and
  mathematical models}, Blood, 92 (1998), pp.~2629--2640.

\bibitem{Hayes_1950}
{\sc N.~D. Hayes}, {\em Roots of the transcendental equation associated with a
  certain difference-differential equation}, J. London Math. Soc., s1-25
  (1950), pp.~226--232, \url{https://doi.org/10.1112/jlms/s1-25.3.226}.

\bibitem{Insperger_2011}
{\sc T.~Insperger and G.~St\'ep\'an}, {\em Semi-Discretization for Time-Delay
  Systems}, Springer, 2011.

\bibitem{Izhikevich_2007}
{\sc E.~Izhikevich}, {\em Dynamical Systems in Neuroscience: The Geometry of
  Excitability and Bursting}, The MIT Press, 2007.

\bibitem{KaplanYorke1979}
{\sc J.~L. Kaplan and J.~A. Yorke}, {\em Chaotic behavior of multidimensional
  difference equations}, in Functional Differential Equations and Approximation
  of Fixed Points: Proceedings, Bonn, July 1978, H.-O. Peitgen and H.-O.
  Walther, eds., Springer, 1979, pp.~204--227,
  \url{https://doi.org/10.1007/BFb0064319}.

\bibitem{Kaushansky_2016}
{\sc K.~Kaushansky}, {\em Hematopoietic stem cells, progenitors, and
  cytokines.}, in Williams {H}ematology, K.~Kaushansky, M.~A. Lichtman, J.~T.
  Prchal, et~al., eds., Mc{G}raw-{H}ill, 9th~ed., 2016.

\bibitem{Kolmanovskii_1999}
{\sc V.~Kolmanovskii and A.~Myshkis}, {\em Introduction to the theory and
  applications of functional differential equations}, Mathematics and Its
  Applications, Kluwer Academic Publishers, 1999.

\bibitem{KrupaTouboul2016}
{\sc M.~Krupa and J.~D. Touboul}, {\em Canard explosion in delay differential
  equations}, J. Dyn. Differ. Equ., 28 (2016), pp.~471--491,
  \url{https://doi.org/10.1007/s10884-015-9478-2}.

\bibitem{Kuznetsov_2004}
{\sc Y.~A. Kuznetsov}, {\em Elements of applied bifurcation theory}, vol.~112
  of Applied Mathematical Sciences, Springer, 3rd~ed., 2004.

\bibitem{Langlois_CCR2018}
{\sc G.~P. Langlois, D.~M. Arnold, J.~Potts, B.~Leber, D.~C. Dale, and M.~C.
  Mackey}, {\em Cyclic thrombocytopenia with statistically significant
  neutrophil oscillations}, Clin. Case. Rep., 6 (2018), pp.~1347--1352,
  \url{https://doi.org/10.1002/ccr3.1611}.

\bibitem{Langlois2017}
{\sc G.~P. Langlois, M.~Craig, A.~R. Humphries, et~al.}, {\em Normal and
  pathological dynamics of platelets in humans}, J. Math. Biol., 75 (2017),
  pp.~1411--1462, \url{https://doi.org/10.1007/s00285-017-1125-6}.

\bibitem{Lorenz_1963}
{\sc E.~N. Lorenz}, {\em Deterministic nonperiodic flow}, J. Atmospheric Sci.,
  20 (1963), pp.~130--141.

\bibitem{Mackey_1978}
{\sc M.~C. Mackey}, {\em Unified hypothesis for the origin of aplastic anemia
  and periodic haematopoiesis}, Blood, 51 (1978), pp.~941--956.

\bibitem{Pujo_Menjouet_2006}
{\sc M.~C. Mackey, C.~Ou, L.~Pujo-Menjouet, and J.~Wu}, {\em Periodic
  oscillations of blood cell populations in chronic myelogenous leukemia}, SIAM
  J. Math. Anal., 38 (2006), pp.~166--187,
  \url{https://doi.org/10.1137/04061578X}.

\bibitem{Mackey_1994}
{\sc M.~C. Mackey and R.~Rudnicki}, {\em Global stability in a delayed partial
  differential equation describing cellular replication}, J. Math. Biol., 33
  (1994), pp.~89--109, \url{https://doi.org/10.1007/BF00160175}.

\bibitem{JMPRN86}
{\sc J.~Mallet-Paret and R.~D. Nussbaum}, {\em Global continuation and
  asymptotic behavior for periodic solutions of a differential-delay equation},
  Annali di Mat. Pura ed Appl., 145 (1986), pp.~33--128,
  \url{https://doi.org/10.1007/BF01790539}.

\bibitem{JMPRNIII}
{\sc J.~Mallet-Paret and R.~D. Nussbaum}, {\em Boundary layer phenomena for
  differential-delay equations with state-dependent time lags: {III}}, Discrete
  Contin. Dyn. Syst. Ser. A, 189 (2003), pp.~640--692,
  \url{https://doi.org/10.1016/S0022-0396(02)00088-8}.

\bibitem{JMPRN11}
{\sc J.~Mallet-Paret and R.~D. Nussbaum}, {\em Superstability and rigorous
  asymptotics in singularly perturbed state-dependent delay-differetnial
  equations}, J. Diff. Eqns., 250 (2011), pp.~4037--4084,
  \url{https://doi.org/10.1016/j.jde.2010.10.024}.

\bibitem{Matlab}
{\sc Mathworks}, {\em MATLAB 2015b}, Mathworks, Natick, Massachusetts, 2015.

\bibitem{Longtin1998}
{\sc B.~Mensour and A.~Longtin}, {\em Power spectra and dynamical invariants
  for delay-differential and difference equations}, Physica D, 113 (1998),
  pp.~1--25, \url{https://doi.org/10.1016/S0167-2789(97)00185-1}.

\bibitem{Milton_1989}
{\sc J.~G. Milton and M.~C. Mackey}, {\em Periodic haematological diseases:
  mystical entities or dynamical disorders?}, J. Roy. Coll. Phys. (Lond.), 23
  (1989), pp.~236--241.

\bibitem{Morley_AAM69}
{\sc A.~Morley}, {\em Blood-cell cycles in polycythæmia vera}, Australasian
  Annals of Medicine, 18 (1969), pp.~124--126,
  \url{https://doi.org/10.1111/imj.1969.18.2.124}.

\bibitem{Pujo_Menjouet_2005}
{\sc L.~Pujo-Menjouet, S.~Bernard, and M.~C. Mackey}, {\em Long period
  oscillations in a $\mathrm{G_{0}}$ model of hematopoietic stem cells}, SIAM
  J. Appl. Dyn. Syst., 4 (2005), pp.~312--332,
  \url{https://doi.org/10.1137/030600473}.

\bibitem{Pujo_Menjouet_2004}
{\sc L.~Pujo-Menjouet and M.~C. Mackey}, {\em Contribution to the study of
  periodic chronic myelogenous leukemia}, C. R. Biol., 327 (2004),
  pp.~235--244, \url{https://doi.org/10.1016/j.crvi.2003.05.004}.

\bibitem{Qu_2010}
{\sc Y.~Qu, J.~Wei, and S.~Ruan}, {\em Stability and bifurcation analysis in
  hematopoietic stem cell dynamics with multiple delays}, Physica D, 239
  (2010), pp.~2011--2024, \url{https://doi.org/10.1016/j.physd.2010.07.013}.

\bibitem{Roeker_CyclicBicyt_CRH2013}
{\sc L.~E. Roeker, V.~Gupta, W.~I. Gonsalves, A.~P. Wolanskyj, and N.~Gangat},
  {\em Cyclic bicytopenia in a patient with shapiro syndrome}, Case Rep.
  Hematol., 2013 (2013), p.~Article ID 231713,
  \url{https://doi.org/10.1155/2013/231713}.

\bibitem{DDEBiftool15}
{\sc J.~Sieber, K.~Engelborghs, T.~Luzyanina, G.~Samaey, and D.~Roose}, {\em
  DDE-BIFTOOL Manual - Bifurcation analysis of delay differential equations},
  2015, \url{https://arxiv.org/abs/1406.7144}.
\newblock \ Eprint arXiv:1406.7144 [math.DS].

\bibitem{Smith_2010}
{\sc H.~Smith}, {\em An Introduction to Delay Differential Equations with
  Applications to the Life Sciences}, Springer, 2010.

\bibitem{Wech13}
{\sc M.~Wechselberger, J.~Mitry, and J.~Rinzel}, {\em Canard theory and
  excitability}, in Nonautonomous Dynamical Systems in the Life Sciences,
  P.~Kloeden and C.~P{\"o}tzsche, eds., Springer, 2013.

\bibitem{ZhugeMackeyLeiJTB2019}
{\sc C.~Zhuge, M.~C. Mackey, and J.~Lei}, {\em Origins of oscillation patterns
  in cyclical thrombocytopenia}, J. Theor. Biol., 462 (2019), pp.~432 -- 445,
  \url{https://doi.org/https://doi.org/10.1016/j.jtbi.2018.11.024}.

\end{thebibliography}

\end{document}